\documentclass[11pt,oneside]{amsart}

\usepackage{fullpage}

\usepackage{amsthm,amsfonts,amssymb,amsmath,amsxtra,mathtools}
\usepackage[dvipsnames]{xcolor}
\usepackage[all]{xy}
\SelectTips{cm}{}
\usepackage{tikz}
\usepackage{verbatim}
\usepackage{todonotes}
\usepackage{mathrsfs}
\usepackage{booktabs,makecell}
\usepackage{diagbox}
\usepackage{epigraph}
\usepackage{amsthm}
\usepackage{bm}
\usepackage{tikz-cd}
\usepackage{amsmath}
\usepackage{amssymb}
\usepackage{relsize}
\usepackage{amsfonts}
\usepackage{amsxtra}
\usepackage{braket}
\usepackage{lmodern}
\usepackage{mathrsfs}
\usepackage{array}
\usepackage{xcolor}
\usepackage{mathtools}
\definecolor{citecol}{rgb}{0.07,0.07,0.05}
\definecolor{urlcol}{rgb}{0.06,0.04,0.09}
\definecolor{linkcol}{rgb}{0.01,0.03,0.08}
\usepackage[colorlinks,linkcolor=linkcol,urlcolor=urlcol,citecolor=citecol,pagebackref,breaklinks]{hyperref}
\usepackage{appendix} 
\usepackage[lite,abbrev,msc-links,alphabetic]{amsrefs}
\usepackage{indentfirst}
\usepackage{mathtools}
\usepackage[all]{xy}
\usepackage{yhmath}
\allowdisplaybreaks

\RequirePackage{xspace}
\RequirePackage{etoolbox}
\RequirePackage{varwidth}
\RequirePackage{enumitem}
\RequirePackage{tensor}
\RequirePackage{mathtools}
\RequirePackage{longtable}
\RequirePackage{multirow}
\RequirePackage{makecell}
\RequirePackage{bigints}

\usepackage[pagebackref]{hyperref}  
\hypersetup{  
	colorlinks=true,
	linkcolor=blue,
	citecolor=MidnightBlue,
	urlcolor=magenta,
}

\newcommand{\sH}{\ensuremath{\mathscr{H}}\xspace}
\newcommand{\sV}{\ensuremath{\mathscr{V}}\xspace}

\DeclareMathOperator{\Hom}{Hom}
\DeclareMathOperator{\End}{End}

\DeclareMathOperator{\Int}{\ensuremath{\mathrm{Int}}\xspace}

\DeclareMathOperator{\Lie}{Lie}

\newcommand{\val}{{\mathrm{val}}}

\newcommand{\U}{\mathrm{U}}

\newcommand{\SO}{\mathrm{SO}}
\DeclareMathOperator{\vol}{vol}

\newcommand{\isoarrow}{%
	\ifbool{@display}{\overset{\sim}{\longrightarrow}}{\xrightarrow\sim}%
}

\newcommand{\Fb}{{\breve F}}

\newcommand{\OFb}{{O_{\breve F}}}
\newcommand{\Herm}{\mathrm{Herm}}

\newcommand{\rd}{\mathrm{d}}

\newcommand{\Den}{\mathrm{Den}}

\newcommand{\pDen}{\partial\mathrm{Den}}

\newcommand{\pPden}{\partial \mathrm{Pden}}

\theoremstyle{plain}
\newtheorem{theorem}{Theorem}[section]
\newtheorem{lemma}[theorem]{Lemma}
\newtheorem{proposition}[theorem]{Proposition}

\newtheorem{corollary}[theorem]{Corollary}
\newtheorem{conjecture}[theorem]{Conjecture}

\theoremstyle{definition}
\newtheorem{definition}[theorem]{Definition}
\newtheorem{remark}[theorem]{Remark}

\theoremstyle{remark}

\newcommand{\BB}{{\mathbb B}}
\newcommand{\BC}{{\mathbb C}}

\newcommand{\BE}{{\mathbb E}}
\newcommand{\BF}{{\mathbb F}}

\newcommand{\BL}{{\mathbb L}}

\newcommand{\BP}{{\mathbb P}}
\newcommand{\BQ}{{\mathbb Q}}

\newcommand{\BV}{{\mathbb V}}

\newcommand{\BX}{{\mathbb X}}

\newcommand{\BZ}{{\mathbb Z}}

\newcommand{\CA}{{\mathcal A}}
\newcommand{\CB}{{\mathcal B}}
\newcommand{\CC}{{\mathcal C}}

\newcommand{\CF}{{\mathcal F}}
\newcommand{\CG}{{\mathcal G}}

\newcommand{\CK}{{\mathcal K}}

\newcommand{\CN}{{\mathcal N}}
\newcommand{\CO}{{\mathcal O}}

\newcommand{\CR}{{\mathcal R}}
\newcommand{\CS}{{\mathcal S}}

\newcommand{\CY}{{\mathcal Y}}
\newcommand{\CZ}{{\mathcal Z}}

\newcommand{\FC}{{\mathfrak C}}

\newcommand{\FL}{{\mathfrak L}}
\newcommand{\FM}{{\mathfrak M}}

\newcommand{\FX}{{\mathfrak X}}

\newcommand{\ScF}{{\mathscr F}}

\newcommand{\ScH}{{\mathscr H}}

\newcommand{\ScV}{{\mathscr V}}

\newcommand{\alphad}{\mathrm{Den}}

\DeclareFontFamily{U}{matha}{\hyphenchar\font45}
\DeclareFontShape{U}{matha}{m}{n}{
	<5> <6> <7> <8> <9> <10> gen * matha
	<10.95> matha10 <12> <14.4> <17.28> <20.74> <24.88> matha12
}{}
\DeclareSymbolFont{matha}{U}{matha}{m}{n}
\DeclareFontFamily{U}{mathx}{\hyphenchar\font45}
\DeclareFontShape{U}{mathx}{m}{n}{
	<5> <6> <7> <8> <9> <10>
	<10.95> <12> <14.4> <17.28> <20.74> <24.88>
	mathx10
}{}
\DeclareSymbolFont{mathx}{U}{mathx}{m}{n}

\DeclareMathSymbol{\obot}         {2}{matha}{"6B}

\newcommand{\bX}{\mathbb{X}}

\newcommand{\bV}{\mathbb{V}}

\newcommand{\bW}{\mathbb{W}}

\newcommand{\Fil}{\mathrm{Fil}}

\newcommand{\bL}{\mathbb{L}}

\newcommand{\Q}{\mathbb{Q}}
\newcommand{\Z}{\mathbb{Z}}

\newcommand{\spa}{\mathrm{Span}}

\newcommand{\Oo}{\mathcal{O}}

\newcommand{\cZ}{\mathcal{Z}}

\newcommand{\cE}{\mathcal{E}}

\newcommand{\cY}{\mathcal{Y}}

\newcommand{\cN}{\mathcal{N}}

\newcommand{\Spf}{\mathrm{Spf}\, }

\newcommand{\q}{q}
\newcommand{\pden}{\partial \mathrm{Den}}

\newcommand{\diag}{\mathrm{Diag}}

\newcommand{\Pden}{\mathrm{Pden}}
\newcommand{\ppden}{\partial \mathrm{Pden}}
\newcommand{\den}{\mathrm{Den}}

\newcommand{\m}{\mathrm{m}}

\newcommand{\cO}{\mathcal{O}}
\newcommand{\rD}{\mathrm{D}}

\numberwithin{equation}{section}

\title{On the Kudla-Rapoport conjecture for unitary Shimura varieties with maximal parahoric level structure at unramified primes}

\author[Sungyoon Cho]{Sungyoon Cho}
\address{Department of Mathematics, POSTECH, 77, Cheongam-ro, Nam-gu, Pohang-si, Gyeongsangbuk-do, Korea}
\email{sungyooncho@postech.ac.kr}

\author[Qiao He]{Qiao He}
\address{Department of Mathematics, 
	Columbia University,
	2990 Broadway,
	New York, NY 10027, USA}
\email{qh2275@columbia.edu}

\author[Zhiyu Zhang]{Zhiyu Zhang}
\address{Department of Mathematics, Stanford University,  450 Jane Stanford Way Building 380, Stanford, CA 94305}
\email{zyuzhang@stanford.edu}

\date{\today}
\begin{document}
	\begin{abstract}
		In this article, we prove that a version of Tate conjectures for certain Deligne-Lusztig varieties implies the Kudla-Rapoport conjecture for unitary Shimura varieties with maximal parahoric level at unramified primes. Furthermore, we prove that the Kudla-Rapoport conjecture holds unconditionally for several new cases in any dimension.
	\end{abstract}

	\maketitle

	\tableofcontents

	\section{Introduction}

	\subsection{Introduction}

	\subsubsection{Background}
	The classical \emph{Siegel--Weil formula} (\cite{Sie35,Siegel1951,Weil1965}) relates   special \emph{values} of certain Eisenstein series as theta functions, which are generating series of representation numbers of quadratic forms. Later on, Kudla (\cite{Kudla97, Kudla2004}) proposed an influential program and introduced analogues of theta series in arithmetic geometry. One of the goals of the program is to prove the so-called \emph{arithmetic Siegel--Weil formula} relating the \emph{central derivative} of certain Eisenstein series with a certain arithmetic analogue of theta functions, which is a generating series of arithmetic intersection numbers of $n$ special divisors on Shimura varieties associated to $\SO(n-1,2)$ or $\U(n-1,1)$.
	
	For  $\U(n-1,1)$-Shimura varieties,  Kudla and Rapoport (\cite{KR1}) formulated a conjectural  \emph{local arithmetic Siegel--Weil formula} at an \emph{unramified} place with hyperspecial level, now known as the  \emph{Kudla--Rapoport conjecture}. As a local analogue of the arithmetic Siegel-Weil formula, it relates the \emph{central derivative} of local densities of hermitian forms with the arithmetic intersection number of special cycles on unitary Rapoport--Zink spaces. Now this conjectural identity is also known as the Kudla-Rapoport conjecture and was recently proved by Li and Zhang in \cite{LZ}. We refer the readers to the introduction of \cite{LZ} for more backgrounds and related results. 
	
	One of the distinguished features of the hyperspecial case \cite{KR1} is that the corresponding Rapoport-Zink space has good reduction. Accordingly, the analytic side has a clear formulation. A natural and important question is to formulate and prove analogues of the Kudla--Rapoport conjecture when  the level structure is non-trivial, where many unexpected new phenomenons occur.
	
	At a ramified place,
	there are two well-studied unitary Rapoport--Zink spaces with different level structures. One of them is the \emph{exotic smooth model} which has good reduction, and the other one is the \emph{Kr\"amer model} which has bad (semistable) reduction. The analogue of Kudla--Rapoport conjecture for the even dimensional exotic smooth model was formulated and proved by Li and Liu in \cite{LL2} using a strategy similar to \cite{LZ}. For the Kr\"amer model, corresponding to the bad reduction in this case, the analytic side is more involved. In fact, even the formulation of the conjecture is not clear and needs to be modified.  This phenomenon in the presence of bad reduction was first discovered by Kudla and Rapoport in \cite{KRshimuracurve} via explicit computation in their study of the Drinfeld $p$-adic half plane.  A similar computation was also done in \cite{San3,HSY} for unitary special cycles on the Kr\"amer models of the Drinfeld $p$-adic half plane. The Kudla--Rapoport conjecture for Kr\"amer models, in general, was formulated in \cite{HSY3} with conceptual formulation for the modification and proved in \cite{HLSY}.
	
	The present paper focus on Kudla--Rapoport conjectures with maximal parahoric level structures at an unramified prime, where more new phenomenons occur. If the level structure is almost self-dual, a Kudla-Rapoport type formula was obtained in \cite{San3} when $n=2$ by explicit computation and established in general in \cite{LZ} by relating it with the hyperspecial case.   At an unramified prime with general maximal parahoric level structure,  such formulation was first given in \cite{Cho} in terms of weighted local density taking advantage of a duality between two Rapoport-Zink spaces (see \S \ref{sec: weighted conj}). In this paper, we also give another formulation in the spirit of \cite{HSY3,HLSY} when the intersection number involved is purely contributed by $\cZ$-cycles.

	Assuming a version of Tate conjectures for certain Deligne-Lusztig varieties, the present paper settles this conjecture for any $n$ and any maximal parahoric level structures. We are able to prove the conjecture unconditionally for several special level structures for any $n$. In particular, we will give a new proof of the almost self-dual case first proved in \cite{LZ}.   The main results we obtained should be useful to relax the local assumptions in the arithmetic Siegel--Weil formula for $\U(n-1,1)$--Shimura varieties by allowing maximal parahoric levels at unramified primes. Also,  it may be applied to relax the local assumptions  in the arithmetic inner product formula in \cite{LL2020,LL2} and the $p$-adic arithmetic inner product formula by Disegni and Liu in \cite{DL22}.

	\subsubsection{Kudla--Rapoport conjecture}\label{sec:kudla-rapop-conj}
	
	Let $p$ be an odd prime. Let $F_0$ be a finite unramified extension of $\mathbb{Q}_p$ with residue field $k=\mathbb{F}_q$. Let $F$ be an unramified quadratic extension of $F_0$. Let $\pi$ be a uniformizer of both $F$ and $F_0$. Let $\breve F$ be the completion of the maximal unramified extension of $F$. Let $ O_F, \OFb$ be the ring of integers of $F,\Fb$ respectively.
	
	Let $n\ge2$ be an integer. To define the  unitary Rapoport--Zink space with maximal parahoric level structure, we fix  a  supersingular hermitian $O_F$-modules $\mathbb{X}$ of signature $(1,n-1)$ over $\bar{k}$. The Rapoport-Zink space $\mathcal{N}=\mathcal{N}_n^{[h]}$ is the formal scheme over $\Spf O_{\breve F}$ parameterizing hermitian formal $O_F$-modules $X$ of signature $(1,n-1)$ and type $h$ (see Definition \ref{def: hermitian formal module}) within the quasi-isogeny class of $\mathbb{X}$. The space $\mathcal{N}$ is locally of finite type, and semistable of relative dimension $n-1$ over $\Spf \OFb$.  
	
	Let $\overline{\mathbb{E}}$ be the framing hermitian $O_F$-modules of signature $(0,1)$ over $\bar{k}$. We define  \emph{space of quasi-homomorphisms} to be $\mathbb{V}=\mathbb{V}_n\coloneqq \Hom_{O_F}(\overline{\mathbb{E}}, \mathbb{X}) \otimes_{O_F}F$. We can associate $\bV$ with  a natural $F/F_0$-hermitian form to make $\mathbb{V}$ a non-degenerate $F/F_0$-hermitian space of dimension $n$.  For any subset $L\subset \mathbb{V}$, we define the \emph{special cycle} $\mathcal{Z}(L)$ (resp. $\mathcal{Y}(L)$ ) (see \S\ref{subsec: special cycles}) to be the deformation locus of $ L$ (resp. $\lambda\circ L$) in $\cN_{n}^{[h]}$.
	
	Given  an $O_F$-lattice $L\subset \mathbb{V}$ of full rank $n$, we can define two integers: the \emph{arithmetic intersection number} $\Int(L)$ and the \emph{modified derived local density} $\pDen(L)$.
	
	\begin{definition}
		Let $L\subset \mathbb{V}$ be an $O_F$-lattice and $x_1,\ldots, x_n$ be an $O_F$-basis of $L$. We define the \emph{arithmetic intersection number}
		\begin{equation}
			\label{eq:IntL}
			\mathrm{Int}_{n,h}(L)\coloneqq \chi(\mathcal{N},\mathcal{O}_{\mathcal{Z}(x_1)} \otimes^\mathbb{L}\cdots \otimes^\mathbb{L}\mathcal{O}_{\mathcal{Z}(x_n)} )\in \mathbb{Z},
		\end{equation}
		where $\mathcal{O}_{\mathcal{Z}(x_i)}$ denotes the structure sheaf of the special divisor $\mathcal{Z}(x_i)$, $\otimes^\mathbb{L}$ denotes the derived tensor product of coherent sheaves on $\mathcal{N}$, and $\chi$ denotes the Euler--Poincar\'e characteristic. By Proposition \ref{prop: linear invariance},   $\mathrm{Int}_{n,h}(L)$ is independent of the choice of the basis $x_1,\ldots, x_n$ and hence is a well-defined invariant of $L$ itself.
	\end{definition}

	To define the \emph{modified derived local density} $\pDen(L)$, we need to introduce local densities first. Let $M$ be another hermitian $O_F$-lattice  (of arbitrary rank) and $\Herm_{L,M}$  denote the $O_{F_0}$-scheme of hermitian $O_F$-module homomorphisms from $L$ to $M$.   Then we  define the corresponding \emph{local density}  to be 
	$$\Den(M,L)\coloneqq \lim_{d\rightarrow +\infty}\frac{|\Herm_{L,M}(O_{F_0}/\pi^{d})|}{q^{d\cdot d_{L,M}}},$$ 
	where $d_{L,M}$ is the dimension of $\Herm_{L,M} \otimes_{O_{F_0}}F_0$.  Let $I_k$ be an unimodular hermitian $O_F$-lattice of rank $k$. It is well-known that  there exists a \emph{local density polynomial} $\den(M,L,X)\in \mathbb{Q}[X]$  such that for any integer $k\ge0$,
	\begin{equation}\label{eq:Den poly introduction}
		\Den(M, L, (-\q)^{-k})=\Den(I_k\obot M,L).
	\end{equation}
	Here $I_k\obot M$ denotes the orthogonal direct sum of $I_k$ and $M$.       
	
	When $M$    also has rank $n$ and $M\otimes_{O_F}F$ is not isometric to $L\otimes_{O_F}F$, we have $\Den(M,L)=0$. In this case we write $$\den'(M, L)\coloneqq -\frac{\rd}{\rd X}\bigg|_{X=1} \den(M,L, X),$$ and define the (normalized) \emph{derived local density}
	\begin{equation}\label{eq: def of Den'}
		\mathrm{Den}_{n,h}'(L)\coloneqq \frac{\Den'(I_{n,h}, L)}{\Den(I_{n,h},I_{n,h})}\in \mathbb{Q}.\end{equation}
	Here $I_{n,h}$ is a hermitian lattice with moment matrix $\diag((1)^{n-h},(\pi)^h)$. When the $n$ and $h$ are clear in the context, we also simply denote it as $\Den'(L)$.
	
	Then the naive analogue of the Kudla--Rapoport conjecture for $\cN_n^{[h]}$ should be the identity
	$$\mathrm{Int}_{n,h}(L)\stackrel{?}=\mathrm{Den}_{n,h}'(L).$$ 
	However, as we mentioned before, since $\cN_n^{[h]}$ has bad reduction, some modification for the analytic side is needed. Indeed, a similar consideration as in \cite{HSY3,HLSY} shows that naive analogue cannot be true for trivial reasons. For example, if $h>0$, then  $\Int(L)$ vanishes by definition if $L$ has $I_{n-h+1}$ as a direct summand while $\Den'(L)$ does not vanish for such $L$ by some direct computation. We call an integral $O_F$-lattice $\Lambda\subset \mathbb{V}$  a \emph{vertex lattice of type $t$}  if $\Lambda^\vee/\Lambda$ is a $k$-vector space of dimension $t$. In particular, for a vertex lattice $\Lambda\subset \mathbb{V}$ of type $t<h$,  we have   $\Int(\Lambda)=0$, while $\Den'(\Lambda)\ne0$ in general. 
	
	In order to have $\Int(\Lambda_t)=\pden(\Lambda_t)$ for vertex lattice $\Lambda_t$ of type $t<h$, we define $\pDen(L)$ by modifying $\Den'(L)$ with a linear combination of the (normalized) \emph{local densities}  
	\begin{equation}
		\label{eq: def of den(L)}
		\mathrm{Den}_{n,t}(L)\coloneqq \frac{\Den(\Lambda_t, L)}{\Den(\Lambda_t,\Lambda_t)}\in \mathbb{Z}.
	\end{equation}
	In fact $\bV \not \approx I_{n}^{[h]}\otimes_{O_F}F$. As a result, if $\Lambda_t\subset \bV$, then $t$ and $h$ have different parity.

	\begin{definition}[Definition \ref{definition3.2}]
		Let $L\subset \mathbb{V}$ be an $O_F$-lattice. Define the \emph{modified derived local density} 
		\begin{equation}
			\label{eq: def of pdenL}
			\pDen_{n,h}(L)\coloneqq \mathrm{Den}_{n,h}'(L)+\sum_{i=0}^{\lfloor\frac{(h-1)}{2}\rfloor}c_{n,h-1-2i}\cdot\mathrm{Den}_{n,h-1-2i}(L).
		\end{equation}
		The coefficients $c_{n,i}\in\mathbb{Q}$ here are chosen to satisfy
		\begin{equation} \label{eq:coeff}
			\pDen_{n,h}(\Lambda_{i})=0,\quad \text{ for } 0\le i\le h-1 \text{ and }  i\equiv h-1 \mod{2},
		\end{equation}
		which turns out to be a linear system in $(c_{n,i})$ with a unique solution.
	\end{definition}
	
	Finally, we propose the following Kudla-Rapoport conjecture for $\cN_n^{[h]}$.
	\begin{conjecture}[Conjecture \ref{conj: main section den}, Conjecture \ref{conj: specialized}, Conjecture \ref{conjecture7.8}]\label{conj: main}
		Let $L\subset \mathbb{V}$ be an $O_F$-lattice.  Then we have 
		$$\mathrm{Int}_{n,h}(L)=\pDen_{n,h}(L).$$
	\end{conjecture}
	
	\begin{remark}
		We remark that $\den_{n,h}'(L)$ is not an integer in general. Only the modified $\pden_{n,h}(L)$ is an integer. However, a priori, this is not clear at all. 
	\end{remark}

	The main purpose of this paper is to prove Conjecture \ref{conj: main} assuming a version of Tate conjectures for certain Deligne-Lusztig varieties.  
	
	\begin{theorem}[Theorem \ref{theorem9.1.3}]\label{thm:main}
		Let $L\subset \mathbb{V}$ be an $O_F$-lattice of rank $n$. Assuming Conjecture \ref{conjecture9.2.1}, we have  
		$$\mathrm{Int}_{n,h}(L)=\pDen_{n,h}(L).$$
	\end{theorem}
	
	In this paper,  we also verified Conjecture \ref{conjecture9.2.1} for $\cN_{n}^{[1]}$, $ \cN_{n}^{[n-1]}$ and $\cN_{4}^{[2]}$ unconditionally.
	\begin{theorem}[Theorem \ref{theorem9.2.1}, Theorem \ref{theorem9.1.4}]\label{thm: intro uncond}
		Assume $(n,h)$ is one of the following cases: $(n,1),(n,n-1)$ and $(4,2)$. Then Conjecture \ref{conj: main} holds unconditionally. 
	\end{theorem}
	\begin{remark}
		Conjecture \ref{conj: main} for $\cN_{n}^{[0]}$ was the original Kudla-Rapoport conjecture proposed in \cite{KR1} and proved in \cite{LZ}. The case $(n,h)=(n,1)$ was first proved by \cite{LZ} using certain Hecke correspondence that relates $\cN_{n}^{[1]}$ with $\cN_{n+1}$. Our proof is an attempt to prove the general case uniformly in a way closer to the proof for $\cN_{n}^{[0]}$ in \cite{LZ}.
	\end{remark}
	\begin{remark}
		Although    $\cN_{n}^{[1]}\cong \cN_{n}^{[n-1]}$ by a natural duality, $\cZ$-cycles on   $\cN_{n}^{[1]}$  will be transformed into $\cY$-cycles on   $\cN_{n}^{[n-1]}$. Hence, Theorem \ref{thm: intro uncond} for   $\cN_{n}^{[n-1]}$  is different with the case for  $\cN_{n}^{[1]}$. In particular, Theorem \ref{thm: intro uncond} for $\cN_{n}^{[n-1]}$ and $\cN_{4}^{[2]}$ are new.   
	\end{remark}

	We remark that   Conjecture \ref{conj: main} is based on a different viewpoint from the one formulated in \cite{Cho}. The conjecture formulated in \cite{Cho} is inspired by the duality $\cN_{n}^{[h]}\cong \cN_{n}^{[n-h]}$ and is more general in the sense that it also considers the case when the intersection is between $\cZ$-cycles and $\cY$-cycles. However, since our main theorem in this paper is about intersections between $\cZ$-cycles  and the above formulation is closer to previously studied cases e.g. \cite{LZ},\cite{LL2},\cite{LZ2},\cite{HLSY}, we state this new formulation in the introduction. We refer the reader to Conjecture \ref{conjecture5.5}   for the general conjecture formulated in \cite{Cho} and Conjecture \ref{conj: specialized} for the specialized conjecture about the intersection between $\cZ$-cycles. We show that conjectures \ref{conj: main} and \ref{conj: specialized} are in fact equivalent in Proposition \ref{prop: equiv of conjs}, which  is interestingly a nontrivial fact to prove.

	\subsubsection{Strategy and novelty}\label{sec:whats-new} Our general strategy is similar to the strategy of \cite{LZ} using the local modularity and uncertainty principle.   More precisely, fix an $O_F$-lattice $L^\flat\subset \mathbb{V}$ of rank $n-1$ and consider functions on $\mathbb{V}\setminus L^\flat_F$, $$\mathrm{Int}_{L^\flat}(x)\coloneqq \Int(L^\flat+\langle x\rangle),\quad \pDen_{L^\flat}(x)\coloneqq \pDen(L^\flat+\langle x\rangle).$$ We need to show  $\mathrm{Int}_{L^\flat}=\pDen_{L^\flat}$ as functions on $\bV\setminus L_F^\flat$. First, we have a natural decomposition of $\mathrm{Int}_{L^\flat}$ and $\pDen_{L^\flat}$ into horizontal parts and vertical parts: 
	$$\mathrm{Int}_{L^\flat}=\mathrm{Int}_{L^\flat,\sH}+\mathrm{Int}_{L^\flat,\sV},\quad \pDen_{L^\flat}=\pDen_{L^\flat,\sH}+\pDen_{L^\flat,\sV}.$$ 
	For the horizontal parts, we have $\mathrm{Int}_{L^\flat,\sH}=\pDen_{L^\flat,\sH}$ by direct comparison. For the vertical parts,  $\mathrm{Int}_{L^\flat,\sV}$ and $\pDen_{L^\flat,\sV}$ satisfy ``\emph{local modularity}" in the sense that their Fourier-transforms have nice behaviors.
	
	Due to the existence of nontrivial levels, new phenomenons have already shown up for the horizontal part.  We can decompose the horizontal cycles into primitive horizontal cycles indexed by ``horizontal lattices" and essentially reduce the computation to $n=2$, similar to \cite{LZ}. However, first of all,  we have two different types of the ``horizontal lattices" indexing the primitive horizontal cycle. Moreover, the primitive piece indexed by one of them  admits a further decomposition into  mixed special cycles. Here mixed special cycles mean that they are intersections between both $\cZ$-cycles and $\cY$-cycles. We refer the readers to Theorem \ref{thm: hori part} for more details.

	Now we turn our attention to the vertical part. 
	We discuss the geometric side first.  For a curve $C \subset \cN_{n}^{[h]}$, let $\mathrm{Int}_{C}(\cZ(x))\coloneqq\chi(\cN_{n}^{[h]},\Oo_C\otimes^{\bL}\Oo_{\cZ(x)})$ and $\mathrm{Int}_{C}(\cY(x))\coloneqq\chi(\cN_{n}^{[h]},\Oo_C\otimes^{\bL}\Oo_{\cY(x)})$. If  $C$ lies in $\cN_2^{[1]}$ or $\cN_3^{[0]}$ and is embedded into $\cN_{n}^{[h]}$ via an embedding of $\cN_2^{[1]}$ or $\cN_3^{[0]}$ into $\cN_n^{[h]}$, then explicit computation shows that
	\begin{align}\label{eq: local modularity speculation} 
		\widehat{\Int}_{C}(\cZ(x)) =-q^{-h}\cdot  \mathrm{Int}_{C}(\cY(x)), 
	\end{align}
	where $ \widehat{\Int}_{C}(\cZ(x)) $ denotes the Fourier transform of $\mathrm{Int}_C(\cZ(x))$ (see \S \ref{subsec: notation} for more details). We call such a curve $C$ a special curve.   This was first observed in \cite[]{LZ} when $h=0$. In fact, when $h=0$, the more general identity $$\widehat{\Int}_{L^\flat,\mathscr{V}}(\cZ(x)) =-  \mathrm{Int}_{L^\flat,\mathscr{V}}(\cY(x))$$
	was proved in \cite{LZ}, which we call \emph{local modularity} in this case.   When $h\ge 0$ and $C$ is special, \eqref{eq: local modularity speculation}  was observed in \cite{ZhiyuZhang} based on computation of intersection numbers on $\cN_{2}^{[1]}$ established in \cite{San3}. In \cite{ZhiyuZhang}, this local modularity was used to prove arithmetic transfer conjecture in a similar setting. 
	
	Nevertheless, it is by no means clear that this will be true in general. One observation is that if $\mathrm{Int}_{L^\flat,\sV}(x)$ can be written as a linear combination of $\mathrm{Int}_{C}(x)$ for special $C$, then we have
	\begin{align}\label{eq: speculation for FT Int}
		\widehat{\Int}_{L^\flat,\sV}(x)``="\begin{cases}
			-q^{-h}\cdot  \mathrm{Int}_{n-1,h-1,\sV}(L^\flat) &\text{ if $\val(x)=-1$},\\
			0 &\text{ if $\val(x)\le -2$}.
		\end{cases}
	\end{align}

	Although it is not clear whether $\mathrm{Int}_{L^\flat,\sV}(x)$ can be written as a linear combination of $\mathrm{Int}_{C}(x)$ for special $C$, one can always try to test the corresponding speculation on the analytic side, which is a central idea for many of the past works. Indeed, for example, we  recognized that there should exist two different types of horizontal lattices by testing the analytic side first (the horizontal part goes to infinity when $\val(x)$ goes to infinity). Also, it is really the computation of $\widehat{\pDen}_{L^\flat,\sV}(x)$ in the $\cN_4^{[2]}$ case inspires us what we should expect in general for the geometric side as we explained in \S \ref{section4}.
	
	On the other hand, once we have a big picture for the geometric side with the help of explicit computation from the analytic side, the insight from the geometric side also serves as a guiding principle to obtain purely analytic results. For example, \cite[Proposition 7.5]{HLSY} is inspired by the fact that $\cZ(L)$ is empty for non-integral $L$, and the proof of \cite[Proposition 7.5]{HLSY}  gives important hints for how to prove the rest major analytic results of \cite{HLSY}.  This is also what happened in the current case.

	Indeed, as one of the main results, we manage to prove an analytic analogue of \eqref{eq: speculation for FT Int} unconditionally. More precisely, we prove $\pDen_{L^\flat,\sV}(x)$ can be extended to $\bV$ as a locally constant function and
	\begin{align}\label{eq: FT of pden}
		\widehat{\pden}_{L^\flat,\sV}(x)=\begin{cases} 
			-q^{-h}\cdot   \mathrm{Int}_{n-1,h-1,\sV}(L^\flat) &\text{ if $\val(x)=-1$},\\
			0 &\text{ if $\val(x)\le -2$}.
		\end{cases}
	\end{align}
	As a result, assuming \eqref{eq: speculation for FT Int},  for any $x$ with $\val(x)<0$, we have 
	\begin{align*}
		\widehat{\Int}_{L^\flat}(x)-\widehat{\pden}_{L^\flat}(x)=\widehat{\Int}_{L^\flat,\sV}(x)-\widehat{\pden}_{L^\flat,\sV}(x)=0.
	\end{align*}
	Then the identity $\Int_{L^\flat}(x)-\pDen_{L^\flat}(x)=0$  follows from the uncertainty principle as in \cite{LZ}.
	
	We remark that \eqref{eq: FT of pden} suggests that  $\mathrm{Int}_{L^\flat,\sV}(x)$ might be written as a linear combination of $\mathrm{Int}_{C}(x)$ for special curve $C$. When $h=0$,  this is indeed the case and was proved in \cite[Corollary 5.3.3]{LZ}. Therefore we propose this statement as a conjecture in Conjecture \ref{conjecture9.2.1}.

	In order to establish \eqref{eq: FT of pden}, we make use of the \emph{primitive decomposition} of the local density polynomial into primitive local density polynomials and obtain a decomposition of  $\pDen(L)$:
	\begin{equation}
		\label{eq:decompostionP}\pDen(L)=\sum_{L\subset L'}\pPden(L'),
	\end{equation}
	where $L'$ runs over $O_F$-lattices in $L_F$ containing $L$, and the symbol $\Pden$ stands for the primitive version of $\Den$ (see \eqref{eq:prim local density M,L}). The key reason to consider this primitive decomposition and $\ppden(L)$ is that we usually have a very simple formula for $\ppden(L)$ which eventually makes the computation about $\pden(L)$ possible. This is the case for all the previously proved Kudla-Rapoport type formulas, e.g. \cite{LZ},\cite{LL2},\cite{LZ2}, and \cite{HLSY}. 
	
	Although a similar approach to compute $\ppden(L)$ as in \cite{HLSY} may be directly generalized to the current situation, we decide to generalize the main result of \cite{Cho3} to obtain an explicit formula of $\ppden(L)$ (see Proposition \ref{proposition5.16}). One of the reasons we choose this approach is due to the fact that Proposition \ref{proposition5.16} itself is already a certain induction formula that reduces the computation of $\ppden(L)$ to the good reduction case (see Remark \ref{rem: ind formula}), which is particularly handy in order to prove Theorem \ref{thm: ind formula for Pden} below.  
	
	One of the major new phenomenons and difficulties  we found and overcame in this paper is the fact that when there is a non-trivial level structure, $\ppden(L)$ \emph{no longer} has a simple formula in general. This can be seen via some explicit computation using Proposition \ref{proposition5.16}.   
	Nevertheless, inspired by an attempt to compute the Fourier transform of $\pDen_{L^\flat,\sV}(x)$ (e.g. to obtain \eqref{eq: FT of pden}),  we find simple inductive formulas for $\ppden(L)$ which suffices to control the Fourier transform of $\pDen_{L^\flat,\sV}(x)$. 
	In fact, such inductive formulas also hold in all the previously studied cases and the simple formulas for $\ppden(L)$ in these cases follow as a direct corollary.  
	
	More precisely, let $t_i(L)$ and $t_{\ge i}(L)$ be the number of the fundamental invariants  of $L$ that is exactly $i$ and at least $i$ respectively.  Then we have the following. 
	\begin{theorem}[Theorem \ref{theorem5.24}]\label{thm: ind formula for Pden}  
		Let $L$ be a hermitian lattice with $\val(L) \not\equiv h \mod 2$. We have  $\ppden_{n,h}(L)$ depends only on $(t_{\ge 2}(L),t_1(L),t_0(L))$. For simplicity, we denote it as $D_{n,h}(t_2 ,t_1 ,t_0 )$. Moreover, assume that $(t_{2}-1,t_1+1,t_0) \neq (n-h,h,0)$, $t_0 \leq n-h$ and $t_2\ge 1$, then we have   
		\begin{align*}
			&D_{n,h}(t_2,t_1,t_0)-D_{n,h}(t_2-1,t_1+1,t_0)
			=-(-q)^{2n-h-1-t_1-2t_0}D_{n-1,h-1}(t_2-1,t_1,t_0).
		\end{align*}
	\end{theorem}

	The above inductive formula is complemented by the  following simple formulas for special $D_{n,h}(t_2,t_1,t_0)$.
	\begin{theorem}\label{thm: complement formula}
		With the same notations and assumptions as in Theorem \ref{thm: ind formula for Pden}, we have the following.
		\begin{enumerate}
			\item (Lemma \ref{lemma7.5}, Proposition \ref{prop: vanishing of D(lambda)})  If $t_0 > n-h$, then $D_{n,h}(t_2,t_1,t_0)=0$.
			
			\item (Theorem \ref{theorem5.31} (1)) If $t_2=0$ and $h+1 \leq t_1$. Then, we have
			\begin{equation*}
				D_{n,h}(0,t_1,t_0)=\dfrac{\prod_{l=h+1}^{t_1}(1-(-q)^l)}{(1-(-q)^{t_1-h})}.
			\end{equation*}
			
			\item  (Theorem \ref{theorem5.31} (2)) If $t_2=1$ and $h-1 \leq t_1$. Then, we have
			\begin{equation*}
				D_{n,h}(1,t_1,t_0)=\left\lbrace \begin{array}{ll}
					1 & \text{ if } t_1=h-1, h;\\
					\prod_{l=h+1}^{t_1}(1-(-q)^l) & \text{ if } t_1 \geq h+1.
				\end{array}\right.
			\end{equation*}
		\end{enumerate}
	\end{theorem}
	Theorems \ref{thm: ind formula for Pden} and \ref{thm: complement formula} will be proved in \S \ref{sec: ind formula} using the formula for $\ppden(L)$ obtained in \cite{Cho3}. As we have remarked, the method developed in \cite{HLSY} may also be adapted to the current case. However, we find that using the formulas in \cite{Cho3} is easier for our purpose so we stick with this approach. Note that the formulas in \cite{Cho3} are derived via a duality between weighted local densities in analogy with the duality between $\cN_{n}^{[h]}$ and $\cN_{n}^{[n-h]}$, and this might be a particular reason for why this formula so applicable in the current case.

	With the help of formulas for $\pPden(L)$, we finally prove \eqref{eq: FT of pden} via certain involved weighted lattice counting in \S\ref{sec: FT} via a similar method as in \cite{LZ2,HLSY}. However, since $\ppden(L)$ now depends on $t_{\geq 2}(L),t_1(L)$ and $t_0(L)$, the counting becomes much more involved compared with the ones  in \cite{LZ2,HLSY}.

	\subsection{Notation and terminology}\label{subsec: notation}
	\begin{itemize}
		\item
		Let $p$ be an odd prime. Let $F_0$ be a finite unramified extension of $\mathbb{Q}_p$ with residue field $k=\mathbb{F}_q$. Let $F$ be the unramified quadratic extension of $F_0$. Let $\pi$ be a uniformizer of $F$ and $F_0$.   Let $\breve F$ be the completion of the maximal unramified extension of $F$. Let $ O_F, \OFb$ be the ring of integers of $F,\Fb$ respectively.
		
		\item We say a sublattice  of a hermitian space is non-degenerate if the restriction of the hermitian form to it is non-degenerate.
		
		\item
		In this paper, without explicit mentioning a lattice means a non-degenerate hermitian $O_F$-lattice. Unless otherwise stated, the symbol $L$ always means a non-degenerate lattice of rank $n$ with a hermitian form $(\, , \,)$.

		\item We define $L^{\vee}$ to be the dual lattice of $L$ with respect to the hermitian form $(\, , \,)$. If $L\subset L^\vee$, we say $L$ is integral. If $L \subset L^\vee \subset \pi^{-1}L$, we say $L$ is a vertex lattice. 
		
		\item 
		We say that a basis $\left\{\ell_{1}, \ldots, \ell_{n}\right\}$ of $L$ is a normal basis (which always exists) if its moment matrix $T= ( (\ell_{i}, \ell_{j} ) )_{1\le i, j\le n}$ is
		\begin{align*} 
			\diag(\pi^{\alpha_1},\ldots,\pi^{\alpha_n})
		\end{align*}
		where $\alpha_{1}, \ldots, \alpha_{n} \in \mathbb{Z}$. Moreover, we define its fundamental invariants $(a_1,\cdots,a_n)$ to be the unique nondecreasing rearrangement of $(\alpha_{1}, \ldots, \alpha_{n})$. 
		
		\item We define the valuation of $L$ to be $\mathrm{val}(L)\coloneqq\sum_{i=1}^{n}a_i$, where $(a_1,\cdots,a_n)$ are the fundamental invariants of $L$.  For $x\in L$, we define $\val(x)=\val((x,x))$, where $\val(\pi)=1$.  
		
		\item We call a sublattice $N\subset M$ primitive in $M$ if $\mathrm{dim}_{\mathbb F_\q}\overline{N}=r(N)$, where $\overline{N}=(N+\pi M)/\pi M$. We also use $\overline{L}$ to denote $L\otimes_{O_F} O_F/(\pi).$
		
		\item We use $I_{m}$ to denote a unimodular lattice of rank $m$, and $I_{n,h}$ to denote a hermitian lattice with moment matrix $\diag((1)^{n-h},(\pi)^{h})$.

		\item For a hermitian space $\bV$, we let $\bV^{? i}\coloneqq \{x\in \bV\mid \val(x)? i\}$ where $?$ can be $\ge$, $\leq$ or $=$.

		\item Fix an unramified additive character $\psi: F_0 \rightarrow \mathbb{C}^{\times}$. Here "unramifiedness" means that the conductor of $\psi$ (i.e., the largest fractional ideal in $F_0$ on which $\psi$ is trivial) is $O_{F_0}$. For an integrable function $f$ on $\mathbb{V}$, we define its Fourier transform $\widehat{f}$ to be 
		$$
		\widehat{f}(x)\coloneqq \int_{\mathbb{V}} f(y) \psi (\operatorname{tr}_{F / F_0}(x, y) ) \mathrm{d} y, \quad x \in \mathbb{V} .
		$$
		We normalize the Haar measure on $\mathbb{V}$ to be self-dual, so $\hat{f}(x)=f(-x)$. For an $O_F$-lattice $L \subset \mathbb{V}$ of rank $n$, we have (under the assumption that $F / F_0$ is unramified)
		$$
		\widehat{\mathbf{1}}_L=\operatorname{vol}(L) \mathbf{1}_{L^{\vee}}, \quad \text { and } \quad \operatorname{vol}(L)=\left[L^{\vee}: L\right]^{-1 / 2}=q^{-\operatorname{\val}(L)} .
		$$
		Note that $\operatorname{\val}(L)$ can be defined for any lattice $L$ (not necessarily integral) so that the above equality for $\operatorname{vol}(L)$ holds.
	\end{itemize}

	\subsection{Acknowledgement} We would like to thank Sungmun Cho, Benjamin Howard, Chao Li, Yifeng Liu, Michael Rapoport, Tonghai Yang and Wei Zhang for helpful discussions or comments. Part of the work was done while Q. H. and Z. Z. were participating in the ``Algebraic Cycles, $L$-values, and Euler Systems" program held in MSRI during Spring 2023. We would like to thank MSRI for the excellent work conditions, financial support, and hospitality. S.C. was supported by POSTECH Basic Science Research Institute under the NRF grant number NRF2021R1A6A1A1004294412.

	\section{Rapoport-Zink space and special cycles}

	\subsection{Rapoport-Zink space}
	In this section, we review the definition and basic properties of Rapoport-Zink space and special cycles. 
	
	\begin{definition}\label{def: hermitian formal module}
		For any $\Spf O_{\breve F}$-scheme $S$, \emph{a hermitian formal $O_F$-moodule} $(X, \iota, \lambda)$ of   signature $(1, n-1)$ and type $h$ over $S$ is the following data:
		\begin{itemize}
			\item
			$X$ is a strict formal $O_{F_0}$-module over $S$ of relative height $2n$ and dimension $n$. Strictness means the induced action of $O_{F_0}$ on $\Lie X$ is via the structure morphism $O_{F_0} \to \CO_S$. 
			\item $\iota: O_F \to \End(X)$ is an action of $O_F$ on $X$ that extends the action of $O_{F_0}$. We require that the \emph{Kottwitz condition} of signature $(1, n-1)$ holds for all $a \in O_{F} $:
			\begin{equation}\label{eq: Kottwitz}
				\mathrm{char} (\iota(a)\mid \Lie X) = (T-a)(T-\overline{a})^{n-1} \in \CO_S[T].
			\end{equation} 
			\item $\lambda$ is a polarization on $X$, which is $O_F/O_{F_0}$ semi-linear in the sense that the Rosati involution $\mathrm{Ros}_\lambda$ induces the non-hrivial involution on $\iota: O_{F} \to \End(X)$.
			\item We require that the finite flat group scheme $\mathrm{Ker}\, \lambda$ over $S$ lies in $X[\pi]$ and is of order $q^{2h}$.
		\end{itemize}
	\end{definition}
	An isomorphism $(X_1, \iota_1, \lambda_1) \overset{\sim}{\longrightarrow} (X_2, \iota_2, \lambda_2)$ between two such triples is an $O_{F}$-linear isomorphism $\varphi\colon X_1 \overset{\sim}{\longrightarrow} X_2$ such that $\varphi^*(\lambda_2)=\lambda_1$.   Up to $O_F$-linear quasi-isogeny compatible with the polarization, there exists a unique such triple $(\BX, \iota_{\BX}, \lambda_{\BX})$ over $\BF$. Fix one choice of $(\BX, \iota_{\BX}, \lambda_{\BX})$ as the {\em framing object}.  

	\begin{definition}\label{defn: local RZ space}
		Let (Nilp) be the category of $O_{F}$-schemes $S$ such that $\pi$ is locally nilpotent on $S$.   Then the Rapoport--Zink space associated with $(\BX, \iota_{\BX}, \lambda_{\BX})$ is the functor 
		$$\CN_{(\BX, \iota_{\BX}, \lambda_{\BX})}=\CN_n^{[h]} \to \Spf O_{\breve F} $$
		sending $S\in$ (Nilp) to the set of isomorphism classes of tuples $(X, \iota, \lambda, \rho)$, where
		\begin{itemize}
			\item $(X, \iota, \lambda)$ is a hermitian formal $O_F$-module of dimension $n$ and type $h$ over $S$;
			\item $\rho: X \times_{S} {\overline{S} } \to \BX \times_\BF \overline{S} $ is an $O_F$-linear quasi-isogeny of height $0$ over the reduction $\overline{S} \coloneqq S \times_{O_{\breve F_0}} \mathbb F$ such that $\rho^*(\lambda_{\BX, \overline{S} }) = \lambda_{\overline{S} }$.
		\end{itemize}
	\end{definition}
	The functor $\mathcal{N}_n^{[h]}$ is representable by a formal scheme over $\operatorname{Spf} O_{\breve{F}}$ which is locally formally of finite type by \cite{RZ}.  Moreover, this formal scheme is regular (see \cite[Proposition 3.33]{cho2018basic}).  We often simply denote $\CN_n^{[h]}$ as $\cN$ if the signature and type are clear in the context.

	There is an isomorphism $\theta: \CN_n^{[h]}\stackrel{\sim}{\to}\cN_n^{[n-h]}$ constructed as follows (\cite[Remark 5.2]{cho2018basic}). For each $S\in$ (Nilp), 
	\begin{align*}
		\CN_n^{[h]}(S)&\stackrel{\theta}{\to} \cN_n^{[n-h]}(S),\\
		(X, i_X, \lambda_X, \rho_X)&\mapsto   (X^{\vee}, \bar{i}_X^{\vee}, \lambda_X^{\prime}, (\rho_X^{\vee} )^{-1}),
	\end{align*}
	where $\lambda_X^{\prime}: X^{\vee} \rightarrow X$ is the unique polarization such that $\lambda_X^{\prime} \circ \lambda_X=$ $i_X(\pi)$, and for $a \in O_{F}$, the action $\bar{i}_X^{\vee}$ is defined as $\bar{i}_X^{\vee}(a)\coloneqq i_X (a^* )^{\vee}$.
	When we denote $\CN_n^{[h]}$ as $\cN$, we use $\cN^\vee$ to denote $\cN_n^{[n-h]}$.

	\subsection{Special cycles}\label{subsec: special cycles}
	To define special cycles, we need to fix a hermitian formal $O_F$-moodule $ (\overline{\mathbb{E}}, i_{\overline{\mathbb{E}}}, \lambda_{\overline{\mathbb{E}}} )$ of signature $(0,1)$ and type $0$ over $\mathbb{F}$. Then we can similarly define a Rapoport-Zink Space $\cN_{(\overline{\mathbb{E}}, i_{\overline{\mathbb{E}}}, \lambda_{\overline{\mathbb{E}}})}$ which we denote as $\cN^0$ for simplicity. Recall that there is a unique lifting $(\overline{\cE},\iota_{\overline{\cE}},\lambda_{\overline{\cE}})$ of $ (\overline{\mathbb{E}}, i_{\overline{\mathbb{E}}}, \lambda_{\overline{\mathbb{E}}} )$ over $O_{\breve{F}}$.
	
	\begin{definition}
		We define the space of special homomorphism to be the $F$-vector space 
		\begin{align*}
			\bV\coloneqq \Hom_{O_F}(\overline{\BE},\bX)\otimes_{\Z} \BQ.
		\end{align*}
		We can associate $\bV$ with a naturally defined hermitian form as follows. For $x, y \in \mathbb{V}$, we define a hermitian form $h$ on $\mathbb{V}$ as
		$$
		(x, y)=\lambda_{\overline{\mathbb{E}}}^{-1} \circ y^{\vee} \circ \lambda_{\mathbb{X}} \circ x \in \operatorname{End}_{O_F}(\overline{\mathbb{E}}) \otimes \mathbb{Q} \stackrel{i_{\mathbb{\mathbb { X }}}^{\simeq}}{\simeq} F.
		$$
		We often omit $i_{\overline{\mathbb{E}}}^{-1}$ via the identification $\operatorname{End}_{O_{F}}(\overline{\mathbb{E}}) \otimes \mathbb{Q} \simeq F$.
	\end{definition}

	\begin{definition}\cite[Definition 3.2]{KR1}, \cite[Definition 5.4]{cho2018basic}\hfill
		\begin{enumerate}
			\item 
			For $x \in \mathbb{V}$, we define the special cycle $\mathcal{Z}(x)$ to be the closed formal subscheme of $\mathcal{N}^0 \times \mathcal{N}$ with the following property: For each $O_F$-scheme $S$ such that $\pi$ is locally nilpotent, $\mathcal{Z}(x)(S)$ is the set of all points $\xi= (\overline{\cE}_S,\iota_{\overline{\cE}_S},\lambda_{\overline{\cE}_S}, X, i_X, \lambda_X, \rho_X )$ in $ \mathcal{N}^0 \times \mathcal{N} )(S)$ such that the quasi-homomorphism
			$$
			\rho_X^{-1} \circ x \circ \rho_{\overline{\cE}_S}: \overline{\cE}_S \times_S \overline{S} \rightarrow X \times_S \overline{S}
			$$
			extends to a homomorphism from $\overline{\mathcal{E}}_S$ to $X$.
			\item For $y \in \mathbb{V}$, we define the special cycle $\mathcal{Y}(y)$ in $\mathcal{N}^0 \times \mathcal{N}$ as follows. First, consider the special cycle $\mathcal{Z} (\lambda_{\mathbb{X}} \circ y )$ in $\mathcal{N}^0 \times \cN^\vee$. This is the closed formal subscheme of $\mathcal{N}^0 \times \cN^\vee$. We define $\mathcal{Y}(y)$ as $ (i d \times \theta^{-1} ) (\mathcal{Z} (\lambda_{\mathbb{X}} \circ y ) )$ in $\mathcal{N}^0 \times \mathcal{N}$
		\end{enumerate}
	\end{definition}
	
	Note that $\mathcal{N}^0$ can be identified with $\operatorname{Spf} O_{\breve{F}}$, and hence $\mathcal{Z}(x), \mathcal{Y}(y)$ can be identified with closed formal subschemes of $\mathcal{N}$.

	The same proof of \cite[Proposition 5.9]{KR2} gives us the following.
	\begin{proposition}\cite[Proposition 5.9]{KR2} The functors $\mathcal{Z}(x)$ and $\mathcal{Y}(y)$ are representable by Cartier divisors of $\mathcal{N}$.
	\end{proposition}
	
	Special cycles have the following properties.
	
	\begin{proposition}\label{proposition2.5}\cite[Proposition 5.10]{cho2018basic} Let $x,y \in \BV$.
		\begin{enumerate}
			\item If $\val((x,x))=0$, then $\CZ(x) \simeq \CN^{[h]}_{n-1}$.
			\item If $\val((y,y))=-1$, then $\CY(y) \simeq \CN^{[h-1]}_{n-1}$.
		\end{enumerate}
	\end{proposition}
	
	\begin{proposition}\label{proposition2.6}\cite[Proposition 5.11]{cho2018basic}
		Assume that $\val((x,x))=0$ and $\val((y,y))=-1$. Assume further that by rescaling as in the proof of \cite[Proposition 5.10]{cho2018basic}, $x^*\circ x=1$, and $ (\lambda_{\BX} \circ y)^*\circ (\lambda_{\BX} \circ y)=1$. Here, $x^*$ (resp. $(\lambda_{\BX} \circ y)^*$) is the adjoint of $x$ (resp. $(\lambda_{\BX} \circ y)$)  with respect to the polarizations $\lambda_{\BX}$ and $\lambda_{\overline{\BE}}$ (resp. $\lambda'_{\BX}$ and $\lambda_{\overline{\BE}}$). We define $e_x\coloneqq x \circ x^*$ and $e_y\coloneqq (\lambda_{\BX} \circ y) \circ (\lambda_{\BX} \circ y)^*$. Fix isomorphisms
		\begin{equation*}
			\begin{split}
				\Phi:\CZ(x) \simeq \CN^{[h]}_{n-1},\\
				\Psi:\CY(y) \simeq \CN^{[h-1]}_{n-1},
			\end{split}
		\end{equation*}
		as in Proposition \ref{proposition2.5}. Then the following statements hold.
		\begin{enumerate}
			\item For $z\in \BV$ such that $(x,z)=0$, let $z'\coloneqq (1-e_x)\circ z$. Then, we have $\Phi(\CZ(x) \cap \CZ(z))=\CZ(z')$ in $\CN^{[h]}_{n-1}$ and $(z',z')=(z,z)$.
			
			\item For $w\in \BV$ such that $(x,w)=0$, let $w'\coloneqq (1-e_x)\circ w$. Then, we have $\Phi(\CZ(x) \cap \CY(w))=\CY(w')$ in $\CN^{[h]}_{n-1}$ and $(w',w')=(w,w)$.
			
			\item For $z\in \BV$ such that $(y,z)=0$, let $z'\coloneqq (1-e_y^{\vee})\circ z$. Then, we have $\Psi(\CY(y) \cap \CZ(z))=\CZ(z')$ in $\CN^{[h-1]}_{n-1}$ and $(z',z')=(z,z)$.
			
			\item For $w\in \BV$ such that $(y,w)=0$, let $w'\coloneqq (1-e_y^{\vee})\circ w$. Then, we have $\Psi(\CY(y) \cap \CY(w))=\CY(w')$ in $\CN^{[h-1]}_{n-1}$ and $(w',w')=(w,w)$.
		\end{enumerate}
	\end{proposition}

	\subsection{Horizontal and vertical part of special cycles}\label{subsection2.3}
	We closely follow \cite[\S 2.9]{LZ}  in this subsection.
	We call a formal scheme $Z$ over $\operatorname{Spf}$ $O_{\breve{F}}$   vertical (resp. horizontal) if $\pi$ is locally nilpotent on $Z$ (resp. flat over $\operatorname{Spf} O_{\breve{F}}$ ). In particular, the formal scheme-theoretic union of two vertical (resp. horizontal) formal subschemes of a formal scheme is again vertical (resp. horizontal).

	Now we define the horizontal part and vertical part of $Z$ respectively. The horizontal part $Z_{\mathscr{H}}$ of $Z$ is defined to be the closed formal subscheme with ideal sheaf $\mathcal{O}_Z\left[\pi^{\infty}\right] \subset \mathcal{O}_Z$. Then $Z_{\mathscr{H}}$ is the maximal horizontal closed formal subscheme of $Z$.
	For noetherian $Z$, we can find $N \gg 0$ such that $\pi^N \mathcal{O}_Z\left[\pi^{\infty}\right]=0$.  Then the vertical part $Z_{\mathscr{V}} \subset Z$ is defined to be the closed formal subscheme with ideal sheaf $\pi^N \mathcal{O}_Z$. 
	
	Note that $\mathcal{O}_Z\left[\pi^{\infty}\right] \cap \pi^N \mathcal{O}_Z=0$ implies the following  decomposition:
	$$
	Z=Z_{\mathscr{H}} \cup Z_{\mathscr{V}}.
	$$

	The same proof of \cite[Lemma 2.9.2]{LZ} gives the following.
	\begin{lemma}\label{lemma2.7}\cite[Lemma 2.9.2]{LZ}
		Let $L \subset \mathbb{V}$ be a $O_F$-lattice of rank $r \geq n-1$ such that $L_F$ is non-degenerate. Then $\mathcal{Z}(L)$ is noetherian.
	\end{lemma}

	The following lemma follows from the same proof of \cite[Lemma 5.1.1]{LZ}.
	\begin{lemma}\cite[Lemma 5.1.1]{LZ}
		Let $L^\flat$ be an $O_F$-lattice of rank $n-1$ in $\bV_n$. Then $\mathcal{Z} (L^b )_{\mathscr{V}}$ is supported on $\mathcal{N}_n^{\mathrm{red}}$, i.e., $\mathcal{O}_{\mathcal{Z} (L^b )_{\mathscr{V}}}$ is annihilated by a power of the ideal sheaf of $\mathcal{N}_n^{\mathrm{red}} \subset \mathcal{N}_n$.
	\end{lemma}

	\subsection{Linear invariance}
	Following \cite{Ho2}, we show the linear invariance of intersection numbers. In this subsection, we use $\cN$ to denote $\cN_{n}^{[h]}$ and  $(X,\iota_X,\lambda_X)$ to denote the  the universal object over $\cN$. Let $\mathrm{D}(X)$ denote the covariant Grothendieck-Messing crystal of $X$  restricted to the Zariski site. Then we have a short exact sequence of locally free $\mathcal{O}_{\cN}$-modules:
	\begin{align*}
		0 \rightarrow \operatorname{Fil}(X) \rightarrow \mathrm{D}(X) \rightarrow \operatorname{Lie}(X) \rightarrow 0.
	\end{align*}
	which is $O_F$-linear via the action given by $\iota_X$. Let 
	\begin{align*}
		\epsilon&\coloneqq \pi \otimes 1+ 1\otimes \pi \in O_{F}\otimes_{O_F}\mathcal{O}_{\cN},\\
		\bar{\epsilon}&\coloneqq -\pi \otimes 1 +1\otimes \pi \in O_{F}\otimes_{O_F}\mathcal{O}_{\cN}.
	\end{align*}
	\begin{definition}
		Let $L_X$ be the image of $\iota(\pi)+\pi$ on $\mathrm{Lie}(X)$. In other words, $L_X\coloneqq \epsilon \Lie(X).$
	\end{definition}
	According to the Kottwitz signature condition, we know $L_X\subset \mathrm{Lie}(X)$ is locally a $\mathcal{O}_{\cN}$-module  direct summand of rank $1$.

	For a closed formal subscheme $Z$ of $\mathcal{N}$ with ideal sheaf $\mathcal{I}_Z$, we denote by $\tilde{Z}$ the closed formal subscheme defined by the sheaf $\mathcal{I}_Z^2$. Let $x \in \bV$ be a non-zero special homomorphism. Let $X_0$ be the universal object of $\CN^{[0]}_{1}$. By the very definition of $\mathcal{Z}(x)$,  we have
	$$
	X_0 |_{\mathcal{Z}(x)} \stackrel{x}{\rightarrow} X |_{\mathcal{Z}(x)},
	$$
	which induces an $O_F$-linear morphism of vector bundles
	$$
	\mathrm{D} (X_0 ) |_{\mathcal{Z}(x)} \stackrel{x}{\rightarrow} \mathrm{D}(X) |_{\mathcal{Z}(x)}.
	$$
	By the Grothendieck--Messing theory, we may canonically extend the above morphism to a morphism
	$$
	\mathrm{D} (X_0 ) |_{\tilde{\mathcal{Z}}(x)} \stackrel{\tilde{x}}{\rightarrow} \mathrm{D}(X) |_{\tilde{\mathcal{Z}}(x)},
	$$
	which no longer preserves the Hodge filtrations and hence induces a nontrivial morphism
	\begin{align}\label{eq: morphism}
		\operatorname{Fil} (X_0 ) |_{\tilde{\mathcal{Z}}(x)} \stackrel{\tilde{x}}{\rightarrow} \operatorname{Lie}(X) |_{\tilde{\mathcal{Z}}(x)}. 
	\end{align}
	\begin{proposition}\label{prop: obstruction}
		The morphism \eqref{eq: morphism} induces a morphism
		$$
		\operatorname{Fil} (X_0 ) |_{\tilde{Z}(x)} \stackrel{\tilde{x}}{\rightarrow} L_X |_{\tilde{Z}(x)} .
		$$
		Moreover, $Z(x)$ is the vanishing locus of $\tilde{x}$.
	\end{proposition}
	\begin{proof}
		According to the signature condition of $X_0$, we have $\Fil(X_0)=\epsilon \rD(X_0)$ since both are locally $\cO_{\cN}$ direct summands of $\rD(X_0)$ of rank $1$. Since $\tilde{x}$ is $O_F$-linear, we have 
		\begin{align*}
			\tilde{x}(\Fil(X_0))=\tilde{x}(\epsilon \rD(X_0))\in \epsilon \rD(X)=L_X.
		\end{align*}
		Now the second claim follows from the Grothendieck--Messing theory.
	\end{proof}

	Now given a nonzero element $x \in \bV$, we define a chain complex of locally free $\mathcal{O}_{\mathcal{N}}$-modules
	$$
	C(x)\coloneqq (\cdots \rightarrow 0 \rightarrow \mathcal{I}_{\mathcal{Z}(x)} \rightarrow\mathcal{O}_{\mathcal{N}} \rightarrow 0 )
	$$
	supported in degrees 1 and 0 with the map $\mathcal{I}_{\mathcal{Z}(x)} \rightarrow\mathcal{O}_{\mathcal{N}}$ being the natural inclusion. We extend the definition to $x=0$ by setting
	$$
	C(0)\coloneqq  (\cdots \rightarrow 0 \rightarrow \omega \stackrel{0}{\rightarrow}\mathcal{O}_{\mathcal{N}} \rightarrow 0 )
	$$
	supported in degrees 1 and 0 , where $\omega$ is the line bundle such that $\omega^{-1}  =\underline{\operatorname{Hom}} (\operatorname{Fil} (X_0 ), L_X ).$
	
	The following is our main result of this subsection which follows from Proposition \ref{prop: linear invariance} by the same argument as in \cite{Ho2}.
	\begin{proposition}\label{prop: linear invariance}
		Let $0 \leqslant m \leqslant n$ be an integer. Suppose that $x_1, \ldots, x_m \in \bV$ and $y_1, \ldots, y_m \in \bV$ generate the same $O_E$-submodule. Then we have an isomorphism
		$$
		\mathrm{H}_i (C (x_1 ) \otimes_{\mathcal{O}_{\mathcal{N}}} \cdots \otimes_{\mathcal{O}_{\mathcal{N}}} C (x_m ) ) \simeq \mathrm{H}_i (C (y_1 ) \otimes_{\mathcal{O}_{\mathcal{N}}} \cdots \otimes_{\mathcal{O}_{\mathcal{N}}} C (y_m ) )
		$$
		of $\mathcal{O}_{\mathcal{N}}$-modules for every $i$.
	\end{proposition}

	\section{Local density and the modified Kudla--Rapoport conjecture}
	In this section, we discuss the conjecture \ref{conj: main}  proposed in the introduction in detail.

	\subsection{Local density} 
	Let $L, M$ be two integral hermitian $O_F$-lattices with rank $n, m$ respectively.  Let $\operatorname{Herm}{ }_{L,M}$ be the scheme of integral representations of $M$ by $L$, an $O_{F_0}$-scheme such that for any $O_{F_0}$-algebra $R$,
	$$
	\operatorname{Herm}_{L,M}(R)=\operatorname{Herm} (L \otimes_{O_{F_0}} R, M \otimes_{O_{F_0}} R )
	$$
	where $\operatorname{Herm}$ denotes the set of hermitian module homomorphisms. The local density of integral representations of $M$ by $L$ is defined to be
	$$
	\operatorname{Den}(M, L)\coloneqq \lim _{d \rightarrow+\infty} \frac{\# \operatorname{Herm}_{L,M} (O_{F_0} / (\pi^d) )}{q^{d \cdot \operatorname{dim} (\operatorname{Herm}_{L, M} )_{F_0}}}.
	$$
	If the generic fiber $ (\operatorname{Herm}_{L, M} )_{F_0}$ is non-empty, then  we have $n \leq m$ and
	$$
	\operatorname{dim} (\operatorname{Herm}_{M, L} )_{F_0}=\operatorname{dim} \mathrm{U}_m-\operatorname{dim} \mathrm{U}_{m-n}=n \cdot(2 m-n).
	$$
	
	Now we consider the local density polynomial.
	Let $I_k$ be an unimodular hermitian $O_F$-lattice of rank $k$. It is well-known that  there exists a \emph{local density polynomial} $\den(M,L,X)\in \mathbb{Q}[X]$  such that for any integer $k\ge0$,
	\begin{equation}\label{eq:Den poly introduction}
		\Den(M, L, (-\q)^{-k})=\Den(I_k\obot M,L).
	\end{equation}
	When $M$  has also rank $n$ and $\chi(M)=-\chi(L)$, we have $\Den(M,L)=0$   and in this case we write $$\den'(M, L)\coloneqq - \frac{\rd}{\rd X}\bigg|_{X=1} \den(M,L, X).$$ Define the (normalized) \emph{derived local density}
	\begin{equation}\label{eq: def of Den'}
		\mathrm{Den}_{n,h}'(L)\coloneqq \frac{\Den'(I_{n,h}, L)}{\Den(I_{n,h},I_{n,h})}\in \mathbb{Q}.
	\end{equation}
	Here $I_{n,h}$ is a hermitian lattice with moment matrix $\diag((1)^{n-h},(\pi)^h)$.   When the $n$ and $h$ are clear in the context, we also simply denote it as $\Den'(L)$.
	
	Recall that an integral $O_F$-lattice $\Lambda\subset \mathbb{V}$ is a \emph{vertex lattice of type $t$}  if $\Lambda^\vee/\Lambda$ is a $\kappa$-vector space of dimension $t$. 
	\begin{lemma}\label{lem: vanishing of Z(Lambda)}
		Assume $t<h$,  we have   $\Int_{n,h}(\Lambda_t)=0$.
	\end{lemma}
	\begin{proof}
		First, we write $\Lambda_t=I_{n-t}\obot J_{t}$ where $J_t$ is a lattice with moment matrix $(p)^t$.  According to Proposition \ref{proposition2.5}, $\cZ(\Lambda_t)$ in $\cN_{n}^{[h]}$ is isomorphic to $\cZ(I_{h-t}\obot J_{t})$ in $\cN_{h}^{[h]}$. However, for $(X,\iota,\lambda,\rho)\in \cN_{h}^{[h]}$, we have $\mathrm{Ker}\lambda =X[\pi]$ by definition of $\cN_{h}^{[h]}$. Hence, $\lambda=\pi \lambda'$ for some principal $\lambda'$. Assume $(X,\iota,\lambda,\rho)\in \cZ(x)$ for some $x\in \bV$, then it is clear from the definition of $\cZ(x)$ such that $\val(x)\ge 1$. Hence, if $\val(x)=0$, then $\cZ(x)=\emptyset$ in $\cN_{h}^{[h]}$. In particular, $\cZ(I_{h-t}\obot J_{t})$ in $\cN_{h}^{[h]}$ is empty.
	\end{proof}

	The naive analogue of the Kudla--Rapoport conjecture for $\cN_n^{[h]}$ states that $$\mathrm{Int}_{n,h}(L)\stackrel{?}=\mathrm{Den}_{n,h}'(L).$$ However, this can not be true since $\Int_{n,h}(\Lambda_t)=0$ by Lemma \ref{lem: vanishing of Z(Lambda)}, but $\Den'_{n,h}(\Lambda_t)\neq 0$.

	Now a similar consideration as in \cite{HSY3,HLSY} suggests the following.   In order to have $\Int(\Lambda_t)=\pden(\Lambda_t)$ for vertex lattice $\Lambda_t$ of type $t<h$, we define $\pDen(L)$ by modifying $\Den'(L)$ with a linear combination of the (normalized) \emph{local densities}  
	\begin{equation}
		\label{eq: def of den(L)}
		\mathrm{Den}_{n,t}(L)\coloneqq \frac{\Den(\Lambda_t, L)}{\Den(\Lambda_t,\Lambda_t)}\in \mathbb{Z}.
	\end{equation}
	Note that since $F/F_0$ is unramified, the hermitian space over $F$ is determined by the parity of valuation of the hermitian form. As a result, if $\Lambda_t\subset \bV$, then $t$ and $h$ have different parity.

	\begin{definition}\label{definition3.2}
		Let $L\subset \mathbb{V}$ be an $O_F$-lattice. Define the \emph{modified derived local density} 
		\begin{equation}
			\label{eq: def of pdenL}
			\pDen_{n,h}(L)\coloneqq \mathrm{Den}_{n,h}'(L)+\sum_{i=0}^{\lfloor\frac{(h-1)}{2}\rfloor}c_{n,h-1-2i}\cdot\mathrm{Den}_{n,h-1-2i}(L).
		\end{equation}
		The coefficients $c_{n,i}\in\mathbb{Q}$ here are chosen to satisfy
		\begin{equation} \label{eq:coeff}
			\pDen_{n,h}(\Lambda_{i})=0,\quad \text{ for } 0\le i\le h-1 \text{ and }  i\equiv h-1 \mod{2},
		\end{equation}
		which turns out to be a linear system in $(c_{n,i})$ with a unique solution since $\Den(\Lambda_j,\Lambda_i)=0$ if $j>i$.
	\end{definition}

	\begin{conjecture}\label{conj: main section den}
		Let $L\subset \mathbb{V}$ be an $O_F$-lattice.  Then we have 
		$$\mathrm{Int}_{n,h}(L)=\pDen_{n,h}(L).$$
	\end{conjecture}

	Although the definition of $\pDen(L)$ is very explicit, the computation of $\pDen(L)$ is a  challenging task, especially when $h>0$,  $I_{n,h}$ and the unimodular lattice $I_k$ lie in two different Jordan block. One way to compute $\pDen(L)$ is to decompose $\pDen(L)$ into primitive pieces as we introduce now. 
	
	Similarly to the local density polynomial, we define the primitive local density polynomial $\Pden(M, L, X)$    to be the polynomial in $\Q[X]$ such that
	\begin{equation}\label{eq:prim local density M,L}
		\Pden(M, L, (-\q)^{-k})\coloneqq\lim _{d \rightarrow+\infty} \frac{\# \operatorname{Pherm}_{L,M} (O_{F_0} / (\pi^d))}{q^{d \cdot \operatorname{dim} (\operatorname{Herm}_{L, M} )_{F_0}}},
	\end{equation}
	where
	\begin{align*}
		\mathrm{Pherm}_{L,M\obot H^k}(O_{F_0}/(\pi^d))&\coloneqq\{ \phi \in  	\mathrm{Herm}_{L,M\obot H^k}(O_{F_0}/(\pi^d)) \mid  \text{ $\phi$ is primitive}\}.
	\end{align*}
	Recall that $\phi\in	\mathrm{Herm}_{L,M\obot I_k}(O_{F_0}/(\pi^d))$ is primitive if $\dim_{\mathbb{F}_\q}((\phi (L)+\pi (M\obot I_k))/\pi (M\obot I_k)=n$. In particular, we have $\den(M,M)=\Pden(M,M)$ for any hermitian $O_F$-lattice $M$. We can also similarly define the normalized primitive local densities: 
	$$
	\Pden_{n,h}'(L)=\frac{\Pden'(I_{n,h}, L)}{\Den(I_{n,h},I_{n,h})},\quad \text{  } \quad \mathrm{Pden}_{n,t}(L)\coloneqq \frac{\Pden(\Lambda_t, L)}{\Den(\Lambda_t,\Lambda_t)},
	$$
	and 
	\begin{align*}
		\ppden_{n,h}(L)\coloneqq \mathrm{Pden}_{n,h}'(L)+\sum_{i=0}^{\lfloor\frac{(h-1)}{2}\rfloor}c_{n,h-1-2i}\cdot\mathrm{Pden}_{n,h-1-2i}(L).
	\end{align*}

	The following lemma decomposes local density polynomials into a summation of primitive local density polynomials. The following is essentially due to \cite{CY}. See also \cite[Theorem 3.5.1]{LZ}.
	\begin{lemma}\label{lem: ind formula to prim ld}
		Let $M$ and $L$ be lattices of rank $m$ and $n$. Then we have
		$$\den(M,L,X)=\sum_{L\subset L'\subset L_{F}}(\q^{n-m}X)^{\ell(L'/L)}\Pden(M,L',X),$$
		where $\ell(L'/L)=\mathrm{length}_{O_F}L'/L$. Here $\Pden(M,L',X)=0$ for $L'$ with fundamental invariant strictly less than the smallest fundamental invariant of $M$. In particular, the summation is finite.
	\end{lemma}
	
	Conversely, we can recover primitive local density polynomials as a linear combination of local density polynomials.
	\begin{theorem}\cite[Theorem 5.2]{HSY3}\label{thm:ind formula reducing valuation}
		Let $M$ and $L$ be lattices of rank $m$ and $n$. We have
		\begin{align*}
			\Pden(M, L,X)
			= \sum_{i=0}^{n} (-1)^{i} \q^{i(i-1)/2+i(n-m)}X^{i}
			\sum_{\substack{L \subset L' \subset \pi^{-1}L \\ \ell(L'/L)=i}} \den(M, L',X).
		\end{align*}
	\end{theorem}
	
	\begin{corollary}\label{cor: ind structure of partial Den(T)} Let $L$ be a lattice of rank $n$.		Then
		\begin{align*}	
			\ppden_{n,h}(L) = \sum_{i=0}^{n} (-1)^{i} \q^{i(i-1)/2}
			\sum_{\substack{L \subset L' \subset \pi^{-1} L \\ \ell(L'/L)=i}} \pden_{n,h}(L').
		\end{align*}

	\end{corollary}

	\begin{lemma}\label{lem: 0 vanish of error term}
		For two lattices $L$ and $M$ of the same rank $n$, we have
		\begin{align}\label{eq: Pden(M,L)}
			\Pden(M,L)=\begin{cases}
				\den(M,L) & \text{ if $M\cong L$},\\
				0 & \text{ if $M\not\cong L$}.
			\end{cases}
		\end{align}
		Moreover,
		\begin{align*}
			\den(M,L)=n(M,L)\cdot \den(M,M),
		\end{align*}
		where for two lattices $M,L\subset \bV$ of rank $n$,  $n(M,L)=|\{L'\subset L_F\mid  L\subset L', L'\cong M\}|. $
	\end{lemma}

	\begin{corollary}\label{cor: vanishing of error term}
		Assume $L\not \cong \Lambda_{t}$ for any vertex lattice $\Lambda_{t}$ with  $t <h$.  Then
		\begin{align*}
			\ppden_{n,h}(L)=\Pden_{n,h}'(L).
		\end{align*}
	\end{corollary}
	
	\begin{corollary} \label{cor: coefficient}
		Let $c_{n,t} $ be the coefficients in \eqref{eq:coeff} with  even $t$ and  $0<t\le t_{\max}$.  Then
		\begin{align*}
			c_{n,t} =-\Pden_{n,h}'(\Lambda_t).
		\end{align*}
	\end{corollary}
	\begin{proof}
		On the one hand, combining Corollary \ref{cor: ind structure of partial Den(T)} with \eqref{eq:coeff}, we obtain
		\begin{align*}
			\ppden_{n,h}(\Lambda_t)=0.
		\end{align*}
		On the other hand, by  Lemma \ref{lem: 0 vanish of error term} and \eqref{eq: def of pdenL},
		\begin{align*}
			\ppden_{n,h}(\Lambda_{t})=\Pden_{n,h}'(\Lambda_{t})+c_{n,t}.
		\end{align*}
	\end{proof}

	We will give another formulation of the conjecture \ref{conj: main section den} (conjecture \ref{conj: specialized}) in \S \ref{sec: weighted conj} which is based on the duality between $\cN_{n}^{[h]}$ and $\cN_{n}^{[n-h]}$, and is in fact more general (it takes care of the intersections between $\cZ$-cycles and $\cY$-cycles). The main terms of these two formulations are the same by direct calculations. However, interestingly, it is not clear that these two formulations have the same correction terms from the definition.

	\section{Our strategy: $\CN^{[2]}_4$}\label{section4}
	Our general strategy is closest to the unramified unitary case \cite{LZ} since $\CN^{[0]}_n \subset \CN^{[h]}_{n+h}$, and has several new ingredients which are quite complicated. Therefore, in this section, we consider the case $\CN^{[2]}_4$ and explain our strategy. Indeed, this is the first new case we proved and almost all ideas are essentially from this case. Therefore, we believe that this section will be helpful for readers. Since we want to explain our strategy in detail, we will freely use notations from the following sections.
	
	Let $\BV$ be the space of special homomorphisms (dimension $4$), and let $L^{\flat}\subset \BV$ be an $O_F$-lattice of rank 3. For any lattice $L'^{\flat}$ such that $L^{\flat} \subset L'^{\flat} \subset (L'^{\flat})^{\vee}\subset L_F^{\flat}$, we define the primitive part $\CZ(L'^{\flat})^{\circ}$ of the special cycle $\CZ(L'^{\flat})$ inductively by setting
	\begin{equation*}
		\CZ(L'^{\flat})^{\circ}\coloneqq \CZ(L'^{\flat})-\mathlarger{\sum}_{\substack{L'^{\flat} \subset L''^{\flat}\\L''^{\flat} \subset (L''^{\flat})^{\vee} \subset L'^{\flat}_F}} \CZ(L''^{\flat})^{\circ}.
	\end{equation*}
	
	Then, for $x \in \BV \backslash L_F^{\flat}$, the Kudla-Rapoport conjecture (Conjecture \ref{conjecture5.5}) is equivalent to
	\begin{equation*}
		\text{Int}_{L'^{\flat\circ}}(x)\coloneqq \chi(\CN^{[2]}_4, \text{}^{\BL}\CZ(L'^{\flat})^{\circ} \otimes^{\BL}  O_{\CZ(x)})=\mathlarger{\sum}_{L'^{\flat} \subset L' \subset L'^{\vee}, L'\cap L_F^{\flat}=L'^{\flat}}D_{4,2}(L')1_{L'}(x)=:\partial Den^{4,2}_{L'^{\flat\circ}}(x),
	\end{equation*}
	where $D_{4,2}(L')$ is the Cho-Yamauchi constant for $\CN^{[2]}_4$ (Definition \ref{definition5.9}).
	
	The first step is decomposing $\text{Int}_{L'^{\flat\circ}}(x)$ and $\partial \text{Den}^{4,2}_{L'^{\flat\circ}}(x)$ into horizontal and vertical parts:
	\begin{equation*}
		\begin{array}{l}
			\text{Int}_{L'^{\flat\circ}}(x)=\text{Int}_{L'^{\flat\circ},\ScH}(x)+\text{Int}_{L'^{\flat\circ},\ScV}(x),\\
			\partial \text{Den}^{4,2}_{L'^{\flat\circ}}(x)=\partial \text{Den}^{4,2}_{L'^{\flat\circ},\ScH}(x)+\partial \text{Den}^{4,2}_{L'^{\flat\circ},\ScV}(x).
		\end{array}
	\end{equation*}
	In the case of the good reduction $\CN^{[0]}_n$, this decomposition is relatively simple: if $L'^{\flat}$ has the fundamental invariants $(0,0,\dots,0,\alpha), \alpha \geq 1$, then $\text{Int}_{L'^{\flat\circ}}(x)$ and $\partial \text{Den}^{4,2}_{L'^{\flat\circ}}(x)$ are horizontal. Otherwise, $\text{Int}_{L'^{\flat\circ}}(x)$ and $\partial \text{Den}^{4,2}_{L'^{\flat\circ}}(x)$ are vertical.
	
	However, when $h \geq 1$, i.e., in the case of bad reduction, horizontal parts and vertical parts cannot be separated just by the fundamental invariants of $L'^{\flat}$. Indeed, even in $\CN^{[1]}_2$, $\CZ(L_1)^{\circ}$ $(L_1 \simeq (\pi))$ is a sum of horizontal parts and vertical parts. To understand this phenomenon, we did some explicit computation on the analytic side by using \cite{Cho4}: for $L_{\beta\gamma\delta} \in \BV$ of rank 3 with fundamental invariants $(\delta,\gamma,\beta)$, and $x \in \BV \backslash (L_{\beta\gamma\delta})_F$ with $\val ((x,x))=\alpha \geq \beta \geq \gamma\geq \delta$, we have
	\begin{equation*}\begin{array}{l}
			\partial \text{Den}_{L_{100}^{\circ}}^{4,2}(x)=\alpha/2, \quad\quad\quad\quad\quad\quad\quad\quad\quad\quad \quad\quad\quad\,\,\  \partial \text{Den}_{L_{300}^{\circ}}^{4,2}(x)=(q^2+q)\alpha/2+1-q^2,\\
			\partial \text{Den}_{L_{311}^{\circ}}^{4,2}(x)=(q^7+q^6)\alpha/2-(q^7-q^5-1), \quad\quad\quad \partial \mathrm{Den}_{L_{210}^{\circ}}^{4,2}(x)=q^3-q+1,\\
			\partial \text{Den}_{L_{221}^{\circ}}^{4,2}(x)=-(q^2-1)(q^2-q+1)(q^3+q^2+q+1),\\
			\partial \text{Den}_{L_{331}^{\circ}}^{4,2}(x)=-(q^2-1)(q^6+q^5+q^4+2q^3+q^2+1),\\
			\partial \text{Den}_{L_{322}^{\circ}}^{4,2}(x)=-(q^2-1)(q^2-q+1)(q^4+2q^3+q^2+q+1),\\
			\quad\quad \quad \quad \quad\quad \quad \quad\quad\quad \quad \quad \vdots
		\end{array}
	\end{equation*}
	The most interesting thing in this computation is the fact that $\partial \text{Den}_{L_{\beta\gamma\delta}^{\circ}}^{4,2}(x)$ does not depend on $\alpha$ if $(\beta,\gamma,\delta) \neq (\beta,0,0), (\beta,1,1)$. A similar phenomenon happens in the case of $\CN^{[0]}_n$ for vertical components since these are locally constant. Therefore, it is reasonable to guess that the horizontal parts of $\CZ(L^{\flat})$ are contained in $\CZ(L_{\beta\gamma\delta})^{\circ}$s for $(\beta,\gamma,\delta)=(\beta,0,0), (\beta,1,1)$ (which turns out to be true by Theorem \ref{thm: hori part}). Since $\CZ(L_{\beta00})^{\circ}, \CZ(L_{\beta11})^{\circ}$ have some vertical parts too, it should be handled carefully. Anyway, horizontal parts can be understood in this way, so for now, let us focus on the cases where $(\beta,\gamma,\delta) \neq (\beta,0,0), (\beta,1,1)$.
	
	Now, we guessed that $\CZ(L_{\beta\gamma\delta})^{\circ}$ is purely vertical if $(\beta,\gamma,\delta) \neq (\beta,0,0),(\beta,1,1)$. The next step is to understand the Fourier transforms
	\begin{equation*}\begin{array}{c}
			\widehat{\text{Int}}_{L_{\beta\gamma\delta}^{\circ},\ScV}(x) \overset{guess}{=}\widehat{\text{Int}}_{L_{\beta\gamma\delta}^{\circ}}(x),\\
			\widehat{\partial \text{Den}}_{L_{\beta\gamma\delta}^{\circ},\ScV}^{4,2}(x)\overset{guess}{=}\widehat{\partial \text{Den}}_{L_{\beta\gamma\delta}^{\circ}}^{4,2}(x),
		\end{array}
	\end{equation*}
	for $x \perp L_{\beta\gamma\delta}$, $\val((x,x))<0$. In the case of good reduction $\CN^{[0]}_n$, both $\widehat{\text{Int}}_{L'^{\flat\circ},\ScV}(x)$ and $\widehat{\partial \text{Den}}_{L'^{\flat\circ},\ScV}^{n,0}(x)$ vanish when $\val((x,x)) <0$, and hence
	\begin{equation}\label{eqint1.1}
		\widehat{\text{Int}}_{L'^{\flat\circ},\ScV}(x)-\widehat{\partial \text{Den}}_{L'^{\flat\circ},\ScV}^{n,0}(x)=0, \quad \text{for }\val((x,x))<0.
	\end{equation}
	
	Since \eqref{eqint1.1} is the most crucial property to prove the Kudla-Rapoport conjecture inductively, we have to show that \eqref{eqint1.1} holds in our cases.
	
	In $\CN^{[2]}_4$, we still have that $\widehat{\partial \text{Den}}_{L_{\beta\gamma\delta}^{\circ}}^{4,2}(x)=0$ if $\val((x,x)) \leq -2$, but $\widehat{\partial \text{Den}}_{L_{\beta\gamma\delta}^{\circ}}^{4,2}(x)$ is not zero for $\val((x,x))=-1$ (see Theorem \ref{theorem5.46}, Theorem \ref{theorem5.47}, Theorem \ref{theorem8.18}, Theorem \ref{theorem8.19}, Theorem \ref{theorem9.1.1}). Indeed, we can compute that for $\val((x,x))=-1$,
	\begin{equation*}
		\begin{array}{ll}
			\widehat{\partial \text{Den}}_{L_{444}^{\circ}}^{4,2}(x)=\dfrac{1}{q^2}(q^2-1)(q^3+1),&
			\widehat{\partial \text{Den}}_{L_{431}^{\circ}}^{4,2}(x)=-\dfrac{1}{q^2}(q+1)(q^3-q+1),\\\\
			\widehat{\partial \text{Den}}_{L_{440}^{\circ}}^{4,2}(x)=\dfrac{1}{q^2}(q^2-1),&
			\widehat{\partial \text{Den}}_{L_{310}^{\circ}}^{4,2}(x)=-\dfrac{1}{q^2}.
		\end{array}
	\end{equation*}
	What is the meaning of these numbers? This is the main obstruction when we try to prove the conjecture. Since we cannot compute the geometric side directly, it is not possible to prove the conjecture without knowing the meaning of these numbers.
	
	Fortunately, we had a table of the Cho-Yamauchi constants $D_{3,1}(L)$:
	\begin{equation*}
		\begin{array}{ll}
			D_{3,1}(L_{444})=-(q^2-1)(q^3+1),&
			D_{3,1}(L_{431})=(q+1)(q^3-q+1),\\
			D_{3,1}(L_{440})=-(q^2-1),&
			D_{3,1}(L_{310})=1.
		\end{array}
	\end{equation*}
	Now, it is easy to see that
	\begin{equation}\label{eqint1.2}
		\widehat{\partial \text{Den}}_{L_{\beta\gamma\delta}^{\circ}}^{4,2}(x)=-\dfrac{1}{q^2}D_{3,1}(L_{\beta\gamma\delta}).
	\end{equation}
	This is the most important observation in our work since this gives the following crucial ideas.
	
	First, note that $\widehat{\pDen}_{L_{\beta\gamma\delta}^{\circ}}^{4,2}(x)$ is a certain linear sum of the Cho-Yamauchi constants $D_{4,2}(L)$ for $\CN^{[2]}_4$ and the right-hand side of \eqref{eqint1.2} is the Cho-Yamauchi constants $D_{3,1}(L)$ for $\CN^{[1]}_3$. This suggests that there should be certain inductive relations among $D_{n,h}(L)$ and $D_{n-1,h-1}(L)$, $\forall 0 \leq h \leq n$ (see Theorem \ref{theorem5.24}). Since the Cho-Yamauchi constants for $\CN^{[h]}_n$ are very complicated (for example, $D_{6,2}(L)=(q-1)(q+1)^3(q^2-q+1)(q^{13}-q^{12}+q^{11}+q^{10}-2q^9+3q^8-3q^7+q^6-2q^4+2q^3-2q^2+q-1)$ for a lattice $L$ with fundamental invariants $(1,1,a_4,a_3,a_2,a_1)$, $a_i \geq 2$), we may not be able to find these inductive relations without the above observation \eqref{eqint1.2}.
	
	These inductive relations are the most important ingredients to understand the analytic side and by using them, we have a quite complete understanding on the analytic side of the Kudla-Rapoport conjecture for $\CN^{[h]}_n$, $\forall 0 \leq h \leq n$. 
	
	Second, note that the right-hand side of \eqref{eqint1.2} is the Cho-Yamauchi constant for $\CN^{[1]}_3$, not $\CN^{[2]}_3$. Also, note that we need a $Y$-cycle to get a reduction from $\CN^{[2]}_4$ to $\CN^{[1]}_3$ (see Proposition \ref{proposition2.5}). This means that $Y$-cycles appear when we take the Fourier transform of $\text{Int}_{L_{\beta\gamma\delta}^{\circ}}(x)=\chi(\CN^{[2]}_4, \text{}^{\BL}\CZ(L_{\beta\gamma\delta}^{\circ}) \otimes^{\BL}  O_{\CZ(x)})$. Indeed, the Kudla-Rapoport conjecture for $\CN_{3}^{[1]}$ holds, therefore,
	\begin{equation*}
		D_{3,1}(L_{\beta\gamma\delta})=\chi(\CN^{[2]}_4, \text{}^{\BL}\CZ(L_{\beta\gamma\delta})^{\circ} \otimes^{\BL}  O_{\CY(x)}), \quad \val((x,x))=-1.
	\end{equation*}
	Therefore, \eqref{eqint1.2} suggests that
	\begin{equation*}\begin{array}{rl}
			\widehat{\chi(\CN^{[2]}_4, \text{}^{\BL}\CZ(L_{\beta\gamma\delta})^{\circ} \otimes^{\BL}  O_{\CZ(x)})}=\widehat{\text{Int}}_{L_{\beta\gamma\delta}^{\circ}}(x)&\\
			\overset{conjecture}{=}\widehat{\partial \text{Den}}_{L_{\beta\gamma\delta}^{\circ}}^{4,2}(x)&=-\dfrac{1}{q^2}D_{3,1}(L_{\beta\gamma\delta})=-\dfrac{1}{q^2}\chi(\CN^{[2]}_4, \text{}^{\BL}\CZ(L_{\beta\gamma\delta})^{\circ} \otimes^{\BL}  O_{\CY(x)}).
		\end{array}
	\end{equation*}
	If we can prove this relation, then we can show that \eqref{eqint1.1} holds for all $\val(x,x)<0$, and we can use an inductive argument to prove the Kudla-Rapoport conjecture for $\CN^{[2]}_4$.
	
	By \cite[Lemma 6.3.1]{LZ} and \cite[Theorem 8.1]{ZhiyuZhang}, we know that
	\begin{equation*}
		\widehat{\chi(\CN^{[2]}_4, O_C \otimes^{\BL}  O_{\CZ(x)})}=-\dfrac{1}{q^2}\chi(\CN^{[2]}_4, O_C \otimes^{\BL}  O_{\CY(x)}),
	\end{equation*}
	if $C$ is a Deligne-Lusztig curve or $\BP^1$. Therefore, if $\CZ(L_{\beta\gamma\delta})^{\circ}$ is a linear sum of Deligne-Lusztig curves or $\BP^1$ (in the Grothendieck group of coherent sheaves), then we have
	\begin{equation*}
		\widehat{\chi(\CN^{[2]}_4, \text{}^{\BL}\CZ(L_{\beta\gamma\delta})^{\circ} \otimes^{\BL}  O_{\CZ(x)})}=-\dfrac{1}{q^2}\chi(\CN^{[2]}_4, \text{}^{\BL}\CZ(L_{\beta\gamma\delta})^{\circ} \otimes^{\BL}  O_{\CY(x)}).
	\end{equation*}
	
	This is how we make the Conjecture \ref{conjecture9.2.1} (this can be regarded as a variant of Tate conjectures for certain Deligne-Lusztig varieties), and we prove that if Conjecture \ref{conjecture9.2.1} holds, then the Kudla-Rapoport conjecture holds (see Theorem \ref{theorem9.1.3}). Then, we prove that Conjecture \ref{conjecture9.2.1} holds for $\CN^{[2]}_{4}$ and some other cases (see Theorem \ref{theorem9.2.1}). This is how we prove the Kudla-Rapoport conjecture for $\CN^{[2]}_4$.

	\section{Horizontal parts of Kudla–Rapoport cycles}

	In this section, we describe  the horizontal parts of special cycles following the approach of \cite[\S 4]{LZ}.  Due to the existence of non-trivial level structures, there are some new phenomena.
	
	Let $K$ denote a finite extension of $\breve{F}$. Consider $z \in \mathcal{N}_n^{[h]}(O_K)$ which corresponds to an $O_F$-hermitian module $G$ of signature $(1, n-1)$ over $O_K$. Let $T_p(-)$ denote the integral $p$-adic Tate modules and
	$$
	L\coloneqq \operatorname{Hom}_{O_F} (T_p \overline{\cE}, T_p G ).
	$$
	We can associate  $L$ with a hermitian form $\{x,y\}$ given by
	$$ (T_p \overline{\cE} \stackrel{x}{\rightarrow} T_p G \stackrel{\lambda_G}{\longrightarrow} T_p G^{\vee} \stackrel{y^{\vee}}{\longrightarrow} T_p \overline{\cE}^{\vee} \stackrel{\lambda_{\overline{\mathcal{\mathcal { E }}}}^{\vee}}{\longrightarrow} T_p \overline{\cE} ) \in \operatorname{End}_{O_F} (T_p \overline{\cE}) \cong O_F.$$
	One can check that $L$ is represented by the hermitian matrix $\diag((1)^{n-h},(\pi)^h)$.

	Following \cite[\S 4]{LZ}, we consider two injective $O_F$-linear isometric homomorphisms
	$$
	i_K: \operatorname{Hom}_{O_F}(\overline{\cE}, G)_F \rightarrow L_{F},
	$$
	and
	$$
	i_{\bar{k}}: \operatorname{Hom}_{O_F}(\overline{\cE}, G)_F \rightarrow \mathbb{V} .
	$$
	By \cite[Lemma 4.4.1]{LZ}, we have  
	\begin{align}\label{eq: intersection}
		\operatorname{Hom}_{O_F}(\overline{\cE}, G)=i_K^{-1}(L).
	\end{align}

	By the definition of special cycles, for any $O_F$-lattice $M \subset \mathbb{V}$, we have $z \in \mathcal{Z}(M)(O_K)$ if and only if $M \subset i_{\bar{k}} (\operatorname{Hom}_{O_F}(\overline{\cE}, G) )$. By \eqref{eq: intersection}, $z \in \mathcal{Z}(M)(O_K)$ if and only if
	\begin{align}\label{eq: equiv}
		M \subset i_{\bar{k}} (i_K^{-1}(L)).
	\end{align}

	Now assume that $z \in \mathcal{Z}(L^\flat)(O_K)$ corresponds to  an $O_F$-hermitian module $G$ of signature $(1, n-1)$ over $O_K$. By \eqref{eq: equiv}, we have
	$$
	L^\flat \subset i_{\bar{k}}(i_K^{-1}(L) ).
	$$
	Define $M^\flat\coloneqq L_F^\flat \cap i_{\bar{k}}(i_K^{-1}(L) )$. By \eqref{eq: equiv}, we still have $z \in \mathcal{Z} (M^\flat )(O_K)$. Moreover, we have
	$$
	M^\flat \stackrel{\sim}{\longrightarrow} L \cap i_K (i_{\bar{k}}^{-1} (L_F^\flat ) ).
	$$
	Set $\mathbb{W}=i_K (i_{\bar{k}}^{-1} (L_F^\flat ) )$, which has the same dimension as $L_F^\flat$.

	\begin{definition}
		Define $H(\bV)$ to be the collection of  $O_F$-lattices $M^\flat \subset \bV$ of rank $n-1$ such that 
		\begin{align*}
			M^\flat \approx \diag((1)^{n-h},(\pi)^{h-2},(\pi^a)) \text{ with $a\in \Z_{> 0}$}  \quad \text{or} \quad    M^\flat \approx \diag(1^{n-h-2},(\pi)^{h},(\pi^a))  \text{ with $a\in \Z_{\ge 0}$}.
		\end{align*}
	\end{definition}
	
	\begin{definition}  
		For any lattice $L'^{\flat}$ such that $L^{\flat} \subset L'^{\flat} \subset (L'^{\flat})^{\vee}\subset L_F^{\flat}$, we define the primitive part $\CZ(L'^{\flat})^{\circ}$ of the special cycle $\CZ(L'^{\flat})$ inductively by setting
		\begin{equation*}
			\CZ(L'^{\flat})^{\circ}\coloneqq \CZ(L'^{\flat})-\mathlarger{\sum}_{\substack{L'^{\flat} \subset L''^{\flat}\\L''^{\flat} \subset (L''^{\flat})^{\vee} \subset L'^{\flat}_F}} \CZ(L''^{\flat})^{\circ}.
		\end{equation*}
		Moreover, we define $\cY(L'^{\flat})^\circ$ similarly.
	\end{definition}

	\begin{theorem}\label{thm: hori part}
		Let $L^\flat \subseteq \mathbb{V}_n$ be a hermitian $O_F$-lattice of rank $n-1$. Then
		$$
		\mathcal{Z} (L^\flat )_{\mathscr{H}}=\bigcup_{\substack{L^\flat \subseteq M^\flat, \\ M^\flat \in H(\bV)}} \mathcal{Z} (M^\flat )_{\mathscr{H}}^{\circ}.
		$$
		Moreover, we have the following. 
		\begin{enumerate}
			\item If $M^\flat\approx \diag((1)^{n-h},(\pi)^{h-2},(\pi^a)) \text{ with $a\in \Z_{> 0}$}$, then $\cZ(M^\flat)_{\mathscr{H}}\cong \cZ(M^\flat)\simeq \cZ(x)\subset \cN_2^{[0]}$ with $\val(x)=a-1,$ which is a quasi-canonical lifting of degree $a-1$.  
			\item If  $  M^\flat \approx \diag(1^{n-h-2},(\pi)^{h},(\pi^a))  \text{ with $a\in \Z_{\ge 0}$}$, then  
			\begin{align*}  
				\cZ(M^\flat)^\circ_{\mathscr{H}}=\sum_{\substack{M^\flat \subset  M_1\oplus N_2 \subset \pi^{-1}M^\flat \\ N_2\approx  (\pi^{-1})^h}} \cZ(M_1)^\circ\cdot \cY(N_2)^\circ.
			\end{align*}
			Here, each $\cZ(M_1)^\circ\cdot \cY(N_2)^\circ$ is a quasi-canonical lifting of degree $a$. The summation index   has cardinality:
			\begin{align*}
				\begin{cases}
					q^{2n-2} & \text{ if } a \ge 1,\\
					q^{2n-2}\dfrac{1-(-q)^{-n}}{1-(-q)^{-1}} &\text{ if } a=0.
				\end{cases}
			\end{align*}
		\end{enumerate}
		
	\end{theorem}
	
	The rest of this section is devoted to the proof of Theorem \ref{thm: hori part}. We give the proof at the end of this section after some preparations.

	First, by the similar method as in the appendix of \cite{HSY3}, we can prove the following two lemmas.
	\begin{lemma}\label{lem: orbit of prim vector}
		Assume $L=L_1\obot L_2$ where $L_1\approx (1)^{n-h}$ and $L_2\approx (\pi)^h$. Let $x$ be a primitive vector in $L$. 
		\begin{enumerate}
			\item If $\mathrm{Pr}_{L_1}(x)$ is primitive in $L_1$, then there exists $L_1'\approx (1)^{n-h}$ such that $x\in L_1'$.
			\item If  $\mathrm{Pr}_{L_1}(x)$ is not primitive in $L_1$, then there exists $L_1'\approx (1)^{n-h}$ and $L_2'\approx (\pi)^h$ such that $L=L_1'\obot L_2'$ and $x\in L_2'$.
		\end{enumerate}

	\end{lemma}

	\begin{lemma}\label{lem: orbit of prim vector II}
		Assume $L\approx (1)^{n-h}$ or $(\pi)^h$. Then for any primitive vectors $x,x'\in L$ with $q(x)=q(x')$, there exists $g\in \mathrm{U}(L)$ such that $g(x)=x'$.
	\end{lemma}

	With the help of the above two lemmas, we can prove the following.
	\begin{lemma}\label{lemma4.3}
		Assume $L$ is a hermitian lattice represented by $\diag((1)^{n-h},(\pi)^h)$ and $\bW \subset L_F$ is a subspace of dimension $n-1$. Then $M^\flat \coloneqq \bW\cap L$ is represented either by   $\diag(1^{n-h-2},(\pi)^{h},(\pi^a))$ with $a\in \Z_{\ge 0}$ or $\diag((1)^{n-h},(\pi)^{h-2},(\pi^a))$ with $a\in \Z_{> 0}$. In the first case, we can write $M^\flat=M_1\obot M_2$ such that $L=L_1\obot M_2$.
	\end{lemma}
	
	\begin{proof}
		First, we assume $M^\flat_1\coloneqq M^\flat \cap L_{1,F}$ is not unimodular. Then the rank of $L_1$ is at least 2 and we choose a basis $\{e_1,\ldots,e_{n-h}\}$ of $L_1$ such that $e_1,e_2\in L_1$ and $(e_1,e_1)=(e_2,e_2)=0$, $(e_1,e_2)=1$ and $(e_i,e_j)=0$ for $i,j>2, i\neq j$. Since  $M_1^\flat$ is primitive in $L_{1}$ but non-isometric to $L_1$, we know the rank of $M_1^\flat$ is smaller than $n-h$. Hence the rank of $M_1^\flat$ has to be $n-h-1$ and $M_2^\flat\coloneqq M^\flat \cap L_{2,F}=L_2$. Now we choose a orthogonal basis $\{x_1,\ldots,x_{n-h-1}\}$ of $M_1^\flat$, where   we assume $q(x_{1})=(x_{1},x_{1})$ has largest valuation among $\{q(x_1),\ldots,q(x_{n-h-1})\}$.  In particular, $\val(q(x_1))>0$. By Lemmas \ref{lem: orbit of prim vector} and \ref{lem: orbit of prim vector II}, we may assume $x_{1}=e_1+\frac{q(x_1)}{2}e_2$.  Since $M_1^\flat $ is primitive, we can write $x_i=a_i(e_1-\frac{q(x_1)}{2}e_2)+\sum_{j=3}^{n-h}a_{ij}e_j$ and for each $j$ with $3\le j\le n-h$, there exists an $i$ such that $a_{ij}$ is a unit. Hence by possibly choosing a different basis $\{x_2,\ldots,x_{n-h-1}\}$, we may assume $x_i=a_i (e_1-\frac{q(x_1)}{2}) e_2)+e_{i+1}$. Now it is clear that $\spa\{x_2,\ldots,x_{n-h-1}\}$ is unimodular and $M^\flat \approx \diag((1)^{n-h-2},(\pi)^{h},(\pi^a))$.

		Now we assume  $M_1^\flat\coloneqq M^\flat \cap L_{1,F}$ is unimodular. If $\mathrm{rank}M_1^\flat=n-h-1$, then an argument as above shows that $M^\flat\approx \diag((1)^{n-h-1},(\pi)^h)$. If If $\mathrm{rank}M_1^\flat=n-h$, then $\mathrm{rank}M_2^\flat=h-1$, and a similar argument as before shows that $M_2^\flat \approx (\pi)^{h-2}\obot (\pi^a)$ with $a\in \Z_{>0}$. Then $M^\flat\approx (1)^{n-h}\obot (\pi)^{h-2}\obot (\pi^a)$ and $(1)^{n-h-1}\obot (\pi)^h$ respectively.    
	\end{proof}
	
	In particular,  Lemma \ref{lemma4.3} implies that if $z\in \cZ(L^\flat)(O_K)$, then $z \in  \cZ(M^\flat)(O_K)$ where $M^\flat \in H(\bV)$. 
	
	Now we assume $z\in \cZ(L^\flat)(O_K)$ corresponds to a $M^\flat$ represented by 
	$(1)^{n-h-2}\obot (\pi^a) \obot (\pi)^h$. By   Proposition \ref{proposition2.6}, 
	we may reduce to the case $M^\flat$ represented by 
	$(\pi^a)\obot (\pi)^{h}.$
	
	\begin{proposition}\label{prop: cancellation law full type part}
		Assume $z\in \cZ(L^\flat)(O_K)$ corresponds to a $M^\flat=M_1\obot M_2$, where $M_1$ is represented by 
		$(\pi^a)$ and $M_2$ is represented by $(\pi)^h$. If $L=L_1\obot M_2$, then $z \in \cZ(M_1)^\circ\cdot \cY(\pi^{-1}M_2)^\circ.$
	\end{proposition}
	\begin{proof}
		First, notice that $z\in \cZ(M^\flat)^\circ$ essentially by the definition of $M^\flat$. We may choose a basis $\{x_1,x_2\}$ of $L_1$ and a basis $\{x_3,\cdots,x_{h+2}\}$ of $M_2$ such that the moment matrix of $\{x_1,x_2\}$ and $\{x_3,\cdots,x_{h+2}\}$ is $(1)^2$ and $(\pi)^h$.   Let 
		\begin{align*}
			L_{z^\vee}\coloneqq  \operatorname{Hom}_{O_F} (T_p \overline{\cE}, T_p G^{\vee} ).
		\end{align*}
		Composing $L$ with $\lambda_G$, we obtain an embedding $L\hookrightarrow L_{z^\vee}$ such that $L_{z^\vee}=\lambda_G(L_1)\obot \pi^{-1} \lambda_G(M_2).$ Moreover, the moment matrix of $\lambda_G(L_1)\obot \pi^{-1} \lambda_G(M_2)$ is $(\pi)^2\obot (1)^h$. We can also directly check that $L_{z^\vee}\cap \lambda_G(i_K(i_{\bar{k}}^{-1}(L^\flat_F)))=M_1\obot \pi^{-1} \lambda_G(M_2)$. Hence, $z \in \cZ(M_1)^\circ\cdot \cY(\pi^{-1}M_2)^\circ$ by the definition of special cycles.
	\end{proof}
	
	\begin{lemma}
		Assume $M^\flat=M_1\obot M_2=M_1'\oplus M_2$. Then 
		$$\cZ(M_1)^\circ\cdot \cY(\pi^{-1}M_2)^\circ=\cZ(M_1')^\circ\cdot \cY(\pi^{-1}M_2)^\circ.$$ 
	\end{lemma}
	\begin{proof}
		First of all, we have $\cZ(M_1)^\circ\cdot \cY(\pi^{-1}M_2)^\circ\subset  \cZ(M_1)^\circ\cdot \cZ(M_2)^\circ=\cZ(M_1')^\circ\cdot \cZ(M_2)^\circ\subset \cZ(M_1')^\circ.$ Hence,  $\cZ(M_1)^\circ\cdot \cY(\pi^{-1}M_2)^\circ\subset \cZ(M_1')^\circ\cdot \cY(\pi^{-1}M_2)^\circ$. By switching the role of $M_1$ and $M_1'$, the lemma is proved.
	\end{proof}
	
	\begin{lemma}\label{lem: possible hor lattice}
		Let $M^\flat=M_1\obot M_2$, where $M_1=\mathrm{span}\{x_0\}$, and $M_2$ has a basis $\{x_1,\cdots,x_h\}$   with moment matrix $(\pi)^h$. 
		Assume $M^\flat \subset N^\flat \subset \pi^{-1} M^\flat$ and $N^\flat \approx \diag( \pi^a,(\pi^{-1})^h)$. Then $N^\flat = M_1\oplus \mathrm{span}\{\pi^{-1}(x_1+\alpha_1 x_0),\cdots,\pi^{-1}(x_h+\alpha_h x_0)\}$, where $\alpha_i$ is a representative of $O_F/(\pi)$.
	\end{lemma}
	\begin{proof}
		Since $M^\flat   \approx \diag( \pi^a,(\pi)^h)$ and $N^\flat \approx \diag( \pi^a,(\pi^{-1})^h)$,
		$N^\flat$ must contain a sub-lattice of the form $\mathrm{span}\{\pi^{-1}(x_1+\alpha_1 x_0),\cdots,\pi^{-1}(x_h+\alpha_h x_0)\}$. 
		Moreover, we have $M^\flat \stackrel{h}{\subset}N^\flat\subset \pi^{-1}M^\flat$. 
		
		Notice that $M^\flat \stackrel{h}{\subset} M_1\oplus \mathrm{span}\{\pi^{-1}(x_1+\alpha_1 x_0),\cdots,\pi^{-1}(x_h+\alpha_h x_0)\}\subset N^\flat$. Hence $N^\flat = M_1\oplus \mathrm{span}\{\pi^{-1}(x_1+\alpha_1 x_0),\cdots,\pi^{-1}(x_h+\alpha_h x_0)\}$. 
	\end{proof}
	
	\begin{lemma}
		Assume $M^\flat=M_1\obot M_2$ with basis $\{x_0,x_1,\cdots,x_h\}$ as in Lemma \ref{lem: possible hor lattice}. Let $N^\flat $ and $(N^\flat)'$ be two different lattices such that $M^\flat \subset N^\flat \subset \pi^{-1} M^\flat$, $M^\flat \subset (N^\flat)' \subset \pi^{-1} M^\flat$, and $N^\flat \approx (N^\flat)' \approx \diag(\pi^a,(\pi^{-1})^h).$ Write $N^\flat=M_1 \oplus N_2$ and $(N^\flat)'=M_1\oplus N_2'$. Then $$\cZ(M_1)^\circ\cdot \cY(N_2)^\circ(O_K)\neq \cZ(M_1)^\circ\cdot \cY^\circ(N_2')(O_K).$$
	\end{lemma}
	\begin{proof}
		Let $N^\flat = M_1\oplus \mathrm{span}\{\pi^{-1}(x_1+\alpha_1 x_0),\cdots,\pi^{-1}(x_h+\alpha_h x_0)\}$ and $(N^\flat)' = M_1\oplus \mathrm{span}\{\pi^{-1}(x_1+\alpha_1' x_0),\cdots,\pi^{-1}(x_h+\alpha_h' x_0)\}$ where $\alpha_i$ and $\alpha_i'$ are representatives of $O_F/(\pi)$.  If $\cZ(M_1)^\circ\cdot \cY(N_2)^\circ(O_K)= \cZ(M_1)^\circ\cdot \cY^\circ(N_2')(O_K),$ then we have nontrivial $z\in \cZ(M_1)^\circ\cdot \cY(N_2)^\circ(O_K)\cap \cZ(M_1)^\circ\cdot \cY^\circ(N_2')(O_K)$. This implies that $z\in \cY(\pi^{-1}M_1)$.   In particular, $z\in \cY(\pi^{-1}M_1)\cdot \cY(\mathrm{span}\{\pi^{-1}x_1 ,\cdots,\pi^{-1}x_h\})(O_K).$ However, 
		notice that 
		\begin{align*}
			\cY(\pi^{-1}M_1)\cdot \cY(\mathrm{span}\{\pi^{-1}x_1 ,\cdots,\pi^{-1}x_h\}) \cong \cZ(\pi^{-1}M_1)\cdot \cY(\mathrm{span}\{\pi^{-1}x_1 ,\cdots,\pi^{-1}x_h\}) 
		\end{align*}
		by cancellation law and the fact that $\cZ(\pi^{-1}M_1)=\cY(\pi^{-1}M_1)$ in $\cN_2$. This contradicts the fact $z\not\in \cZ(\pi^{-1}M_1)^{\circ}$.
	\end{proof}

	\begin{proof}[Proof of Theorem \ref{thm: hori part}]
		Assume $z\in \cZ(M^\flat)^\circ(O_K)$. By Lemma \ref{lemma4.3}, we know that   $z \in  \cZ(M^\flat)(O_K)$ where $M^\flat \in H(\bV)$.

		If $M^\flat =M_1\obot M_2$, where $M_1\approx (1)^{n-h}$ and $M_2\approx \diag((\pi)^{h-2},(\pi^a))$ with $a>0$, then according to Proposition \ref{proposition2.6}, $\cZ(M^\flat)\cong \cZ(M_2)\subset \cN_{h}^{[h]}$. Then applying the duality between $\cN_{h}^{[h]}$  and  $\cN_{h}^{[0]}$, we have $\cZ(M_2)\subset \cN_{h}^{[h]}$ is isomorphic to $\cZ(M_2')\subset \cN_{h}^{[0]}$, where $M_2'\cong \diag((1)^{h-2},(\pi^{a-1}))$. Then   by   Proposition \ref{proposition2.6} and \cite[Proposition 8.1]{KR1}, $\cZ(M^\flat)^\circ $ is isomorphic to a quasi-canonical lifting of degree $a-1$. 
		
		Now we assume $M^\flat \approx \diag(1^{n-h-2},(\pi)^{h},(\pi^a))  \text{ with $a\in \Z_{\ge 0}$}.$ Then according to Proposition \ref{prop: cancellation law full type part}, and Lemma \ref{lem: possible hor lattice}, we have $z \in \cZ(M_1)^\circ\cdot \cY(N_2)^\circ(O_K)$ for some $N_2$, where $N_2=\mathrm{span}\{\pi^{-1}(x_1+\alpha_1 x_0),\cdots,\pi^{-1}(x_h+\alpha_h x_0)\}$ and each $\alpha_i$ is some representative of $O_F/(\pi)$. If   $M_1\cong \diag(1^{n-h-2},\pi^a)$ and $M_2\cong (\pi)^{h}$, then by  Proposition \ref{proposition2.6} and \cite[Proposition 8.1]{KR1} again, $\cZ(M_1)^\circ\cdot \cY(\pi^{-1}M_2)^\circ$ is isomorphic to a quasi-canonical lifting of degree $a$.

		According to the above discussion, we have
		\begin{align}\label{eq: hor dec}
			\mathcal{Z} (L^\flat )_{\mathscr{H}}=\bigcup_{\substack{L^\flat \subseteq M^\flat, \\ M^\flat \in H(\bV)}} m(M^\flat) \mathcal{Z} (M^\flat )_{\mathscr{H}}^{\circ},
		\end{align}
		and $\mathcal{Z} (M^\flat )_{\mathscr{H}}^{\circ}$ has the desired description.

		Now we show $m(M^\flat)=1$ for any $M^\flat$ in the above identity  following the  proof of \cite[Theorem 4.2.1]{LZ} closely. It suffices to show $O_K[\epsilon]$-points of both sides of \eqref{eq: hor dec} are the same where $\epsilon^2=0$. First, each $O_K$-point of the right hand side of \eqref{eq: hor dec} has a unique lift to an $O_K[\epsilon]$-point. Therefore we only need to show each $z\in \cZ(M^\flat)_{\mathscr{H}}^\circ(O_K)$ has a unique lift in $\cZ(M^\flat)_{\mathscr{H}}^\circ(O_K[\epsilon])$.

		Let $G$ be the corresponding $O_F$-hermitian module of signature $(1,n-1)$ over $O_K$ and $\mathbb{D}(G)$ be the (covariant) $O_{F_0}$-relative Dieudonné crystal of $G$.  First, we have an action $O_F \otimes_{O_{F_0}} O_K \simeq O_K \oplus O_K$ on $\mathbb{D}(G)(O_K)$ induced by the action of  $O_F$ via $\iota: O_F \rightarrow \operatorname{End}(G)$, and hence a $\mathbb{Z} / 2 \mathbb{Z}$-grading on $\mathbb{D}(G)(O_K)$. Then we let $\mathscr{A}=$ $\operatorname{gr}_0 \mathbb{D}(G)(O_K)$. It is a free $O_K$-module of rank $n$  equipped with an $O_K$-hyperplane: $\mathscr{H}\coloneqq \operatorname{Fil}^1 \mathbb{D}(G)(O_K) \cap \mathscr{A}$ by the Kottwitz signature condition. Note that $\mathscr{H}$ contains the image of $L^\flat$ under the identification of \cite[Lemma 3.9]{KR1}.

		Note that the kernel of $O_K[\epsilon] \rightarrow O_K$ has a natural nilpotent divided power structure. Then according to  Grothendieck-Messing theory, a lift $\tilde{z} \in \mathcal{Z}(L^\flat)(O_K[\epsilon])$ of $z$ corresponds to an $O_K[\epsilon]$ direct summand  of $\mathbb{D}(G)(O_K[\epsilon])$ $\widetilde{\mathrm{Fil}}$ that lifts $\mathrm{Fil}^1 \mathbb{D}(G)$  and  contains the image of $L^\flat$ in $\widetilde{\mathscr{A}}$.    Here,  $\widetilde{\mathrm{Fil}}$ is isotropic under the natural pairing $\langle \, , \, \rangle_{\mathbb{D}(G)(O_K[\epsilon])}$ on $\mathbb{D}(G)(O_K[\epsilon])$ induced by the polarization. Since $L^\flat \subset \operatorname{Hom}_{O_F}(\overline{\cE}, G)$ has rank $n-1$, by Breuil's theorem \cite[§4.3]{LZ},   we know that the image of $L^\flat$ in $\operatorname{gr}_0 \mathbb{D}(G)(S)$ has rank $n-1$ over $S$ (the Breuil's ring) and thus its image in the base change $\mathscr{A}$ has rank $n-1$ over $O_K$.    In particular, $\mathrm{gr_0}\widetilde{\Fil}$  is the unique $O_K[\epsilon]$-hyperplane $\widetilde{\mathscr{H}}$ of $\mathrm{gr}_0\mathbb{D}(G)(O_K[\epsilon])$ that contains the $O_K[\epsilon]$-module spanned by the image of $L^\flat$ in $\widetilde{\mathscr{A}}$.
		
		To determine $\mathrm{gr}_1\widetilde{\mathrm{Fil}}$, we note that  $\widetilde{\mathrm{Fil}}$ is a direct summand of $\mathbb{D}(G)(O_K[\epsilon])$ with rank $n$ containing $\widetilde{\mathscr{H}}$. Since $\widetilde{\mathrm{Fil}}$ is isotropic under $\langle \, , \, \rangle_{\mathbb{D}(G)(O_K[\epsilon])}$, we have $\mathrm{gr}_1\widetilde{\mathrm{Fil}}\subset (\widetilde{\mathscr{H}})^\perp \cap \mathrm{gr}_1 \mathbb{D}(G)(O_K[\epsilon])$. Here $(\widetilde{\mathscr{H}})^\perp$ is the perpendicular subspace in $\mathbb{D}(G)(O_K[\epsilon])$ with respect to $\langle \, , \, \rangle_{\mathbb{D}(G)(O_K[\epsilon])}$. Moreover, since $\mathrm{det}\langle \, , \, \rangle_{\mathbb{D}(G)(O_K[\epsilon])}\neq 0$ in $O_K[\epsilon]$, we have  $(\widetilde{\mathscr{H}})^\perp$ has rank $n+1$. Note that  $\mathrm{gr}_0\mathbb{D}(G)(O_K[\epsilon])$ is also isotropic under $\langle \, , \, \rangle_{\mathbb{D}(G)(O_K[\epsilon])}$. In particular, $\mathrm{gr}_0\mathbb{D}(G)(O_K[\epsilon])\subset (\widetilde{\mathscr{H}})^\perp$ which has rank $n$. Hence $(\widetilde{\mathscr{H}})^\perp\cap \mathrm{gr}_1 \mathbb{D}(G)(O_K[\epsilon])$ is of rank one.  Since $\widetilde{\Fil}$ is a direct summand of $\mathbb{D}(G)(O_K[\epsilon])$, we know $\mathrm{gr}_1\widetilde{\mathrm{Fil}}=(\widetilde{\mathscr{H}})^\perp \cap \mathrm{gr}_1 \mathbb{D}(G)(O_K[\epsilon])$.  Hence $\widetilde{\Fil}=\mathrm{gr}_0\widetilde{\Fil}\oplus \mathrm{gr}_1\widetilde{\Fil}$ is uniquely determined, and the lift $\tilde{z}\in \cZ(L^\flat)(O_K[\epsilon])$ of $z$ is unique. Hence $\widetilde{\Fil}=\mathrm{gr}_0\widetilde{\Fil}\oplus \mathrm{gr}_1\widetilde{\Fil}$ is uniquely determined, and the lift $\tilde{z}\in \cZ(L^\flat)(O_K[\epsilon])$ of $z$ is unique.
		
		We defer the proof of the cardinality of the summation index to Lemma \ref{lemma5.48}, where it is proved via analytic method.
	\end{proof}

	\section{Local modularity and Tate conjectures}
	First, as we have discussed in the introduction, we propose the following local modularity conjecture   motivated by the analytic computation and the special case for $\cN_n^{[0]}$ (see \cite[Corollary 5.3.3]{LZ}).
	\begin{conjecture}\label{conjecture6.1.1}
		For the Rapoport-Zink space $\CN^{[h]}_n$ and an $O_F$-lattice $L^{\flat} \subset \BV$ of rank $n-1$, we have
		\begin{equation*}
			\widehat{\Int}_{L^\flat,\mathscr{V}}(\cZ(x)) =-  \dfrac{1}{q^h}\mathrm{Int}_{L^\flat,\mathscr{V}}(\cY(x)).
		\end{equation*}
	\end{conjecture}
	\begin{remark}
		More precisely, as we wrote in Section \ref{section4}, Conjecture \ref{conjecture6.1.1} is primarily motivated by \eqref{eqint1.2} and, more generally, by Theorem \ref{theorem9.1.1} below along with \cite[Lemma 6.3.1]{LZ} and \cite[Theorem 8.1]{ZhiyuZhang}.
	\end{remark}
	
	\begin{conjecture}\label{conjecture9.2.1}(cf. \cite[Corollary 5.3.3]{LZ}) For the Rapoport-Zink space $\CN^{[h]}_{n}$ and an $O_F$-lattice $L^{\flat} \subset \BV$ of rank $n-1$, there are finitely many Deligne-Lusztig curves $C_i\subset \cN_{3}^{[0]}\hookrightarrow \cN_{n}^{[h]}$, projective lines $D_i\subset \cN_{2}^{[1]}\hookrightarrow \cN_{n}^{[h]}$ and $\text{mult}_{C_i}$, $\text{mult}_{D_i} \in \BQ$ such that for any $x \in \BV \backslash L^{\flat}_F$,
		\begin{equation*}
			\chi(\CN^{[h]}_n, \text{}^{\BL}\CZ(L^{\flat})_{\ScV} \otimes^{\BL} O_{\CZ(x)}) = \sum_i \text{mult}_{C_i} \chi(\CN^{[h]}_n, O_{C_i} \otimes^{\BL} O_{\CZ(x)})+\sum_i \text{mult}_{D_i} \chi(\CN^{[h]}_n, O_{D_i} \otimes^{\BL} O_{\CZ(x)}).
		\end{equation*}
	\end{conjecture}
	
	\begin{remark}
		Indeed, as we wrote in Section \ref{section4}, Conjecture \ref{conjecture6.1.1} follows from Conjecture \ref{conjecture9.2.1}. Conjecture \ref{conjecture9.2.1} is implied by a stronger version of Tate conjectures on $1$-cycles for Deligne--Lusztig varieties $Y_\Lambda$ in Proposition \ref{proposition6.4.1} (cf. \cite[Theorem 5.3.2]{LZ}).
	\end{remark}
	
	\begin{theorem}\label{theorem9.2.1} Conjecture \ref{conjecture9.2.1} holds for $\CN^{[0]}_{n}, \CN^{[1]}_{n},\CN^{[n-1]}_{n}, \CN^{[n]}_{n}$, and $\CN^{[2]}_{4}$.
	\end{theorem}
	
	To prove this theorem, we follow \cite[Section 5.3]{LZ}. First, we need to recall several notations and theorems from \cite{cho2018basic}. In \cite[Theorem 1.1]{cho2018basic}, we proved that the reduced subscheme of $\CN^{[h]}_{n}$ has a Bruhat-Tits stratification and their components are certain Deligne-Lusztig varieties $Y_{\Lambda}$ where $\Lambda$ is a vertex lattice of type $t(\Lambda)$. More precisely, we have the following proposition.
	\begin{proposition}\cite[Theorem 1.1]{cho2018basic}\label{proposition6.4.1} Let $\CN^{[h]}_{n,red}$ be the underlying reduced subscheme of $\CN^{[h]}_{n}$. Then, we have
		\begin{equation*}
			\CN^{[h]}_{n,red}=\mathlarger{\cup}_{t(\Lambda) \leq h-1} Y_{\Lambda} \text{ } \mathlarger{\cup}\text{ }  \mathlarger{\cup}_{t(\Lambda) \geq h+1} Y_{\Lambda},
		\end{equation*}
		where $Y_{\Lambda}$ denotes certain Deligne-Lusztig varieties associated with vertex lattices $\Lambda$. Also, the dimension of $Y_{\Lambda}$ is $\frac{1}{2}(t(\Lambda)+h-1)$ (resp. $\frac{1}{2}(t(\Lambda)+n-h-1)$) if $t(\Lambda) \geq h+1$ (resp. $t(\Lambda) \leq h-1$).
	\end{proposition}
	
	By \cite[Corollary 5.3.3]{LZ}, we know that the Conjecture \ref{conjecture9.2.1} holds for $\CN^{[0]}_{n}$ and $\CN^{[n]}_{n}$. Also, by \cite[Theorem 1.1]{cho2018basic}, the irreducible components of the reduced subscheme of $\CN^{[2]}_4$ are $\BP^2$ and their Chow groups are well-known. Therefore, let us focus on $\CN^{[1]}_{n},\CN^{[n-1]}_{n}$.
	Since $\CN^{[1]}_{n}$ is isomorphic to $\CN^{[n-1]}_{n}$, we only need to consider $\CN^{[1]}_{n}$. In this case, by Proposition \ref{proposition6.4.1} (see \cite[Theorem 1.1]{cho2018basic} for more detail), we know that the reduced subscheme of $\CN^{[1]}_n$ has a Bruhat-Tits stratification and their components are Deligne-Lusztig varieties $Y_{\Lambda}$ where $t(\Lambda) \geq 2$ or projective spaces $\BP^n_{\Lambda}$, where $t(\Lambda)=0$. Let us describe $Y_{\Lambda}$ more precisely.
	
	Let $k_F$ be the residue field of $F$ and let $V_{2d+2}$ be the unique (up to isomorphism) $k_F/k$-hermitian space of dimension $2d+2$. Let $\Lambda/\pi \Lambda^{\vee}=V_{2d+2}$ where $t(\Lambda)=2d+2$, and let $J_{2d+2}$ be the special unitary group associated to $(V_{2d+2},(\cdot,\cdot))$. Let $(W_{2d+2},S_{2d+2})$ be the Weyl system of $J_{2d+2}$ and let $B_{2d+2}$ be the standard Borel subgroup. For $I \in S_{2d+2}$, we define $W_I$ as the subgroup of $W_{2d+2}$ generated by $I$ and $P_I\coloneqq B_{2d+2}W_IB_{2d+2}$. Note that $W_{2d+2}$ can be identified with a symmetric group and $S_{2d+2}$ with $\lbrace s_1,\dots,s_{2d+1} \rbrace$ where $s_i$ is the transposition of $i$ and $i+1$. We write
	\begin{equation*}\begin{array}{l}
			I_0\coloneqq \lbrace s_1,\dots s_d,s_{d+2},\dots,s_{2d+1}\rbrace,\\
			I_i\coloneqq \lbrace s_1,\dots, s_{d-i},s_{d+i+2},\dots,s_{2d+1}\rbrace, 1 \leq i \leq d,\\
			P_i\coloneqq P_{I_i}.
		\end{array}
	\end{equation*}
	Note that $P_d=B_{2d+2}$. Also, note that the elements in $J_{2d+2}/P_i$ parametrize flags
	\begin{equation*}
		0 \subset T_{-i}\overset{1}{\subset} T_{-i+1} \dots \overset{1}{\subset} T_{-1} \overset{1}{\subset} \overline{A} \overset{1}{\subset} \overline{B} \overset{1}{\subset} T_1 \dots \overset{1}{\subset} T_i \subset V_{2d+2}.
	\end{equation*}
	
	Now, by \cite[Theorem 1.1]{cho2018basic}, we know that the reduced subscheme of $\CN^{[1]}_n$ is
	\begin{equation}\label{eq9.6}
		\CN^{[1]}_{n,red}=\cup_{\Lambda, t(\Lambda)\geq2} Y_{\Lambda} \cup \cup_{\Lambda, t(\Lambda)=0} \BP^{n}_{\Lambda}.
	\end{equation}
	
	Here, $Y_{\Lambda}=X_{P_0}(id) \sqcup X_{P_0}(s_{d+1})$ where $2d+2=t(\Lambda)$(see \cite[Definition 3.10]{cho2018basic}). By \cite[Lemma 2.21]{cho2018basic} (cf. \cite[Lemma 2.1]{Vo}), we have the following statement (cf. \cite[Theorem 2.15]{Vo}).
	
	\begin{proposition}(cf. \cite[Theorem 2.15]{Vo})
		There is a decomposition of $X_{P_0}(id) \sqcup X_{P_0}(s_{d+1})$ into a disjoint union of locally closed subvarieties
		\begin{equation*}
			X_{P_0}(id) \sqcup X_{P_0}(s_{d+1})=\sqcup_{i=0}^{d} \lbrace X_{P_i}(s_{d+2}\dots s_{d+i+1}) \sqcup X_{P_i}(s_{d+1}s_{d+2}\dots s_{d+i+1})\rbrace.
		\end{equation*}
	\end{proposition}
	\begin{proof}
		We can follow the proof of \cite[Theorem 2.15]{Vo} with \cite[Lemma 2.21]{cho2018basic}.
	\end{proof}
	
	Since $P_d=B_{2d+2}$, we have that $X_{P_d}(s_{d+2}\dots s_{2d+1})$ and $X_{P_d}(s_{d+1}s_{d+2}\dots s_{2d+1})$ are classical Deligne-Lusztig varieties. For $0 \leq i \leq d$, let us write
	\begin{equation*}\begin{array}{l}
			X_{i,1}^{\circ}\coloneqq X_{P_i}(s_{d+1}\dots s_{d+i+1}),\\
			X_{i-1,2}^{\circ}\coloneqq X_{P_i}(s_{d+2}\dots s_{d+i+1}),\\
			Y_{i,1}^{\circ}\coloneqq X_{B_{2i+2}}(s_{i+1}\dots s_{2i+1}),\\
			Y_{i-1,2}^{\circ}\coloneqq X_{B_{2i+2}}(s_{i+2}\dots s_{2i+1}),\\
			\widetilde{X}_i^{\circ}\coloneqq X_{i,1}^{\circ} \sqcup X_{i,2}^{\circ},\\
			\widetilde{X}_i\coloneqq \sqcup_{m=0}^{i}\widetilde{X}_i^{\circ},\\
			Y_{d}\coloneqq \sqcup_{i=0}^{d} \lbrace X_{P_i}(s_{d+2}\dots s_{d+i+1}) \sqcup X_{P_i}(s_{d+1}s_{d+2}\dots s_{d+i+1})\rbrace.
		\end{array}
	\end{equation*}
	
	Then, by the above proposition and a Bruhat-Tits stratification in \cite[Theorem 1.1]{cho2018basic}, we have that $X_{i,1}^{\circ}$ (resp. $X_{i,2}^{\circ}$) is a disjoint union of isomorphic copies of $Y_{i,1}^{\circ}$ (resp. $Y_{i,2}^{\circ}$). Also, the dimensions of $Y_{i,1}^{\circ}$ and $Y_{i,2}^{\circ}$ are $i+1$.
	
	For any $k_F$-variety $S$, we write $H^j(S)(i)$ for $H^j(S_{\overline{k}},\overline{\BQ_l}(i))$ where $l \neq p$ is a prime and $\overline{k}$ is an algebraically closed field containing $k_F$. Let $\ScF=\text{Fr}_{k_F}$ be the $q^2$-Frobenius on $H^j(S)(i)$. Then, the following analogous statement of \cite[Lemma 5.3.1]{LZ} holds.

	\begin{lemma}\label{lemma9.6}(cf. \cite[Lemma 5.3.1]{LZ}) For any $d,i \geq 0$ and $s \geq 1$, the action of $\ScF^s$ on the following cohomology groups are semisimple, and the space of $\ScF^s$-invariants is zero when $j \geq 1$.
		\begin{enumerate}
			\item $H^{2j}(Y_{d,1}^{\circ})(j)$.
			\item $H^{2j}(Y_{d-1,2}^{\circ})(j)$.
			\item $H^{2j}(\widetilde{X}_i^{\circ})(j)$.
			\item $H^{2j}(Y_d-\widetilde{X}_i)(j)$.
		\end{enumerate}
		
	\end{lemma}
	\begin{proof} Here, we follow the proof of \cite[Lemma 5.3.1]{LZ} with some modification.
		\begin{enumerate}
			\item By \cite[(7.3) $(\text{}^2\text{A}_{2d+1},\ScF)$]{lusztig1976coxeter} (or we refer to the proof of \cite[Lemma 2]{Ohm}), we have the following table on the eigenvalues of $\ScF$ on $H^{j}_{c}(Y_{d,1}^{\circ})$.
			\begin{equation*}
				\begin{array}{c|c|c|c|c|c|c|c|c}
					j & d+1 & d+2 & d+3 & \dots & 2d-1 & 2d & 2d+1 & 2d+2\\
					\hline
					& 1 & q^2 & q^4 & \dots & q^{2d-4} & q^{2d-2} & q^{2d} & q^{2d+2}\\
					\hline
					& -q^3 & -q^5 & -q^7 &\dots & -q^{2d-1} &&&.
				\end{array}
			\end{equation*}
			
			By the Poincar\'e duality, we have a perfect pairing
			\begin{equation*}
				H^{2d+2-j}_{c}(Y_{d,1}^{\circ}) \times H^{j}(Y_{d,1}^{\circ})(d+1) \rightarrow H^{2d+2}_c(Y_{d,1}^{\circ})(d+1)\simeq \overline{\BQ}_l.
			\end{equation*}
			Therefore, the eigenvalues of $\ScF$ on $H^{2j}(Y^{\circ}_{d,1})(j)$ are given by $q^{2(d+1-j)}$ times the inverse of the eigenvalues in $H^{2(d+1-j)}_c(Y_{d,1}^{\circ})$. More precisely,
			\begin{equation*}
				\begin{array}{ccc}
					& H^{2(d+1-j)}_c(Y_{d,1}^{\circ}) & q^{2(d+1-j)} \times \text{the inverse}\\
					j=0 & q^{2d+2} & 1\\
					j=1 & q^{2d-2} & q^2 \\
					j \geq 2 & q^{2d+2-4j}, -q^{2d-4j+5} & q^{2j}, -q^{2j-3}.
				\end{array}
			\end{equation*}
			Therefore, the eigenvalue of $\ScF^s$ cannot be $1$ when $j \geq 1$. The semisimplicity of the action of $\ScF^s$ is from \cite[6.1]{lusztig1976coxeter}.
			
			\item By \cite[(7.3) $(\text{A}_{d+1},\ScF)$]{lusztig1976coxeter} (or we refer to the proof of \cite[Lemma 2]{Ohm}, here, note that $\ScF$ is $q^2$-Frobenius), we have the following table on the eigenvalues of $\ScF$ on $H^{j}_{c}(Y_{d,2}^{\circ})$.
			
			\begin{equation*}
				\begin{array}{c|c|c|c|c|c|c|c|c}
					j & d+1 & d+2 & d+3 & \dots & 2d-1 & 2d & 2d+1 & 2d+2\\
					\hline
					& 1 & q^2 & q^4 & \dots & q^{2d-4} & q^{2d-2} & q^{2d} & q^{2d+2}.
				\end{array}
			\end{equation*}
			By the Poincar\'e duality, we have a perfect pairing
			\begin{equation*}
				H^{2d+2-j}_{c}(Y_{d,2}^{\circ}) \times H^{j}(Y_{d,2}^{\circ})(d+1) \rightarrow H^{2d+2}_c(Y_{d,2}^{\circ})(d+1)\simeq \overline{\BQ}_l.
			\end{equation*}
			Therefore, the eigenvalues of $\ScF$ on $H^{2j}(Y^{\circ}_{d,2})(j)$ are given by $q^{2(d+1-j)}$ times the inverse of the eigenvalues in $H^{2(d+1-j)}_c(Y_{d,2}^{\circ})$. More precisely,
			\begin{equation*}
				\begin{array}{ccc}
					& H^{2(d+1-j)}_c(Y_{d,2}^{\circ}) & q^{2(d+1-j)} \times \text{the inverse}\\
					j=0 & q^{2d+2} & 1\\
					j=1 & q^{2d-2} & q^2 \\
					j \geq 2 & q^{2d+2-4j} & q^{2j}.
				\end{array}
			\end{equation*}
			Therefore, the eigenvalue of $\ScF^s$ cannot be $1$ when $j \geq 1$. The semisimplicity of the action of $\ScF^s$ is from \cite[6.1]{lusztig1976coxeter}.
			
			\item This follows from (1) and (2) since $\widetilde{X}_i^{\circ}$ is a disjoint union of $Y_{i,1}^{\circ}$ and $Y_{i,2}^{\circ}$.
			
			\item This follows from (3) since $Y_d-\widetilde{X}_i^{\circ}=\sqcup_{m=i+1}^{d} \widetilde{X}_m^{\circ}$.
			
		\end{enumerate}
	\end{proof}
	
	\begin{theorem}\label{theorem9.7}(cf. \cite[Theorem 5.3.2]{LZ}) For any $0 \leq i \leq d+1$ and any $s \geq 1$, we have
		\begin{enumerate}
			\item The space of Tate classes $H^{2i}(Y_{d})(i)^{\ScF^s=1}$ is spanned by the cycle classes of the irreducible components of $\widetilde{X}_i$.
			
			\item Let $H^{2i}(Y_d)(i)_1 \subset H^{2i}(Y_d)(i)$ be the generalized eigenspace of $\ScF^s$ for the eigenvalue $1$. Then $H^{2i}(Y_d)(i)_1=H^{2i}(Y_d)(i)^{\ScF^s=1}$.
		\end{enumerate}
	\end{theorem}
	\begin{proof}
		The proof is the same as \cite[Theorem 5.3.2]{LZ} with Lemma \ref{lemma9.6}
	\end{proof}
	
	\begin{proof}[Proof of Theorem \ref{theorem9.2.1}] Here, we follow the proof of \cite[Corollary 5.3.3]{LZ}. For $\CN^{[0]}_n$ and $\CN^{[n]}_n$, this is from \cite[Corollary 5.3.3]{LZ}. For $\CN^{[1]}_{n}$ and $\CN^{[n-1]}_{n}$, note that $\CN^{[1]}_{n}$ and $\CN^{[n-1]}_{n}$ are isomorphic, so we only need to consider the case $\CN^{[1]}_{n}$. By \cite[Theorem 1.1]{cho2018basic}, we have \eqref{eq9.6}:
		\begin{equation*}
			\CN^{[1]}_{n,red}=\cup_{\Lambda, t(\Lambda)\geq2} Y_{\Lambda} \cup \cup_{\Lambda, t(\Lambda)=0} \BP^{n}_{\Lambda},
		\end{equation*}
		and any curve $C$ in $\CN^{[1]}_{n,red}$ lies in some $Y_{\Lambda} \simeq Y_d$ or the projective space $\BP^{n}_{\Lambda}$ for some vertex lattice $\Lambda$. If it lies on $Y_d$, then by Theorem \ref{theorem9.7}, the cycle class of $C$ can be written as a $\BQ$-linear combination of the cycle classes of the irreducible components of $\widetilde{X}_1$ and these are projective lines. Similarly, if it lies on $\BP^{n}_{\Lambda}$, then by the Chow group of $\BP^{n}_{\Lambda}$, we have that the cycle class of $C$ can be written as a $\BQ$-linear combination of projective lines. This finishes the proof of the theorem for $\CN^{[1]}_{n}$ and $\CN^{[n-1]}_{n}$. 
		
		For $\CN^{[2]}_{4}$, we know that the irreducible components of $\CN^{[2]}_{4,red}$ are $\BP^{2}_{\Lambda}$ for vertex lattices $\Lambda$, and hence by the description of Chow group of $\BP^2_{\Lambda}$, the cycle class of $C$ can be written as a $\BQ$-linear combination of projective lines. Moreover, the projective lines here can be realized as images of embeddings of projective lines in $\cN_{2}^{[1]}$.    
		
		Finally, the finiteness of curves is from Lemma \ref{lemma2.7}. This finishes the proof of the theorem.
	\end{proof}

	\section{Weighted representation densities and conjectures}\label{sec: weighted conj}
	In this section, we first recall the definition of weighted representation densities and formulas from \cite[Section 3.1]{Cho}. Then, we will recall the conjectural formula in \cite[Conjecture 3.17, Conjecture 3.25]{Cho}.
	
	We denote by $^*$ the nontrivial Galois automorphism of $F$ over $F_0$. We fix the standard additive character $\psi:F_0 \rightarrow \BC^{\times}$ that is trivial on $O_{F_0}$. Let $V^+$ (resp. $V^{-}$) be a split (resp. non-split) $2n$-dimensional hermitian vector space over $F$ and let $\CS((V^{\pm})^{2n})$ be the space of Schwartz functions on $(V^{\pm})^{2n}$. Let $V_{r,r}$ be the split hermitian space of signature $(r,r)$ and let $L_{r,r}$ be a self-dual lattice in $V_{r,r}$. Let $\phi_{r,r}$ be the characteristic function of $(L_{r,r})^{2n}$. Let $(V^{\pm})^{[r]}$ be the space $V^{\pm} \otimes V_{r,r}$. For any function $\phi \in \CS((V^{\pm})^{2n})$, we define a function $\phi^{[r]}$ by $\phi \otimes \phi_{r,r} \in \CS(((V^{\pm})^{[r]})^{2n})$.
	
	Let $\Gamma_{n}$ be the Iwahori subgroup
	\begin{equation*}
		\Gamma_{n}\coloneqq \lbrace \gamma=(\gamma_{ij}) \in GL_{n}(O_{F}) \mid \gamma_{ij} \in \pi O_{F} \text{ if }i>j \rbrace.
	\end{equation*}
	
	We define the set $V_{n}(F)$ by
	\begin{equation*}
		V_{n}(F)=\lbrace Y \in M_{n,n}(F) \mid \text{}^tY^*=Y \rbrace.
	\end{equation*}
	
	We define the set $X_{n}(F)$ by
	\begin{equation*}
		X_{n}(F)=\lbrace X \in GL_{n}(F) \mid \text{}^tX^*=X \rbrace.
	\end{equation*}
	
	For $g \in GL_{n}(F)$ and $X \in X_{n}(F)$, we define the group action of $GL_{n}(F)$ on $X_{n}(F)$ by $g \cdot X=gX^tg^*$. For $X, Y \in V_{n}(F)$, we denote by $\langle X,Y \rangle=\mathrm{Tr}(XY)$. For $X \in M_{m,n}(F)$ and $A \in V_{m}(F)$, we denote by $A[X]=^tX^*AX.$ For a hermitian matrix $A \in X_m(F)$, we define
	\begin{equation*} A^{[r]}=\left(\begin{array}{cc} A & \\ & I_{2r} \end{array}\right).\end{equation*}
	
	Now, let us recall the definition of usual representation densities.
	\begin{definition}
		For $A \in X_m(O_F)$ and $B \in X_n(O_F)$, we define $\alphad(A,B)$ by
		\begin{equation*}
			\alphad(A,B)=\lim_{d \rightarrow \infty} (q^{-d})^{n(2m-n)}\vert \lbrace x \in M_{m,n}(O_F/\pi^dO_F) \mid A[x]\equiv B (\text{mod} \pi^d)\rbrace\vert.
		\end{equation*}
	\end{definition}
	
	Now, let us recall the definition of weighted representation densities in \cite{Cho}.
	\begin{definition}\cite[Definition 3.1]{Cho} Let $0 \leq h,t \leq n$. Let $L_t$ be a lattice of rank $2n$ in $V^+$ if $t$ is even (resp. in $V^{-}$ if $t$ is odd) with hermitian form 
		\begin{equation*}
			A_t\coloneqq \left(\begin{array}{cc} I_{2n-t} &\\
				&\pi^{-1}I_t
			\end{array}\right).
		\end{equation*}
		Let $1_{h,t} \in \CS((V^{\pm})^{2n})$ be the characteristic function of $(L_t^{\vee})^{2n-h} \times L_t^h$.
		
		For $B \in X_{2n}(F)$, we define
		\begin{equation*}
			W_{h,t}(B,(-q)^{-2r})\coloneqq \int_{V_{2n}(F)}\int_{M_{2n+2r,2n}(F)}\psi(\langle Y,A_t^{[r]}[X]-B\rangle)1_{h,t}^{[r]}(X)dXdY.
		\end{equation*}
		Here, $dY$ (resp. $dX$) is the Haar measure on $V_{2n}(F)$ (resp. $M_{2n+2r,2n}(F)$) such that
		\begin{equation*}
			\int_{V_{2n}(O_F)}dY=1 \text{ }(\text{resp. }\int_{M_{2n+2r,2n}(O_F)}dX=1 ).
		\end{equation*}
	\end{definition}
	
	\bigskip
	The functions $\alphad(A,B)$ and $W_{h,t}(B,r)$ have the following formulas.
	
	\begin{lemma}\label{lemma5.3}(\cite{Hi}, \cite[Lemma 3.5]{Cho}) For $A \in X_{m}(F)$ and $B \in X_{2n}(F)$, we have that
		
		\begin{equation*}
			\alphad(A^{[r]},B)=\mathlarger{\sum}_{Y \in \Gamma_{2n} \backslash X_{2n}(F)}\dfrac{\CG(Y,B)\CF(Y,A^{[r]})}{\alpha(Y;\Gamma_{2n})},
		\end{equation*}
		and
		\begin{equation*}
			W_{h,t}(B,(-q)^{-2r})=\mathlarger{\sum}_{Y \in \Gamma_{2n} \backslash X_{2n}(F)}\dfrac{\CG(Y,B)\CF_{h}(Y,A_t^{[r]})}{\alpha(Y;\Gamma_{2n})}.
		\end{equation*}
		Here, we define $\CF(Y,A^{[r]})$ by
		\begin{equation*}
			\CF(Y,A^{[r]})\coloneqq \mathlarger{\int}_{M_{m,2n}(F)}\psi(\langle Y,A^{[r]}[X] \rangle)dX,
		\end{equation*}
		and we define $\CF_h(Y,A_t^{[r]})$ by
		\begin{equation*}
			\CF_h(Y,A_t^{[r]})\coloneqq \mathlarger{\int}_{M_{2n+2r,2n}(F)}\psi(\langle Y,A_t^{[r]}[X] \rangle)1_{h,t}^{[r]}(X)dX.
		\end{equation*}
		
		We define $\CG(Y,B)$ by
		\begin{equation*}
			\CG(Y,B)\coloneqq \mathlarger{\int}_{\Gamma_{2n}}\psi(\langle Y, -B[\gamma]\rangle)d\gamma,
		\end{equation*}
		where $d\gamma$ is the Haar measure on $M_{2n,2n}(O_F)$ such that $\int_{M_{2n,2n}(O_F)}d\gamma=1$.
		
		Also, we define $\alpha(Y;\Gamma_{2n})$ by
		\begin{equation*}
			\alpha(Y;\Gamma_{2n})\coloneqq \lim_{d\rightarrow \infty} q^{-4dn^2}N_d(Y;\Gamma_{2n}),
		\end{equation*}
		where $N_d(Y;\Gamma_{2n})=\vert \lbrace \gamma \in \Gamma_{2n}( \text{mod } \pi^d) \vert \gamma \cdot Y \equiv Y ( \text{mod } \pi^{d})\rbrace \vert$.

	\end{lemma}
	
	\begin{definition} For $r \geq 0$, $A \in X_{n+2r}(O_F)$, and $B \in X_n(O_F)$, we can regard $\CF(Y,A^{[r]})$, $\alphad(A^{[r]},B)$, $\CF_h(Y,A_t^{[r]})$, and $W_{h,t}(B,(-q)^{-2r})$ as functions of $X=(-q)^{-2r}$. We define
		\begin{equation*}
			\CF'(Y,A)\coloneqq -\dfrac{d}{dX}\CF(Y,A^{[r]})\vert_{X=1},
		\end{equation*}
		and
		\begin{equation*}
			\CF_h'(Y,A_t)\coloneqq -\dfrac{d}{dX}\CF_h(Y,A_t^{[r]})\vert_{X=1}.
		\end{equation*}
		
		Also, we define
		\begin{equation*}
			\alphad'(A,B)=-\dfrac{d}{dX}\alphad(A,B;X)\vert_{X=1},
		\end{equation*}
		and
		\begin{equation*}
			W'_{h,t}(B)\coloneqq -\dfrac{d}{dX}W_{h,t}(B,r)\vert_{X=1}.
		\end{equation*}
	\end{definition}
	
	\begin{definition}\label{definition5.5}\cite[Proposition 2.7]{Cho3} For $0 \leq i, h \leq 2n$, we define the constant $\beta_i^h$ by
		
		\begin{equation*}
			\beta^h_i=\alpha_{i+1,h}^{-1}\Biggl(\dfrac{\mathlarger{\prod}_{1 \leq m \leq 2n, m \neq i+1} (1-x_m)}{\mathlarger{\prod}_{1 \leq m \leq 2n+1, m \neq i+1}(x_m-x_{i+1})}\Biggl),
		\end{equation*}
		where
		\begin{equation*}
			\begin{array}{lr}
				\alpha_{i,h}=(-q)^{(n+1-i)(2n-h)},& 1 \leq i \leq n;\\
				\alpha_{i,h}=(-q)^{(2n+1-i)(2n+h)},& n+1 \leq i \leq 2n;\\
				\alpha_{2n+1,h}=1,
			\end{array}
		\end{equation*}
		and
		\begin{equation*}
			\begin{array}{lr}
				x_i=(-q)^{n+1-i},& 1 \leq i \leq n; \\
				x_i=(-q)^{i-2n-1},& n+1 \leq i \leq 2n;\\
				x_{2n+1}=1.
			\end{array}
		\end{equation*}
		
	\end{definition}
	
	Now, we can state \cite[Conjecture 3.17, Conjecture 3.25]{Cho}.
	
	\begin{conjecture}\label{conjecture5.5}\cite[Conjecture 3.17, Conjecture 3.25]{Cho}
		
		\quad For a basis $\lbrace x_1, \dots, x_{2n-m}, y_1, \dots, y_m\rbrace$ of $\BV$, and special cycles $\CZ(x_1)$,$\dots$,$\CZ(x_{2n-m})$, $\CY(y_1)$,$\dots$, $\CY(y_m)$ in $\CN^{[n]}_{2n}$, we have
		\begin{equation*}
			\begin{array}{l}
				\chi(\CN_{2n}^{[n]},O_{\CZ(x_1)}\otimes^{\BL}\dots\otimes^{\BL}O_{\CY(y_m)}) 
				=\dfrac{1}{W_{n,n}(A_n,1)}\lbrace W'_{m,n}(B)-\mathlarger{\sum}_{0 \leq i \leq n-1} \beta_i^mW_{m,i}(B,1)\rbrace.
			\end{array}
		\end{equation*}
		Here, $\chi$ is the Euler-Poincar\'e characteristic and $\otimes^{\BL}$ is the derived tensor product. Also, $B$ is the matrix
		\begin{equation*}
			B=\left(\begin{array}{cc} 
				(x_i,x_j) & ( x_i,y_l)\\
				( y_k,x_j)& (y_k,y_l)
			\end{array} 
			\right)_{1\leq i,j \leq 2n-m, 1\leq k,l \leq m}.
		\end{equation*}
	\end{conjecture}
	Assume that special homomophisms $\lbrace x_1, x_2, \dots, x_n, x_{n+1},\dots,x_{n+h},y_1,\dots,y_{n-h} \rbrace$ has the hermitian matrix:
	\begin{equation}\label{eq5.1}
		B=\left(\begin{array}{cc} 
			( x_i,x_j) & (x_i,y_l)\\
			( y_k,x_j)& ( y_k,y_l)
		\end{array} 
		\middle)_{\substack{1\leq i,j \leq n+h,\\ 1\leq k,l \leq n-h}}=\middle(\begin{array}{ccc}
			T & &\\
			& I_{h} & \\
			&& \pi^{-1}I_{n-h}
		\end{array}\right),
	\end{equation}
	for some $n \times n$ matrix $T$. Then, by Proposition \ref{proposition2.5} and Proposition \ref{proposition2.6}, the arithmetic intersection number of special cycles $\chi(\CN_{2n}^{[n]},O_{\CZ(x_1)}\otimes^{\BL}\dots\otimes^{\BL}O_{\CZ(x_{n+h})} \otimes^{\BL}O_{\CY(y_1)}\otimes^{\BL}\dots\otimes^{\BL}O_{\CY(y_{n-h})})$ in $\CN^{[n]}_{2n}$ can be identified with $\Int_{n,h}(T)=\Int_{n,h}(L)=\chi(\CN_{n}^{[h]},O_{\CZ(x_1)}\otimes^{\BL}\dots\otimes^{\BL}O_{\CZ(x_n)})$ in $\CN^{[h]}_{n}$, where $L=\spa_{O_F}\{x_1,\cdots,x_n\}$. We note that the valuation of the determinant of $B$ and $h+1$ have the same parity. Now, Conjecture \ref{conjecture5.5} is specialized to the following conjecture. 
	
	\begin{conjecture}\label{conj: specialized}
		\quad Consider a basis $\lbrace x_1, \dots, x_{n+h}, y_1, \dots, y_{n-h}\rbrace$ of $\BV$ with moment matrix $B$ as in \eqref{eq5.1}. Let $L=\spa_{O_F}\{x_1,\cdots,x_n\}$. Then
		\begin{equation}\label{eq: specialized conj}
			\begin{array}{l}
				\Int_{n,h}(L)
				=\Int_{n,h}(T)=\dfrac{1}{W_{n,n}(A_n,1)}\lbrace W'_{n-h,n}(B)-\mathlarger{\sum}_{0 \leq i \leq n-1} \beta_i^{n-h}W_{n-h,i}(B,1)\rbrace.
			\end{array}
		\end{equation}
	\end{conjecture}
	
	Note that $\Int_{n,h}(L)$ is exactly the intersection number considered in Conjecture \ref{conj: main section den}. We show the analytic sides of Conjectures \ref{conj: main section den} and \ref{conj: specialized} also match in \S \ref{sec: CY constant}.

	\section{Cho-Yamauchi constants}\label{sec: CY constant}
	In this section, we will modify the result of \cite{Cho3} to get the Cho-Yamauchi constants in the case of $\CN^{[h]}_n$. More precisely, we want to write the conjectural formula for the arithmetic intersection numbers of special cycles $\chi(\CN_{n}^{[h]},O_{\CZ(x_1)}\otimes^{\BL}\dots\otimes^{\BL}O_{\CZ(x_n)})$ in $\CN^{[h]}_n$ as a linear sum of representation densities. First, we start with the following proposition.

	\begin{proposition}\label{proposition5.7}
		Assume that $B$ is of the form in \eqref{eq5.1}. Then, we have
		\begin{align*}
			W_{n-h,n}'(B)&=q^{-4n^2+(n+h)(n-h)}\alphad(\pi A_n,I_{n-h})\alphad'(I_{n+h,h},\left( \begin{array}{cc}T&\\& I_h\end{array}\right) )
			\\&=q^{-4n^2+(n+h)(n-h)}\alphad(\pi A_n,I_{n-h})\alphad(I_{n+h,h},I_h)\alphad'(I_{n,h},T).
		\end{align*}
	\end{proposition}
	\begin{proof}
		One can use a similar method as in \cite[Corollary 9.12]{KR1} to prove this. For example, see \cite[Proposition A.3, Proposition A.4, (A.0.4), (A.0.5)]{Cho}.
	\end{proof}
	
	Note that in Proposition \ref{proposition5.7}, the terms $\alphad(\pi A_n, I_{n-h})$ and $\alphad(I_{n+h,h},I_h)$ are constants, and $\alphad'(I_{n,h},T)$ is the derivative of a usual representation density. Now, by \cite{Cho3}, this can be written as a linear sum of usual representation densities. Let us follow the steps in \cite[Section 4.1]{Cho3}. For this, we need to introduce some notations.
	
	\begin{definition}\begin{enumerate}
			\item	We write $\CR_n$ for the set
			\begin{equation*}
				\CR_n=\lbrace Y_{\sigma,e} \mid (\sigma,e) \in \CS_n \times \BZ^n, \sigma^2=1, e_i=e_{\sigma(i)}, \forall i \rbrace,
			\end{equation*}
			where $\CS_n$ is the symmetric group of degree $n$, and
			\begin{equation*}
				Y_{\sigma,e}=\sigma \left( \begin{array}{ccc} \pi^{e_1} & & 0 \\
					& \ddots & \\ 0 & & \pi^{e_n} \end{array}\right).
			\end{equation*}
			Then $\CR_n$ forms a complete set of representatives of $\Gamma_n \backslash X_{n}(F)$.
			
			\item	We write $\CR_n^{0+}$ for the set
			\begin{equation*}
				\CR_n^{0+}=\lbrace \lambda=(\lambda_1,\dots,\lambda_n) \in \BZ^n \mid \lambda_1 \geq \dots \geq \lambda_n \geq 0 \rbrace.
			\end{equation*}
			
			\item For $\lambda \in \CR_n^{0+}$, we define $A_{\lambda}$ by
			\begin{equation*}
				A_{\lambda}=\left(\begin{array}{ccc} \pi^{\lambda_1} & & \\ & \ddots & \\ & & \pi^{\lambda_n} \end{array}\right).
			\end{equation*}
			
			\item For $\lambda=(\lambda_1, \dots ,\lambda_n) \in \CR_n^{0+}$, we define $\vert \lambda \vert$ by
			\begin{equation*}
				\vert \lambda \vert= \sum_{i=1}^{n}\lambda_i.
			\end{equation*}
		\end{enumerate}
	\end{definition}
	
	\begin{definition}[Cho-Yamauchi constant]\label{definition5.9} Assume that $B$ is of the form in \eqref{eq5.1}.
		For $\lambda \in \CR_n^{0+}$, we define $D_{n,h}(\lambda)$ to be the constant satisfying
		\begin{equation*}
			\dfrac{1}{W_{n,n}(A_n,1)}\lbrace W'_{n-h,n}(B)-\mathlarger{\sum}_{0 \leq t \leq n-1} \beta_t^{n-h}W_{n-h,t}(B,1)\rbrace=\mathlarger{\sum}_{\lambda \in \CR_n^{0+}}D_{n,h}(\lambda)\dfrac{\alphad(A_{\lambda},T)}{\alphad(A_{\lambda},A_{\lambda})}.
		\end{equation*}
		The existence and uniqueness of these constants are from \cite{Cho3}. The constant $D_{n,h}(\lambda)$ is a version of the Cho-Yamauchi constant in \cite{CY}.
	\end{definition}
	
	\begin{remark}\label{remark5.10}
		In $\CN^h(1,n-1)$, let $L$ be a rank $n$ $O_F$-lattice generated by special homomorphisms $x_1, \dots, x_{n}$ in $\BV$. Assume that $T$ is the hermitian matrix of $L$. Then, the valuation of the determinant of $T$ and $h+1$ have the same parity. Therefore, in Definition \ref{definition5.9}, the terms $\alphad(A_{\lambda},T)$ such that
		\begin{equation*}
			\val(\det(A_{\lambda}))=\sum_i \lambda_i \not\equiv h+1 (\text{mod 2})
		\end{equation*}
		are always equal to 0.
	\end{remark}
	
	Now, let us compute the correction terms $\dfrac{-\mathlarger{\sum}_{0 \leq t \leq n-1} \beta_t^{n-h}W_{n-h,t}(B,1)}{W_{n,n}(A_n,1)}$.
	
	\begin{proposition}\label{proposition5.27}
		Assume that $B$ is of the form in \eqref{eq5.1}. If $n-h \leq t \leq n-1$, we have
		\begin{align*}
			W_{n-h,t}(B,1)&=q^{-4n^2+(n+h)(3n-2t-h)}\alphad(\pi A_t, I_{n-h})
			\alphad(I_{n+h,t-n+h},I_h)\alphad(I_{n,t-n+h},T).
		\end{align*}
		If $t < n-h$, we have that $W_{n-h,t}(B,1)=0$.
	\end{proposition}
	\begin{proof}
		One can use a similar method as in \cite[Corollary 9.12]{KR1} to prove this. For example, see \cite[Proposition A.3, Proposition A.4, (A.0.4), (A.0.5)]{Cho}.
	\end{proof}
	
	\begin{proposition}\label{proposition5.28} For $n-h \leq t \leq n-1$, we have
		\begin{equation*}
			\dfrac{\beta_t^{n-h}W_{n-h,t}(B,1)}{W_{n,n}(A_n,1)}=\dfrac{-(-q)^{\frac{(n-t)(n-t-1-2h)}{2}}}{1-(-q)^{-(n-t)}}\dfrac{\alphad(I_{n,t-n+h},T)}{\alphad(I_{n,t-n+h},I_{n,t-n+h})}.
		\end{equation*}
	\end{proposition}
	\begin{proof}
		By Definition \ref{definition5.5}, we have that
		\begin{equation*}\begin{array}{l}
				\beta_t^{n-h}=(-q)^{-(n-t)(n+h)}\dfrac{(-1)^{n-1}(-q)^{\frac{n(n+1)}{2}-(n-t)}\mathlarger{\prod}_{l=n+1-t}^{n}(1-(-q)^{-l})\mathlarger{\prod}_{l=1}^{n-t-1}(1-(-q)^{-l})\mathlarger{\prod}_{l=1}^{n}(1-(-q)^{-l})}{(-1)^{t}(-q)^{2n(n-t)+\frac{t(t+1)}{2}}\mathlarger{\prod}_{l=1}^{t}(1-(-q)^{-l})\mathlarger{\prod}_{l=1}^{2n-t}(1-(-q)^{-l})}\\\\
				=(-1)^{n-t-1}(-q)^{-\frac{(n-t)(1+2h+5n-t)}{2}}\dfrac{\prod_{l=n+1-t}^{n}(1-(-q)^{-l})\prod_{l=1}^{n-t-1}(1-(-q)^{-l})\prod_{l=1}^{n}(1-(-q)^{-l})}{\prod_{l=1}^{t}(1-(-q)^{-l})\prod_{l=1}^{2n-t}(1-(-q)^{-l})}.
			\end{array}
		\end{equation*}
		Also, by \cite[Proposition A.4]{Cho}, we have that
		\begin{equation*}
			\alphad(\pi A_t,I_{n-h})=\alphad(\left(\begin{array}{cc}\pi I_{2n-t} & \\ & I_{t}\end{array}\right),I_{n-h})=\prod_{l=t-n+h+1}^{t}(1-(-q)^{-l}),
		\end{equation*}
		\begin{equation*}
			\alphad(I_{n+h,t-n+h},I_{h})=\prod_{l=2n-h-t+1}^{2n-t}(1-(-q)^{-l}),
		\end{equation*}
		\begin{equation*}
			\alphad(I_{n,t-n+h},I_{n,t-n+h})=q^{(t-n+h)^2}\prod_{l=1}^{2n-t-h}(1-(-q)^{-l})\prod_{l=1}^{t-n+h}(1-(-q)^{-l}),
		\end{equation*}
		and
		\begin{equation*}
			W_{n,n}(A_n,1)=q^{-3n^2}\prod_{l=1}^{n}(1-(-q)^{-l})^2.
		\end{equation*}
		Combining these and Proposition \ref{proposition5.27}, we get the proposition.
	\end{proof}
	
	Combining Proposition \ref{proposition5.7} and Proposition \ref{proposition5.28}, we get the following corollary which compares the analytic sides of Conjecture \ref{conj: main section den} and Conjecture \ref{conjecture5.5}.
	
	\begin{corollary}\label{corollary7.7}
		Assume that $B$ is of the in \eqref{eq5.1}. By Proposition \ref{proposition5.7}, Proposition \ref{proposition5.27}, Proposition \ref{proposition5.28}, and \cite[Proposition A.4]{Cho}, we can write $\dfrac{1}{W_{n,n}(A_n,1)}\lbrace W'_{n-h,n}(B)-\mathlarger{\sum}_{0 \leq t \leq n-1} \beta_t^{n-h}W_{n-h,t}(B,1)\rbrace$ in terms of usual representation densities as follows:
		\begin{equation*}
			\dfrac{1}{W_{n,n}(A_n,1)}\lbrace W'_{n-h,n}(B)-\mathlarger{\sum}_{0 \leq t \leq n-1} \beta_t^{n-h}W_{n-h,t}(B,1)\rbrace
		\end{equation*}
		\begin{equation*}\begin{aligned}
				&=\dfrac{\alphad'(I_{n,h},T)}{\alphad(I_{n,h},I_{n,h})}
				+\mathlarger{\sum}_{k=0}^{h-1}\frac{(-q)^{-\frac{(h-k)(h+k+1)}{2}}}{1-(-q)^{-(h-k)}}\dfrac{\alphad(I_{n,k},T)}{\alphad(I_{n,k},I_{n,k})}.
			\end{aligned}
		\end{equation*}
	\end{corollary}
	\bigskip
	By Corollary \ref{corollary7.7}, Conjecture \ref{conj: specialized} can be rewritten as follows.
	\begin{conjecture}[$\CZ$-cycles in $\CN^{[h]}_{n}$]\label{conjecture7.8}
		For a basis $\lbrace x_1, \dots, x_{n}\rbrace$ of $\BV$, and special cycles $\CZ(x_1)$,$\dots$,$\CZ(x_{n})$ in $\CN^{[h]}_{n}$, we have
		\begin{equation*}
			\begin{array}{l}
				\chi(\CN_{n}^{[h]},O_{\CZ(x_1)}\otimes^{\BL}\dots\otimes^{\BL}O_{\CZ(x_n)}) 
				=\dfrac{\alphad'(I_{n,h},T)}{\alphad(I_{n,h},I_{n,h})}
				+\mathlarger{\sum}_{k=0}^{h-1}\frac{(-q)^{-\frac{(h-k)(h+k+1)}{2}}}{1-(-q)^{-(h-k)}}\dfrac{\alphad(I_{n,k},T)}{\alphad(I_{n,k},I_{n,k})}.
			\end{array}
		\end{equation*}
		
		Here, $\chi$ is the Euler-Poincar\'e characteristic and $\otimes^{\BL}$ is the derived tensor product. Also, $T$ is the matrix
		\begin{equation*}
			T=\left(\begin{array}{c} 
				(x_i,x_j)
			\end{array} 
			\right)_{1\leq i,j \leq n}.
		\end{equation*}
		
	\end{conjecture}

	\begin{proposition}\label{prop: equiv of conjs}
		Conjecture \ref{conj: main section den} is equivalent to Conjecture \ref{conjecture7.8}.
	\end{proposition}
	\begin{proof}
		The intersection numbers from both conjectures are by definition the same. Hence we only need to show the analytic sides of both conjectures agree. By Corollary \ref{corollary7.7}, we only need to show $\beta_t^{n-h}$ from Conjecture \ref{conjecture5.5}  is the same as $c_{n,t}$ from Conjecture \ref{conj: main section den}.    Since the $c_{n,t}$ is characterized by $\ppden_{n,h}(I_{n,t})=0$ for $t\le h-1$ and $t\equiv h+1 \pmod 2$, we only need to show $D_{n,h}(I_{n,t})=0$ for $t\le h-1$ and $t\equiv h+1 \pmod 2$. This is proved in Proposition \ref{prop: vanishing of D(lambda)} via a method we used throughout \S \ref{sec: ind formula} so we postpone the proof.
	\end{proof}
	
	\begin{remark}
		Before we start to find the constants $D_{n,h}(\lambda)$, let us provide a brief explanation of the forthcoming steps. By Proposition \ref{proposition5.7}, we know that $\dfrac{W'_{n-h,n}(B)}{W_{n,n}(A_n,1)}$ is a constant multiple of $\alphad'(I_{n,h},T)$. Also, we know how to write $\alphad'(I_n,T)$ in terms of a linear sum of representation densities by \cite{CY} and \cite[Theorem 3.5.1]{LZ} (see Proposition \ref{proposition5.13} below).
		
		Now, we consider $\alphad'(I_{n,h},T)-(-q)^{nh}\alphad'(I_n,T)$. Then, it is possible to write the difference between these two in terms of a certain linear sum of representation densities (see \eqref{eq5.11}). Therefore, we can use the linear sum expression of $(-q)^{nh}\alphad'(I_n,T)$ and the difference $\alphad'(I_{n,h},T)-(-q)^{nh}\alphad'(I_n,T)$ to find all constants $D_{n,h}(\lambda)$. This is what we will do in the next few pages.
	\end{remark}
	
	\begin{definition}
		\begin{enumerate}
			\item  For $Y \in \CR_n$, we define
			\begin{equation*}
				t_0(Y)=\vert \lbrace e_i \mid e_i \geq 0 \rbrace \vert,
			\end{equation*}
			and for $k \geq 1$, we define
			\begin{equation*}
				t_k(Y)=\vert \lbrace e_i \mid e_i=-k \rbrace \vert.
			\end{equation*}
			
			\item For $\eta \in \CR_n^{0+}$, and $k \geq 0$, we define
			\begin{equation*}
				t_k(\eta)=\vert \lbrace \eta_i \mid \eta_i=k \rbrace \vert,
			\end{equation*}
			and
			\begin{equation*}
				t_{\ge k}(\eta)=\vert \lbrace \eta_i \mid \eta_i \geq k \rbrace \vert.
			\end{equation*}
			
			\item For $Y \in \CR_n$, we define
			\begin{equation*}
				t(Y)=(t_0(Y),t_1(Y),\dots).
			\end{equation*}
			Similarly, for $\eta \in \CR^{0+}_n$, we define
			\begin{equation*}
				t(\eta)=(t_0(\eta),t_1(\eta),\dots).
			\end{equation*}
		\end{enumerate}
	\end{definition}

	\begin{definition}\begin{enumerate}
			\item For $\lambda=(\overline{\lambda},\overset{t_1(\lambda)}{\overbrace{1,\dots,1}},\overset{t_0(\lambda)}{\overbrace{0,\dots,0}}) \in \CR_n^{0+}$, and $0 \leq s \leq t_0(\lambda)$, we define
			\begin{equation*}
				\lambda_s^+\coloneqq (\overline{\lambda},\overset{t_1(\lambda)+s}{\overbrace{1,\dots,1}},\overset{t_0(\lambda)-s}{\overbrace{0,\dots,0}}),
			\end{equation*}
			by replacing $s$ zeros by $s$ 1's.
			
			\item For $\lambda=(\overline{\lambda},\overset{t_1(\lambda)}{\overbrace{1,\dots,1}},\overset{t_0(\lambda)}{\overbrace{0,\dots,0}}) \in \CR_n^{0+}$ and $0 \leq s \leq t_1(\lambda)$, we define
			\begin{equation*}
				\lambda_{s}^{-}=(\overline{\lambda},\overset{t_1(\lambda)-s}{\overbrace{1,\dots,1}},\overset{t_0(\lambda)+s}{\overbrace{0,\dots,0}}),
			\end{equation*}
			by replacing $s$ 1's by $s$ zeros.
			
			\item For $0 \leq l \leq n$ and $\lambda \in \CR^{0+}_n$ such that $t_0(\lambda) \geq l$, we define $\lambda^{\vee_l}$ as the element in $\CR_{n-l}^{0+}$ such that $\lambda=(\lambda^{\vee_l},0,\dots,0)$.
		\end{enumerate}
		
	\end{definition}
	
	\begin{definition} Assume that $k \geq 0$, $\alpha, \eta \in \CR_{n}^{0+}$, $Y \in \CR_n^{0+}$, and $t(\eta)=t(Y)$.
		\begin{enumerate}
			\item We define
			\begin{equation*}
				\CB_k(Y)=\mathlarger{\sum}_{i} \min(0,e_i+k)-\min(0,e_i),
			\end{equation*}
			and
			\begin{equation*}
				\CB_k(\eta)=\mathlarger{\sum}_{i} \min(k,\eta_i).
			\end{equation*}
			Note that $\CB_k(\eta)=\CB_k(Y)$ since $t(\eta)=t(Y)$.
			
			\item We define
			\begin{equation*}
				\CB_{\alpha}(Y)=\mathlarger{\sum}_{i}\CB_{\alpha_i}(Y),
			\end{equation*}
			and
			\begin{equation*}
				\CB_{\alpha}(\eta)=\mathlarger{\sum}_{i}\CB_{\alpha_i}(\eta).
			\end{equation*}
			Note that $\CB_{\alpha}(\eta)=\CB_{\alpha}(Y)$ since $t(\eta)=t(Y)$.
			
			\item We define $f(Y)$ by
			\begin{equation*}
				f(Y)=\prod_{i}(-q)^{n\min(0,e_i)}.
			\end{equation*}
		\end{enumerate}
	\end{definition}

	By Lemma \ref{lemma5.3} and the fact that $\CR_n$ forms a complete set of representatives of $\Gamma_n \backslash X_{n}(F)$, we have that
	\begin{equation}\label{eq5.2'}
		\alphad'(I_{n,h},B)=\mathlarger{\sum}_{Y \in \CR_n} \dfrac{\CG(Y,B)\CF'(Y,I_{n,h})}{\alpha(Y;\Gamma_{n})},
	\end{equation}
	
	\begin{equation}\label{eq5.3'}
		\alphad'(I_{n},B)=\mathlarger{\sum}_{Y \in \CR_n} \dfrac{\CG(Y,B)\CF'(Y,I_n)}{\alpha(Y;\Gamma_{n})},
	\end{equation}
	and for $\lambda \in \CR_n^{0+}$,
	\begin{equation}\label{eq5.4'}
		\alphad(A_{\lambda},B)=\mathlarger{\sum}_{Y \in \CR_n} \dfrac{\CG(Y,B)\CF(Y,A_{\lambda})}{\alpha(Y;\Gamma_{n})}.
	\end{equation}
	
	By \cite[Lemma 3.2, Section 3.2]{Cho3}, we know that $\CF'(Y,I_{n,h})$ can be written uniquely as a linear sum of $\CF(Y,A_{\lambda})$, $\lambda \in \CR_n^{0+}$.  As in the proof of \cite[Lemma 3.15]{Cho}, we can compute that for
	\begin{equation*}
		Y=Y_{\sigma,e}=\sigma \left(\begin{array}{ccc}\pi^{e_1} & & 0 \\ & \ddots& \\ 0 & & \pi^{e_n} \end{array}\right), 
	\end{equation*}
	we have
	\begin{equation*}
		\CF'(Y,I_{n,h})=\mathlarger{\sum}_{j}\min(0,e_j) (-q)^{h\CB_1(Y)}f(Y),
	\end{equation*}
	\begin{equation*}
		\CF'(Y,I_n)=\mathlarger{\sum}_{j}\min(0,e_j)f(Y),
	\end{equation*}
	and
	\begin{equation*}
		\CF(Y,A_{\lambda})=(-q)^{\CB_{\lambda}(Y)}f(Y).
	\end{equation*}
	Therefore, we have that (cf. \cite[(4.1.1)]{Cho3})
	\begin{equation}\label{eq5.2}\begin{array}{l}
			\CF'(Y,I_{n,h})-(-q)^{hn}\CF'(Y,I_n)
			=\left\lbrace \begin{array}{cl} 
				0 & \text{if } t_0(Y)=0,\\
				(\mathlarger{\sum}_{j}\min(0,e_j))f(Y)((-q)^{h(n-1)}-(-q)^{hn})& \text{ if } t_0(Y)=1,\\
				\quad\quad\quad\quad\vdots\\
				(\mathlarger{\sum}_{j}\min(0,e_j))f(Y)((-q)^{h(n-k)}-(-q)^{hn})& \text{ if } t_0(Y)=k,\\
				\quad\quad\quad\quad\vdots\\
				(\mathlarger{\sum}_{j}\min(0,e_j))f(Y)(1-(-q)^{hn})& \text{ if } t_0(Y)=n.

			\end{array}\right.
		\end{array}
	\end{equation}
	
	Now, let us define the following constants and matrices.
	\begin{definition}\label{definition5.12}\hfill
		\begin{enumerate}
			\item (cf. \cite[Lemma 4.4]{Cho3}) For $0 \leq i \leq l$, we define constants $d_{il}$ by 
			\begin{equation*}
				d_{il}=(-q)^{-in} \prod_{\substack{0 \leq m \leq l\\m \neq i}}\dfrac{1}{((-q)^{-i}-(-q)^{-m})}.
			\end{equation*}
			Therefore, we have
			\begin{equation*}
				\dfrac{d_{il}}{d_{i+1,l}}=-(-q)^{n+i+1-l}\dfrac{(1-(-q)^{-(i+1)})}{(1-(-q)^{(l-i)})}.
			\end{equation*}
			
			\item We define the upper triangular $(n+1) \times (n+1)$ matrix $\Delta$ by	\begin{equation*}
				\Delta=\left( \begin{array}{ccccc}
					d_{00} & d_{01} & d_{02} & \dots &d_{0n}\\
					0 & d_{11} & d_{12} & \dots & d_{1n}\\
					0 &  0& d_{22} &\dots & d_{2n}\\
					\vdots & \ddots & \ddots & \ddots& \vdots \\
					0 & 0 & 0 & \dots & d_{nn} \\
				\end{array}\right).
			\end{equation*}
			
			\item For $l \leq 0$, we define $\FM_l$ by
			\begin{equation*}
				\FM_{l}:\left(\begin{array}{cccc}
					1 & (-q)^{2n} & \dots & (-q)^{2ln}\\
					1 & (-q)^{2n-1} & \dots & (-q)^{l(2n-1)}\\
					\vdots & \vdots & \ddots & \vdots\\
					1 & (-q)^{2n-l} & \dots & (-q)^{l(2n-l)}
				\end{array}\right).
			\end{equation*}
			
			\item For $0 \leq i, j \leq n$, we define constants $\CA_{ij}$ by
			\begin{equation*}
				(\CA_{ij})_{0 \leq i, j \leq n}=\FM_n \Delta.
			\end{equation*}
			\item For $0 \leq i \leq n$, we define constants $\CK_i$ by
			\begin{equation*}
				(\CA_{ij})\left(\begin{array}{c}
					\CK_0\\
					\CK_1\\
					\vdots \\
					\CK_{n}
				\end{array}\middle)=\middle(\begin{array}{c}
					0\\
					(-q)^{h(n-1)}-(-q)^{hn}\\
					\vdots\\
					1-(-q)^{hn}
				\end{array} \right).
			\end{equation*}
		\end{enumerate}
	\end{definition}
	\begin{proposition}\label{proposition5.13} \cite[Proposition 3.3, Proposition 3.4]{Cho3} We define the following constants.
		\begin{enumerate}
			\item For $\alpha \in \CR_n^{0+}$ such that $\sum \alpha_i=$odd, we write $\CC_{\alpha}$ for
			\begin{equation*}
				\CC_{\alpha}=\prod_{i=1}^{t_{\ge 1}(\alpha)-1}(1-(-q)^i).
			\end{equation*}
			Here, we define $\CC_{\alpha}=1$ if $t_{\ge 1}(\alpha)=1$.
			
			\item  For $\alpha \in \CR_n^{0+}$ such that $\sum \alpha_i=$even and $\alpha\neq (0,0,\dots,0)$, we write $\CC_{\alpha}$ for
			\begin{equation*}
				\CC_{\alpha}=-\prod_{i=1}^{t_{\ge 1}(\alpha)-1}(1-(-q)^i).
			\end{equation*}
			Here, we define $\CC_{\alpha}=-1$ if $t_{\ge 1}(\alpha)=1$.
			
			Also, if $\alpha=(0,0,\dots,0)$, we define
			\begin{equation*}
				\CC_{\alpha}=\dfrac{\alphad'(I_n,I_n)}{\alphad(I_n,I_n)}.
			\end{equation*}
		\end{enumerate}
		Then, we have that
		\begin{equation*}
			\dfrac{\alphad'(I_n,B)}{\alphad(I_n,I_n)}=\mathlarger{\sum}_{\alpha \in \CR^{0+}_n}\CC_{\alpha}\dfrac{\alphad(A_{\alpha},B)}{\alphad(A_{\alpha},A_{\alpha})}.
		\end{equation*}
		
	\end{proposition}
	
	\bigskip
	
	Now, by \cite[(4.1.5)]{Cho3}, we have that for $\lambda \in \CR_n^{0+}$ with $t_0(Y)\geq l$,
	\begin{equation}\label{eq5.3}
		\mathlarger{\sum}_{0 \leq i \leq l}d_{il}\CF(Y,A_{\lambda_i^+})=\left\lbrace\begin{array}{ll}0=A_{0l}\CF(Y,A_{\lambda}) & \text{ if }t_0(Y)=0,\\
			\vdots & \vdots\\
			0=A_{l-1,l}\CF(Y,A_{\lambda}) & \text{ if }t_0(Y)=l-1,\\
			(-q)^{\CB_{\lambda}(Y)}f(Y)=\CF(Y,A_{\lambda})=A_{ll}\CF(Y,A_{\lambda}) & \text{ if }t_0(Y)=l,\\
			A_{kl}\CF(Y,A_{\lambda}) & \text{ if } t_0(Y)=k, \text{ }l+1 \leq k \leq n.
		\end{array}\right.
	\end{equation}
	
	Also, by Proposition \ref{proposition5.13}, we have that for $Y \in \CR_{n-l}^{0+}$,
	\begin{equation*}
		\begin{array}{ll}
			& \dfrac{\CF'(Y,I_{n-l})}{\alphad(I_{n-l},I_{n-l})}=\mathlarger{\sum}_{\overline{\lambda}\in \CR_{n-l}^{0+}}\CC_{\overline{\lambda}} \dfrac{\CF(Y,A_{\overline{\lambda}})}{\alphad(A_{\overline{\lambda}},A_{\overline{\lambda}})}\\
			\Longleftrightarrow &\sum_{i} \min(0,e_i)f(Y)=\mathlarger{\sum}_{\overline{\lambda}\in \CR_{n-l}^{0+}}\CC_{\overline{\lambda}} \dfrac{\alphad(I_{n-l},I_{n-l})}{\alphad(A_{\overline{\lambda}},A_{\overline{\lambda}})}(-q)^{\CB_{\overline{\lambda}}(Y)}f(Y)\\
			\Longleftrightarrow & \sum_{i} \min(0,e_i)=\mathlarger{\sum}_{\overline{\lambda}\in \CR_{n-l}^{0+}}\CC_{\overline{\lambda}} \dfrac{\alphad(I_{n-l},I_{n-l})}{\alphad(A_{\overline{\lambda}},A_{\overline{\lambda}})}(-q)^{\CB_{\overline{\lambda}}(Y)}.
		\end{array}
	\end{equation*}
	Since running $\lambda=(\lambda^{\vee_l},0,\dots,0)$ over  $\CR_n^{0+}$ with $t_0(Y)\geq l$ is equivalent to running $\lambda^{\vee_l}$ over $\CR_{n-l}^{0+}$ by removing $l$ zeros, we have that
	\begin{equation}\label{eq5.4}\begin{array}{l}
			\mathlarger{\sum}_{\lambda \in \CR_{n}^{0+}, t_0(Y)\geq l} \CC_{\lambda^{\vee_l}}\dfrac{\alphad(I_{n-l},I_{n-l})}{\alphad(A_{\lambda^{\vee_l}},A_{\lambda^{\vee_l}})}\mathlarger{\sum}_{0 \leq i\leq l} d_{il}\CF_0(Y,A_{\lambda_{i}^+})\\\\
			=\left\lbrace \begin{array}{ll}
				0\quad=\quad\CA_{il}\sum_{j}\min(0,e_j)f(Y)  & 
				\text{ if } t_0(Y) \leq l-1\\
				\CA_{il}\sum_{j}\min(0,e_j)f(Y)&
				\text{ if } t_0(Y) \geq l.
			\end{array}\right.
		\end{array}
	\end{equation}
	
	In Definition \ref{definition5.12} (5), we defined constants $\CK_i$ such that
	\begin{equation*}
		(\CA_{ij})\left(\begin{array}{c}
			\CK_0\\
			\CK_1\\
			\vdots \\
			\CK_{n}
		\end{array}\middle)=\middle(\begin{array}{c}
			0\\
			(-q)^{h(n-1)}-(-q)^{hn}\\
			\vdots\\
			1-(-q)^{hn}
		\end{array} \right).
	\end{equation*}
	
	Therefore, we have that
	\begin{equation*}\begin{array}{l}
			\mathlarger{\sum}_{l=0}^{n}\CK_l \lbrace\mathlarger{\sum}_{\lambda \in \CR_{n}^{0+}, t_0(Y) \geq l} \CC_{\lambda^{\vee_l}}\dfrac{\alphad(I_{n-l},I_{n-l})}{\alphad(A_{\lambda^{\vee_l}},A_{\lambda^{\vee_l}})}\mathlarger{\sum}_{0 \leq i\leq l} d_{il}\CF_0(Y,A_{\lambda_{i}^+})\rbrace\\\\
			=\left\lbrace \begin{array}{cl} 
				0 & \text{if } t_0(Y)=0,\\
				(\mathlarger{\sum}_{j}\min(0,e_j))f(Y)((-q)^{h(n-1)}-(-q)^{hn})& \text{ if } t_0(Y)=1,\\
				\quad\quad\quad\quad\vdots\\
				(\mathlarger{\sum}_{j}\min(0,e_j))f(Y)((-q)^{h(n-k)}-(-q)^{hn})& \text{ if } t_0(Y)=k,\\
				\quad\quad\quad\quad\vdots\\
				(\mathlarger{\sum}_{j}\min(0,e_j))f(Y)(1-(-q)^{hn})& \text{ if } t_0(Y)=n.
			\end{array}\right.
		\end{array}
	\end{equation*}
	
	By comparing this with \eqref{eq5.2}, we have that
	
	\begin{equation}\label{eq5.8}\begin{array}{l}
			\CF'(Y,I_{n,h})-(-q)^{hn}\CF'(Y,I_n)\\
			=\mathlarger{\sum}_{l=0}^{n}\CK_l \lbrace\mathlarger{\sum}_{\lambda \in \CR_n^{0+}, t_0(Y) \geq l} \CC_{\lambda^{\vee_l}}\dfrac{\alphad(I_{n-l},I_{n-l})}{\alphad(A_{\lambda^{\vee_l}},A_{\lambda^{\vee_l}})}\mathlarger{\sum}_{0 \leq i\leq l} d_{il}\CF_0(Y,A_{\lambda_{i}^+})\rbrace.
		\end{array}
	\end{equation}
	
	Now, we have the following lemma.
	
	\begin{lemma}\label{lemma5.14}(cf. \cite[Lemma 4.5]{Cho3}) For $i=0$ and $h+1 \leq i \leq n$, $\CK_i=0$. Also, $d_{hh}\CK_h=1$ and for $1 \leq l \leq h-1$, we have
		\begin{equation*}
			d_{h-l,h-l}\CK_{h-l}=(-q)^n\dfrac{1-(-q)^{-h+l-1}}{1-(-q)^{-l}}d_{h-l+1,h-l+1}\CK_{h-l+1}.
		\end{equation*}
	\end{lemma}
	\begin{proof}
		As in \cite[Lemma 4.3]{Cho3}, one can show that
		\begin{equation}\label{eq5.5}
			\Delta\left(\begin{array}{c}
				\CK_0\\
				\CK_1\\
				\dots \\
				\CK_{n}\\
				\dots\\
				\CK_{2n}
			\end{array}\middle)=\middle(\begin{array}{c}
				-(-q)^{nh}\\
				0\\
				0\\
				1\\
				0\\
				0
			\end{array} \right)\begin{array}{l} 
				\text{1st entry}\\
				\\
				\\
				\text{(h+1)-th entry}
				\\
				\\
				.
			\end{array}
		\end{equation}
		Since $\Delta$ is an upper triangular matrix, we have that $\CK_i=0$ for $h+1 \leq i \leq n$. Also, the $(h+1)$-th row of \eqref{eq5.5} implies that
		\begin{equation*}
			d_{hh}\CK_h+d_{h,h+1}\CK_{h+1}+\dots+d_{h,n}\CK_{n}=1,
		\end{equation*}
		and hence $d_{hh}\CK_h=1$. Now, the proof of the last statement is almost identical to the proof of \cite[Lemma 4.5]{Cho3}.
	\end{proof}

	Now, we are ready to write the following proposition on the constants $D_{n,h}(\lambda)$. 
	
	\begin{proposition}\label{proposition5.16}(cf. \cite[Theorem 4.7]{Cho3})
		For $0 \leq h \leq n$ and $\lambda \in \CR_n^{0+}$ such that $\lambda \neq (1^t,0^n-t), t \leq h-1$, we have
		
		\begin{equation}\label{eq5.10.1}\begin{aligned}
				D_{n,h}(\lambda)=&\dfrac{q^{h(n-h)}\prod_{l=1}^{n}(1-(-q)^{-l})}{\prod_{l=1}^{h}(1-(-q)^{-l})\prod_{l=1}^{n-h}(1-(-q)^{-l})}\CC_{\lambda}\\
				&+\mathlarger{\sum}_{\substack{1 \leq i \leq h\\ \max\lbrace i-t_0(\lambda),0 \rbrace  \leq s \\ \leq \min\lbrace i , t_1(\lambda) \rbrace}}(-q)^{n(h-i)+(i-s)(2n-i+s+1)/2-h^2+s(2n-2t_0(\lambda)-s)}(-1)^{i+h}\\
				& \qquad \times \dfrac{\prod_{l=1}^{n-i}(1-(-q)^{-l})}{\prod_{l=1}^{n-h}(1-(-q)^{-l})\prod_{l=1}^{h}(1-(-q)^{-l})}
				\times \dfrac{\prod_{l=s+1}^{h}(1-(-q)^{-l})}{\prod_{l=1}^{h-i}(1-(-q)^{-l})\prod_{l=1}^{i-s}(1-(-q)^{-l})}
				\\
				&\qquad\times
				\dfrac{\prod_{l=1}^{l=t_0(\lambda)}(1-(-q)^{-l})\prod_{l=1}^{l=t_1(\lambda)}(1-(-q)^{-l})}{\prod_{l=1}^{l=t_0(\lambda)-i+s}(1-(-q)^{-l})\prod_{l=1}^{l=t_1(\lambda)-s}(1-(-q)^{-l})} \times \CC_{(\lambda_s^-)^{\vee_i}}.
			\end{aligned}
		\end{equation}
		Here, we choose the following convention: for $k \leq 0$, we assume that
		\begin{equation*}
			\prod_{l=1}^{k}(*)=1 \quad \text{ and }\quad \sum_{l=1}^k (*)=0.
		\end{equation*}
		If $\lambda=(1^t,0^{n-t}), t \leq h-1$, then $D_{n,h}(\lambda)=\text{ the right hand side of \eqref{eq5.10.1}}+\dfrac{(-q)^{\frac{-(h-t)(h+t+1)}{2}}}{1-(-q)^{-(h-t)}}$.
		In particular, $D_{n,h}(\lambda)$ depends only on $n$, $t_0(\lambda)$, $t_1(\lambda)$, and the parity of $\sum_i\lambda_i$.
	\end{proposition}
	
	\begin{remark}\label{rem: ind formula}
		Note that when $h=0$, the sum $\mathlarger{\sum}_{\substack{1 \leq i \leq h\\ \max\lbrace i-t_0(\lambda),0 \rbrace  \leq s \leq \min\lbrace i , t_1(\lambda) \rbrace}}(*)$ in \eqref{eq5.10.1} is empty, and hence $D_{n,0}(\lambda)=\CC_{\lambda}$ which is obvious by construction. Therefore, Proposition \ref{proposition5.16} decomposes $D_{n,h}$ into a weighted summation of $D_{n,0}$, which is the Cho-Yamauchi constant in the good reduction case.
	\end{remark}

	\begin{proof}[Proof of Proposition \ref{proposition5.16}] The proof is almost identical to the proof of \cite[Theorem 4.7]{Cho3}; therefore, let me just write which parts are different. First, note that for $B \in X_n(F)$ of the form in \eqref{eq5.1}, Proposition \ref{proposition5.7} implies that
		\begin{equation}\label{eq5.10}\begin{aligned}
				\dfrac{W_{n-h,n}'(B)}{W_{n,n}(A_n,1)}&=\dfrac{q^{-4n^2+(n+h)(n-h)}\alphad(\pi A_n,I_{n-h})\alphad(I_{n+h,h},I_h)\alphad'(I_{n,h},T)}{W_{n,n}(A_n,1)}.
			\end{aligned}
		\end{equation}
		Now, by \eqref{eq5.2'}, \eqref{eq5.3'}, \eqref{eq5.4'}, and \eqref{eq5.8}, we have
		\begin{equation}\label{eq5.11}\begin{aligned}
				\alphad'(I_{n,h},T)&=(-q)^{nh}\alphad'(I_n,T)
				+\mathlarger{\sum}_{l=0}^{n}\CK_l \lbrace\mathlarger{\sum}_{\lambda \in \CR_n^{0+},t_0(Y) \geq l} \CC_{\lambda^{\vee_l}}\dfrac{\alphad(I_{n-l},I_{n-l})}{\alphad(A_{\lambda^{\vee_l}},A_{\lambda^{\vee_l}})}\mathlarger{\sum}_{0 \leq i\leq l} d_{il}\alphad(A_{\lambda_{i}^+},T)\rbrace.
			\end{aligned}
		\end{equation}
		
		Now, we can easily follow the proof of \cite[Theorem 4.7]{Cho3} by using Lemma \ref{lemma5.14}, \eqref{eq5.10}, and \eqref{eq5.11}. Also, for $\lambda=(1^t,0^{n-t}), t\leq h-1$, we use Proposition \ref{proposition5.28}.
	\end{proof}
	
	Note that $D_{n,h}(\lambda)$ depends only on $t_{\ge 2}(\lambda)$, $t_1(\lambda)$, $t_0(\lambda)$, and the parity of $\sum_i\lambda_i$. Also, by Remark \ref{remark5.10}, we only need to consider $\lambda$ such that $\sum_i \lambda_i \equiv h+1 \pmod 2$. Also, later on, we will establish inductive formulas relating $D_{n,h}(\lambda)$ for $\lambda$ with different $t_{\ge 2}(\lambda),t_1(\lambda)$ and $t_0(\lambda)$. So we introduce the following notation to streamline the computation.

	\begin{definition}\label{definition5.19} Assume that $h,n,a,b,c,i,j,s \geq 0$, $0 \leq s \leq i \leq h \leq n$, $0\leq j \leq n$, and $n=a+b+c$.
		\begin{enumerate}
			\item  If $c \neq n$, we write $\CC_{j}(a,b,c)$ for
			\begin{equation*}
				\CC_{j}(a,b,c)=(-1)^{j+1}\prod_{i=1}^{a+b-1}(1-(-q)^i).
			\end{equation*}
			Also, if $c=n$, we define
			\begin{equation*}
				\CC_{j}(0,0,n)=\dfrac{\alphad'(I_n,I_n)}{\alphad(I_n,I_n)}.
			\end{equation*}
			Then, for $\lambda \in \CR_n^{0+}$, we have that
			\begin{equation*}
				\CC_{\lambda}=\CC_{\vert \lambda \vert}(t_{\ge 2}(\lambda),t_1(\lambda),t_0(\lambda)).
			\end{equation*}
			\item Let $M_{n,h}(a,b,c,i,s)$ be the constant 
			\begin{equation*}\begin{aligned}
					M_{n,h}(a,b,c,i,s)=&(-q)^{n(h-i)+(i-s)(2n-i+s+1)/2-h^2+s(2n-2c-s)}(-1)^{i+h}\\
					&
					\times \dfrac{\prod_{l=1}^{n-i}(1-(-q)^{-l})}{\prod_{l=1}^{n-h}(1-(-q)^{-l})\prod_{l=1}^{h}(1-(-q)^{-l})}
					\times \dfrac{\prod_{l=s+1}^{h}((1-(-q)^{-l})}{\prod_{l=1}^{h-i}(1-(-q)^{-l})\prod_{l=1}^{i-s}(1-(-q)^{-l})}\\
					&\times
					\dfrac{\prod_{l=1}^{c}(1-(-q)^{-l})\prod_{l=1}^{b}(1-(-q)^{-l})}{\prod_{l=1}^{c-i+s}(1-(-q)^{-l})\prod_{l=1}^{b-s}(1-(-q)^{-l})}
					\times\CC_{h+1-s}(a,b-s,c+s-i).
				\end{aligned}
			\end{equation*}
			
			\item For $(a,b,c) \neq (0,t,n-t), t\leq h-1$, we write $D_{n,h}(a,b,c)$ for
			\begin{equation*}
				D_{n,h}(a,b,c)=\mathlarger{\sum}_{\substack{0 \leq i \leq h\\ \max(i-c,0)  \leq s \leq \min(i , b )}}M_{n,h}(a,b,c,i,s)=\mathlarger{\sum}_{\substack{0 \leq s \leq \min(h,b)\\ s \leq i \leq \min( s+c,h)}}M_{n,h}(a,b,c,i,s).
			\end{equation*}
			
			\item For $(a,b,c)=(0,t,n-t), t \leq h-1$, $t \equiv h+1 \pmod 2$, we write $D_{n,h}(a,b,c)$ for
			\begin{equation*}\begin{aligned}
					D_{n,h}(0,t,n-t)&=\mathlarger{\sum}_{\substack{0 \leq i \leq h\\ \max(i-c,0)  \leq s \leq \min(i , b )}}M_{n,h}(a,b,c,i,s)+\dfrac{(-q)^{\frac{-(h-t)(h+t+1)}{2}}}{1-(-q)^{-(h-t)}}\\
					&=\mathlarger{\sum}_{\substack{0 \leq s \leq \min(h,b)\\ s \leq i \leq \min( s+c,h)}}M_{n,h}(a,b,c,i,s)+\dfrac{(-q)^{\frac{-(h-t)(h+t+1)}{2}}}{1-(-q)^{-(h-t)}}.
				\end{aligned}
			\end{equation*}
		\end{enumerate}
	\end{definition}
	
	\begin{proposition}\label{proposition5.20}
		For $\lambda \in \CR_n^{0+}$, we have that
		\begin{equation*}
			D_{n,h}(\lambda)=D_{n,h}(t_{\ge 2}(\lambda),t_1(\lambda),t_0(\lambda)).
		\end{equation*}
	\end{proposition}
	\begin{proof}
		This follows from the definitions of $M_{n,h}(a,b,c,i,s)$, $D_{n,h}(a,b,c)$, Proposition \ref{proposition5.16}, and the fact that
		\begin{equation*}
			M_{n,h}(a,b,c,0,0)=\dfrac{q^{h(n-h)}\prod_{l=1}^{n}(1-(-q)^{-l})}{\prod_{l=1}^{h}(1-(-q)^{-l})\prod_{l=1}^{n-h}(1-(-q)^{-l})}\CC_{h+1}(a,b,c).
		\end{equation*}
	\end{proof}

	\section{Inductive relations among Cho-Yamauchi constants}\label{sec: ind formula}
	
	In this section, we will prove some inductive relations among $D_{n,h}(\lambda)$. The main result of this section can be summarized as follows: Let $0 \leq a,b,c,h, \leq n$, $a+b+c=n$. Assume further that $(a-1,b+1,c) \neq (0,h,n-h)$. Then, we have the following result.
	\begin{enumerate}
		\item (Theorem \ref{theorem5.24}) If $c \leq n-h$ and $a \geq 1$, then
		\begin{equation*}
			D_{n,h}(a,b,c)-D_{n,h}(a-1,b+1,c)=-(-q)^{2n-h-1-b-2c}D_{n-1,h-1}(a-1,b,c).
		\end{equation*}
		
		\item (Lemma \ref{lemma7.5}, Proposition \ref{prop: vanishing of D(lambda)}) If $c > n-h$, then $D_{n,h}(a,b,c)=0$.
		
		\item (Theorem \ref{theorem5.31} (1)) Assume that $a=0$ and $h+1 \leq b$. Then, we have
		\begin{equation*}
			D_{n,h}(0,b,c)=\dfrac{\prod_{l=h+1}^{b}(1-(-q)^l)}{(1-(-q)^{b-h})}.
		\end{equation*}
		
		\item (Theorem \ref{theorem5.31} (2)) Assume that $a=1$ and $h-1 \leq b$. Then, we have
		\begin{equation*}
			D_{n,h}(1,b,c)=\left\lbrace \begin{array}{ll}
				1 & \text{ if } b=h-1, h;\\
				\prod_{l=h+1}^{b}(1-(-q)^l) & \text{ if } b \geq h+1.
			\end{array}\right.
		\end{equation*}
	\end{enumerate}
	
	Here, the main reason for excluding the case $(a-1,b+1,c)=(0,h,n-h)$ is because it is not necessary by Remark \ref{remark5.10}.
	\bigskip

	First, let us prove the following lemma.	
	\begin{lemma}\label{lemma5.21} Let $0 \leq a,b,c,h \leq n$, $a+b+c=n$, and $0 \leq s \leq \min(h,b)$.
		\begin{enumerate}
			\item If $c \leq n-h$ and $(a,b,c) \neq (0,h,n-h)$, we have that
			\begin{align*} 
				&\sum_{i=s}^{\min(s+c,h)}M_{n,h}(a,b,c,i,s)\\
				&=(-1)^{n+c+s+1}(-q)^{\frac{c}{2}+\frac{c^2}{2}-h^2-\frac{n}{2}-cn+hn+\frac{n^2}{2}-ch+\frac{s}{2}-\frac{s^2}{2}}\\
				&\quad \times \dfrac{\prod_{l=1}^{b}(1-(-q)^{-l})\prod_{l=1}^{n-c-s}(1-(-q)^{-l})\prod_{l=1}^{n-c-s-1}(1-(-q)^{-l})}{\prod_{l=1}^{b-s}(1-(-q)^{-l})\prod_{l=1}^{s}(1-(-q)^{-l})\prod_{l=1}^{n-h-c}(1-(-q)^{-l})\prod_{l=1}^{h-s}(1-(-q)^{-l})}.
			\end{align*}
			
			\item If $c>n-h$ and $a \neq 0$, we have that $\sum_{i=s}^{\min(s+c,h)}M_{n,h}(a,b,c,i,s)=0$.
		\end{enumerate}
	\end{lemma}
	\begin{proof}
		By our assumption, we have that $(a,b-s,c+s-i) \neq (0,0,n-i)$ for any $i \leq s \leq \min(s+c,h)$. Therefore, by Definition \ref{definition5.19} (1), we have that
		\begin{equation*}\begin{array}{ll}
				\CC_{h+1-s}(a,b-s,c+s-i)&=(-1)^{h-s}\prod_{l=1}^{n-c-s-1}(1-(-q)^l)\\
				&=(-1)^{n+c+h+1}(-q)^{(n-c-s-1)(n-c-s)/2}\prod_{l=1}^{n-c-s-1}(1-(-q)^{-l}),
			\end{array}
		\end{equation*}
		which is independent of $i$.
		Also, by Definition \ref{definition5.19} (2), we have that
		\begin{equation*}
			M_{n,h}(a,b,c,i,s)=\CC_{h+1-s}(a,b-s,c+s-i)Y_{n,h}(b,c,s)X_{n,h}(c,i,s),
		\end{equation*}
		where
		\begin{equation*}\begin{aligned}
				Y_{n,h}(b,c,s)&
				\coloneqq	(-1)^{h}(-q)^{-h^2+hn-\frac{s}{2}-\frac{3s^2}{2}-2cs+ns}
				\dfrac{\mathlarger{\prod}_{l=1}^{c}(1-(-q)^{-l})    \mathlarger{\prod}_{l=1}^{b}(1-(-q)^{-l})}{\mathlarger{\prod}_{l=1}^{n-h}(1-(-q)^{-l})\mathlarger{\prod}_{l=1}^{s}(1-(-q)^{-l})\mathlarger{\prod}_{l=1}^{b-s}(1-(-q)^{-l})},
			\end{aligned}
		\end{equation*}
		which is independent of $i$,		and
		\begin{equation*}\begin{aligned}
				X_{n,h}(c,i,s)\coloneqq
				\dfrac{(-1)^{i}(-q)^{\frac{i}{2}-\frac{i^2}{2}+is}\mathlarger{\prod}_{l=1}^{n-i}(1-(-q)^{-l})}{\mathlarger{\prod}_{l=1}^{h-i}(1-(-q)^{-l})\mathlarger{\prod}_{l=1}^{i-s}(1-(-q)^{-l})\mathlarger{\prod}_{l=1}^{c-i+s}(1-(-q)^{-l})}.
			\end{aligned}
		\end{equation*}

		Now, we note that Lemma \ref{lemma5.21} follows from the following claim
		\begin{equation}\label{lemma5.22}
			\sum_{i=s}^{\min(s+c,h)}X_{n,h}(c,i,s)=\left\lbrace\begin{array}{cl}  \dfrac{(-1)^s(-q)^{-ch+cs+\frac{s}{2}+\frac{s^2}{2}}\mathlarger{\prod}_{l=1}^{n-h}(1-(-q)^{-l})\mathlarger{\prod}_{l=1}^{n-s-c}(1-(-q)^{-l})}{\mathlarger{\prod}_{l=1}^{h-s}(1-(-q)^{-l})\mathlarger{\prod}_{l=1}^{c}(1-(-q)^{-l})\mathlarger{\prod}_{l=1}^{n-h-c}(1-(-q)^{-l})} & \text{ if } c \leq n-h,\\\\
				
				0 & \text{ if } c > n-h.
				
			\end{array}\right.
		\end{equation}
		Hence it suffices to show that \eqref{lemma5.22} holds.
		
		If $h=s$, then 
		\begin{equation*}\begin{aligned}
				\sum_{i=s}^{\min(s+c,h)}X_{n,h}(c,i,s)&=X_{n,h}(c,s,s) =\dfrac{(-1)^s(-q)^{\frac{s}{2}+\frac{s^2}{2}}\prod_{l=1}^{n-s}(1-(-q)^{-l})}{\prod_{l=1}^{c}(1-(-q)^{-l})}.
			\end{aligned}
		\end{equation*}
		Therefore, \eqref{lemma5.22} holds in this case.
		
		Similarly, if $c=0$, then
		\begin{equation*}\begin{aligned}
				\sum_{i=s}^{\min(s+c,h)}X_{n,h}(c,i,s)&=X_{n,h}(c,s,s) 
				&=\dfrac{(-1)^s(-q)^{\frac{s}{2}+\frac{s^2}{2}}\prod_{l=1}^{n-s}(1-(-q)^{-l})}{\prod_{l=1}^{h-s}(1-(-q)^{-l})}.
			\end{aligned}
		\end{equation*}
		Therefore, \eqref{lemma5.22} holds in this case.
		
		Now, assume that $c\geq 1$ and $h > s$. In Lemma \ref{lemma5.22.1} below, we will prove that
		\begin{equation}\label{eq5.12.1}
			\sum_{i=s}^{\min(s+c,h)}\dfrac{X_{n,h}(c,i,s)}{X_{n,h}(c,s,s)}=\dfrac{(-q)^{-(h-s)}(1-(-q)^{-(n-h)})}{1-(-q)^{-(n-s)}}\sum_{i=s}^{\min(s+c-1,h)}\dfrac{X_{n-1,h}(c-1,i,s)}{X_{n-1,h}(c-1,s,s)}.
		\end{equation}
		Assume  \eqref{eq5.12.1} holds for now. Then we are ready to prove \eqref{lemma5.22}.
		
		If $c \leq n-h$, by applying \eqref{eq5.12.1} repeatedly, we have that
		\begin{equation*}\begin{aligned}
				\sum_{i=s}^{\min(s+c,h)}\dfrac{X_{n,h}(c,i,s)}{X_{n,h}(c,s,s)}&=(-q)^{-c(h-s)}\dfrac{\prod_{l=n-h-c+1}^{n-h}(1-(-q)^{-l})}{\prod_{l=n-s-c+1}^{n-s}(1-(-q)^{-l})}\sum_{i=s}^{\min(s+0,h)}\dfrac{X_{n-c,h}(0,i,s)}{X_{n-c,h}(0,s,s)}\\
				&=(-q)^{-c(h-s)}\dfrac{\prod_{l=n-h-c+1}^{n-h}(1-(-q)^{-l})}{\prod_{l=n-s-c+1}^{n-s}(1-(-q)^{-l})}.
			\end{aligned}
		\end{equation*}
		Combining this with
		\begin{equation*}\begin{aligned}
				X_{n,h}(c,s,s)=
				\dfrac{(-1)^{s}(-q)^{\frac{s}{2}+\frac{s^2}{2}}\prod_{l=1}^{n-s}(1-(-q)^{-l})}{\prod_{l=1}^{h-s}(1-(-q)^{-l})\prod_{l=1}^{c}(1-(-q)^{-l})},
			\end{aligned}
		\end{equation*}
		we see that \eqref{lemma5.22} holds when $c \leq n-h$.
		
		Now, assume that $c > n-h$. By \eqref{eq5.12.1}, we have that
		\begin{equation*}\begin{aligned}
				\mathlarger{\sum}_{i=s}^{\min(s+c,h)}\dfrac{X_{n,h}(c,i,s)}{X_{n,h}(c,s,s)}=(-q)^{-(n-h)(h-s)}\dfrac{\prod_{l=1}^{n-h}(1-(-q)^{-l})}{\prod_{l=h-s+1}^{n-s}(1-(-q)^{-l})}\mathlarger{\sum}_{i=s}^{\min(s+c-n+h,h)}\dfrac{X_{h,h}(c-n+h,i,s)}{X_{h,h}(c-n+h,s,s)}.
			\end{aligned}
		\end{equation*}
		Therefore, it suffices to show that $\mathlarger{\sum}_{i=s}^{\min(s+c-n+h,h)}\dfrac{X_{h,h}(c-n+h,i,s)}{X_{h,h}(c-n+h,s,s)}=0$.
		
		Note that $s \leq \min(h,b)$ and hence $s+c \leq b+c \leq n$. Therefore, $\min(s+c-n+h,h)=s+c-n+h$. For simplicity, let us write $c'=c-n+h\geq 1$. 
		
		We claim that
		\begin{equation}\label{eq5.15}
			\mathlarger{\sum}_{i=s}^{s+k}\dfrac{X_{h,h}(c',i,s)}{X_{h,h}(c',s,s)}=(-1)^{k}(-q)^{-\frac{k(k+1)}{2}}\dfrac{\prod_{l=c'-k}^{c'-1}(1-(-q)^{-l})}{\prod_{l=1}^{k}(1-(-q)^{-l})},
		\end{equation}
		which specializes to 
		\begin{equation*}
			\mathlarger{\sum}_{i=s}^{s+c'}\dfrac{X_{h,h}(c',i,s)}{X_{h,h}(c',s,s)}=0
		\end{equation*}
		when $k=c'$.
		
		Therefore it suffices to prove \eqref{eq5.15}. We prove this by induction on $k$. Recall that by definition, we have
		\begin{equation*}
			\dfrac{X_{h,h}(c',s+k,s)}{X_{h,h}(c',s,s)}=(-1)^k(-q)^{-\frac{k(k-1)}{2}}\dfrac{\prod_{l=c'-k+1}^{c'}(1-(-q)^{-l})}{\prod_{l=1}^{k}(1-(-q)^{-l})}.
		\end{equation*}
		When $k=0$, both sides are $1$. When $k=1$,
		\begin{equation*}\begin{aligned}
				\mathlarger{\sum}_{i=s}^{s+1}\dfrac{X_{h,h}(c',i,s)}{X_{h,h}(c',s,s)}=1-\dfrac{1-(-q)^{-c'}}{1-(-q)^{-1}}=-(-q)^{-1}\dfrac{1-(-q)^{-(c'-1)}}{1-(-q)^{-1}}.
			\end{aligned}
		\end{equation*}
		Now, assume that \eqref{eq5.15} holds for $k$. Then, we have
		\begin{equation*}\begin{aligned}
				\mathlarger{\sum}_{i=s}^{s+k+1}\dfrac{X_{h,h}(c',i,s)}{X_{h,h}(c',s,s)}&=(-1)^{k}(-q)^{-\frac{k(k+1)}{2}}\dfrac{\prod_{l=c'-k}^{c'-1}(1-(-q)^{-l})}{\prod_{l=1}^{k}(1-(-q)^{-l})}+\dfrac{X_{h,h}(c',s+k+1,s)}{X_{h,h}(c',s,s)}\\
				&=(-1)^k(-q)^{-\frac{k(k+1)}{2}}\dfrac{\prod_{l=c'-k}^{c'-1}(1-(-q)^{-l})}{\prod_{l=1}^{k}(1-(-q)^{-l})}\mathlarger{\lbrace}1-\dfrac{1-(-q)^{-c'}}{1-(-q)^{-k-1}} \mathlarger{\rbrace}\\
				&=(-1)^{k+1}(-q)^{-\frac{(k+1)(k+2)}{2}}\dfrac{\prod_{l=c'-k-1}^{c'-1}(1-(-q)^{-l})}{\prod_{l=1}^{k+1}(1-(-q)^{-l})}.
			\end{aligned}
		\end{equation*}
		This shows that \eqref{eq5.15} holds and finishes the proof of \eqref{lemma5.22}, and hence the proof of the lemma.
	\end{proof}

	\begin{lemma}\label{lemma5.22.1}
		For $c \geq 1$ and $h>s$, the function $X_{n,h}(c,i,s)$ (defined in the proof of the Lemma \ref{lemma5.21}) satisfies
		\begin{equation}\label{eq5.12}
			\sum_{i=s}^{\min(s+c,h)}\dfrac{X_{n,h}(c,i,s)}{X_{n,h}(c,s,s)}=\dfrac{(-q)^{-(h-s)}(1-(-q)^{-(n-h)})}{1-(-q)^{-(n-s)}}\sum_{i=s}^{\min(s+c-1,h)}\dfrac{X_{n-1,h}(c-1,i,s)}{X_{n-1,h}(c-1,s,s)}.
		\end{equation}
	\end{lemma}
	\begin{proof}
		To prove \eqref{eq5.12}, we define $Z_{n,h}(c,j,s)$ to be
		\begin{equation*}
			Z_{n,h}(c,j,s)=(-1)^{j-s}(-q)^{-\frac{(j-s)(j-s+1)}{2}}\dfrac{\prod_{l=c-j+s}^{c-1}(1-(-q)^{-l})\prod_{l=h-j+1}^{h-s}(1-(-q)^{-l})}{\prod_{l=1}^{j-s}(1-(-q)^{-l})\prod_{l=n-j+1}^{n-s}(1-(-q)^{-l})}.
		\end{equation*}
		Note that if $c \geq 1$ and $h>s$, we have that $\min(s+c,h)>s$. 
		
		Now, we claim the following statement: for $k \geq 1$, we have
		\begin{equation}\label{eq5.13}\begin{array}{l}
				\mathlarger{\sum}_{i=s}^{s+k}\dfrac{X_{n,h}(c,i,s)}{X_{n,h}(c,s,s)}-\dfrac{(-q)^{-(h-s)}(1-(-q)^{-(n-h)})}{1-(-q)^{-(n-s)}}\mathlarger{\sum}_{i=s}^{s+k-1}\dfrac{X_{n-1,h}(c-1,i,s)}{X_{n-1,h}(c-1,s,s)}
				=Z_{n,h}(c,s+k,s).
			\end{array}
		\end{equation}
		By definition, we have that
		\begin{equation*}
			\dfrac{X_{n,h}(c,i,s)}{X_{n,h}(c,s,s)}=(-1)^{i-s}(-q)^{-\frac{(i-s)(i-s-1)}{2}}\dfrac{\prod_{l=c-i+s+1}^{c}(1-(-q)^{-l})\prod_{l=h-i+1}^{h-s}(1-(-q)^{-l})}{\prod_{l=1}^{i-s}(1-(-q)^{-l})\prod_{l=n-i+1}^{n-s}(1-(-q)^{-l})}.
		\end{equation*}
		In particular, when $k=1$, we have that
		\begin{equation*}\begin{aligned}
				\mathlarger{\sum}_{i=s}^{s+1}\dfrac{X_{n,h}(c,i,s)}{X_{n,h}(c,s,s)}&=1-\dfrac{(1-(-q)^{-c})(1-(-q)^{-(h-s)})}{(1-(-q)^{-1})(1-(-q)^{-(n-s)})}\\
				&=\dfrac{(-q)^{-(h-s)}(1-(-q)^{-(n-h)})}{1-(-q)^{-(n-s)}}-\dfrac{(-q)^{-1}(1-(-q)^{-(c-1)})(1-(-q)^{-(h-s)})}{(1-(-q)^{-1})(1-(-q)^{-(n-s)})}\\
				&=\dfrac{(-q)^{-(h-s)}(1-(-q)^{-(n-h)})}{1-(-q)^{-(n-s)}}\dfrac{X_{n-1,h}(c-1,s,s)}{X_{n-1,h}(c-1,s,s)}+Z_{n,h}(c,s+1,s).
			\end{aligned}
		\end{equation*}
		Therefore, \eqref{eq5.13} holds for $k=1$.
		
		Now, assume that \eqref{eq5.13} is true for $k$, then it suffices to show that
		
		\begin{equation}\label{eq5.14}\begin{array}{l}
				Z_{n,h}(c,s+k,s)+\dfrac{X_{n,h}(c,s+k+1,s)}{X_{n,h}(c,s,s)}\\
				=\dfrac{(-q)^{-(h-s)}(1-(-q)^{-(n-h)})}{1-(-q)^{-(n-s)}}\dfrac{X_{n-1,h}(c-1,s+k,s)}{X_{n-1,h}(c-1,s,s)}+Z_{n,h}(c,s+k+1,s).
			\end{array}
		\end{equation}
		
		Indeed,
		\begin{equation*}
			\begin{array}{l}
				Z_{n,h}(c,s+k,s)+\dfrac{X_{n,h}(c,s+k+1,s)}{X_{n,h}(c,s,s)}\\
				=(-1)^k(-q)^{-\frac{k(k+1)}{2}}\dfrac{\mathlarger{\prod}_{l=c-k}^{c-1}(1-(-q)^{-l})\mathlarger{\prod}_{l=h-s-k+1}^{h-s}(1-(-q)^{-l})}{\mathlarger{\prod}_{l=1}^{k}(1-(-q)^{-l})\mathlarger{\prod}_{l=n-s-k+1}^{n-s}(1-(-q)^{-l})}\mathlarger{\mathlarger{\lbrace}}1-\dfrac{(1-(-q)^{-c})(1-(-q)^{-(h-s-k)})}{(1-(-q)^{-(k+1)})(1-(-q)^{-(-n-s-k)}}\mathlarger{\mathlarger{\rbrace}},
			\end{array}
		\end{equation*}
		and
		\begin{equation*}\begin{aligned}
				1-\dfrac{(1-(-q)^{-c})(1-(-q)^{-(h-s-k)})}{(1-(-q)^{-(k+1)})(1-(-q)^{-(-n-s-k)})}&=\dfrac{(-q)^k(-q)^{-(h-s)}(1-(-q)^{-(n-h)})}{1-(-q)^{-(n-s-k)}}\\
				&\quad -\dfrac{(-q)^{-(k+1)}(1-(-q)^{-(-c-k-1)})(1-(-q)^{-(h-s-k)})}{(1-(-q)^{-(k+1)})(1-(-q)^{-(n-s-k)})}.
			\end{aligned}
		\end{equation*}
		
		Now, note that
		\begin{equation*}
			\begin{aligned}
				&(-1)^k(-q)^{-\frac{k(k+1)}{2}}\frac{\prod_{l=c-k}^{c-1}(1-(-q)^{-l})\prod_{l=h-s-k+1}^{h-s}(1-(-q)^{-l})}{\prod_{l=1}^{k}(1-(-q)^{-l})\prod_{l=n-s-k+1}^{n-s}(1-(-q)^{-l})} \times \dfrac{(-q)^k(-q)^{-(h-s)}(1-(-q)^{-(n-h)})}{1-(-q)^{-(n-s-k)}}\\
				&=(-1)^k(-q)^{-\frac{k(k-1)}{2}}\frac{\prod_{l=c-k}^{c-1}(1-(-q)^{-l})\prod_{l=h-s-k+1}^{h-s}(1-(-q)^{-l})}{\prod_{l=1}^{k}(1-(-q)^{-l})\prod_{l=n-s-k}^{n-s-1}(1-(-q)^{-l})} \times \dfrac{(-q)^{-(h-s)}(1-(-q)^{-(n-h)})}{1-(-q)^{-(n-s)}}\\
				&=\dfrac{X_{n-1,h}(c-1,s+k,s)}{X_{n-1,h}(c-1,s,s)}\dfrac{(-q)^{-(h-s)}(1-(-q)^{-(n-h)})}{1-(-q)^{-(n-s)}},
			\end{aligned}
		\end{equation*}
		and
		\begin{equation*}
			\begin{aligned}
				&(-1)^k(-q)^{-\frac{k(k+1)}{2}}\frac{\prod_{l=c-k}^{c-1}(1-(-q)^{-l})\prod_{l=h-s-k+1}^{h-s}(1-(-q)^{-l})}{\prod_{l=1}^{k}(1-(-q)^{-l})\prod_{l=n-s-k+1}^{n-s}(1-(-q)^{-l})} \\
				&\quad \times \lbrace-\dfrac{(-q)^{-(k+1)}(1-(-q)^{-(-c-k-1)})(1-(-q)^{-(h-s-k)})}{(1-(-q)^{-(k+1)})(1-(-q)^{-(n-s-k)})} \rbrace\\
				&=(-1)^{k+1}(-q)^{-\frac{(k+1)(k+2)}{2}}\frac{\prod_{l=c-k-1}^{c-1}(1-(-q)^{-l})\prod_{l=h-s-k}^{h-s}(1-(-q)^{-l})}{\prod_{l=1}^{k+1}(1-(-q)^{-l})\prod_{l=n-s-k}^{n-s}(1-(-q)^{-l})}\\
				&=Z_{n,h}(c,s+k+1,s).
			\end{aligned}
		\end{equation*}
		
		Therefore, combining these two, we have that \eqref{eq5.14} holds, and hence \eqref{eq5.13} holds. 
		
		Now, we are ready to prove \eqref{eq5.12}.
		When $h \geq s+c$, we have that $\min(s+c,h)=s+c$ and $\min(s+c-1,h)=s+c-1$. Also, by \eqref{eq5.13}, we have
		\begin{equation*}
			\mathlarger{\sum}_{i=s}^{s+c}\dfrac{X_{n,h}(c,i,s)}{X_{n,h}(c,s,s)}-\mathlarger{\sum}_{i=s}^{s+c-1}\dfrac{X_{n-1,h}(c-1,i,s)}{X_{n-1,h}(c-1,s,s)}\dfrac{(-q)^{-(h-s)}(1-(-q)^{-(n-h)})}{1-(-q)^{-(n-s)}}
			=Z_{n,h}(c,s+c,s).
		\end{equation*}
		
		Now, note that $Z_{n,h}(c,s+c,s)=0$. This implies that \eqref{eq5.12} holds when $h \geq s+c$.
		
		When $h \leq s+c-1$, we have that $\min(s+c,h)=h$ and $\min(s+c-1,h)=h$. Therefore, by \eqref{eq5.13}, we have
		\begin{equation*}\begin{array}{l}
				\sum_{i=s}^{\min(s+c,h)}\dfrac{X_{n,h}(c,i,s)}{X_{n,h}(c,s,s)}-\dfrac{(-q)^{-(h-s)}(1-(-q)^{-(n-h)})}{1-(-q)^{-(n-s)}}\sum_{i=s}^{\min(s+c-1,h)}\dfrac{X_{n-1,h}(c-1,i,s)}{X_{n-1,h}(c-1,s,s)}\\
				
				=\mathlarger{\sum}_{i=s}^{h}\dfrac{X_{n,h}(c,i,s)}{X_{n,h}(c,s,s)}-\dfrac{(-q)^{-(h-s)}(1-(-q)^{-(n-h)})}{1-(-q)^{-(n-s)}}\mathlarger{\sum}_{i=s}^{h}\dfrac{X_{n-1,h}(c-1,i,s)}{X_{n-1,h}(c-1,s,s)}\\
				=Z_{n,h}(c,h,s)-\dfrac{(-q)^{-(h-s)}(1-(-q)^{-(n-h)})}{1-(-q)^{-(n-s)}}\dfrac{X_{n-1,h}(c-1,h,s)}{X_{n-1,h}(c-1,s,s)}\\
				=0.
			\end{array}
		\end{equation*}
		
		This finishes the proof of the lemma.
	\end{proof}

	\begin{lemma}\label{lemma5.23} Let $0 \leq a,b,c,h \leq n$, $a+b+c=n$, and $0 \leq s \leq \min(h,b)$. 
		If $c \leq n-h$ and $(a,b,c) \neq (0,h,n-h)$, we have that
		\begin{equation*}\begin{array}{l}
				\sum_{i=s}^{\min(s+c,h)}M_{n,h}(a,b,c,i,s)-\sum_{i=s}^{\min(s+c,h)}M_{n,h}(a-1,b+1,c,i,s)\\
				=-(-q)^{2n-h-1-b-2c}\sum_{i=s-1}^{\min(s+c-1,h-1)}M_{n-1,h-1}(a-1,b,c,i,s-1).
			\end{array}
		\end{equation*}
	\end{lemma}
	\begin{proof}
		By Lemma \ref{lemma5.21} (1), we have that
		\begin{equation*}
			\begin{aligned}
				&\sum_{i=s}^{\min(s+c,h)}M_{n,h}(a,b,c,i,s)-\sum_{i=s}^{\min(s+c,h)}M_{n,h}(a-1,b+1,c,i,s)\\
				&=(-1)^{n+c+s+1}(-q)^{\frac{c}{2}+\frac{c^2}{2}-h^2-\frac{n}{2}-cn+hn+\frac{n^2}{2}-ch+\frac{s}{2}-\frac{s^2}{2}}\\
				&\quad \times \mathlarger{\mathlarger{\mathlarger{\lbrace}}}\dfrac{\prod_{l=1}^{b}(1-(-q)^{-l})\prod_{l=1}^{n-c-s}(1-(-q)^{-l})\prod_{l=1}^{n-c-s-1}(1-(-q)^{-l})}{\prod_{l=1}^{b-s}(1-(-q)^{-l})\prod_{l=1}^{s}(1-(-q)^{-l})\prod_{l=1}^{n-h-c}(1-(-q)^{-l})\prod_{l=1}^{h-s}(1-(-q)^{-l})}\\
				&\quad\quad\quad -\dfrac{\prod_{l=1}^{b+1}(1-(-q)^{-l})\prod_{l=1}^{n-c-s}(1-(-q)^{-l})\prod_{l=1}^{n-c-s-1}(1-(-q)^{-l})}{\prod_{l=1}^{b+1-s}(1-(-q)^{-l})\prod_{l=1}^{s}(1-(-q)^{-l})\prod_{l=1}^{n-h-c}(1-(-q)^{-l})\prod_{l=1}^{h-s}(1-(-q)^{-l})}\mathlarger{\mathlarger{\mathlarger{\rbrace}}}\\
				&=(-1)^{n+c+s}(-q)^{\frac{c}{2}+\frac{c^2}{2}-h^2-\frac{n}{2}-cn+hn+\frac{n^2}{2}-ch+\frac{s}{2}-\frac{s^2}{2}-(b+1-s)}\\
				&\quad \times \dfrac{\prod_{l=1}^{b}(1-(-q)^{-l})\prod_{l=1}^{n-c-s}(1-(-q)^{-l})\prod_{l=1}^{n-c-s-1}(1-(-q)^{-l})}{\prod_{l=1}^{b+1-s}(1-(-q)^{-l})\prod_{l=1}^{s-1}(1-(-q)^{-l})\prod_{l=1}^{n-h-c}(1-(-q)^{-l})\prod_{l=1}^{h-s}(1-(-q)^{-l})}.
			\end{aligned}
		\end{equation*}
		
		On the other hand,
		\begin{equation*}\begin{aligned}
				&-(-q)^{2n-h-1-b-2c}\sum_{i=s-1}^{\min(s+c-1,h-1)}M_{n-1,h-1}(a-1,b,c,i,s-1)\\
				&=(-1)^{n+c+s}(-q)^{\frac{c}{2}+\frac{c^2}{2}-(h-1)^2-\frac{n-1}{2}-c(n-1)+(h-1)(n-1)+\frac{(n-1)^2}{2}-c(h-1)+\frac{s-1}{2}-\frac{(s-1)^2}{2}+2n-h-1-b-2c}\\
				&\quad \times
				\dfrac{\prod_{l=1}^{b}(1-(-q)^{-l})\prod_{l=1}^{n-c-s}(1-(-q)^{-l})\prod_{l=1}^{n-c-s-1}(1-(-q)^{-l})}{\prod_{l=1}^{b+1-s}(1-(-q)^{-l})\prod_{l=1}^{s-1}(1-(-q)^{-l})\prod_{l=1}^{n-h-c}(1-(-q)^{-l})\prod_{l=1}^{h-s}(1-(-q)^{-l})}\\
				&=(-1)^{n+c+s}(-q)^{\frac{c}{2}+\frac{c^2}{2}-h^2-\frac{n}{2}-cn+hn+\frac{n^2}{2}-ch+\frac{3s}{2}-\frac{s^2}{2}-b-1}\\
				&\quad \times
				\dfrac{\prod_{l=1}^{b}(1-(-q)^{-l})\prod_{l=1}^{n-c-s}(1-(-q)^{-l})\prod_{l=1}^{n-c-s-1}(1-(-q)^{-l})}{\prod_{l=1}^{b+1-s}(1-(-q)^{-l})\prod_{l=1}^{s-1}(1-(-q)^{-l})\prod_{l=1}^{n-h-c}(1-(-q)^{-l})\prod_{l=1}^{h-s}(1-(-q)^{-l})}.
			\end{aligned}
		\end{equation*}
		
		Now, it is easy to see that the above two are the same. This finishes the proof of the lemma.
		
	\end{proof}
	\begin{theorem}\label{theorem5.24}  Let $0 \leq a,b,c,h \leq n$ and $a+b+c=n$. If $c \leq n-h$, $1\le a$, and $(a-1,b+1,c) \neq (0,h,n-h)$, we have that
		\begin{equation*}D_{n,h}(a,b,c)-D_{n,h}(a-1,b+1,c)=-(-q)^{2n-h-1-b-2c}D_{n-1,h-1}(a-1,b,c).
		\end{equation*}
	\end{theorem}
	\begin{proof}
		First, by Lemma \ref{lemma5.21}, we have that
		\begin{equation*}\begin{aligned}
				\sum_{i=0}^{\min(0+c,h)}M_{n,h}(a,b,c,i,0)=&(-1)^{n+c+1}(-q)^{\frac{c}{2}+\frac{c^2}{2}-h^2-\frac{n}{2}-cn+hn+\frac{n^2}{2}-ch}\\
				&\dfrac{\prod_{l=1}^{n-c}(1-(-q)^{-l})\prod_{l=1}^{n-c-1}(1-(-q)^{-l})}{\prod_{l=1}^{n-h-c}(1-(-q)^{-l})\prod_{l=1}^{h}(1-(-q)^{-l})}.
			\end{aligned}
		\end{equation*}
		
		Since this does not depend on $a,b$, we have that
		\begin{equation*}
			\sum_{i=0}^{\min(0+c,h)}M_{n,h}(a,b,c,i,0)=\sum_{i=0}^{\min(0+c,h)}M_{n,h}(a-1,b+1,c,i,0).
		\end{equation*}
		
		Since
		\begin{equation*}\begin{array}{l}
				D_{n,h}(a,b,c)=\mathlarger{\sum}_{s=0}^{\min(h,b)}\mathlarger{\sum}_{i=s}^{\min( s+c,h)}M_{n,h}(a,b,c,i,s),
			\end{array}
		\end{equation*}
		we have that
		\begin{equation*}\begin{array}{l}
				D_{n,h}(a,b,c)-D_{n,h}(a-1,b+1,c)
				\\\\=\mathlarger{\sum}_{s=0}^{\min(h,b)}\mathlarger{\sum}_{i=s}^{\min( s+c,h)}M_{n,h}(a,b,c,i,s)-\mathlarger{\sum}_{s=0}^{\min(h,b+1)}\mathlarger{\sum}_{i=s}^{\min(s+c,h)}M_{n,h}(a-1,b+1,c,i,s)\\\\
				=\mathlarger{\sum}_{s=1}^{\min(h,b)}\mathlarger{\sum}_{i=s}^{\min( s+c,h)}M_{n,h}(a,b,c,i,s)-\mathlarger{\sum}_{s=1}^{\min(h,b+1)}\mathlarger{\sum}_{i=s}^{\min(s+c,h)}M_{n,h}(a-1,b+1,c,i,s).
			\end{array}
		\end{equation*}
		
		Now, assume that $h \leq b$. Then $\min(h,b)=h$, $\min(h,b+1)=h$, $\min(h-1,b)=h-1$. Therefore, by Lemma \ref{lemma5.23}, we have
		\begin{equation*}
			\begin{array}{l}
				D_{n,h}(a,b,c)-D_{n,h}(a-1,b+1,c)\\
				=\mathlarger{\sum}_{s=1}^{h}\lbrace\mathlarger{\sum}_{i=s}^{\min( s+c,h)}M_{n,h}(a,b,c,i,s)-\mathlarger{\sum}_{i=s}^{\min(s+c,h)}M_{n,h}(a-1,b+1,c,i,s)\rbrace\\
				=\mathlarger{\sum}_{s=0}^{h-1}-(-q)^{2n-h-1-b-2c}\mathlarger{\sum}_{i=s}^{\min( s+c,h-1)}M_{n-1,h-1}(a-1,b,c,i,s)\\
				=-(-q)^{2n-h-1-b-2c}D_{n-1,h-1}(a-1,b,c).
			\end{array}
		\end{equation*}
		This finishes the proof of the theorem when $h \leq b$.
		
		Now, assume that $h > b$. Then $\min(h,b)=b$, $\min(h,b+1)=b+1$. Therefore, by Lemma \ref{lemma5.23}, we have that
		\begin{equation}\label{eq5.16,1}
			\begin{array}{l}
				D_{n,h}(a,b,c)-D_{n,h}(a-1,b+1,c)\\
				=\mathlarger{\sum}_{s=1}^{b}\lbrace\mathlarger{\sum}_{i=s}^{\min( s+c,h)}M_{n,h}(a,b,c,i,s)-\mathlarger{\sum}_{i=s}^{\min(s+c,h)}M_{n,h}(a-1,b+1,c,i,s)\rbrace\\
				\quad -\mathlarger{\sum}_{i=b+1}^{\min(b+1+c,h)}M_{n,h}(a-1,b+1,c,i,b+1)\\
				=\mathlarger{\sum}_{s=0}^{b-1}-(-q)^{2n-h-1-b-2c}\mathlarger{\sum}_{i=s}^{\min( s+c,h-1)}M_{n-1,h-1}(a-1,b,c,i,s)\\
				\quad -\mathlarger{\sum}_{i=b+1}^{\min(b+1+c,h)}M_{n,h}(a-1,b+1,c,i,b+1).
			\end{array}
		\end{equation}
		By Lemma \ref{lemma5.21}, we have that
		\begin{equation*}\begin{array}{l}
				-\mathlarger{\sum}_{i=b+1}^{\min(b+1+c,h)}M_{n,h}(a-1,b+1,c,i,b+1)\\
				=-(-1)^{n+c+b}(-q)^{\frac{c}{2}+\frac{c^2}{2}-h^2-\frac{n}{2}-cn+hn+\frac{n^2}{2}-ch+\frac{(b+1)}{2}-\frac{(b+1)^2}{2}}\\
				\quad \times \dfrac{\prod_{l=1}^{n-c-b-1}(1-(-q)^{-l})\prod_{l=1}^{n-c-b-2}(1-(-q)^{-l})}{\prod_{l=1}^{n-h-c}(1-(-q)^{-l})\prod_{l=1}^{h-b-1}(1-(-q)^{-l})}\\
				=-(-1)^{n-1+c+b+1}(-q)^{2n-h-1-b-2c+\frac{c}{2}+\frac{c^2}{2}-(h-1)^2-\frac{n-1}{2}-c(n-1)+(h-1)(n-1)+\frac{(n-1)^2}{2}-c(h-1)+\frac{b}{2}-\frac{b^2}{2}}
				\\
				\quad \times \dfrac{\prod_{l=1}^{n-1-c-b}(1-(-q)^{-l})\prod_{l=1}^{n-1-c-b-1}(1-(-q)^{-l})}{\prod_{l=1}^{(n-1)-(h-1)-c}(1-(-q)^{-l})\prod_{l=1}^{h-1-b}(1-(-q)^{-l})}\\
				=-(-q)^{2n-h-1-b-2c}\mathlarger{\sum}_{i=b}^{\min(b+c,h-1)}M_{n-1,h-1}(a-1,b,c,i,b).
			\end{array}
		\end{equation*}
		Therefore, since $\min(h-1,b)=b$, \eqref{eq5.16,1} can be written as
		\begin{equation*}
			\begin{array}{l}
				D_{n,h}(a,b,c)-D_{n,h}(a-1,b+1,c)\\
				=-(-q)^{2n-h-1-b-2c}\quad \mathlarger{\sum}_{s=0}^{b}\quad\mathlarger{\sum}_{i=s}^{\min( s+c,h-1)}M_{n-1,h-1}(a-1,b,c,i,s)\\
				=-(-q)^{2n-h-1-b-2c}D_{n-1,h-1}(a-1,b,c).
			\end{array}
		\end{equation*}
		This finishes the proof of the theorem.
	\end{proof}
	
	\begin{lemma}\label{lemma7.5}  Let $0 \leq a,b,c,h \leq n$, and $a+b+c=n$. If $c > n-h$ and $a \neq 0$, we have that
		\begin{equation*}D_{n,h}(a,b,c)=0.
		\end{equation*}
		Therefore, for $\lambda \in \CR_n^{0+}$ such that $t_0(\lambda) > n-h$, and $t_{\ge 2}(\lambda) \neq 0$, we have that
		\begin{equation*}
			D_{n,h}(\lambda)=0.
		\end{equation*}
	\end{lemma}
	\begin{proof}
		This follows from Lemma \ref{lemma5.21} (2), Definition \ref{definition5.19} (3) and Proposition \ref{proposition5.20}.
	\end{proof}

	\begin{proposition}\label{prop: vanishing of D(lambda)}
		For $0 \leq t \leq h-1$, $t \equiv h+1 \pmod 2$, and $\lambda=(1^{t},0^{n-t}) \in \CR_n^{0+}$, we have that
		\begin{equation*}
			D_{n,h}(\lambda)=0.
		\end{equation*}
	\end{proposition}
	\begin{proof}
		First, by Proposition \ref{proposition5.20}, we have that
		\begin{equation*}
			D_{n,h}(\lambda)=D_{n,h}(0,t,n-t)=\mathlarger{\sum}_{\substack{0 \leq s \leq \min(h,b)\\ s \leq i \leq \min( s+c,h)}}M_{n,h}(a,b,c,i,s)+\dfrac{(-q)^{\frac{-(h-t)(h+t+1)}{2}}}{1-(-q)^{-(h-t)}}.
		\end{equation*}
		
		Before we start the proof of the proposition, let us say why we consider this case separately. Recall from Definition \ref{definition5.19} (1) that $\CC_l(0,0,k)$ is defined separately. If $\lambda$ is not of the form $(1^{t},0^{n-t})$, $t <h$, these constants $\CC_l(0,0,k)$'s do not appear in the sum $\mathlarger{\sum}_{\substack{0 \leq s \leq \min(h,b)\\ s \leq i \leq \min( s+c,h)}}M_{n,h}(a,b,c,i,s).$ However, if $\lambda=(1^t,0^{n-t})$, $t<h$, $M_{n,h}(0,t,n-t,i,t)$ has the term $\CC_{h+1-t}(0,0,n-i)$. Therefore, we need to be careful when we compute $D_{n,h}(0,t,n-t)$.

		Now, let us start the proof of the proposition. Recall from the proof of the Lemma \ref{lemma5.21} that we have
		\begin{equation*}
			M_{n,h}(a,b,c,i,s)=\CC_{h+1-s}(a,b-s,c+s-i)Y_{n,h}(b,c,s)X_{n,h}(c,i,s),
		\end{equation*}
		
		Also, by \eqref{lemma5.22}, we have that
		\begin{equation*}\begin{array}{l}
				\sum_{i=s}^{\min(s+n-t,h)}Y_{n,h}(t,n-t,s)X_{n,h}(n-t,i,s)=Y_{n,h}(t,n-t,s)\sum_{i=s}^{\min(s+n-t,h)}X_{n,h}(n-t,i,s)=0
			\end{array}
		\end{equation*}
		
		Also, note that $\min(h,t)=t$ and $\CC_{h+1-s}(0,t-s,n-t+s-i)$ does not depend on $i$ if $s \neq t$. Therefore, we have that
		
		\begin{equation*}
			\begin{array}{l}
				D_{n,h}(0,t,n-t)-\dfrac{(-q)^{\frac{-(h-t)(h+t+1)}{2}}}{1-(-q)^{-(h-t)}}\\
				=\mathlarger{\sum}_{\substack{0 \leq s \leq \min(h,t)\\ s \leq i \leq \min(s+n-t,h)}}\CC_{h+1-s}(0,t-s,n-t+s-i) Y_{n,h}(t,n-t,s)X_{n,h}(n-t,i,s)\\
				=\mathlarger{\sum}_{s=0}^{t-1}\CC_{h+1-s}(0,t-s,n-t+s)Y_{n,h}(t,n-t,s) \mathlarger{\sum}_{i=s}^{\min(s+n-t,h)}X_{n,h}(n-t,i,s)\\
				\quad +\mathlarger{\sum}_{i=t}^h\CC_{h+1-t}(0,0,n-i) Y_{n,h}(t,n-t,t)X_{n,h}(n-t,i,t)\\\\
				=\mathlarger{\sum}_{i=t}^h\CC_{h+1-t}(0,0,n-i)Y_{n,h}(t,n-t,t)X_{n,h}(n-t,i,t).
			\end{array}
		\end{equation*}
		
		Now, recall that $\CC_{h+1-t}(0,0,n-i)=\dfrac{\alphad'(I_{n-i},I_{n-i})}{\alphad(I_{n-i},I_{n-i})}=\mathlarger{\sum}_{l=1}^{n-i}\dfrac{1}{(-q)^l-1}.$
		
		This implies that it suffices to show that
		\begin{equation}\label{eq5.17}
			Y_{n,h}(t,n-t,t)\mathlarger{\sum}_{i=t}^h(\mathlarger{\sum}_{l=1}^{n-i}\dfrac{1}{(-q)^l-1}) X_{n,h}(n-t,i,t)=-\dfrac{(-q)^{\frac{-(h-t)(h+t+1)}{2}}}{1-(-q)^{-(h-t)}}.
		\end{equation}
		
		First, we have that
		\begin{equation*}\begin{aligned}
				\mathlarger{\sum}_{i=t}^h(\mathlarger{\sum}_{l=1}^{n-i}\dfrac{1}{(-q)^l-1}) X_{n,h}(n-t,i,t)=&\mathlarger{\sum}_{i=t}^hX_{n,h}(n-t,i,t)(\mathlarger{\sum}_{l=1}^{n-h}\dfrac{1}{(-q)^l-1}) \\
				&+\mathlarger{\sum}_{k=0}^{h-t-1}\dfrac{1}{(-q)^{n-t-k}-1}\mathlarger{\sum}_{i=t}^{t+k}X_{n,h}(n-t,i,t).
			\end{aligned}
		\end{equation*}
		Since $\mathlarger{\sum}_{i=t}^hX_{n,h}(n-t,i,t)=0$ (by \eqref{lemma5.22}), we have that
		\begin{equation}\label{eq5.19}
			\mathlarger{\sum}_{i=t}^h(\mathlarger{\sum}_{l=1}^{n-i}\dfrac{1}{(-q)^l-1}) X_{n,h}(n-t,i,t)=\mathlarger{\sum}_{k=0}^{h-t-1}\dfrac{1}{(-q)^{n-t-k}-1}\mathlarger{\sum}_{i=t}^{t+k}X_{n,h}(n-t,i,t).
		\end{equation}
		
		We claim that
		\begin{equation}\label{eq5.20}
			\mathlarger{\sum}_{i=t}^{t+k} X_{n,h}(n-t,i,t)=\dfrac{(-1)^{t+k}(-q)^{\frac{(t+k+1)(t-k)}{2}}}{\prod_{l=1}^{h-t-k-1}(1-(-q)^{-l})\prod_{l=1}^{k}(1-(-q)^{-l})(1-(-q)^{-(h-t)})}.
		\end{equation}
		
		Indeed, it is easy to see that this holds for $k=0$. Now, assume that this holds for $k$. Then, we have that
		\begin{equation*}
			\begin{array}{l}
				\mathlarger{\sum}_{i=t}^{t+k+1}X_{n,h}(n-t,i,t)=\mathlarger{\sum}_{i=t}^{t+k+1} \dfrac{(-1)^i(-q)^{\frac{i}{2}-\frac{i^2}{2}+it}}{\prod_{l=1}^{h-i}(1-(-q)^{-l})\prod_{l=1}^{i-t}(1-(-q)^{-l})}\\
				=\dfrac{(-1)^{t+k}(-q)^{\frac{(t+k+1)(t-k)}{2}}}{\mathlarger{\prod}_{l=1}^{h-t-k-1}(1-(-q)^{-l})\mathlarger{\prod}_{l=1}^{k}(1-(-q)^{-l})(1-(-q)^{-(h-t)})}+\dfrac{(-1)^{t+k+1}(-q)^{\frac{t+k+1}{2}-\frac{(t+k+1)^2}{2}+(t+k+1)t}}{\mathlarger{\prod}_{l=1}^{h-t-k-1}(1-(-q)^{-l})\mathlarger{\prod}_{l=1}^{k+1}(1-(-q)^{-l})}\\
				=\dfrac{(-1)^{t+k+1}(-q)^{\frac{(t+k+1)(t-k)}{2}}}{\prod_{l=1}^{h-t-k-1}(1-(-q)^{-l})\prod_{l=1}^{k+1}(1-(-q)^{-l})}\lbrace-\dfrac{(1-(-q)^{-(k+1)})}{1-(-q)^{-(h-t)}}+1\rbrace\\
				=\dfrac{(-1)^{t+k+1}(-q)^{\frac{(t+k+1)(t-k)}{2}}}{\prod_{l=1}^{h-t-k-1}(1-(-q)^{-l})\prod_{l=1}^{k+1}(1-(-q)^{-l})}\dfrac{(-q)^{-(k+1)}(1-(-q)^{-h-t+k+1})}{(1-(-q)^{-(h-t)}}\\
				=\dfrac{(-1)^{t+k+1}(-q)^{\frac{(t+k+2)(t-k-1)}{2}}}{\prod_{l=1}^{h-t-k-2}(1-(-q)^{-l})\prod_{l=1}^{k+1}(1-(-q)^{-l})(1-(-q)^{-(h-t)})}.
			\end{array}
		\end{equation*}
		This shows that \eqref{eq5.20} holds.
		
		Now, by \eqref{eq5.20}, we have that \eqref{eq5.19} can be written as
		\begin{equation}\label{eq5.21}
			\begin{array}{ll}
				\mathlarger{\sum}_{i=t}^h(\mathlarger{\sum}_{l=1}^{n-i}\dfrac{1}{(-q)^l-1}) X_{n,h}(n-t,i,t)=\mathlarger{\sum}_{k=0}^{h-t-1}\dfrac{1}{(-q)^{n-t-k}-1}\mathlarger{\sum}_{i=t}^{t+k}X_{n,h}(n-t,i,t)\\
				=\dfrac{(-1)^{t}(-q)^{\frac{t^2}{2}+\frac{3t}{2}-n}}{(1-(-q)^{-(h-t)})}\times
				\mathlarger{\sum}_{k=0}^{h-t-1}\dfrac{(-1)^k(-q)^{-\frac{k(k-1)}{2}}}{(1-(-q)^{-(n-t-k)})\prod_{l=1}^{h-t-k-1}(1-(-q)^{-l})\prod_{l=1}^{k}(1-(-q)^{-l})}.
			\end{array}
		\end{equation}
		
		In Lemma \ref{lemma5.30} below, we will prove that
		\begin{equation}\label{eq5.22}
			\begin{array}{l}
				\mathlarger{\sum}_{k=0}^{h-t-1}\dfrac{(-1)^k(-q)^{-\frac{k(k-1)}{2}}}{(1-(-q)^{-(n-t-k)})\prod_{l=1}^{h-t-k-1}(1-(-q)^{-l})\prod_{l=1}^{k}(1-(-q)^{-l})}\\
				=\dfrac{(-1)^{h-t-1}(-q)^{-\frac{(h-t)(h-t-1)}{2}}(-q)^{-(n-h)(h-t-1)}}{\prod_{l=n-h+1}^{n-t}(1-(-q)^{-l})}.
			\end{array}
		\end{equation}
		
		Combining \eqref{eq5.21} and \eqref{eq5.22}, we have that
		\begin{equation*}\begin{array}{l}
				Y_{n,h}(t,n-t,t)\mathlarger{\sum}_{i=t}^h(\mathlarger{\sum}_{l=1}^{n-i}\dfrac{1}{(-q)^l-1}) X_{n,h}(n-t,i,t)\\
				=\dfrac{(-1)^{h}(-q)^{-h^2+hn-\frac{t}{2}+\frac{t^2}{2}-nt}\prod_{l=1}^{n-t}(1-(-q)^{-l})}{\prod_{l=1}^{n-h}(1-(-q)^{-l})}\\
				\times \dfrac{(-1)^t(-q)^{\frac{t^2}{2}+\frac{3t}{2}-n}}{(1-(-q)^{-(h-t)})}
				\times \dfrac{(-1)^{h-t-1}(-q)^{-\frac{(h-t)(h-t-1)}{2}}(-q)^{-(n-h)(h-t-1)}}{\prod_{l=n-h+1}^{n-t}(1-(-q)^{-l})}\\
				=\dfrac{-(-q)^{\frac{-(h-t)(h+t+1)}{2}}}{1-(-q)^{-(h-t)}}.
				
			\end{array}
		\end{equation*}

		This finishes the proof of the proposition.
		
	\end{proof}

	\begin{lemma}\label{lemma5.30} For $0 \leq t <h \leq n$, we have
		\begin{equation*}
			\begin{array}{l}
				\mathlarger{\sum}_{k=0}^{h-t-1}\dfrac{(-1)^k(-q)^{-\frac{k(k-1)}{2}}}{(1-(-q)^{-(n-t-k)})\prod_{l=1}^{h-t-k-1}(1-(-q)^{-l})\prod_{l=1}^{k}(1-(-q)^{-l})}\\\\
				=\dfrac{(-1)^{h-t-1}(-q)^{-\frac{(h-t)(h-t-1)}{2}}(-q)^{-(n-h)(h-t-1)}}{\prod_{l=n-h+1}^{n-t}(1-(-q)^{-l})}.
			\end{array}
		\end{equation*}
	\end{lemma}
	\begin{proof}
		For $N >0$, we claim the following statement: 
		\begin{equation}\label{eq5.23}
			\sum_{k=0}^{N-1}\frac{(-1)^{k}(-q)^{-\frac{k(k-1)}{2}}}{(1-(-q)^{-(N-k)}X)\prod_{l=1}^{N-k-1}(1-(-q)^{-l})\prod_{l=1}^{k}(1-(-q)^{-l})}=\frac{(-1)^{N-1}(-q)^{-\frac{N(N-1)}{2}}X^{N-1}}{\prod_{l=1}^{N}(1-(-q)^{-l}X)},
		\end{equation}
		which specializes to the statement of the lemma when $X=(-q)^{-(n-h)}$. Therefore it suffices to show the claim \eqref{eq5.23}.
		
		Consider a Vandermonde matrix
		\begin{equation*}
			\FX=\left(\begin{array}{cccc}
				1 & 1& \dots & 1\\
				x_1& x_2 & \dots & x_N\\
				\vdots &\vdots &\ddots& \vdots\\
				x_1^{N-1} & x_2^{N-1} &\dots & x_N^{N-1}
			\end{array}\right),
		\end{equation*}
		and let $\FX^{-1}=(y_{ij})_{1 \leq i,j \leq N}$. Note that $y_{ij}$ is the $X^{N-j}$-coefficient of
		\begin{equation*}
			\prod_{l=1,l\neq i}^{N} \dfrac{1-x_l X}{x_i-x_l}.
		\end{equation*}
		
		Now, assume that $x_l=(-q)^{-l}$. Then, we have that $y_{(N-k)j}$ is the $X^{N-j}$-coefficient of
		\begin{equation*}\begin{array}{l}
				\mathlarger{\prod}_{l=1,l\neq (N-k)}^{N} \dfrac{(1-(-q)^{-l} X)}{(-q)^{-(N-k)}-(-q)^{-l}}= \dfrac{(-1)^{N-k-1}(-q)^{\frac{(N-k)(N+k-1)}{2}}\prod_{l=1,l\neq (N-k)}^{N}(1-(-q)^{-l} X)}{\prod_{l=1}^{N-k-1}(1-(-q)^{-l})\prod_{l=1}^{k}(1-(-q)^{-l})}\\\\
				=(-1)^{N-1}(-q)^{\frac{N(N-1)}{2}}\dfrac{(-1)^{k}(-q)^{-\frac{k(k-1)}{2}}\prod_{l=1,l\neq (N-k)}^{N}(1-(-q)^{-l} X)}{\prod_{l=1}^{N-k-1}(1-(-q)^{-l})\prod_{l=1}^{k}(1-(-q)^{-l})}.
			\end{array}
		\end{equation*}
		Therefore, we have that 
		\begin{equation}\label{eq5.24} 
			\mathlarger{\sum}_{j=1}^{N}\mathlarger{\sum}_{k=0}^{N-1}y_{(N-k)j}X^{N-j}=(-1)^{N-1}(-q)^{\frac{N(N-1)}{2}}\mathlarger{\sum}_{k=0}^{N-1}\dfrac{(-1)^{k}(-q)^{-\frac{k(k-1)}{2}}\prod_{l=1, l \neq N-k}^{N}(1-(-q)^{-l}X)}{\prod_{l=1}^{N-k-1}(1-(-q)^{-l})\prod_{l=1}^{k}(1-(-q)^{-l})}.
		\end{equation}

		Now, note that
		\begin{equation*}
			(1 \quad 0 \dots 0) \FX=(1 \quad 1 \dots 1) \Longleftrightarrow (1 \quad 0 \dots 0)=(1 \quad 1 \dots 1)\FX^{-1}=(\mathlarger{\sum}_{i=1}^{N}y_{i1} \quad \mathlarger{\sum}_{i=1}^{N}y_{i2} \dots \mathlarger{\sum}_{i=1}^{N}y_{iN}).
		\end{equation*}
		Hence, we have 
		\begin{equation}\label{eq5.25}
			\mathlarger{\sum}_{j=1}^{N}\mathlarger{\sum}_{k=0}^{N-1}y_{(N-k)j}X^{N-j}=X^{N-1}.
		\end{equation}
		Therefore, \eqref{eq5.24} and \eqref{eq5.25} imply that \eqref{eq5.23} holds.
		This finishes the proof of the lemma.
	\end{proof}
	
	Now, let us state the following theorem.
	\begin{theorem}\label{theorem5.31}
		Assume that $0 \leq h \leq n$.
		\begin{enumerate}
			\item Assume that $h+1 \leq b \leq n$, and $b+c=n$. Then, we have
			\begin{equation*}
				D_{n,h}(0,b,c)=\dfrac{\prod_{l=h+1}^{b}(1-(-q)^l)}{(1-(-q)^{b-h})}.
			\end{equation*}
			
			\item Assume that $h-1 \leq b \leq n$, and $b+c=n-1$. Then, we have
			\begin{equation*}
				D_{n,h}(1,b,c)=\left\lbrace \begin{array}{ll}
					1 & \text{ if } b=h-1, h,\\
					\prod_{l=h+1}^{b}(1-(-q)^l) & \text{ if } b \geq h+1.
				\end{array}\right.
			\end{equation*}
		\end{enumerate}
	\end{theorem}
	\begin{proof}
		Recall from Lemma \ref{lemma5.21} that for $a=0$ or $1$ (i.e., $b+c=n$ or $n-1$, respectively), we have
		\begin{equation}\label{eq5.26.3}\begin{array}{l}
				D_{n,h}(a,b,c)=\sum_{s=0}^{\min(h,b)}\sum_{i=s}^{\min(s+c,h)}M_{n,h}(a,b,c,i,s)\\
				=(-1)^{n+c+1}(-q)^{\frac{c}{2}+\frac{c^2}{2}-h^2-\frac{n}{2}-cn+hn+\frac{n^2}{2}-ch}\dfrac{\prod_{l=1}^{n-c-a}(1-(-q)^{-l})}{\prod_{l=1}^{n-h-c}(1-(-q)^{-l})}\\
				\quad \times \sum_{s=0}^{\min(h,b)} \dfrac{(-1)^s(-q)^{\frac{s}{2}-\frac{s^2}{2}}\prod_{l=1}^{n-c-s+a-1}(1-(-q)^{-l})}{\prod_{l=1}^{s}(1-(-q)^{-l})\prod_{l=1}^{h-s}(1-(-q)^{-l})}.
				
			\end{array}
		\end{equation}
		
		First, assume that $a=1$, $b=h-1$, and $c=n-h$. Then, we have
		\begin{equation*}
			\begin{array}{l}
				\mathlarger{\sum}_{s=0}^{\min(h,b)} \dfrac{(-1)^s(-q)^{\frac{s}{2}-\frac{s^2}{2}}\prod_{l=1}^{n-c-s+a-1}(1-(-q)^{-l})}{\prod_{l=1}^{s}(1-(-q)^{-l})\prod_{l=1}^{h-s}(1-(-q)^{-l})}=\mathlarger{\sum}_{s=0}^{h-1} \dfrac{(-1)^s(-q)^{\frac{s}{2}-\frac{s^2}{2}}}{\prod_{l=1}^{s}(1-(-q)^{-l})}
				=\dfrac{(-1)^{h-1}(-q)^{-\frac{h(h-1)}{2}}}{\prod_{l=1}^{h-1}(1-(-q)^{-l})}.
			\end{array}
		\end{equation*}
		
		Therefore, we have
		\begin{equation*}\begin{aligned}
				D_{n,h}(1,h-1,n-h)&=(-1)^{n+c+h}(-q)^{\frac{c}{2}+\frac{c^2}{2}-h^2-\frac{n}{2}-cn+hn+\frac{n^2}{2}-ch-h(h-1)/2}\dfrac{\prod_{l=1}^{h-1}(1-(-q)^{-l})}{\prod_{l=1}^{h-1}(1-(-q)^{-l})}\\
				&=(-q)^{\frac{(n-c-h)(n-c+3h-1)}{2}}=1.
			\end{aligned}
		\end{equation*}

		Now, assume that $b \geq h$. In this case, we have that $\min(h,b)=h$. We claim that
		\begin{equation}\label{eq5.26.2}
			\mathlarger{\sum}_{s=0}^{h} \dfrac{(-1)^s(-q)^{\frac{s}{2}-\frac{s^2}{2}}\prod_{l=1}^{N-s}(1-(-q)^{-l})}{\prod_{l=1}^{s}(1-(-q)^{-l})\prod_{l=1}^{h-s}(1-(-q)^{-l})}=\dfrac{(-1)^h(-q)^{-hN+\frac{h(h-1)}{2}}\prod_{l=1}^{N-h}(1-(-q)^{-l})}{\prod_{l=1}^{h}(1-(-q)^{-l})}, 
		\end{equation}
		where $N=n-c+a-1$.
		
		To prove \eqref{eq5.26.2}, we define the following constants: for $0 \leq k \leq h$, $1 \leq t \leq k+1$, 
		\begin{equation*}
			\begin{aligned}
				\omega_{k,t}=\left\lbrace \begin{array}{cl} \dfrac{(-1)^k(-q)^{-(t-1)N+\frac{(t-1)(t-2)}{2}-\frac{(k-t+1)(k-t+2)}{2}}\prod_{l=1}^{N-k}(1-(-q)^{-l})}{\prod_{l=1}^{k-t+1}(1-(-q)^{-l})\prod_{l=1}^{h-k-1}(1-(-q)^{-l})\prod_{l=h-t+1}^{h}(1-(-q)^{-l})}& \text{ if } k \leq h-1,\\
					0 &\text{ if }k=h, t\neq h+1,\\
					\dfrac{(-1)^h(-q)^{-hN+\frac{h(h-1)}{2}}\prod_{l=1}^{N-h}(1-(-q)^{-l})}{\prod_{l=1}^{h}(1-(-q)^{-l})} & \text{ if } k=h, t=h+1,\end{array}\right.
			\end{aligned}
		\end{equation*}
		and  
		\begin{equation*}\begin{array}{l}
				\tau_{k,t}=\dfrac{(-1)^k(-q)^{-tN+\frac{t(t-1)}{2}-\frac{(k-t-1)(k-t)}{2}}\prod_{l=1}^{N-k}(1-(-q)^{-l})}{\prod_{l=1}^{k-t}(1-(-q)^{-l})\prod_{l=1}^{h-k}(1-(-q)^{-l})\prod_{l=h-t+1}^{h}(1-(-q)^{-l})}.
			\end{array}
		\end{equation*}
		Note that $$\omega_{0,1}=\dfrac{\prod_{l=1}^{N}(1-(-q)^{-l})}{\prod_{l=1}^{h}(1-(-q)^{-l})}=\dfrac{(-1)^0(-q)^{\frac{0}{2}-\frac{0^2}{2}}\prod_{l=1}^{N-0}(1-(-q)^{-l})}{\prod_{l=1}^{0}(1-(-q)^{-l})\prod_{l=1}^{h-0}(1-(-q)^{-l})},$$
		and
		\begin{equation*}
			\tau_{k,k}=\dfrac{(-1)^k(-q)^{-kN+\frac{k(k-1)}{2}}\prod_{l=1}^{N-k}(1-(-q)^{-l})}{\prod_{l=1}^{h}(1-(-q)^{-l})}=\omega_{k,k+1}.
		\end{equation*}
		
		We claim the following two equations: for $k \geq 1$, $2 \leq t \leq k$, we have
		\begin{equation}\label{eq5.27.2}
			\dfrac{(-1)^k(-q)^{\frac{k}{2}-\frac{k^2}{2}}\prod_{l=1}^{N-k}(1-(-q)^{-l})}{\prod_{l=1}^{k}(1-(-q)^{-l})\prod_{l=1}^{h-k}(1-(-q)^{-l})}+\omega_{k-1,1}=\omega_{k,1}+\tau_{k,1},
		\end{equation}
		and
		\begin{equation}\label{eq5.28.2}
			\tau_{k,t-1}+\omega_{k-1,t}=\omega_{k,t}+\tau_{k,t}.
		\end{equation}

		First, we have
		\begin{equation*}
			\begin{aligned}
				&\dfrac{(-1)^k(-q)^{\frac{k}{2}-\frac{k^2}{2}}\prod_{l=1}^{N-k}(1-(-q)^{-l})}{\prod_{l=1}^{k}(1-(-q)^{-l})\prod_{l=1}^{h-k}(1-(-q)^{-l})}+\omega_{k-1,1}\\
				&	=\dfrac{(-1)^k(-q)^{\frac{k}{2}-\frac{k^2}{2}}\prod_{l=1}^{N-k}(1-(-q)^{-l})}{\prod_{l=1}^{k-1}(1-(-q)^{-l})\prod_{l=1}^{h-k}(1-(-q)^{-l})}\mathlarger{\lbrace} \dfrac{1}{1-(-q)^{-k}}-\dfrac{1-(-q)^{-(N-k+1)}}{1-(-q)^{-h}}\mathlarger{\rbrace}\\
				&	=\dfrac{(-1)^k(-q)^{\frac{k}{2}-\frac{k^2}{2}}\prod_{l=1}^{N-k}(1-(-q)^{-l})}{\prod_{l=1}^{k-1}(1-(-q)^{-l})\prod_{l=1}^{h-k}(1-(-q)^{-l})}\mathlarger{\lbrace} \dfrac{(-q)^{-k}(1-(-q)^{-(h-k)})}{(1-(-q)^{-k})(1-(-q)^{-h})}+\dfrac{(-q)^{-(N-k+1)}}{1-(-q)^{-h}}\mathlarger{\rbrace}\\
				&	=\omega_{k,1}+\tau_{k,1}.
			\end{aligned}
		\end{equation*}
		This shows that \eqref{eq5.27.2} holds.
		
		For \eqref{eq5.28.2}, we have \begin{align*}
			\tau_{k,t-1}+\omega_{k-1,t}=&\dfrac{(-1)^k(-q)^{-(t-1)N+\frac{(t-1)(t-2)}{2}-\frac{(k-t)(k-t+1)}{2}}\prod_{l=1}^{N-k}(1-(-q)^{-l})}{\prod_{l=1}^{k-t}(1-(-q)^{-l})\prod_{l=1}^{h-k}(1-(-q)^{-l})\prod_{l=h-t+2}^{h}(1-(-q)^{-l})}\\
			&\times \mathlarger{\lbrace} \dfrac{1}{(1-(-q)^{-(k-t+1)}}-\dfrac{1-(-q)^{-(N-k+1)}}{1-(-q)^{-(h-t+1)}}\mathlarger{\rbrace}\\ 
			=&\dfrac{(-1)^k(-q)^{-(t-1)N+\frac{(t-1)(t-2)}{2}-\frac{(k-t)(k-t+1)}{2}}\prod_{l=1}^{N-k}(1-(-q)^{-l})}{\prod_{l=1}^{k-t}(1-(-q)^{-l})\prod_{l=1}^{h-k}(1-(-q)^{-l})\prod_{l=h-t+2}^{h}(1-(-q)^{-l})}\\
			&\times \mathlarger{\lbrace} \dfrac{(-q)^{-(k-t+1)}(1-(-q)^{-(h-k)})}{(1-(-q)^{-(k-t+1)})(1-(-q)^{-(h-t+1)})}+\dfrac{(-q)^{-(N-k+1)}}{1-(-q)^{-(h-t+1)}}\mathlarger{\rbrace}\\
			=&\omega_{k,t}+\tau_{k,t}.
		\end{align*}
		This shows that \eqref{eq5.28.2} holds.
		Now, by \eqref{eq5.27.2} and \eqref{eq5.28.2}, we have that for $k \geq 1$,
		\begin{equation*}\begin{aligned}
				&\dfrac{(-1)^k(-q)^{\frac{k}{2}-\frac{k^2}{2}}\prod_{l=1}^{N-k}(1-(-q)^{-l})}{\prod_{l=1}^{k}(1-(-q)^{-l})\prod_{l=1}^{h-k}(1-(-q)^{-l})}+\mathlarger{\sum}_{t=1}^{k}\omega_{k-1,t}+\mathlarger{\sum}_{t=1}^{k-1}\tau_{k,t}=\mathlarger{\sum}_{t=1}^{k}\omega_{k,t}+\mathlarger{\sum}_{t=1}^{k}\tau_{k,t},
			\end{aligned}
		\end{equation*}
		and hence
		\begin{equation}\label{eq: identity}
			\begin{aligned}
				\dfrac{(-1)^k(-q)^{\frac{k}{2}-\frac{k^2}{2}}\prod_{l=1}^{N-k}(1-(-q)^{-l})}{\prod_{l=1}^{k}(1-(-q)^{-l})\prod_{l=1}^{h-k}(1-(-q)^{-l})}+\mathlarger{\sum}_{t=1}^{k}\omega_{k-1,t}=\mathlarger{\sum}_{t=1}^{k}\omega_{k,t}+\tau_{k,k}=\mathlarger{\sum}_{t=1}^{k+1}\omega_{k,t}.\end{aligned}
		\end{equation}
		Here, we used $\tau_{k,k}=\omega_{k,k+1}$ for the last identity.

		Therefore, by \eqref{eq: identity}, we have
		\begin{equation*}\begin{aligned}
				&\mathlarger{\sum}_{k=1}^{h}\dfrac{(-1)^k(-q)^{\frac{k}{2}-\frac{k^2}{2}}\prod_{l=1}^{N-k}(1-(-q)^{-l})}{\prod_{l=1}^{k}(1-(-q)^{-l})\prod_{l=1}^{h-k}(1-(-q)^{-l})}+\mathlarger{\sum}_{k=1}^{h}\mathlarger{\sum}_{t=1}^{k}\omega_{k-1,t}=\mathlarger{\sum}_{k=1}^{h}\mathlarger{\sum}_{t=1}^{k+1}\omega_{k,t}\\
				\Longleftrightarrow & \mathlarger{\sum}_{k=1}^{h}\dfrac{(-1)^k(-q)^{\frac{k}{2}-\frac{k^2}{2}}\prod_{l=1}^{N-k}(1-(-q)^{-l})}{\prod_{l=1}^{k}(1-(-q)^{-l})\prod_{l=1}^{h-k}(1-(-q)^{-l})}+\omega_{0,1}=\mathlarger{\sum}_{t=1}^{h+1}\omega_{h,t}\\
				\Longleftrightarrow & \mathlarger{\sum}_{k=0}^{h}\dfrac{(-1)^k(-q)^{\frac{k}{2}-\frac{k^2}{2}}\prod_{l=1}^{N-k}(1-(-q)^{-l})}{\prod_{l=1}^{k}(1-(-q)^{-l})\prod_{l=1}^{h-k}(1-(-q)^{-l})}=\omega_{h,h+1} \text{ (since $\omega_{h,t}=0$ for all $t<h+1$)}\\
				&=\dfrac{(-1)^h(-q)^{-hN+\frac{h(h-1)}{2}}\prod_{l=1}^{N-h}(1-(-q)^{-l})}{\prod_{l=1}^{h}(1-(-q)^{-l})}.
			\end{aligned}
		\end{equation*}
		This shows that \eqref{eq5.26.2} holds.
		
		Combining \eqref{eq5.26.3} and \eqref{eq5.26.2}, we have
		\begin{equation*}\begin{aligned}
				D_{n,h}(a,b,c)
				&=(-1)^{n+c+h+1}(-q)^{\frac{c}{2}+\frac{c^2}{2}-h^2-\frac{n}{2}-cn+hn+\frac{n^2}{2}-ch-h(n-c+a-1)+\frac{h(h-1)}{2}}\\
				&\quad \times \dfrac{\prod_{l=1}^{n-c-a}(1-(-q)^{-l})\prod_{l=1}^{n-c+a-1-h}(1-(-q)^{-l})}{\prod_{l=1}^{n-h-c}(1-(-q)^{-l})\prod_{l=1}^{h}(1-(-q)^{-l})}.
			\end{aligned}
		\end{equation*}
		If $a=1$, then we have $b=n-c-1$, and
		\begin{equation*}
			\begin{aligned}
				D_{n,h}(1,b,c)&=(-1)^{n+c+h+1}(-q)^{\frac{(n-c-h-1)(n-c+h)}{2}}\dfrac{\prod_{l=1}^{b}(1-(-q)^{-l})}{\prod_{l=1}^{h}(1-(-q)^{-l})}\\
				&=\prod_{l=h+1}^{b}(1-(-q)^l).
			\end{aligned}
		\end{equation*}
		If $a=0$, then we have $b=n-c$, and
		\begin{equation*}
			\begin{aligned}
				D_{n,h}(0,b,c)&=(-1)^{n+c+h+1}(-q)^{\frac{(n-c-h)(n-c+h-1)}{2}}\dfrac{\prod_{l=1}^{b}(1-(-q)^{-l})}{(1-(-q)^{-(n-c-h)})\prod_{l=1}^{h}(1-(-q)^{-l})}\\
				&=\dfrac{\prod_{l=h+1}^{b}(1-(-q)^l)}{(1-(-q)^{b-h})}.
			\end{aligned}
		\end{equation*}
		This finishes the proof of the theorem.
	\end{proof}
	
	\section{Fourier transform}\label{sec: FT}
	In this section, we will prove certain theorems on the Fourier transform of the analytic side of Conjecture \ref{conjecture5.5}. Recall that $\BV$ is the space of special homomorphisms associated with $\CN_{n}^{[h]}$ and it is split (resp. non-split) if $h$ is odd (resp. even). For an integrable function $f$ on $\BV$, we denote by $\widehat{f}$ its Fourier transform
	\begin{equation*}
		\widehat{f}(x)\coloneqq \int_{\BV}f(y)\psi(\mathrm{Tr}_{F/F_0}\langle x,y\rangle)dy, \quad x \in \BV.
	\end{equation*}
	For example, for an $O_F$-lattice $L \in \BV$ of rank $n$, we have
	\begin{equation*}
		\widehat{1}_L=\vol(L)1_{L^{\vee}}=q^{-\val(L)/2}1_{L^{\vee}}.
	\end{equation*}

	\begin{definition} Let $L \subset \BV$ be an $O_F$-lattice of rank $n$ and let $L^{\flat} \subset \BV$ be an $O_F$-lattice of rank $n-1$.
		\begin{enumerate}
			\item Assume that $\lambda \in \CR_{n}^{0+}$ be the fundamental invariants of $L$. Then we define 
			\begin{equation*}
				D_{n,h}(L)=D_{n,h}(\lambda).
			\end{equation*}
			\item For $x \in \BV \backslash L_F^{\flat}$, we define
			\begin{equation*}
				\partial \mathrm{Den}_{L^{\flat}}^{n,h}(x)\coloneqq \mathlarger{\sum}_{L^{\flat} \subset L' \subset L'^{\vee}} D_{n,h}(L')1_{L'}(x),
			\end{equation*}
			where $L' \subset \BV$ are $O_F$-lattices of rank $n$.
			
			\item For $x \in \BV \backslash L_F^{\flat}$, and $L'^{\flat} \supset L^{\flat}$, we define
			\begin{equation*}
				\partial \mathrm{Den}_{L'^{\flat\circ}}^{n,h}(x)\coloneqq \mathlarger{\sum}_{L'^{\flat} \subset L' \subset L'^{\vee}, L' \cap L^{\flat}_F=L'^{\flat}} D_{n,h}(L')1_{L'}(x),
			\end{equation*}
			where $L' \subset \BV$ are $O_F$-lattices of rank $n$. \end{enumerate}
	\end{definition}
	Then, we have that
	\begin{equation*}
		\partial \mathrm{Den}_{L^{\flat}}^{n,h}(x)=\mathlarger{\sum}_{{L^{\flat}} \subset L'^{\flat} \subset (L'^{\flat})^{\vee}} \partial \mathrm{Den}_{L'^{\flat\circ}}^{n,h}(x).
	\end{equation*}

	By definition, we have that
	\begin{equation*}\begin{aligned}
			\widehat{\pDen}_{L'^{\flat\circ}}^{n,h}(x)&=\mathlarger{\sum}_{L'^{\flat} \subset L' \subset L'^{\vee}, L' \cap L^{\flat}_F=L'^{\flat}} D_{n,h}(L')\widehat{1_{L'}}(x)\\
			&=\mathlarger{\sum}_{L'^{\flat} \subset L' \subset L'^{\vee}, L' \cap L^{\flat}_F=L'^{\flat}, x \in L'^{\vee}} D_{n,h}(L')\vol(L').
		\end{aligned}
	\end{equation*}
	Also, by \cite[Lemma 7.2.1, Lemma 7.2.2, (7.4.2.3)]{LZ}, we have that
	\begin{equation*}\begin{array}{ccc}
			[((\langle x \rangle+L'^{\flat})^{\vee,\geq 0}/L'^{\flat})\backslash (L^{\flat}_F/L'^{\flat})]&\overset{\sim}{\longrightarrow}&\lbrace L'^{\flat} \subset L' \subset L'^{\vee}, L' \cap L^{\flat}_F=L'^{\flat}, x \in L'^{\vee} \rbrace\\
			u & \longmapsto & L'^{\flat}+\langle u \rangle.
		\end{array}
	\end{equation*}
	
	Now, assume that $x \perp L^{\flat}$ and $\val(\langle x ,x \rangle) \leq -1$. Then, as in the proof of \cite[Lemma 7.4.2]{LZ}, we have that
	\begin{equation*}
		(\langle x \rangle+L'^{\flat})^{\vee,\geq 0}=(L'^{\flat})^{\vee,\geq 0} \obot \langle x \rangle^{\vee},
	\end{equation*}
	and
	\begin{equation*}
		[((\langle x \rangle+L'^{\flat})^{\vee,\geq 0}/L'^{\flat})\backslash (L^{\flat}_F/L'^{\flat})]\overset{\sim}{\longrightarrow} (L'^{\flat})^{\vee,\geq0}/L'^{\flat} \times (\langle x \rangle^{\vee}\backslash \lbrace 0\rbrace )/O_F^{\times}.
	\end{equation*}
	Therefore, we have that
	\begin{equation*}\begin{array}{ccc}
			(L'^{\flat})^{\vee,\geq0}/L'^{\flat} \times (\langle x \rangle^{\vee}\backslash \lbrace 0\rbrace )/O_F^{\times}&	\overset{\sim}{\longrightarrow}&\lbrace L'^{\flat} \subset L' \subset L'^{\vee}, L' \cap L^{\flat}_F=L'^{\flat}, x \in L'^{\vee} \rbrace\\
			(u^{\flat},u^{\perp}) & \longmapsto & L'^{\flat}+\langle u^{\flat}+u^{\perp} \rangle,
		\end{array}
	\end{equation*}
	and
	\begin{equation*}\begin{array}{ll}
			\widehat{\pDen}_{L'^{\flat\circ}}^{n,h}(x)
			&=\mathlarger{\sum}_{(u^{\flat},u^{\perp}) \in (L'^{\flat})^{\vee,\geq0}/L'^{\flat} \times (\langle x \rangle^{\vee}\backslash \lbrace 0\rbrace )/O_F^{\times}} D_{n,h}(L'^{\flat}+\langle u^{\flat}+u^{\perp}\rangle)\vol(L'^{\flat}+\langle u^{\flat}+u^{\perp}\rangle).
		\end{array}
	\end{equation*}
	
	Now, we need to compute $D_{n,h}(L'^{\flat}+\langle u^{\flat}+u^{\perp}\rangle)$. First, let us define the following notations.
	\begin{definition}
		For $\lambda \in \CR_{n-1}^{0+}$, we let $L_{\lambda} \subset \BV$ be an $O_F$-lattice of rank $n-1$ with hermitian matrix $A_{\lambda}$. Consider a basis $\BB=\lbrace x_1, \dots, x_{n-1}\rbrace$ such that the hermitian matrix of $L_{\lambda}$ with respect to $\BB$ is $A_{\lambda}$.
		\begin{enumerate}
			\item For $i \geq 0$, we define $L_{\lambda \geq i}$ to be the sublattice of $L_{\lambda}$ generated by $\lbrace x_1, \dots, x_{t_{\ge i}(\lambda)} \rbrace$. Therefore, the hermitian matrix of $L_{\lambda \geq i}$ is 
			\begin{equation*}
				\left( \begin{array}{cccc}
					\pi^{\lambda_1} &&&\\
					& \pi^{\lambda_2} &&\\
					&&\ddots&\\
					&&&\pi^{\lambda_{t_{\ge i}(\lambda)}}
				\end{array}\right).
			\end{equation*}
			In particular, $L_{\lambda \geq 0}=L_{\lambda}$.
			
			\item For $i \geq 0$, we define $L_{\lambda = i}$ to be the sublattice of $L_{\lambda}$ generated by $x_j$'s such that $\val(( x_j,x_j ))=i$. Therefore, the hermitian matrix of $L_{\lambda =i}$ is $\pi^{i} I_{t_{i}(\lambda)}$.

			\item For $i \geq j \geq 0$, we define $L_{(\lambda \geq i) - j}$ be an $O_F$-lattice of rank $t_{\ge i}(\lambda)$ with hermitian matrix
			\begin{equation*}
				\left( \begin{array}{cccc}
					\pi^{\lambda_1-j} &&&\\
					& \pi^{\lambda_2-j} &&\\
					&&\ddots&\\
					&&&\pi^{\lambda_{t_{\ge i}(\lambda)}-j}
					
				\end{array}\right).
			\end{equation*}
			Note that $L_{(\lambda \geq i)-j}$ is not necessarily a sublattice in $\BV$.
			
			\item For an $O_F$-lattice $L$, we define
			\begin{equation*}
				\begin{array}{l}
					\mu(L)=\vert (L^{\vee})^{\geq0}/L \vert,\\
					\mu^+(L)=\vert (\pi L^{\vee})^{\geq1}/L \vert,\\
					\mu^{++}(L)=\vert (\pi^2 L^{\vee})^{\geq2}/L \vert.\\
				\end{array}
			\end{equation*}
		\end{enumerate}
	\end{definition}
	
	Consider a basis  $\BB=\lbrace x_1, x_2$, $\dots,x_a,$ $x_{a+1},$$\dots$,$x_{a+b}$,$\dots,$ $x_{a+b+c} \rbrace $ of $L'^{\flat}$ ($a+b+c=n-1$) such that the hermitian matrix of $L$ with respect to $\BB$ is $A_{\lambda}$ where $\lambda=(\lambda_1, \dots, \lambda_{a},1^{b},0^{c}) \in \CR_{n-1}^{0+}$, $\lambda_i \geq 2$. To simplify notation, we abbreviate $L_{\lambda \geq2}$, $L_{\lambda =1 }$ to $L_2'^{\flat}$, $L_{1}'^{\flat}$, respectively. In other words,  $L_2'^{\flat}$ is the $O_F$-lattice generated by $\lbrace x_1, \dots, x_a \rbrace$, and $L'^{\flat}_1$ is the $O_F$-lattice generated by $\lbrace x_{a+1},\dots, x_{a+b} \rbrace$.
	
	Note that
	\begin{equation*}
		(L'^{\flat})^{\vee,\geq0}/L'^{\flat}=(L'^{\flat}_{2} \obot L'^{\flat}_{1})^{\vee,\geq 0}/(L'^{\flat}_{2} \obot L'^{\flat}_{1}).
	\end{equation*}
	
	Then, we have the following propositions.
	
	\begin{proposition}\label{proposition5.33}
		
		Assume that $\val(( u^{\perp},u^{\perp} )) \geq 2$, then we have the followings.
		\begin{enumerate}
			\item (Case 1-1) If $u^{\flat} \in (\pi^2(L_{2}'^{\flat})^{\vee}\obot \pi (L_{1}'^{\flat})^{\vee})^{\geq 2}$,
			\begin{equation*}
				D_{n,h}(L'^{\flat}+\langle u^{\flat}+u^{\perp}\rangle)=D_{n,h}(a+1,b,c).
			\end{equation*}
			\item (Case 1-2) If $u^{\flat} \in (\pi^2(L_{2}'^{\flat})^{\vee}\obot \pi (L_{1}'^{\flat})^{\vee})^{\geq 1}-(\pi^2(L_{2}'^{\flat})^{\vee}\obot \pi (L_{1}'^{\flat})^{\vee})^{\geq 2}$,
			\begin{equation*}
				D_{n,h}(L'^{\flat}+\langle u^{\flat}+u^{\perp}\rangle)=D_{n,h}(a,b+1,c).
			\end{equation*}
			\item (Case 1-3) If $u^{\flat} \in (\pi^2(L_{2}'^{\flat})^{\vee}\obot \pi (L_{1}'^{\flat})^{\vee})^{\geq 0}-(\pi^2(L_{ 2}'^{\flat})^{\vee}\obot \pi (L_{1}'^{\flat})^{\vee})^{\geq 1}$,
			\begin{equation*}
				D_{n,h}(L'^{\flat}+\langle u^{\flat}+u^{\perp}\rangle)=D_{n,h}(a,b,c+1).
			\end{equation*}
			
			\item (Case 2-1) If $u^{\flat} \in (\pi^2(L_{ 2}'^{\flat})^{\vee})^{\geq0}\obot ((L_{1}'^{\flat})^{\vee}-\pi (L_{1}'^{\flat})^{\vee})^{\geq 0}$,
			\begin{equation*}
				D_{n,h}(L'^{\flat}+\langle u^{\flat}+u^{\perp}\rangle)=D_{n,h}(a+1,b-2,c+2).
			\end{equation*}
			\item (Case 2-2) If \begin{equation*}
				u^{\flat} \in ((\pi^2(L_{ 2}'^{\flat})^{\vee})\obot ((L_{1}'^{\flat})^{\vee}-\pi (L_{1}'^{\flat})^{\vee}))^{\geq 0}- (\pi^2(L_{ 2}'^{\flat})^{\vee})^{\geq0}\obot ((L_{1}'^{\flat})^{\vee}-\pi (L_{1}'^{\flat})^{\vee})^{\geq 0},
			\end{equation*}
			\begin{equation*}
				D_{n,h}(L'^{\flat}+\langle u^{\flat}+u^{\perp}\rangle)=D_{n,h}(a,b-1,c+2).
			\end{equation*}
			\item (Case 3-1) If $u^{\flat} \in ((\pi (L_{ 2}'^{\flat})^{\vee}-\pi^2(L_{ 2}'^{\flat})^{\vee})\obot (\pi (L_{1}'^{\flat})^{\vee}))^{\geq 1}$,
			\begin{equation*}
				D_{n,h}(L'^{\flat}+\langle u^{\flat}+u^{\perp}\rangle)=D_{n,h}(a-1,b+2,c).
			\end{equation*}
			\item (Case 3-2) If $u^{\flat} \in ((\pi (L_{ 2}'^{\flat})^{\vee}-\pi^2(L_{ 2}'^{\flat})^{\vee})\obot (\pi (L_{1}'^{\flat})^{\vee}))^{\geq 0}-((\pi (L_{ 2}'^{\flat})^{\vee}-\pi^2(L_{ 2}'^{\flat})^{\vee})\obot (\pi (L_{1}'^{\flat})^{\vee}))^{\geq 1}$,
			\begin{equation*}
				D_{n,h}(L'^{\flat}+\langle u^{\flat}+u^{\perp}\rangle)=D_{n,h}(a,b,c+1).
			\end{equation*}
			\item (Case 4-1) If $u^{\flat} \in (\pi (L_{ 2}'^{\flat})^{\vee}-\pi^2(L_{ 2}'^{\flat})^{\vee})^{\geq0}\obot ((L_{1}'^{\flat})^{\vee}-\pi (L_{1}'^{\flat})^{\vee})^{\geq 0}$,
			\begin{equation*}
				D_{n,h}(L'^{\flat}+\langle u^{\flat}+u^{\perp}\rangle)=D_{n,h}(a+1,b-2,c+2).
			\end{equation*}
			\item (Case 4-2) If \begin{equation*}
				u^{\flat} \in ((\pi (L_{ 2}'^{\flat})^{\vee}-\pi^2(L_{ 2}'^{\flat})^{\vee})\obot ((L_{1}'^{\flat})^{\vee}-\pi (L_{1}'^{\flat})^{\vee}))^{\geq 0}- (\pi (L_{ 2}'^{\flat})^{\vee}-\pi^2(L_{ 2}'^{\flat})^{\vee})^{\geq0}\obot ((L_{1}'^{\flat})^{\vee}-\pi (L_{1}'^{\flat})^{\vee})^{\geq 0},
			\end{equation*}
			\begin{equation*}
				D_{n,h}(L'^{\flat}+\langle u^{\flat}+u^{\perp}\rangle)=D_{n,h}(a,b-1,c+2).
			\end{equation*}
			
			\item (Case 5)  If $u^{\flat} \in (((L_{ 2}'^{\flat})^{\vee}-\pi(L_{ 2}'^{\flat})^{\vee})\obot (L_{1}'^{\flat})^{\vee})^{\geq 0}$,
			\begin{equation*}
				D_{n,h}(L'^{\flat}+\langle u^{\flat}+u^{\perp}\rangle)=D_{n,h}(a-1,b,c+2).
			\end{equation*}
		\end{enumerate}
	\end{proposition}
	\begin{proof} This is a linear algebra problem, so we skip the proof of the proposition.
	\end{proof}
	
	\begin{proposition}\label{proposition5.34}
		Assume that $\val((u^{\perp},u^{\perp} )) =1$, then we have the followings.
		\begin{enumerate}
			\item (Case 1-1) If $u^{\flat} \in (\pi^2(L_{2}'^{\flat})^{\vee}\obot \pi (L_{1}'^{\flat})^{\vee})^{\geq 2}$,
			\begin{equation*}
				D_{n,h}(L'^{\flat}+\langle u^{\flat}+u^{\perp}\rangle)=D_{n,h}(a,b+1,c).
			\end{equation*}
			\item (Case 1-2-1) If $u^{\flat} \in (\pi^2(L_{2}'^{\flat})^{\vee}\obot \pi (L_{1}'^{\flat})^{\vee})^{\geq 1}-(\pi^2(L_{2}'^{\flat})^{\vee}\obot \pi (L_{1}'^{\flat})^{\vee})^{\geq 2}$, and $\val(\langle u^{\flat}+u^{\perp},u^{\flat}+u^{\perp} \rangle) =1$
			\begin{equation*}
				D_{n,h}(L'^{\flat}+\langle u^{\flat}+u^{\perp}\rangle)=D_{n,h}(a,b+1,c).
			\end{equation*}
			\item (Case 1-2-2) If $u^{\flat} \in (\pi^2(L_{2}'^{\flat})^{\vee}\obot \pi (L_{1}'^{\flat})^{\vee})^{\geq 1}-(\pi^2(L_{2}'^{\flat})^{\vee}\obot \pi (L_{1}'^{\flat})^{\vee})^{\geq 2}$, and $\val(\langle u^{\flat}+u^{\perp},u^{\flat}+u^{\perp} \rangle) \geq 2$, we have
			\begin{equation*}
				D_{n,h}(L'^{\flat}+\langle u^{\flat}+u^{\perp}\rangle)=D_{n,h}(a+1,b,c).
			\end{equation*}
			\item (Case 1-3) If $u^{\flat} \in (\pi^2(L_{2}'^{\flat})^{\vee}\obot \pi (L_{1}'^{\flat})^{\vee})^{\geq 0}-(\pi^2(L_{ 2}'^{\flat})^{\vee}\obot \pi (L_{1}'^{\flat})^{\vee})^{\geq 1}$,
			\begin{equation*}
				D_{n,h}(L'^{\flat}+\langle u^{\flat}+u^{\perp}\rangle)=D_{n,h}(a,b,c+1).
			\end{equation*}
			\item (Case 2-1) If $u^{\flat} \in (\pi^2(L_{ 2}'^{\flat})^{\vee})^{\geq0}\obot ((L_{1}'^{\flat})^{\vee}-\pi (L_{1}'^{\flat})^{\vee})^{\geq 0}$,
			\begin{equation*}
				D_{n,h}(L'^{\flat}+\langle u^{\flat}+u^{\perp}\rangle)=D_{n,h}(a+1,b-2,c+2).
			\end{equation*}
			\item (Case 2-2) If \begin{equation*}
				u^{\flat} \in ((\pi^2(L_{ 2}'^{\flat})^{\vee})\obot ((L_{1}'^{\flat})^{\vee}-\pi (L_{1}'^{\flat})^{\vee}))^{\geq 0}- (\pi^2(L_{ 2}'^{\flat})^{\vee})^{\geq0}\obot ((L_{1}'^{\flat})^{\vee}-\pi (L_{1}'^{\flat})^{\vee})^{\geq 0},
			\end{equation*}
			\begin{equation*}
				D_{n,h}(L'^{\flat}+\langle u^{\flat}+u^{\perp}\rangle)=D_{n,h}(a,b-1,c+2).
			\end{equation*}
			\item (Case 3-1) If $u^{\flat} \in ((\pi (L_{ 2}'^{\flat})^{\vee}-\pi^2(L_{ 2}'^{\flat})^{\vee})\obot (\pi (L_{1}'^{\flat})^{\vee}))^{\geq 1}$,
			\begin{equation*}
				D_{n,h}(L'^{\flat}+\langle u^{\flat}+u^{\perp}\rangle)=D_{n,h}(a-1,b+2,c).
			\end{equation*}
			\item (Case 3-2) If $u^{\flat} \in ((\pi (L_{ 2}'^{\flat})^{\vee}-\pi^2(L_{ 2}'^{\flat})^{\vee})\obot (\pi (L_{1}'^{\flat})^{\vee}))^{\geq 0}-((\pi (L_{ 2}'^{\flat})^{\vee}-\pi^2(L_{ 2}'^{\flat})^{\vee})\obot (\pi (L_{1}'^{\flat})^{\vee}))^{\geq 1}$,
			\begin{equation*}
				D_{n,h}(L'^{\flat}+\langle u^{\flat}+u^{\perp}\rangle)=D_{n,h}(a,b,c+1).
			\end{equation*}
			\item (Case 4-1) If $u^{\flat} \in (\pi (L_{ 2}'^{\flat})^{\vee}-\pi^2(L_{ 2}'^{\flat})^{\vee})^{\geq0}\obot ((L_{1}'^{\flat})^{\vee}-\pi (L_{1}'^{\flat})^{\vee})^{\geq 0}$,
			\begin{equation*}
				D_{n,h}(L'^{\flat}+\langle u^{\flat}+u^{\perp}\rangle)=D_{n,h}(a+1,b-2,c+2).
			\end{equation*}
			\item (Case 4-2) If \begin{equation*}
				u^{\flat} \in ((\pi (L_{ 2}'^{\flat})^{\vee}-\pi^2(L_{ 2}'^{\flat})^{\vee})\obot ((L_{1}'^{\flat})^{\vee}-\pi (L_{1}'^{\flat})^{\vee}))^{\geq 0}- (\pi (L_{ 2}'^{\flat})^{\vee}-\pi^2(L_{ 2}'^{\flat})^{\vee})^{\geq0}\obot ((L_{1}'^{\flat})^{\vee}-\pi (L_{1}'^{\flat})^{\vee})^{\geq 0},
			\end{equation*}
			\begin{equation*}
				D_{n,h}(L'^{\flat}+\langle u^{\flat}+u^{\perp}\rangle)=D_{n,h}(a,b-1,c+2).
			\end{equation*}
			
			\item (Case 5)  If $u^{\flat} \in (((L_{ 2}'^{\flat})^{\vee}-\pi(L_{ 2}'^{\flat})^{\vee})\obot (L_{1}'^{\flat})^{\vee})^{\geq 0}$,
			\begin{equation*}
				D_{n,h}(L'^{\flat}+\langle u^{\flat}+u^{\perp}\rangle)=D_{n,h}(a-1,b,c+2).
			\end{equation*}
		\end{enumerate}
	\end{proposition}
	\begin{proof} This is a linear algebra problem, so we skip the proof of the proposition.
	\end{proof}
	
	\begin{proposition}\label{proposition5.35}
		Assume that $\lambda \in \CR_{n}^{0+}$ and $L_{\lambda}$ is an $O_F$-lattice with hermitian matrix $A_{\lambda}$.
		\begin{enumerate}
			\item If $\lambda \geq (1,\dots,1)$, we have
			\begin{equation*}
				\mu^+(L_{\lambda})=\mu(L_{\lambda-1}).
			\end{equation*}
			
			\item If $\lambda \geq (2,\dots,2)$, we have
			\begin{equation*}
				\mu^{++}(L_{\lambda})=\mu^+(L_{\lambda-1}).
			\end{equation*}
		\end{enumerate}
	\end{proposition}
	\begin{proof}
		By definition, we have that
		\begin{equation*}
			\begin{array}{l}
				\mu(L_{\lambda-1})=\vert (L_{\lambda-1}^{\vee})^{\geq0}/L_{\lambda-1} \vert,\\
				\mu^+(L_{\lambda})=\vert (\pi L_{\lambda}^{\vee})^{\geq1}/L_{\lambda} \vert.
			\end{array}
		\end{equation*}
		
		Note that the fundamental invariants of $L_{\lambda-1}^{\vee}$ (resp. $\pi L_{\lambda}^{\vee})$) are $(-\lambda_1+1,\dots, -\lambda_n+1)$ (resp. $(-\lambda_1+2,\dots, -\lambda_n+2)$). Now, for $(L_{\lambda-1}^{\vee})^{\geq0}/L_{\lambda-1}$, and the hermitian form $\langle \cdot,\cdot \rangle$, we consider the same set with the hermitian form $\pi \langle \cdot,\cdot\rangle$. Then, this is isomorphic to $(\pi L_{\lambda}^{\vee})^{\geq1}/L_{\lambda}$. This proves (1). The proof of (2) is similar.
	\end{proof}

	\begin{proposition}\label{proposition5.36.1}
		For $\lambda \in \CR_{n-1}^{0+}$ and $L_{\lambda}=L_{\lambda \geq 2} \obot L_{\lambda=1}=:L_2 \obot L_1$, we have the followings.
		
		\begin{enumerate}
			\item (Case 1-1) We have that
			\begin{equation*}
				\vert	(\pi^2(L_2)^{\vee}\obot \pi (L_1)^{\vee})^{\geq 2}/L_2 \obot L_1\vert=\mu(L_{(\lambda \geq 2)-2}).
			\end{equation*}
			
			\item (Case 1-2) We have that
			\begin{equation*}\begin{array}{l}
					\vert	\lbrace(\pi^2(L_2)^{\vee}\obot \pi (L_1)^{\vee})^{\geq 1}-(\pi^2(L_2)^{\vee}\obot \pi (L_1)^{\vee})^{\geq 2}\rbrace/ L_2 \obot L_1\vert\\
					=q^{2t_{\ge 3}(\lambda)}\mu(L_{(\lambda \geq 3)-3})-\mu(L_{(\lambda \geq2) -2}).
			\end{array}\end{equation*}
			\item (Case 1-3+Case 3-2) We have that
			\begin{equation*}\begin{array}{l}
					\vert \lbrace (\pi^2(L_2)^{\vee}\obot \pi (L_1)^{\vee})^{\geq 0}-(\pi^2(L_2)^{\vee}\obot \pi (L_1)^{\vee})^{\geq 1}\rbrace / L_2 \obot L_1 \vert\\
					+\vert\lbrace((\pi (L_2)^{\vee}-\pi^2(L_{\lambda \geq 2})^{\vee})\obot (\pi (L_1)^{\vee}))^{\geq 0}
					-((\pi (L_2)^{\vee}-\pi^2(L_2)^{\vee})\obot (\pi (L_1)^{\vee}))^{\geq 1}\rbrace/ L_2 \obot L_1\vert\\
					=q^{2t_{\ge 2}(\lambda)}\mu(L_{(\lambda\geq 2) -2})-\mu(L_{(\lambda \geq 2)-1}).
			\end{array}\end{equation*}
			
			\item (Case 2-1+Case 4-1) We have that
			\begin{equation*}\begin{array}{l}
					\vert \lbrace (\pi^2(L_{ 2}'^{\flat})^{\vee})^{\geq0}\obot ((L_{1}'^{\flat})^{\vee}-\pi (L_{1}'^{\flat})^{\vee})^{\geq 0}\rbrace/L_2 \obot L_1 \vert\\
					+\vert\lbrace(\pi (L_{ 2}'^{\flat})^{\vee}-\pi^2(L_{ 2}'^{\flat})^{\vee})^{\geq0}\obot ((L_{1}'^{\flat})^{\vee}-\pi (L_{1}'^{\flat})^{\vee})^{\geq 0}\rbrace/ L_2 \obot L_1 \vert\\
					=q^{2t_{\ge 2}(\lambda)}\mu(L_{(\lambda \geq 2) -2})\times (\mu(L_{\lambda=1})-1).
			\end{array}\end{equation*}
			\item (Case 2-2+Case 4-2) We have that
			\begin{equation*}\begin{array}{l}
					\vert \lbrace ((\pi^2(L_{2})^{\vee})\obot ((L_{1})^{\vee}-\pi (L_{1})^{\vee}))^{\geq 0}- (\pi^2(L_{ 2})^{\vee})^{\geq0}\obot ((L_{1})^{\vee}-\pi (L_{1})^{\vee})^{\geq 0}\rbrace/L_2 \obot L_1 \vert\\
					+\vert\lbrace((\pi (L_{ 2})^{\vee}-\pi^2(L_{ 2})^{\vee})\obot ((L_{1})^{\vee}-\pi (L_{1})^{\vee}))^{\geq 0}\rbrace/ L_2 \obot L_1 \vert\\
					-\vert\lbrace (\pi (L_{ 2})^{\vee}-\pi^2(L_{ 2})^{\vee})^{\geq0}\obot ((L_{1})^{\vee}-\pi (L_{1})^{\vee})^{\geq 0}\rbrace/ L_2 \obot L_1 \vert\\
					=q^{2t_{\ge 2}(\lambda)}\lbrace \mu(L_{(\lambda\geq 2)-2} \obot L_{\lambda=1}) - \mu(L_{(\lambda \geq 2) -2})\mu(L_{\lambda=1})\rbrace.
			\end{array}\end{equation*}
			\item (Case 3-1) We have that
			\begin{equation*}\begin{array}{l}
					\vert \lbrace  ((\pi (L_{ 2})^{\vee}-\pi^2(L_{ 2})^{\vee})\obot (\pi (L_{1})^{\vee}))^{\geq 1}\rbrace/ L_2 \obot L_1 \vert
					=\mu(L_{(\lambda \geq 2) -1})-q^{2t_{\ge 3}(\lambda)}\mu(L_{(\lambda \geq3)-3}).
			\end{array}\end{equation*}
			\item (Case 5) We have that
			\begin{equation*}\begin{array}{l}
					\vert \lbrace (((L_{ 2})^{\vee}-\pi(L_{ 2})^{\vee})\obot (L_{1})^{\vee})^{\geq 0}\rbrace /L_2\obot L_1 \vert
					=\mu(L_{\lambda \geq2} \obot L_{\lambda=1})-q^{2t_{\ge 2}(\lambda)}\mu(L_{(\lambda \geq 2)-2}\obot L_{\lambda=1}).
			\end{array}\end{equation*}
		\end{enumerate}
		
		Here, we choose the following convention: $\forall i \geq j \geq 0$, if $L_{\lambda \geq i}$ is empty, then $\mu(L_{(\lambda \geq i) -j })=1$.
	\end{proposition}
	\begin{proof}
		We only prove the cases: (Case 1-2) and (Case 2-2+Case 4-2). The other cases can be proved similarly.
		
		In (Case 1-2), note that $\pi(L_1)^{\vee}=L_1$. Therefore, we have
		\begin{equation*}\begin{array}{l}
				\vert	\lbrace(\pi^2(L_2)^{\vee}\obot \pi (L_1)^{\vee})^{\geq 1}-(\pi^2(L_2)^{\vee}\obot \pi (L_1)^{\vee})^{\geq 2}\rbrace/ L_2 \obot L_1\vert=\vert	\lbrace(\pi^2(L_2)^{\vee})^{\geq 1}-(\pi^2(L_2)^{\vee})^{\geq 2}\rbrace/ L_2\vert.
		\end{array}\end{equation*}
		If $L_{\lambda \geq 3}$ is empty, then we have that $\pi^2 L_2^{\vee}=L_2=L_{\lambda = 2}$. Therefore, 
		\begin{equation*}
			\vert	\lbrace(\pi^2(L_2)^{\vee})^{\geq 1}\rbrace/ L_2\vert=\vert\lbrace(\pi^2(L_2)^{\vee})^{\geq 2}\rbrace/ L_2\vert=1.
		\end{equation*}
		
		If $L_{\lambda \geq 3}$ is not empty, then we have
		\begin{equation*}\begin{array}{l}
				\vert	\lbrace(\pi^2(L_2)^{\vee})^{\geq 1}\rbrace/ L_2\vert=\vert\lbrace(\pi^2 L_{\lambda \geq 3}^{\vee})^{\geq 1}\rbrace/ L_{\lambda \geq 3}\vert=q^{2t_{\ge 3}(\lambda)}\vert\lbrace(\pi (\pi^{-1}L_{\lambda \geq 3})^{\vee})^{\geq 1}\rbrace/ \pi^{-1}L_{\lambda \geq 3}\vert\\
				=q^{2t_{\ge 3}(\lambda)} \mu^+(\pi^{-1}L_{\lambda \geq 3})=q^{2t_{\ge 3}(\lambda)} \mu^+(L_{(\lambda \geq 3)-2})=q^{2t_{\ge 3}(\lambda)} \mu(L_{(\lambda \geq 3)-3}).
			\end{array}
		\end{equation*}
		Here, we used Proposition \ref{proposition5.35}. Also, by definition, we have that $\vert	\lbrace(\pi^2(L_2)^{\vee})^{\geq 2}\rbrace/ L_2\vert=\mu^{++}(L_2)=\mu(L_{(\lambda \geq 2)-2}).$ Therefore, we have that (Case 1-2) is $q^{2t_{\ge 3}(\lambda)} \mu(L_{(\lambda \geq 3)-3})-\mu(L_{(\lambda \geq 2)-2})$.
		
		Next, consider the case (Case 2-2+Case 4-2).
		First, note that
		\begin{equation*}\begin{array}{l}
				\vert \lbrace ((\pi^2(L_{2})^{\vee})\obot ((L_{1})^{\vee}-\pi (L_{1})^{\vee}))^{\geq 0}- (\pi^2(L_{ 2})^{\vee})^{\geq0}\obot ((L_{1})^{\vee}-\pi (L_{1})^{\vee})^{\geq 0}\rbrace/L_2 \obot L_1 \vert\\
				+\vert\lbrace((\pi (L_{ 2})^{\vee}-\pi^2(L_{ 2})^{\vee})\obot ((L_{1})^{\vee}-\pi (L_{1})^{\vee}))^{\geq 0}\rbrace/ L_2 \obot L_1 \vert\\
				-\vert\lbrace (\pi (L_{ 2})^{\vee}-\pi^2(L_{ 2})^{\vee})^{\geq0}\obot ((L_{1})^{\vee}-\pi (L_{1})^{\vee})^{\geq 0}\rbrace/ L_2 \obot L_1 \vert\\
				=\vert \lbrace ((\pi(L_{2})^{\vee})\obot ((L_{1})^{\vee}-\pi (L_{1})^{\vee}))^{\geq 0}- (\pi(L_{ 2})^{\vee})^{\geq0}\obot ((L_{1})^{\vee}-\pi (L_{1})^{\vee})^{\geq 0}\rbrace/L_2 \obot L_1 \vert.
		\end{array}\end{equation*}
		
		Now, note that
		\begin{equation*}\begin{aligned}
				\vert \lbrace ((\pi(L_{2})^{\vee})\obot (L_{1}^{\vee}))^{\geq 0}\rbrace/L_2 \obot L_1 \vert&=q^{2t_{\ge 2}(\lambda)}\vert \lbrace ((\pi^{-1}L_{2})^{\vee}\obot (L_{1}^{\vee}))^{\geq 0}\rbrace/\pi^{-1}L_2 \obot L_1 \vert\\
				&=q^{2t_{\ge 2}(\lambda)}\mu(\pi^{-1}L_2\obot L_1)\\
				&=q^{2t_{\ge 2}(\lambda)}\mu(L_{(\lambda \geq 2)-2}\obot L_{\lambda=1}).
			\end{aligned}
		\end{equation*}
		
		Also, since $\pi L_1^{\vee}=L_1$, we have
		\begin{equation*}
			\vert \lbrace ((\pi(L_{2})^{\vee})\obot (\pi L_{1}^{\vee}))^{\geq 0}\rbrace/L_2 \obot L_1 \vert=\vert \lbrace ((\pi(L_{2})^{\vee}))^{\geq 0}\rbrace/L_2\vert=q^{2t_{\ge 2}(\lambda)}\mu(L_{(\lambda \geq 2)-2}),
		\end{equation*}
		
		\begin{equation*}
			\vert \lbrace(\pi(L_{ 2})^{\vee})^{\geq0}\obot ((L_{1})^{\vee})^{\geq 0}\rbrace/L_2 \obot L_1 \vert=q^{2t_{\ge 2}(\lambda)}\mu(L_{(\lambda \geq 2)-2})\times \mu(L_{\lambda=1}),
		\end{equation*}
		and
		\begin{equation*}
			\vert \lbrace(\pi(L_{ 2})^{\vee})^{\geq0}\obot (\pi(L_{1})^{\vee})^{\geq 0}\rbrace/L_2 \obot L_1 \vert=\vert \lbrace ((\pi(L_{2})^{\vee}))^{\geq 0}\rbrace/L_2\vert=q^{2t_{\ge 2}(\lambda)}\mu(L_{(\lambda \geq 2)-2}).
		\end{equation*}
		
		Combining these, we have that (Case 2-2+Case 4-2) is
		\begin{equation*}
			q^{2t_{\ge 2}(\lambda)}\mu(L_{(\lambda \geq 2)-2}\obot L_{\lambda=1})-q^{2t_{\ge 2}(\lambda)}\mu(L_{(\lambda \geq 2)-2})\times \mu(L_{\lambda=1}).
		\end{equation*}
		
		This finishes the proof of the proposition in the case (Case 2-2+Case 4-2).
	\end{proof}

	\begin{proposition}\label{proposition5.36}(cf. \cite[Lemma 8.2.3]{LZ2}, \cite[Lemma 8.4]{HLSY}) Assume that $L$ (resp. $M$) is an $O_F$-lattice of rank $n$ with hermitian matrix $A_{\lambda}$ (resp. $A_{\eta}$) such that $\lambda, \eta \in \CR_{n}^{0+}$, $\lambda, \eta \geq (1,\dots, 1)$. Assume further that $L \subset M \subset \pi^{-1}L$. Then, we have
		\begin{equation*}
			\vert	((L^{\vee})^{\geq 0} \backslash (M^{\vee})^{\geq0}) / L \vert=q^{2n-1} \vert	((\pi L^{\vee})^{\geq 1} \backslash (\pi M^{\vee})^{\geq1}) / L \vert.
		\end{equation*}
		
	\end{proposition}
	\begin{proof} Here, we follow the proof of \cite[Lemma 8.2.3]{LZ2}.
		Consider the map
		\begin{equation*}\begin{array}{ccc}
				((L^{\vee})^{\geq 0} \backslash (M^{\vee})^{\geq0}) / L & \longrightarrow &	((\pi L^{\vee})^{\geq 1} \backslash (\pi M^{\vee})^{\geq1}) / L\\
				x &\longmapsto & \pi x.
			\end{array}
		\end{equation*}
		To prove the proposition, it suffices to show that the above map is surjective and every fiber has size $q^{2n-1}$. Choose $x \in (\pi L^{\vee})^{\geq 1} \backslash (\pi M^{\vee})^{\geq 1}$. Then the fiber of $x$ is given by
		\begin{equation*}
			\lbrace \dfrac{1}{\pi}(x+y) \in (L^{\vee})^{\geq 0}, y \in L/\pi L \rbrace.
		\end{equation*}
		Since $x \in \pi L^{\vee}$, the condition $\dfrac{1}{\pi}(x+y) \in (L^{\vee})^{\geq 0}$ is equivalent to
		\begin{equation*}
			y \in L \cap \pi L^{\vee} / \pi L, \text{ and } ( y+x, y+x ) \equiv 0 (\text{mod } \pi^2).
		\end{equation*}
		
		Choose a basis $e=\lbrace e_1, \dots, e_n \rbrace$ of $L$ such that the hermitian matrix of $L$ with respect to $e$ is $A_{\lambda}$. Then, we have
		\begin{equation*}
			\pi L^{\vee}=\oplus_{i} O_F (\pi^{-\lambda_i+1}e_i).
		\end{equation*}
		
		Write
		\begin{equation*}\begin{array}{l}
				x=\mathlarger{\sum}_i \mu_i \pi^{-\lambda_i+1}e_i, \quad \mu_i \in O_F,\\
				y=\mathlarger{\sum}_i \nu_i e_i, \quad \nu_i \in O_F.
			\end{array}
		\end{equation*}
		
		Then, we have
		\begin{equation*}\begin{array}{l}
				( y,y)\equiv \mathlarger{\sum}_{\lambda_i=1} ( \nu_i e_i,\nu_i e_i )\equiv \mathlarger{\sum}_{\lambda_i=1}\pi \nu_i \overline{\nu_i}(\text{mod } \pi^2),\\
				( y,x )=\mathlarger{\sum}_i ( \nu_ie_i, \mu_i \pi^{-\lambda_i+1}e_i)=\mathlarger{\sum}_{i}\nu_i \overline{\mu_i}\pi.
			\end{array}
		\end{equation*}
		Since $x \in (\pi L^{\vee})^{\geq 1}$, the condition $( y+x,y+x ) \equiv 0 (\text{mod }\pi^2)$ is equivalent to
		\begin{equation}\label{eq5.26}
			\dfrac{1}{\pi}( x, x )+\mathlarger{\sum}_{\lambda_i=1}(\nu_i \overline{\nu_i}+\nu_i \overline{\mu_i}+\overline{\nu_i}\mu_i)+\mathlarger{\sum}_{\lambda_i >1}(\nu_i \overline{\mu_i}+\overline{\nu_i}\mu_i) \equiv 0 (\text{mod } \pi).
		\end{equation}
		
		Since $x \notin \pi M^{\vee}$, there is at least one $i$ such that $\mu_i \not\equiv 0 (\text{mod }\pi)$, and $\lambda_i > 1$. Let $\mu_k$ be such a coefficient, i.e., $\mu_k \not\equiv 0 (\text{mod }\pi)$, and $\lambda_k > 1$.
		
		Now, let $\beta=\dfrac{1}{\pi}( x,x)+\mathlarger{\sum}_{\lambda_i=1}(\nu_i \overline{\nu_i}+\nu_i \overline{\mu_i}+\overline{\nu_i}\mu_i).$ Then, \eqref{eq5.26} can be written as
		\begin{equation*}\begin{array}{ll}
				&\beta+	\mathlarger{\sum}_{\lambda_i >1}(\nu_i \overline{\mu_i}+\overline{\nu_i}\mu_i) \equiv 0 (\text{mod } \pi)\\
				\Longleftrightarrow & \mathrm{Tr}(\mu_k\overline{\nu_k})\equiv -\beta	-\mathlarger{\sum}_{\lambda_i >1, i\neq k}(\nu_i \overline{\mu_i}+\overline{\nu_i}\mu_i) (\text{mod } \pi).
			\end{array}
		\end{equation*}
		Choose any $\nu_i, i \neq k$, so there are $q^{2(n-1)}$-choices. Also, $\mathrm{Tr}:\BF_{q^2} \rightarrow \BF_q$ is surjective and every fiber has size $q$. Therefore, for each $\lbrace \nu_i\rbrace_{i \neq k}$, there are $q$-solutions of $\nu_k$, and hence \begin{equation*}
			\vert	\lbrace \dfrac{1}{\pi}(x+y) \in (L^{\vee})^{\geq 0}, y \in L/\pi L \rbrace \vert=q^{2n-1}.
		\end{equation*}
		
		This finishes the proof of the proposition.
		
	\end{proof}

	\begin{proposition}\label{proposition5.37}(cf. \cite[Lemma 8.2.6]{LZ2}, \cite[Lemma 8.6]{HLSY})
		Assume that $L$ is an $O_F$-lattice of rank $n$ and $e=\lbrace e_1,\dots,e_n \rbrace$ is a basis of $L$. Also, assume that the hermitian matrix of $L$ with respect to $e$ is $A_{\lambda}$ where $\lambda \in \CR_{n}^{0+}$, $\lambda \geq (1,\dots, 1)$. We choose an $O_F$-lattice $M$ as follows.
		\begin{enumerate}
			\item If $\lambda_1 \geq 3$, we choose an $O_F$-lattice $M \supset L$ such that $M=O_F (\dfrac{1}{\pi}e_1) \mathlarger{\oplus} \oplus_{i \neq 1} O_F (e_i)$ with fundamental invariants $(\lambda_1-2,\lambda_2,\dots,\lambda_n)$.
			
			\item If $\lambda_1=\lambda_2=2$ (so, $\lambda_i \leq 2$ for all $i$), we choose an $O_F$-lattice $M \supset L$ such that the fundamental invariants of $M$ are $( \lambda_1-1,\lambda_2-1,\lambda_3,\dots,\lambda_n )$. 
		\end{enumerate}
		
		Then, we have that
		\begin{equation*}
			\vert	((\pi L^{\vee})^{\geq 0} \backslash (\pi M^{\vee})^{\geq0}) / L \vert=q \vert	((\pi L^{\vee})^{\geq 1} \backslash (\pi M^{\vee})^{\geq1}) / L \vert.
		\end{equation*}
	\end{proposition}
	\begin{proof}
		Here, we follow the proof of \cite[Lemma 8.2.6]{LZ2}.
		\begin{enumerate}
			\item In this case, we have that
			\begin{equation*}
				\pi L^{\vee}=\mathlarger{\oplus}_i O_F(\pi^{-\lambda_i+1} e_i),
			\end{equation*}
			and
			\begin{equation*}
				\pi M^{\vee}=O_F(\pi^{-\lambda_1+2}e_1) \oplus (\bigoplus_{i \neq 1} O_F(\pi^{-\lambda_i+1} e_i)).
			\end{equation*}
			
			Fix an element $x_0=\mathlarger{\sum}_{i \neq 1} \mu_i \pi^{-\lambda_i+1}e_i, \mu_i \in O_F$. Consider the sets
			\begin{equation*}\begin{array}{l}
					S_{x_0}^{\geq 0}\coloneqq \lbrace x \in (\pi L^{\vee})^{\geq 0} \backslash (\pi M^{\vee})^{\geq 0} \mid x=x_0+\mu_1 \pi^{-\lambda_1+1}e_1, \mu_1 \in O_F \rbrace/L\\
					S_{x_0}^{\geq 1}\coloneqq \lbrace x \in (\pi L^{\vee})^{\geq 1} \backslash (\pi M^{\vee})^{\geq 1} \mid x=x_0+\mu_1 \pi^{-\lambda_1+1}e_1, \mu_1 \in O_F \rbrace/L.
				\end{array}
			\end{equation*}
			
			Then, it suffices to show that $\vert S_{x_0}^{\geq 0}\vert=q \vert S_{x_0}^{\geq 1}\vert.$
			
			Note that $x \notin \pi M^{\vee}$ if and only if $\mu_1 \in O_F^{\times}$. Also, we have
			\begin{equation*}
				( x, x )=( x_0,x_0 )+\mu_1 \overline{\mu_1}\pi^{-\lambda_1+2}.
			\end{equation*}
			Therefore, we have that
			\begin{equation*}
				\begin{array}{lll}
					x \in S_{x_0}^{\geq 0} & \text{ if and only if } & \pi^{\lambda_1-2}( x_0,x_0)+ \mu_1 \overline{\mu_1} \equiv 0 (\text{mod } \pi^{\lambda_1-2}), \mu_1 \in O_F^{\times},\\
					x \in S_{x_0}^{\geq 1} & \text{ if and only if } & \pi^{\lambda_1-2}( x_0,x_0 )+ \mu_1 \overline{\mu_1} \equiv 0 (\text{mod } \pi^{\lambda_1-1}), \mu_1 \in O_F^{\times}.
				\end{array}
			\end{equation*}
			Let us write
			\begin{equation*}
				\begin{array}{l}
					\mu_1=b_0+b_1\pi+b_2\pi^2 +\dots,\\
					-\pi^{\lambda_1-2}( x_0,x_0 )=c_0+c_1 \pi +c_2 \pi^2 +\dots .
				\end{array}
			\end{equation*}
			Then $x \in S_{x_0}^{\geq 0}$ if and only if $b_0 \in O_F^{\times}$ and
			\begin{equation}\label{eq5.27}
				\begin{array}{l}
					c_0=b_0\overline{b_0},\\
					c_1=b_0\overline{b_1}+b_1\overline{b_0},\\
					\vdots\\
					c_{\lambda_1-3}=b_0\overline{b_{\lambda_1-3}}+\dots+b_{\lambda_1-3}\overline{b_0}.
				\end{array}
			\end{equation}
			Also, $x \in S_{x_0}^{\geq 1}$ if and only if $x \in S_{x_0}^{\geq 0}$ and 
			\begin{equation}\label{eq5.28}\begin{array}{ll}
					&c_{\lambda_1-2}=b_0\overline{b_{\lambda_1-2}}+\dots+b_{\lambda_1-2}\overline{b_0}\\
					\Longleftrightarrow & \mathrm{Tr}(b_0\overline{b_{\lambda_1-2}})=c_{\lambda_1-2}-b_1\overline{b_{\lambda_1-3}}-\dots-b_{\lambda_1-3}\overline{b_1}.
				\end{array}
			\end{equation}
			Now, for each $\lbrace b_0,b_1,\dots,b_{\lambda_1-3}\rbrace$ satisfying \eqref{eq5.27}, there are $q^2$-choices of $b_{\lambda_1-2}$ so that $x \in S_{x_0}^{\geq 0}$. Also, there are $q$-choices of $b_{\lambda_1-2}$ so that \eqref{eq5.28} is true, and hence $ x \in S_{x_0}^{\geq 1}$. This shows that $\vert S_{x_0}^{\geq 0}\vert=q \vert S_{x_0}^{\geq 1}\vert.$

			\item In this case, we may choose a basis $f=\lbrace f_1, f_2, \dots, f_n \rbrace$ of $\pi M^{\vee}$ such that the hermitian matrix of $\pi M^{\vee}$ with respect to $f$ is
			
			\begin{equation*}
				\left(	\begin{array}{ccccc	}
					\pi^{1} & & & &\\
					& \pi^{1}&&&\\
					&& \pi^{-\lambda_3+2} &&\\
					&&& \ddots&\\
					&&&&\pi^{-\lambda_n+2}
				\end{array}\right),
			\end{equation*}
			and $\pi L^{\vee}=O_F ( \pi^{-1}(\varepsilon f_1+ f_2)) \oplus O_F (\pi^{-1}(f_1-\overline{\varepsilon} f_2) ) \oplus \bigoplus_{i \neq 1,2} O_F(f_i),$ for some $\varepsilon$ such that $1+\varepsilon \overline{\varepsilon}=\pi$. Then, $\lbrace \pi^{-1}(\varepsilon f_1+ f_2), f_2, f_3,\dots,f_n \rbrace$ forms a basis for $\pi L^{\vee}$.
			
			We fix $x_0=\mathlarger{\sum}_{i \neq 1} \mu_i f_i, \mu_i \in O_F$, and let $x_0=x_1+\mu_{2}f_2$. Then, consider the sets  
			\begin{equation*}\begin{array}{l}
					S_{x_0}^{\geq 0}\coloneqq\lbrace x \in (\pi L^{\vee})^{\geq 0} \backslash (\pi M^{\vee})^{\geq 0} \mid x=x_0+\mu_1 \pi^{-1}(\varepsilon f_1+ f_2), \mu_1 \in O_F \rbrace/L,\\
					S_{x_0}^{\geq 1}\coloneqq\lbrace x \in (\pi L^{\vee})^{\geq 1} \backslash (\pi M^{\vee})^{\geq 1} \mid x=x_0+\mu_1 \pi^{-1}(\varepsilon f_1+ f_2), \mu_1 \in O_F \rbrace/L.
				\end{array}
			\end{equation*}
			
			Note that $x \notin \pi M^{\vee}$ if and only if $\mu_1 \in O_F^{\times}$. Also, we have that   
			\begin{align*} 
				( x, x )&=( x_1+\mu_2f_2+\mu_1 \pi^{-1}(\varepsilon f_1+ f_2),x_1+\mu_2f_2+\mu_1 \pi^{-1}(\varepsilon f_1+ f_2)) \\
				&=( x_1,x_1 )+\mu_2 \overline{\mu_2}\pi+\mu_2\overline{\mu_1}+\overline{\mu_2}\mu_1+\mu_1\overline{\mu_1}.
			\end{align*}
			Therefore, we have
			\begin{equation*}
				\begin{array}{lll}
					x \in S_{x_0}^{\geq 0} & \text{ if and only if } & ( x_1,x_1 )+\mu_2 \overline{\mu_2}\pi+\mu_2\overline{\mu_1}+\overline{\mu_2}\mu_1+\mu_1\overline{\mu_1} \in O_F, \text{and }\mu_1 \in O_F^{\times},\\
					x \in S_{x_0}^{\geq 1} & \text{ if and only if } & ( x_1,x_1 )+\mu_2 \overline{\mu_2}\pi+\mu_2\overline{\mu_1}+\overline{\mu_2}\mu_1+\mu_1\overline{\mu_1} \equiv 0 (\text{mod } \pi), \text{and }\mu_1 \in O_F^{\times}.
				\end{array}
			\end{equation*}
			
			Write  
			\begin{align*} 
				-( x_1, x_1 )&=d_0+d_1 \pi + d_2 \pi^2 +\cdots,\\
				\mu_1&=b_0+b_1 \pi +b_2 \pi^2+\cdots,\\
				\mu_2&=c_0+c_1 \pi + c_2\pi^2+\cdots.
			\end{align*}
			Then, $x \in S_{x_0}^{\geq 0}$ if and only if $b_0 \in O_F^{\times}$. Therefore, there are $q^2$-choices of $c_0$.
			
			Also, $x \in S_{x_0}^{\geq 1}$ if and only if $b_0 \in O_F^{\times}$ and
			\begin{equation*}\begin{array}{ll}
					&d_0=c_0 \overline{b_0}+\overline{c_0}b_0+b_0\overline{b_0}\\
					\Longleftrightarrow & \mathrm{Tr}(c_0\overline{b_0})=d_0-b_0\overline{b_0}.
				\end{array}
			\end{equation*}
			Therefore, for each $b_0 \in O_F^{\times}$, there are $q$-choices of $c_0$. This shows that $\vert S_{x_0}^{\geq 0}\vert=q \vert S_{x_0}^{\geq 1}\vert.$ This finishes the proof of the proposition.

		\end{enumerate}
	\end{proof}

	\begin{proposition}\label{proposition5.38}
		Assume that $\lambda \in \CR_{n}^{0+}$.
		\begin{enumerate}
			\item If $\lambda_1 \geq 3$, we define
			\begin{equation*}
				\eta=(\lambda_1-2,\lambda_2,\dots,\lambda_n) \in \CR_{n}^{0+},
			\end{equation*}
			(if necessary, we change the order of $\eta_i$'s so that $\eta \in \CR_{n}^{0+})$.
			
			Then, we have
			\begin{equation*}\begin{aligned}
					\mu(L_{\lambda})&=\mu(L_{\lambda \geq 1})\\
					&=q^2\mu(L_{\eta \geq 1})+q^{2t_{\ge 1}(\lambda)+2t_{\ge 2}(\lambda)-2}\mu(L_{(\lambda \geq 2) -2})-q^{2t_{\ge 1}(\lambda)+2t_{\ge 2}(\lambda)}\mu(L_{(\eta \geq 2) -2}).
				\end{aligned}
			\end{equation*}
			
			\item If $\lambda_1= \lambda_2=2$, we define
			\begin{equation*}
				\eta=(\lambda_1-1,\lambda_2-1,\dots,\lambda_n) \in \CR_{n}^{0+},
			\end{equation*}
			(if necessary, we change the order of $\eta_i$'s so that $\eta \in \CR_{n}^{0+})$.
			
			Then, we have
			\begin{equation*}\begin{aligned}
					\mu(L_{\lambda})&=\mu(L_{\lambda \geq 1})\\
					&=q^2\mu(L_{\eta \geq 1})+q^{2t_{\ge 1}(\lambda)+2t_{\ge 2}(\lambda)-2}\mu(L_{(\lambda \geq 2) -2})-q^{2t_{\ge 1}(\eta)+2t_{\ge 2}(\eta)}\mu(L_{(\eta \geq 2) -2}).
				\end{aligned}
			\end{equation*}
		\end{enumerate}
		
		Here, we choose the following convention: If $L_{\lambda \geq 2}$ is empty, we assume $\mu(L_{(\lambda \geq 2) -2 })=1$.
	\end{proposition}
	\begin{proof}
		By Proposition \ref{proposition5.36} and Proposition \ref{proposition5.37}, we have 
		\begin{equation*}
			\vert	((L_{\lambda}^{\vee})^{\geq 0} \backslash (L_{\eta}^{\vee})^{\geq0}) / L_{\lambda} \vert=q^{2t_{\ge 1}(\lambda)-1} \vert	((\pi L_{\lambda}^{\vee})^{\geq 1} \backslash (\pi L_{\eta}^{\vee})^{\geq1}) / L_{\lambda} \vert,
		\end{equation*}
		and
		\begin{equation*}
			\vert	((\pi L_{\lambda}^{\vee})^{\geq 0} \backslash (\pi L_{\eta}^{\vee})^{\geq0}) / L_{\lambda} \vert=q \vert	((\pi L_{\lambda}^{\vee})^{\geq 1} \backslash (\pi L_{\eta}^{\vee})^{\geq1}) / L_{\lambda} \vert.
		\end{equation*}
		
		Therefore, we have
		\begin{equation}\label{eq5.29}
			\vert	((L_{\lambda}^{\vee})^{\geq 0} \backslash (L_{\eta}^{\vee})^{\geq0}) / L_{\lambda} \vert=q^{2t_{\ge 1}(\lambda)-2}\vert	((\pi L_{\lambda}^{\vee})^{\geq 0} \backslash (\pi L_{\eta}^{\vee})^{\geq0}) / L_{\lambda} \vert.
		\end{equation}
		
		Note that
		\begin{equation*}\begin{array}{l}
				\vert (L_{\lambda}^{\vee})^{\geq 0}/L_{\lambda}\vert=\mu(L_{\lambda}),\\
				\vert (L_{\eta}^{\vee})^{\geq 0}/L_{\lambda}\vert=q^2 \vert (L_{\eta}^{\vee})^{\geq 0}/L_{\eta}\vert=q^2 \mu(L_{\eta}).
			\end{array}
		\end{equation*}
		
		Also, note that $(\pi L^{\vee}_{\lambda})^{\geq 0}/L_{\lambda}=(\pi (L_{\lambda \geq 2})^{\vee})^{\geq 0}/L_{\lambda \geq 2}$, and it is trivial if $L_{\lambda \geq 2}$ is empty. Since the fundamental invariants of $L_{\lambda \geq 2}$ are at least $2$, any element $x$ of $\pi^{-1}L_{\lambda \geq 2}$ has valuation at least $0$. Therefore, we have
		\begin{equation*}\begin{array}{l}
				\vert (\pi L^{\vee}_{\lambda})^{\geq 0}/L_{\lambda}\vert=\vert(\pi (L_{\lambda \geq 2})^{\vee})^{\geq 0}/L_{\lambda \geq 2}\vert
				=q^{2t_{\ge 2}(\lambda)}\vert((\pi^{-1}L_{\lambda \geq 2})^{\vee})^{\geq 0}/ (\pi^{-1}L_{\lambda \geq 2})\vert\\\\
				=q^{2t_{\ge 2}(\lambda)} \mu(\pi^{-1}L_{\lambda \geq 2})
				=q^{2t_{\ge 2}(\lambda)} \mu(L_{(\lambda \geq 2)-2}).
			\end{array}
		\end{equation*}
		Similarly, we have
		\begin{equation*}
			\vert (\pi L_{\eta}^{\vee})^{\geq 0}/L_{\lambda} \vert=q^2 \vert (\pi L_{\eta}^{\vee})^{\geq 0}/L_{\eta} \vert=q^{2+2t_{\ge 2}(\eta)}\mu(L_{(\eta \geq 2) -2}).
		\end{equation*}
		
		Now, by \eqref{eq5.29} and the fact that $t_{\ge 1}(\lambda)=t_{\ge 1}(\eta)$, we have the proof of the proposition.
	\end{proof}

	\begin{lemma}\label{lemma5.40}
		For $\lambda \in \CR_{n}^{0+}$, $\lambda \geq (1,1,\dots,1)$, we have that
		\begin{equation*}
			\mu(L_{\lambda})=q^{2n-1}\mu(L_{\lambda-1})-(-q)^{\vert \lambda \vert-1}(q-1).
		\end{equation*}
	\end{lemma}
	\begin{proof}
		First, consider the case $\lambda=(1,1,\dots,1)$. Then, it is easy to see that (see \cite[Example 5.6]{vollaard2011supersingular}, for example)
		\begin{equation*}\begin{aligned}
				\mu(L_{1,\dots,1})&=\vert\lbrace (x_1,\dots,x_n) \in \BF_q^{n}\mid x_1^{q+1}+\dots+x_n^{q+1}=0  \rbrace\vert\\
				&=q^{2n-1}+(-q)^{n}+(-q)^{n-1}=q^{2n-1}\mu(L_{0,\dots,0})-(-q)^{n-1}(q-1).
			\end{aligned}
		\end{equation*}
		Similarly, if $\lambda=(2,\overset{n-1}{\overbrace{1,\dots,1}})$, we have that \begin{equation*}\begin{aligned}
				\mu(L_{\lambda})=\vert((L_{2,1,\dots,1})^{\vee})^{\geq 0}/L_{2,1,\dots,1}\vert=q^2\mu(L_{\underset{n-1}{\underbrace{1,\dots,1}}})&=q^{2n-1}+(-q)^{n+1}+(-q)^n\\
				&=q^{2n-1}\mu(L_{1})-(-q)^{n}(q-1).
			\end{aligned}
		\end{equation*}
		
		Now, we will prove the lemma by induction on $\vert \lambda \vert$.
		
		Assume that $\lambda=(\overset{a}{\overbrace{2,\dots,2}},\overset{b}{\overbrace{1,\dots,1}})$, $a \geq 2$, and let $\eta=(\overset{a-2}{\overbrace{2,\dots,2}},\overset{b+2}{\overbrace{1,\dots,1}})$. Then, by Proposition \ref{proposition5.38} (2), we have
		\begin{equation*}
			\mu(L_{\lambda})=q^2 \mu(L_{\eta})+q^{2n+2a-2}-q^{2n+2a-4}.
		\end{equation*}
		Also, we have
		\begin{equation*}\begin{array}{l}
				\mu(L_{\lambda-1})=\mu(L_{\underset{a}{\underbrace{1,\dots,1}}})=q^{2a-1}+(-q)^{a}+(-q)^{a-1},\\
				q^2\mu(L_{\eta-1})=q^2\mu(L_{\underset{a-2}{\underbrace{1,\dots,1}}})=q^{2a-3}+(-q)^{a}+(-q)^{a-1},\\
				\mu(L_{\lambda-1})=q^2\mu(L_{\eta-1})+q^{2a-1}-q^{2a-3}.
			\end{array}
		\end{equation*}
		Therefore, we have that
		\begin{equation*}
			\mu(L_{\lambda})-q^{2n-1}\mu(L_{\lambda-1})=q^2\lbrace\mu(L_{\eta})-q^{2n-1}\mu(L_{\eta-1})\rbrace.
		\end{equation*}
		Now, by our induction hypothesis, we have that
		\begin{equation*}
			q^2\lbrace\mu(L_{\eta})-q^{2n-1}\mu(L_{\eta-1})\rbrace=-(-q)^{\vert\eta\vert-1+2}(q-1)=-(-q)^{\vert\lambda\vert-1}(q-1).
		\end{equation*}
		This finishes the proof of the lemma when $\lambda=(\overset{a}{\overbrace{2,\dots,2}},\overset{b}{\overbrace{1,\dots,1}})$, $a \geq 2$.
		
		Now, assume that $\lambda=(\overset{n-a-b}{\overbrace{3,\dots,3}},\overset{a}{\overbrace{2,\dots,2}},\overset{b}{\overbrace{1,\dots,1}})$, and let $\eta=(\overset{n-a-b-1}{\overbrace{3,\dots,3}},\overset{a}{\overbrace{2,\dots,2}},\overset{b+1}{\overbrace{1,\dots,1}})$. Then, by Proposition \ref{proposition5.38} (1), we have
		\begin{equation*}
			\mu(L_{\lambda})=q^2\mu(L_{\eta})+q^{4n-2b-2}\mu(L_{(\lambda \geq 2)-2})-q^{4n-2b-2}\mu(L_{(\eta \geq 2)-2}).
		\end{equation*}
		Also, by the previous case, we have
		\begin{equation*}
			\begin{array}{l}
				\mu(L_{\lambda-1})=q^{2n-2b-1}\mu(L_{(\lambda \geq 2)-2})-(-q)^{\vert \lambda \vert-n-1}(q-1),\\
				\mu(L_{\eta-1})=q^{2n-2b-3}\mu(L_{(\eta \geq 2)-2})-(-q)^{\vert \eta \vert -n-1}(q-1),\\
				\mu(L_{\lambda-1})=q^2\mu(L_{\eta-1})+q^{2n-2b-1}\mu(L_{(\lambda \geq 2)-2})-q^{2n-2b-1}\mu(L_{(\eta\geq 2)-2}).
			\end{array}
		\end{equation*}
		
		Therefore, we have
		\begin{equation*}
			\begin{aligned}
				\mu(L_{\lambda})-q^{2n-1}\mu(L_{\lambda-1})
				&=q^2\mu(L_{\eta})+q^{4n-2b-2}\mu(L_{(\lambda \geq 2)-2})-q^{4n-2b-2}\mu(L_{(\eta \geq 2)-2})\\
				&-q^{2n-1}\lbrace q^2 \mu(L_{\eta-1})+q^{2n-2b-1}\mu(L_{(\lambda\geq 2)-2})-q^{2n-2b-1}\mu(L_{(\eta\geq 2)-2})\\
				&=q^2\lbrace \mu(L_{\eta})-q^{2n-1}\mu(L_{\eta-1})\rbrace\\
				&=-(-q)^{\vert \eta \vert+2-1}(q-1) \text{ }(\text{by our inductive hypothesis})\\
				&=-(-q)^{\vert \lambda \vert -1}(q-1).
			\end{aligned}
		\end{equation*}
		This finishes the proof of the lemma when $\lambda=(\overset{n-a-b}{\overbrace{3,\dots,3}},\overset{a}{\overbrace{2,\dots,2}},\overset{b}{\overbrace{1,\dots,1}})$.
		
		Now, assume that $\lambda=(\lambda_1=4,\lambda_2,\dots,\lambda_{n-a-b},\overset{a}{\overbrace{2,\dots,2}},\overset{b}{\overbrace{1,\dots,1}})$, $\lambda_{n-a-b}\geq3$. Let
		\begin{equation*} \eta=(\lambda_2,\dots,\lambda_{n-a-b},\overset{a+1}{\overbrace{2,\dots,2}},\overset{b}{\overbrace{1,\dots,1}}).
		\end{equation*}
		Then, by Proposition \ref{proposition5.38} (1), we have
		\begin{equation*}
			\mu(L_{\lambda})=q^2\mu(L_{\eta})+q^{4n-2b-2}\mu(L_{(\lambda \geq 2)-2})-q^{4n-2b}\mu(L_{(\eta \geq 2)-2}).
		\end{equation*}
		Similarly, we have $t_{\ge 2}(\lambda-1)=n-a-b$,  $t_{\ge 2}(\eta-1)=n-a-b-1$,  $t_
		{\geq 1}(\lambda-1)=t_{\geq 1}(\eta-1)=n-b$, and
		\begin{equation*}
			\mu(L_{\lambda-1})=q^2\mu(L_{\eta-1})+q^{4n-2a-4b-2}\mu(L_{(\lambda \geq 3)-3})-q^{4n-2a-4b-2}\mu(L_{(\eta \geq 3)-3}).
		\end{equation*}
		Therefore,
		\begin{equation*}
			\begin{array}{cl}
				\mu(L_{\lambda})-q^{2n-1}\mu(L_{\lambda-1})&=q^2\lbrace \mu(L_{\eta})-q^{2n-1}\mu(L_{\eta-1})\rbrace\\
				&+q^{4n-2b-2}\lbrace \mu(L_{(\lambda \geq 2)-2})-q^{2n-2a-2b-1}\mu(L_{(\lambda \geq 3)-3})\rbrace\\
				&-q^{4n-2b}\lbrace \mu(L_{(\eta \geq 2)-2})-q^{2n-2a-2b-3}\mu(L_{(\eta \geq 3)-3})\rbrace\\
				(\text{by our inductive hypothesis})&=-(-q)^{\vert \eta \vert+1}(q-1)+q^{4n-2b-2}(-(-q)^{\vert \lambda \vert-b-2(n-b)-1}(q-1))\\
				&-(-q)^{4n-2b}(-(-q)^{\vert \eta \vert-b-2(n-b)-1}(q-1))\\
				&=-(-q)^{\vert \lambda \vert-1}(q-1).
			\end{array}
		\end{equation*}
		This finishes the proof of the lemma when $\lambda=(\lambda_1=4,\lambda_2,\dots,\lambda_{n-a-b},\overset{a}{\overbrace{2,\dots,2}},\overset{b}{\overbrace{1,\dots,1}})$, $\lambda_{n-a-b}\geq3$.
		
		Finally, assume that $\lambda=(\lambda_1,\dots,\lambda_{n-a-b},\overset{a}{\overbrace{2,\dots,2}},\overset{b}{\overbrace{1,\dots,1}})$, $\lambda_1 \geq 5$. In this case, let $\eta=(\lambda_1-2,\lambda_2,\dots,\lambda_n)$ (if necessary, we change the order of $\eta_i$'s so that $\eta \in \CR_{n}^{0+})$. Then, by Proposition \ref{proposition5.38} (1), we have
		\begin{equation*}
			\mu(L_{\lambda})=q^2\mu(L_{\eta})+q^{4n-2b-2}\mu(L_{(\lambda \geq 2)-2})-q^{4n-2b}\mu(L_{(\eta \geq 2)-2}),
		\end{equation*}
		and
		\begin{equation*}
			\mu(L_{\lambda-1})=q^2\mu(L_{\eta-1})+q^{4n-2a-4b-2}\mu(L_{(\lambda \geq 3)-3})-q^{4n-2a-4b}\mu(L_{(\eta \geq 3)-3}).
		\end{equation*}
		Therefore,
		\begin{equation*}
			\begin{array}{rl}
				\mu(L_{\lambda})-q^{2n-1}\mu(L_{\lambda-1})&=q^2\lbrace \mu(L_{\eta})-q^{2n-1}\mu(L_{\eta-1})\rbrace\\
				&+q^{4n-2b-2}\lbrace \mu(L_{(\lambda \geq 2)-2})-q^{2n-2a-2b-1}\mu(L_{(\lambda \geq 3)-3})\rbrace\\
				&-q^{4n-2b}\lbrace \mu(L_{(\eta \geq 2)-2})-q^{2n-2a-2b-1}\mu(L_{(\eta \geq 3)-3})\rbrace\\
				(\text{by our inductive hypothesis})&=-(-q)^{\vert \eta \vert+1}(q-1)+q^{4n-2b-2}(-(-q)^{\vert \lambda \vert-b-2(n-b)-1}(q-1))\\
				&-(-q)^{4n-2b}(-(-q)^{\vert \eta \vert-b-2(n-b)-1}(q-1))\\
				&=-(-q)^{\vert \lambda \vert-1}(q-1).
			\end{array}
		\end{equation*}
		This finishes the proof of the lemma.
	\end{proof}

	\begin{lemma}\label{lemma5.43}
		Assume that $a \geq 1$, and $j \in \BZ$. Let $\kappa_{a,i}$ be the constants such that
		\begin{equation*}
			(1-X)(1-(-q)X)\dots(1-(-q)^{a-2}X)=\sum_{i=0}^{a-1}\kappa_{a,i}X^i.
		\end{equation*}
		Then, we have
		\begin{equation*}
			D_{n,h}(a,b,c)=\sum_{i=0}^{a-1}\kappa_{a,i} (-q)^{i(a+b-h+1)}D_{n-i,h-i}(1,b+a-1-i,c).
		\end{equation*}
	\end{lemma}
	\begin{proof}
		This follows from Theorem \ref{theorem5.24}.
	\end{proof}
	
	\begin{lemma}\label{lemma5.44}
		Assume that $a \geq 1$. Then, we have
		\begin{equation}\label{eq5.34}\begin{array}{ccl}
				&&D_{n,h}(a+1,b,c)\\
				+((-q)^a-1)&\times&D_{n,h}(a,b+1,c)\\
				+(-q)^a((-q)^a-1)&\times&D_{n,h}(a,b,c+1)\\
				-(-q)^{2a}(1-(-q)^b)(1-(-q)^{b-1})&\times&D_{n,h}(a+1,b-2,c+2)\\
				-(-q)^{2a+b-1}(1-(-q)^b)(1-(-q)^a)&\times&D_{n,h}(a,b-1,c+2)
			\end{array}
			=0.
		\end{equation}
	\end{lemma}
	\begin{proof}
		To simplify notation, we assume that for $k \geq 1$,
		\begin{equation*}\begin{array}{l}
				X_0^{k}\coloneqq D_{n-k+1,h-k+1}(1,b+a-k+1,c),\\
				X_1^{k}\coloneqq D_{n-k+1,h-k+1}(1,b+a-k,c+1),\\
				X_2^{k}\coloneqq D_{n-k+1,h-k+1}(1,b+a-k-1,c+2).
			\end{array}
		\end{equation*}
		
		Then, by Lemma \ref{lemma5.43}, we have that
		\begin{equation*}\begin{array}{l}
				D_{n,h}(a+1,b,c)=X_0(1-(-q)^{a+b-h+2}X_0)\dots(1-(-q)^{2a+b-h+1}X_0),\\
				D_{n,h}(a,b+1,c)=X_0(1-(-q)^{a+b-h+2}X_0)\dots(1-(-q)^{2a+b-h}X_0),\\
				D_{n,h}(a,b,c+1)=X_1(1-(-q)^{a+b-h+1}X_1)\dots(1-(-q)^{2a+b-h-1}X_1),\\
				D_{n,h}(a+1,b-2,c+2)=X_2(1-(-q)^{a+b-h}X_2)\dots(1-(-q)^{2a+b-h-1}X_2),\\
				D_{n,h}(a,b-1,c+2)=X_2(1-(-q)^{a+b-h}X_2)\dots(1-(-q)^{2a+b-h-2}X_2).\\
			\end{array}
		\end{equation*}
		
		Therefore, we have that
		\begin{equation*}\begin{array}{l}
				D_{n,h}(a+1,b,c)+((-q)^a-1)D_{n,h}(a,b+1,c)\\
				=(-q)^aX_0(1-(-q)^{a+b-h+1}X_0)\dots(1-(-q)^{2a+b-h}X_0)\\
				=(-q)^aX_0(1-(-q)^{2a+b-h}X_0)(\sum_{t=0}^{a-1}\kappa_{a,t}(-q)^{t(a+b-h+1)}X_0^t).
			\end{array}
		\end{equation*}
		This implies that for $1 \leq t \leq a+1$, $X_0^t$-terms are
		\begin{equation}\label{eq5.35}
			\left\lbrace \begin{array}{ll}
				(-q)^a\lbrace \kappa_{a,t-1}(-q)^{(a+b-h+1)(t-1)}-(-q)^{2a+b-h}\kappa_{a,t-2}(-q)^{(a+b-h+1)(t-2)}\rbrace X_0^t& \text{ if } t \neq a+1,\\
				-(-q)^{3a+b-h}\kappa_{a,a-1}(-q)^{(a+b-h+1)(a-1)}X_0^{a+1} & \text{ if } t=a+1.
			\end{array}\right.
		\end{equation}
		
		Similarly, for $1 \leq t \leq a$, $X_1^t$-terms in $(-q)^a((-q)^a-1)D_{n,h}(a,b,c+1)$ are
		\begin{equation}\label{eq5.36}
			(-q)^a((-q)^a-1)\kappa_{a,t-1}(-q)^{(a+b-h+1)(t-1)}X_1^t.
		\end{equation}
		
		Finally, for $1 \leq t \leq a+1$, $X_2^t$-terms in
		\begin{equation*}
			\begin{array}{l}
				-(-q)^{2a}(1-(-q)^b)(1-(-q)^{b-1})D_{n,h}(a+1,b-2,c+2)\\
				-(-q)^{2a+b-1}(1-(-q)^b)(1-(-q)^a)D_{n,h}(a,b-1,c+2)
			\end{array}
		\end{equation*}
		are
		\begin{equation}\label{eq5.37}
			\left\lbrace \begin{array}{ll}
				\lbrace-\kappa_{a,t-1}(-q)^{(a+b-h)(t-1)+2a}(1-(-q)^b)(1-(-q)^{a+b-1})&\\ +\kappa_{a,t-2}(-q)^{(a+b-h)(t-2)+4a+b-h-1}(1-(-q)^b)(1-(-q)^{b-1})\rbrace X_2^t&\text{ if } t \neq a+1,\\
				\kappa_{a,a-1}(-q)^{(a+b-h)(a-1)+4a+b-h-1}(1-(-q)^b)(1-(-q)^{b-1})X_2^{a+1}  &\text{ if } t=a+1.
			\end{array}\right.
		\end{equation}

		Note that if $a+b < h-1$, then $X_0^t, X_1^t,$ and $X_2^t$ are all zero. Therefore, all of the degree $t$-terms of the left-hand side of \eqref{eq5.34} are $0$, and hence \eqref{eq5.34} holds in this case.
		
		Now, assume that $a+b \geq h+2$. Then, by Theorem \ref{theorem5.31}, we have that
		\begin{equation}\label{eq5.38}
			\begin{array}{l}
				X_0^t=\prod_{l=h-t+2}^{a+b-t+1}(1-(-q)^l),\\
				X_1^t=\prod_{l=h-t+2}^{a+b-t}(1-(-q)^l),\\
				X_2^t=\prod_{l=h-t+2}^{a+b-t-1}(1-(-q)^l).
			\end{array}
		\end{equation}
		
		Combining \eqref{eq5.35}, \eqref{eq5.36}, \eqref{eq5.37}, and \eqref{eq5.38}, we can see that the degree $t$-terms of the left hand side of \eqref{eq5.34} can be written as follows: if $1 \leq t \leq a$, we have
		\begin{equation}\label{eq5.39}\begin{aligned}
				\prod_{l=h-t+2}^{a+b-t-1} (1-(-q)^l)\Big{\lbrace}&	(-q)^a\lbrace \kappa_{a,t-1}(-q)^{(a+b-h+1)(t-1)}-(-q)^{2a+b-h}\kappa_{a,t-2}(-q)^{(a+b-h+1)(t-2)}\rbrace\\
				&\times (1-(-q)^{a+b-t+1})(1-(-q)^{a+b-t})\\
				&+(-q)^a((-q)^a-1)\kappa_{a,t-1}(-q)^{(a+b-h+1)(t-1)}(1-(-q)^{a+b-t})\\
				&+	\lbrace-\kappa_{a,t-1}(-q)^{(a+b-h)(t-1)+2a}(1-(-q)^b)(1-(-q)^{a+b-1})\\ &+\kappa_{a,t-2}(-q)^{(a+b-h)(t-2)+4a+b-h-1}(1-(-q)^b)(1-(-q)^{b-1})\rbrace \Big{\rbrace}\\
				&=((-q)^{a+2b}-(-q)^t)(-q)^{(a+b-h)(t-2)+3a+b-h-t-2}\\
				&\lbrace -((-q)^{t+1}-(-q)^2)\kappa_{a,t-1}+((-q)^t-(-q)^{a+1})\kappa_{a,t-2}\rbrace. 
			\end{aligned}
		\end{equation}
		
		Now, we claim that
		\begin{equation}\label{eq5.40}
			\lbrace -((-q)^{t+1}-(-q)^2)\kappa_{a,t-1}+((-q)^t-(-q)^{a+1})\kappa_{a,t-2}\rbrace=0.
		\end{equation}
		
		Recall from Lemma \ref{lemma5.43} that
		\begin{equation*}\begin{array}{l}
				(1-X)(1-(-q)X)\dots(1-(-q)^{a-2}X)=\sum_{t=0}^{a-1}\kappa_{a,t}X^t,\\
				(1-(-q)X)(1-(-q)^2X)\dots(1-(-q)^{a-1}X)=\sum_{t=0}^{a-1}\kappa_{a,t}(-q)^tX^t.
		\end{array}	\end{equation*}
		Therefore,
		\begin{equation*}
			(1-(-q)^{a-1}X)\sum_{t=0}^{a-1}\kappa_{a,t}X^t=(1-X)\sum_{t=0}^{a-1}\kappa_{a,t}(-q)^tX^t.
		\end{equation*}
		
		By comparing degree $t-1$ terms, we have that
		\begin{equation}\label{eq5.41.2}\begin{array}{ll}
				&\kappa_{a,t-1}-(-q)^{a-1}\kappa_{a,t-2}=\kappa_{a,t-1}(-q)^{t-1}-\kappa_{a,t-2}(-q)^{t-2}\\
				\Longleftrightarrow & -((-q)^{t-1}-1)\kappa_{a,t-1}+((-q)^{t-2}-(-q)^{a-1})\kappa_{a,t-2}=0.
			\end{array}
		\end{equation}
		
		Therefore, we have that \eqref{eq5.40} holds, and hence \eqref{eq5.39} is zero. This shows that the degree $t$-term of \eqref{eq5.34} is zero for $1 \leq t \leq a$.
		
		For $t=a+1$, we have that the degree $a+1$-terms of the left hand side of \eqref{eq5.34} are
		\begin{equation*}\begin{array}{ll}
				\prod_{l=h-a+1}^{b-2} (1-(-q)^l) &\Big{\lbrace}-(-q)^{3a+b-h}\kappa_{a,a-1}(-q)^{(a+b-h+1)(a-1)}(1-(-q)^b)(1-(-q)^{b-1})\\
				&+\kappa_{a,a-1}(-q)^{(a+b-h)(a-1)+4a+b-h-1}(1-(-q)^b)(1-(-q)^{b-1})\Big{\rbrace}\\
				=0.&
			\end{array}
		\end{equation*}
		
		This shows that \eqref{eq5.34} holds when $a+b \geq h+2$.
		
		Now, the remaining cases are $a+b=h-1$, $h$, and $h+1$.
		
		When $a+b=h-1$, we have that $X_0^t=1$, $X_1^t=X_2^t=0$ by Theorem \ref{theorem5.31}. Therefore, we have that
		\begin{equation*}\begin{array}{l}
				D_{n,h}(a+1,b,c)+((-q)^a-1)D_{n,h}(a,b+1,c)\\
				=(-q)^a(1-(-q)^{a+b-h+1})\dots(1-(-q)^{2a+b-h})\\
				=0 \quad (\text{since }a+b-h+1=0).
			\end{array}
		\end{equation*}
		Also, we have $D_{n,h}(a,b,c+1)=0$,
		$D_{n,h}(a+1,b-2,c+2)=0$, and
		$D_{n,h}(a,b-1,c+2)=0$. This shows that \eqref{eq5.34} holds when $a+b=h-1$.
		
		When $a+b=h$, we have that $X_0^t=X_1^t=1$ and $X_2^t=0$ by Theorem \ref{theorem5.31}. Therefore, we have that
		\begin{equation*}
			\begin{array}{l}
				D_{n,h}(a+1,b,c)+((-q)^a-1)D_{n,h}(a,b+1,c)
				=(-q)^a(1-(-q)^{a+b-h+1})\dots(1-(-q)^{2a+b-h});\\\\
				(-q)^a((-q)^a-1)D_{n,h}(a,b,c+1)=(-q)^a((-q)^a-1)(1-(-q)^{a+b-h+1})\dots(1-(-q)^{2a+b-h-1});\\\\
				-(-q)^{2a}(1-(-q)^b)(1-(-q)^{b-1})D_{n,h}(a+1,b-2,c+2)=0;\\\\
				-(-q)^{2a+b-1}(1-(-q)^b)(1-(-q)^a)D_{n,h}(a,b-1,c+2)=0.\\
			\end{array}
		\end{equation*}
		
		Therefore, the left-hand side of \eqref{eq5.34} is
		\begin{equation*}\begin{array}{l}
				(-q)^a(1-(-q)^{a+b-h+1})\dots(1-(-q)^{2a+b-h})+(-q)^a((-q)^a-1)(1-(-q)^{a+b-h+1})\dots(1-(-q)^{2a+b-h-1})\\
				=(-q)^a(1-(-q)^{a+b-h+1})\dots(1-(-q)^{2a+b-h-1})\lbrace (-q)^a-(-q)^{2a+b-h} \rbrace\\
				=0 \quad (\text{since }2a+b-h=a).
			\end{array}
		\end{equation*}
		
		Now, assume that $a+b=h+1$. In this case, by Theorem \ref{theorem5.31}, we have that $X_0^t=(1-(-q)^{h-t+2})$, $X_1^t=1$, and $X_2^t=1$. Therefore, we have
		\begin{equation}\label{eq5.42.2}
			\begin{array}{l}
				D_{n,h}(a+1,b,c)+((-q)^a-1)D_{n,h}(a,b+1,c)\\
				=(-q)^a(1-(-q)^{a+b-h+1})\dots(1-(-q)^{2a+b-h})\\
				-(-q)^a(-q)^{h+1}(1-(-q)^{a+b-h+1}(-q)^{-1})\dots(1-(-q)^{2a+b-h}(-q)^{-1})\\
				=(-q)^a(1-(-q)^{2})\dots(1-(-q)^{a})(1-(-q)^{a+1}-(-q)^{h+1}+(-q)^{h+2}).
			\end{array}
		\end{equation}
		
		Also, we have
		\begin{equation*}\begin{array}{l}
				(-q)^a((-q)^a-1)D_{n,h}(a,b,c+1)=(-q)^a((-q)^a-1)(1-(-q)^{2})\dots(1-(-q)^{a}),
		\end{array}\end{equation*}
		and
		\begin{equation*}
			\begin{array}{l}
				-(-q)^{2a}(1-(-q)^b)(1-(-q)^{b-1})D_{n,h}(a+1,b-2,c+2)\\
				-(-q)^{2a+b-1}(1-(-q)^b)(1-(-q)^a)D_{n,h}(a,b-1,c+2)\\
				=-(-q)^{2a}(1-(-q)^b)(1-(-q)^{b-1})(1-(-q)^{a+b-h})\dots(1-(-q)^{2a+b-h-1})\\
				-(-q)^{2a+b-1}(1-(-q)^b)(1-(-q)^a)(1-(-q)^{a+b-h})\dots(1-(-q)^{2a+b-h-2})\\
				=-(-q)^{2a}(1-(-q)^b)(1-(-q)^1)\dots(1-(-q)^a).
			\end{array}
		\end{equation*}
		
		Therefore, the left-hand side of \eqref{eq5.34} is
		\begin{equation*}
			\begin{array}{l}
				(-q)^a(1-(-q)^{2})\dots(1-(-q)^{a})(1-(-q)^{a+1}-(-q)^{h+1}+(-q)^{h+2})\\
				+(-q)^a((-q)^a-1)(1-(-q)^{2})\dots(1-(-q)^{a})-(-q)^{2a}(1-(-q)^b)(1-(-q)^1)\dots(1-(-q)^a)\\
				=0.
			\end{array}
		\end{equation*}
		Therefore, \eqref{eq5.34} holds when $a+b=h+1$.
		
		This finishes the proof of the lemma.
	\end{proof}

	\begin{lemma}\label{lemma5.45}
		Assume that $a \geq 2$. Then, we have
		\begin{equation}\label{eq5.42}\begin{array}{ccl}
				&&D_{n,h}(a,b+1,c)\\
				+((-q)^{a-1}-1)&\times&D_{n,h}(a-1,b+2,c)\\
				-(-q)^{a-1}&\times&D_{n,h}(a,b,c+1)\\
				+(-q)^{2a+b-1}(1-(-q)^b)&\times&D_{n,h}(a,b-1,c+2)\\
				+(-q)^{2a+2b-1}(1-(-q)^{a-1})&\times&D_{n,h}(a-1,b,c+2)
			\end{array}
			=0
		\end{equation}
	\end{lemma}
	\begin{proof}
		To simplify notation, we follow the notation in the proof of Lemma \ref{lemma5.44}. Then, by Lemma \ref{lemma5.43}, we have
		\begin{equation*}\begin{array}{l}
				D_{n,h}(a,b+1,c)=X_0(1-(-q)^{a+b-h+2}X_0)\dots(1-(-q)^{2a+b-h}X_0),\\
				D_{n,h}(a-1,b+2,c)=X_0(1-(-q)^{a+b-h+2}X_0)\dots(1-(-q)^{2a+b-h-1}X_0),\\
				D_{n,h}(a,b,c+1)=X_1(1-(-q)^{a+b-h+1}X_1)\dots(1-(-q)^{2a+b-h-1}X_1),\\
				
				D_{n,h}(a,b-1,c+2)=X_2(1-(-q)^{a+b-h}X_2)\dots(1-(-q)^{2a+b-h-2}X_2),\\
				D_{n,h}(a-1,b,c+2)=X_2(1-(-q)^{a+b-h}X_2)\dots(1-(-q)^{2a+b-h-3}X_2).
			\end{array}
		\end{equation*}
		
		This implies that for $1 \leq t \leq a$, $X_0^t$-terms of $D_{n,h}(a,b+1,c)+((-q)^{a-1}-1)D_{n,h}(a-1,b+2,c)$ are
		\begin{equation}\label{eq5.44}
			\left\lbrace \begin{array}{ll}
				(-q)^{a-1}\lbrace \kappa_{a-1,t-1}(-q)^{(a+b-h+1)(t-1)}-(-q)^{2a+b-h-1}\kappa_{a-1,t-2}(-q)^{(a+b-h+1)(t-2)}\rbrace X_0^t,& \text{ if } t \neq a,\\
				-(-q)^{3a+b-h-2}\kappa_{a-1,a-2}(-q)^{(a+b-h+1)(a-2)}X_0^{a} & \text{ if } t=a.
			\end{array}\right.
		\end{equation}
		
		Similarly, for $1 \leq t \leq a$, $X_1^t$-terms in $-(-q)^{a-1}D_{n,h}(a,b,c+1)$ are
		\begin{equation}\label{eq5.45}
			\left\lbrace \begin{array}{ll}
				-(-q)^{a-1}\lbrace \kappa_{a-1,t-1}(-q)^{(a+b-h+1)(t-1)}-(-q)^{2a+b-h-1}\kappa_{a-1,t-2}(-q)^{(a+b-h+1)(t-2)}\rbrace X_1^t,& \text{ if } t \neq a,\\
				(-q)^{3a+b-h-2}\kappa_{a-1,a-2}(-q)^{(a+b-h+1)(a-2)}X_1^{a} & \text{ if } t=a.
			\end{array}\right.
		\end{equation}
		
		Finally, for $1 \leq t \leq a$, $X_2^t$-terms in $(-q)^{2a+b-1}(1-(-q)^b)D_{n,h}(a,b-1,c+2)+(-q)^{2a+2b-1}(1-(-q)^{a-1})D_{n,h}(a-1,b,c+2)$ are
		\begin{equation}\label{eq5.46}
			\left\lbrace \begin{array}{ll}
				\lbrace\kappa_{a-1,t-1}(-q)^{(a+b-h)(t-1)+2a+b-1}(1-(-q)^{a+b-1})&\\ -\kappa_{a-1,t-2}(-q)^{(a+b-h)(t-2)+4a+2b-h-3}(1-(-q)^b\rbrace X_2^t&\text{ if } t \neq a,\\
				-\kappa_{a-1,a-2}(-q)^{(a+b-h)(a-2)+4a+2b-h-3}(1-(-q)^b)X_2^{a}  &\text{ if } t=a.
			\end{array}\right.
		\end{equation}
		
		First, assume that $a+b \geq h+2$. Then, by \eqref{eq5.38}, \eqref{eq5.44}, \eqref{eq5.45}, and \eqref{eq5.46}, we have that the degree $t$-terms of the left hand side of \eqref{eq5.42} are as follows: for $1 \leq t \leq a-1$, we have
		
		\begin{equation}\label{eq5.47}\begin{array}{l}
				\prod_{l=h-t+2}^{a+b-t-1} (1-(-q)^l)\\
				\times\Big{\lbrace}	(-q)^{a-1}\lbrace \kappa_{a-1,t-1}(-q)^{(a+b-h+1)(t-1)}-(-q)^{2a+b-h-1}\kappa_{a-1,t-2}(-q)^{(a+b-h+1)(t-2)}\rbrace\\
				\times (1-(-q)^{a+b-t+1})(1-(-q)^{a+b-t})\\
				-(-q)^{a-1}\lbrace \kappa_{a-1,t-1}(-q)^{(a+b-h+1)(t-1)}-(-q)^{2a+b-h-1}\kappa_{a-1,t-2}(-q)^{(a+b-h+1)(t-2)}\rbrace\\
				\times (1-(-q)^{a+b-t})\\
				+	\lbrace\kappa_{a-1,t-1}(-q)^{(a+b-h)(t-1)+2a+b-1}(1-(-q)^{a+b-1})\\ -\kappa_{a-1,t-2}(-q)^{(a+b-h)(t-2)+4a+2b-h-3}(1-(-q)^b)\rbrace \Big{\rbrace}\\
				=(-q)^{(a+b-h)(t-1)+3a+2b-t-3}\\
				\lbrace -((-q)^{t+1}-(-q)^2)\kappa_{a-1,t-1}+((-q)^t-(-q)^{a})\kappa_{a-1,t-2}\rbrace. 
			\end{array}
		\end{equation}
		
		Now, by \eqref{eq5.40}, we have that \eqref{eq5.47} is zero.
		
		For $t=a$, we have that the degree $a$-terms of the left hand side of \eqref{eq5.42} are
		\begin{equation*}\begin{array}{ll}
				\prod_{l=h-a+2}^{b-1} (1-(-q)^l) &\Big{\lbrace}-(-q)^{3a+b-h-2}\kappa_{a-1,a-2}(-q)^{(a+b-h+1)(a-2)}(1-(-q)^{b+1})(1-(-q)^{b})\\
				&+(-q)^{3a+b-h-2}\kappa_{a-1,a-2}(-q)^{(a+b-h+1)(a-2)}(1-(-q)^b)\\
				&-\kappa_{a-1,a-2}(-q)^{(a+b-h)(a-2)+4a+2b-h-3}(1-(-q)^b)\Big{\rbrace}\\
				=0.&
			\end{array}
		\end{equation*}
		
		This shows that \eqref{eq5.42} holds when $a+b \geq h+2$.
		
		Now, the remaining cases are $a+b <h-1$, $a+b=h-1$, $h$, and $h+1$.
		
		If $a+b<h-1$, then $X_0^t, X_1^t,$ and $X_2^t$ are all zero, and hence \eqref{eq5.42} holds.
		
		If $a+b=h-1$, we have that $X_0^t=1$, and $X_1^t=X_2^t=0$ by Theorem \ref{theorem5.31}. Therefore, we have that \eqref{eq5.42} is
		\begin{equation*}\begin{array}{l}
				(1-(-q)^{a+b-h+2})\dots(1-(-q)^{2a+b-h})+((-q)^{a-1}-1)(1-(-q)^{a+b-h+2})\dots(1-(-q)^{2a+b-h-1})\\
				=0 \quad (\text{ since }2a+b-h=a-1).
			\end{array}		
		\end{equation*}
		This shows that \eqref{eq5.42} holds when $a+b=h-1$.
		
		If $a+b=h$, then we have $X_0^t=X_1^t=1$ and $X_2^t=0$ by Theorem \ref{theorem5.31}. Therefore, we have that \eqref{eq5.42} is
		\begin{equation*}
			\begin{array}{l}
				(1-(-q)^{a+b-h+2})\dots(1-(-q)^{2a+b-h})+((-q)^{a-1}-1)(1-(-q)^{a+b-h+2})\dots(1-(-q)^{2a+b-h-1})\\
				-(-q)^{a-1}(1-(-q)^{a+b-h+1})\dots(1-(-q)^{2a+b-h-1})\\
				=0.
			\end{array}
		\end{equation*}
		This shows that \eqref{eq5.42} holds when $a+b=h$.
		
		When $a+b=h+1$, we have $X_0^t=(1-(-q)^{h-t+2})$, and $X_1^t=X_2^t=1$. Therefore, by \eqref{eq5.42.2}, we have that \eqref{eq5.42} is
		\begin{equation*}
			\begin{array}{l}
				(-q)^{a-1}(1-(-q)^2)\dots(1-(-q)^{a-1})(1-(-q)^a-(-q)^{h+1}+(-q)^{h+2})\\
				-(-q)^{a-1}(1-(-q)^2)\dots(1-(-q)^a)\\
				+(-q)^{2a+b-1}(1-(-q)^b)(1-(-q))\dots(1-(-q)^{a-1})\\
				+(-q)^{2a+2b-1}(1-(-q)^{a-1})(1-(-q))\dots(1-(-q)^{a-2})\\
				=0.
			\end{array}
		\end{equation*}
		
		Therefore, \eqref{eq5.42} holds when $a+b=h+1$, and this finishes the proof of the lemma.
	\end{proof}

	\begin{lemma}\label{lemma5.46}
		Assume that $a=0$. Then, we have
		\begin{equation}\label{eq5.48}\begin{array}{ccl}
				&&D_{n,h}(a+1,b,c)\\
				+((-q)^a-1)&\times&D_{n,h}(a,b+1,c)\\
				+(-q)^a((-q)^a-1)&\times&D_{n,h}(a,b,c+1)\\
				-(-q)^{2a}(1-(-q)^b)(1-(-q)^{b-1})&\times&D_{n,h}(a+1,b-2,c+2)\\
				-(-q)^{2a+b-1}(1-(-q)^b)(1-(-q)^a)&\times&D_{n,h}(a,b-1,c+2)
			\end{array}
		\end{equation}
		\begin{equation*}
			\mathlarger{	\mathlarger{=}}	\quad\quad\begin{array}{cl}
				0 &	\text{ if }b \geq h+2 \text{ or } b \leq h-2,\\
				1 & \text{ if } b=h-1, h,\\
				q^{2h+1}+(-q)^h &\text { if } b=h+1.
			\end{array}
		\end{equation*}
	\end{lemma}
	\begin{proof}
		Note that when $a=0$, we have that \eqref{eq5.48} is
		\begin{equation*}
			D_{n,h}(1,b,c)-(1-(-q)^b)(1-(-q)^{b-1})D_{n,h}(1,b-2,c+2).
		\end{equation*}
		
		If $b \leq h-2$, then we have that both $D_{n,h}(1,b,c)$ and $D_{n,h}(1,b-2,c+2)$ are zero.
		
		If $b \geq h+2$, then by Theorem \ref{theorem5.31}, we have that
		\begin{equation*}\begin{array}{l}
				D_{n,h}(1,b,c)-(1-(-q)^b)(1-(-q)^{b-1})D_{n,h}(1,b-2,c+2)\\
				=\prod_{l=h+1}^{b}(1-(-q)^l)-(1-(-q)^b)(1-(-q)^{b-1})\prod_{l=h+1}^{b-2}(1-(-q)^l)=0.
			\end{array}
		\end{equation*}
		
		Now, assume that $b=h+1$. Then, by Theorem \ref{theorem5.31}, we have
		\begin{equation*}\begin{array}{l}
				D_{n,h}(1,b,c)-(1-(-q)^b)(1-(-q)^{b-1})D_{n,h}(1,b-2,c+2)\\
				=(1-(-q)^{h+1})-(1-(-q)^{h+1})(1-(-q)^{h})=(-q)^h(1-(-q)^{h+1})=q^{2h+1}+(-q)^h.
			\end{array}
		\end{equation*}
		
		Finally, if $b=h, h-1$, then by Theorem \ref{theorem5.31}, we have
		\begin{equation*}\begin{array}{l}
				D_{n,h}(1,b,c)-(1-(-q)^b)(1-(-q)^{b-1})D_{n,h}(1,b-2,c+2)=1.
			\end{array}
		\end{equation*}
		
		This finishes the proof of the lemma.
	\end{proof}

	\begin{lemma}\label{lemma5.47}
		Assume that $a=1$. Then, we have
		\begin{equation}\label{eq5.49}\begin{array}{ccl}
				&&D_{n,h}(a,b+1,c)\\
				+((-q)^{a-1}-1)&\times&D_{n,h}(a-1,b+2,c)\\
				-(-q)^{a-1}&\times&D_{n,h}(a,b,c+1)\\
				+(-q)^{2a+b-1}(1-(-q)^b)&\times&D_{n,h}(a,b-1,c+2)\\
				+(-q)^{2a+2b-1}(1-(-q)^{a-1})&\times&D_{n,h}(a-1,b,c+2)
			\end{array}
		\end{equation}
		\begin{equation*}
			\mathlarger{	\mathlarger{=}}	\quad\quad\begin{array}{cl}
				0 &	\text{ if }b \geq h+1 \text{ or } b \leq h-3,\\
				1 & \text{ if } b=h-2,\\
				0 & \text{ if } b=h-1,\\
				q^{2h+1} &\text { if } b=h.
			\end{array}
		\end{equation*}
	\end{lemma}
	\begin{proof}
		Note that when $a=1$, we have that \eqref{eq5.49} is
		\begin{equation*}
			D_{n,h}(1,b+1,c)-D_{n,h}(1,b,c+1)+(-q)^{b+1}(1-(-q)^b)D_{n,h}(1,b-1,c+2).
		\end{equation*}
		
		If $b\leq h-3$, we have that $D_{n,h}(1,b+1,c)=D_{n,h}(1,b,c+1)=D_{n,h}(1,b-1,c+2)=0$.
		
		If $b \geq h+1$, by Theorem \ref{theorem5.31}, we have
		\begin{equation*}
			\begin{array}{l}
				D_{n,h}(1,b+1,c)-D_{n,h}(1,b,c+1)+(-q)^{b+1}(1-(-q)^b)D_{n,h}(1,b-1,c+2)\\
				=\prod_{l=h+1}^{b+1}(1-(-q)^l)-\prod_{l=h+1}^{b}(1-(-q)^l)+(-q)^{b+1}(1-(-q)^b)\prod_{l=h+1}^{b-1}(1-(-q)^l)=0.
			\end{array}
		\end{equation*}
		If $b=h-2$, then we have
		\begin{equation*}
			\begin{array}{l}
				D_{n,h}(1,b+1,c)-D_{n,h}(1,b,c+1)+(-q)^{b+1}(1-(-q)^b)D_{n,h}(1,b-1,c+2)
				=1-0+0=1.
			\end{array}
		\end{equation*}
		If $b=h-1$, then we have
		\begin{equation*}
			\begin{array}{l}
				D_{n,h}(1,b+1,c)-D_{n,h}(1,b,c+1)+(-q)^{b+1}(1-(-q)^b)D_{n,h}(1,b-1,c+2)
				=1-1+0=0.
			\end{array}
		\end{equation*}
		If $b=h$, then we have
		\begin{equation*}
			\begin{array}{l}
				D_{n,h}(1,b+1,c)-D_{n,h}(1,b,c+1)+(-q)^{b+1}(1-(-q)^b)D_{n,h}(1,b-1,c+2)\\
				=(1-(-q)^{h+1})-1+(-q)^{h+1}(1-(-q)^h)=q^{2h+1}.
			\end{array}
		\end{equation*}
		This finishes the proof of the lemma.
	\end{proof}

	\begin{theorem}\label{theorem5.46}
		Assume that $x \perp L^{\flat}$, $\val(( x,x)) \leq -2$, and $L^{\flat} \subset L'^{\flat} \subset (L'^{\flat})^{\vee} \subset L^{\flat}_F$. Let $\lambda$ be the fundamental invariants of $L'^{\flat}$ and $(a,b,c)=(t_{\ge 2}(\lambda),t_1(\lambda),t_0(\lambda))$. Assume further that $(a,b,c) \neq (1,h,n-h-2), (1,h-2,n-h), (0,h+1,n-h-2), (0,h,n-h-1),(0,h-1,n-h)$. Then, we have
		\begin{equation*}
			\widehat{\pDen}_{L'^{\flat\circ}}^{n,h}(x)=0.
		\end{equation*} 
	\end{theorem}
	\begin{proof}
		Recall that
		\begin{equation*}
			\widehat{\pDen}_{L'^{\flat\circ}}^{n,h}(x)
			=\mathlarger{\sum}_{(u^{\flat},u^{\perp}) \in (L'^{\flat})^{\vee,\geq0}/L'^{\flat} \times (\langle x \rangle^{\vee}\backslash \lbrace 0\rbrace )/O_F^{\times}} D_{n,h}(L'^{\flat}+\langle u^{\flat}+u^{\perp}\rangle)\vol(L'^{\flat}+\langle u^{\flat}+u^{\perp}\rangle).
		\end{equation*}
		Since $\val(( x,x )) \leq -2$, we have that $\val(( u^{\perp},u^{\perp} )) \geq 2$. Therefore, by Proposition \ref{proposition5.33}, $D_{n,h}(L'^{\flat}+\langle u^{\flat}+u^{\perp}\rangle)$ depends only on $u^{\flat}$. Also, we have that
		\begin{equation*}
			\vol(L'^{\flat}+\langle u^{\flat}+u^{\perp}\rangle)=\vol(L'^{\flat})\vol(\langle u^{\perp} \rangle).
		\end{equation*}
		
		This implies that
		\begin{equation*}\begin{aligned}
				\widehat{\pDen}_{L'^{\flat\circ}}^{n,h}(x)
				&=\mathlarger{\sum}_{(u^{\flat},u^{\perp}) \in (L'^{\flat})^{\vee,\geq0}/L'^{\flat} \times (\langle x \rangle^{\vee}\backslash \lbrace 0\rbrace )/O_F^{\times}} D_{n,h}(L'^{\flat}+\langle u^{\flat}+u^{\perp}\rangle)\vol(L'^{\flat}+\langle u^{\flat}+u^{\perp}\rangle)\\
				&=\vol(L'^{\flat})\vol(\langle x \rangle^{\vee})\mathlarger{\sum}_{i \geq 0} q^{-2i}\mathlarger{\sum}_{u^{\flat} \in (L'^{\flat})^{\vee,\geq0}/L'^{\flat}} D_{n,h}(L'^{\flat}+\langle u^{\flat}+u^{\perp}\rangle)\\
				&=\vol(L'^{\flat})\vol(\langle x \rangle^{\vee})(1-q^{-2})^{-1}\mathlarger{\sum}_{u^{\flat} \in (L'^{\flat})^{\vee,\geq0}/L'^{\flat}} D_{n,h}(L'^{\flat}+\langle u^{\flat}+u^{\perp}\rangle).
			\end{aligned}
		\end{equation*}
		
		Now, it suffices to show that $\mathlarger{\sum}_{u^{\flat} \in (L'^{\flat})^{\vee,\geq0}/L'^{\flat}} D_{n,h}(L'^{\flat}+\langle u^{\flat}+u^{\perp}\rangle)=0$.
		
		We claim that $\mathlarger{\sum}_{u^{\flat} \in (L'^{\flat})^{\vee,\geq0}/L'^{\flat}} D_{n,h}(L'^{\flat}+\langle u^{\flat}+u^{\perp}\rangle)$ is a sum of two equations \eqref{eq5.34} and \eqref{eq5.42}. More precisely, we claim that
		\begin{equation}\label{eq5.50}\begin{aligned}
				\mathlarger{\sum}_{u^{\flat} \in (L'^{\flat})^{\vee,\geq0}/L'^{\flat}} D_{n,h}(L'^{\flat}+\langle u^{\flat}+u^{\perp}\rangle)&=\mu(L_{(\lambda\geq 2)-2})\times \eqref{eq5.34}\\&+\lbrace q^{2t_{\ge 3}(\lambda)}\mu(L_{(\lambda \geq3)-3})-(-q)^{t_{\ge 2}(\lambda)}\mu(L_{(\lambda\geq 2)-2}) \rbrace \times\eqref{eq5.42}.
			\end{aligned}
		\end{equation}
		
		Note that both sides of \eqref{eq5.50} are linear sums of the Cho-Yamauchi constants $D_{n,h}(a+1,b,c)$, $D_{n,h}(a,b+1,c)$, $D_{n,h}(a,b,c+1)$, $D_{n,h}(a+1,b-2,c+2)$, $D_{n,h}(a,b-1,c+2)$, $D_{n,h}(a-1,b+2,c)$, $D_{n,h}(a-1,b,c+2)$. Therefore, it suffices to show that both sides have the same coefficients.
		
		First, consider the coefficients of $D_{n,h}(a+1,b,c)$ on both sides. By Proposition \ref{proposition5.36.1} (Case 1-1), we have that the coefficient of $D_{n,h}(a+1,b,c)$ on the left-hand side of \eqref{eq5.50} is $\mu(L_{(\lambda \geq 2) -2}).$ Also, the coefficients of $D_{n,h}(a+1,b,c)$ on the right hand side of \eqref{eq5.50} is $\mu(L_{\lambda \geq 2 -2})$ (note that $D_{n,h}(a+1,b,c)$ appears only in \eqref{eq5.34}). Therefore, the coefficients of $D_{n,h}(a+1,b,c)$ on both sides are the same.
		
		For $D_{n,h}(a,b+1,c)$, by Proposition \ref{proposition5.36.1} (Case 1-2), we have that the coefficient of $D_{n,h}(a,b+1,c)$ on the left hand side of \eqref{eq5.50} is $q^{2t_{\ge 3}(\lambda)}\mu(L_{(\lambda \geq 3)-3})-\mu(L_{(\lambda \geq 2) -2}).$ Now, the coefficient of $D_{n,h}(a,b+1,c)$ on the right hand side of \eqref{eq5.50} is
		\begin{equation*}\begin{array}{l}
				\mu(L_{(\lambda \geq 2)-2}) \lbrace(-q)^{t_{\ge 2}(\lambda)}-1\rbrace+\lbrace q^{2t_{\ge 3}(\lambda)}\mu(L_{(\lambda \geq3)-3})-(-q)^{t_{\ge 2}(\lambda)}\mu(L_{(\lambda\geq 2)-2})\rbrace\\
				=q^{2t_{\ge 3}(\lambda)}\mu(L_{(\lambda \geq 3)-3})-\mu(L_{(\lambda \geq 2) -2}).
			\end{array}
		\end{equation*}
		Therefore, the coefficients of $D_{n,h}(a,b+1,c)$ on both sides are the same.
		
		For $D_{n,h}(a,b,c+1)$, by Proposition \ref{proposition5.36.1} (Case 1-3+Case 3-2), we have that the coefficient of $D_{n,h}(a,b,c+1)$ on the left hand side of \eqref{eq5.50} is $q^{2t_{\ge 2}(\lambda)}\mu(L_{(\lambda \geq 2)-2})-\mu(L_{(\lambda \geq 2) -1}).$ On the other hand, the coefficient of $D_{n,h}(a,b,c+1)$ on the right hand side of \eqref{eq5.50} is
		\begin{equation*}\begin{array}{l}
				\mu(L_{(\lambda \geq 2)-2}) (-q)^{t_{\ge 2}(\lambda)}\lbrace(-q)^{t_{\ge 2}(\lambda)}-1\rbrace-(-q)^{t_{\ge 2}(\lambda)-1}\lbrace q^{2t_{\ge 3}(\lambda)}\mu(L_{(\lambda \geq3)-3})-(-q)^{t_{\ge 2}(\lambda)}\mu(L_{(\lambda\geq 2)-2})\rbrace.
			\end{array}
		\end{equation*}
		
		Therefore, it suffices to show that
		\begin{equation*}
			\mu(L_{(\lambda \geq 2)-1})+(-q)^{2t_{\ge 2}(\lambda)-1}\mu(L_{(\lambda \geq 2)-2})=(-q)^{t_{\ge 2}(\lambda)}\mu(L_{(\lambda \geq 2)-2})+(-q)^{t_{\ge 2}(\lambda)+2t_{\ge 3}(\lambda)-1}\mu(L_{(\lambda\geq3)-3}).
		\end{equation*}
		By Lemma \ref{lemma5.40}, we have that
		\begin{equation*}
			\mu(L_{(\lambda \geq 2)-1})-q^{2t_{\ge 2}(\lambda)-1}\mu(L_{(\lambda \geq 2)-2})=-(q-1)(-q)^{\vert \lambda\vert-t_1(\lambda)-t_{\ge 2}(\lambda)-1},
		\end{equation*}
		and
		\begin{equation*}
			\begin{aligned}
				(-q)^{t_{\ge 2}(\lambda)}\mu(L_{(\lambda \geq 2)-2})+(-q)^{t_{\ge 2}(\lambda)+2t_{\ge 3}(\lambda)-1}\mu(L_{(\lambda\geq3)-3})&=(-q)^{t_{\ge 2}(\lambda)}\lbrace -(q-1)(-q)^{\vert \lambda \vert-t_1(\lambda)-2t_{\ge 2}(\lambda)-1}\rbrace\\
				&=-(q-1)(-q)^{\vert \lambda \vert-t_1(\lambda)-t_{\ge 2}(\lambda)-1}.
			\end{aligned}
		\end{equation*}
		This shows that the coefficients of $D_{n,h}(a,b,c+1)$ on both sides are the same.
		
		For $D_{n,h}(a+1,b-2,c+2)$, by Proposition \ref{proposition5.36.1} (Case 2-1+Case 4-1), we have that the coefficient of $D_{n,h}(a+1,b-2,c+2)$ on the left hand side of \eqref{eq5.50} is $q^{2t_{\ge 2}(\lambda)}\mu(L_{(\lambda \geq 2) -2})\times (\mu(L_{\lambda=1})-1).$ On the other hand, the coefficient of $D_{n,h}(a+1,b-2,c+2)$ on the right hand side of \eqref{eq5.50} is
		\begin{equation*}
			-(-q)^{2t_{\ge 2}(\lambda)}(1-(-q)^{t_1(\lambda)})(1-(-q)^{t_1(\lambda)-1})\mu(L_{(\lambda \geq 2) -2}).
		\end{equation*}
		Now, by Lemma \ref{lemma5.40}, we have that
		\begin{equation}\label{eq5.51}
			\mu(L_{\lambda=1})-1=q^{2t_1(\lambda)-1}-(-q)^{t_1(\lambda)-1}(q-1)-1=-(1-(-q)^{t_1(\lambda)})(1-(-q)^{t_1(\lambda)-1}).
		\end{equation}
		This shows that the coefficients of $D_{n,h}(a+1,b-2,c+2)$ on both sides are the same.
		
		For $D_{n,h}(a,b-1,c+2)$, by Proposition \ref{proposition5.36.1} (Case 2-2+Case 4-2), we have that the coefficient of $D_{n,h}(a,b-1,c+2)$ on the left hand side of \eqref{eq5.50} is
		\begin{equation*}
			q^{2t_{\ge 2}(\lambda)}\lbrace \mu(L_{(\lambda \geq 2)-2} \obot L_{\lambda=1})-\mu(L_{(\lambda \geq 2) -2})\mu(L_{\lambda=1})\rbrace.
		\end{equation*}
		On the other hand, the coefficient of $D_{n,h}(a,b-1,c+2)$ on the right hand side is
		\begin{equation*}\begin{array}{l}
				-(-q)^{2t_{\ge 2}(\lambda)+t_1(\lambda)-1}(1-(-q)^{t_1(\lambda)})(1-(-q)^{t_{\ge 2}(\lambda)}) \mu(L_{(\lambda \geq 2) -2})\\
				+(-q)^{2t_{\ge 2}(\lambda)+t_1(\lambda)-1}(1-(-q)^{t_1(\lambda)})\lbrace q^{2t_{\ge 3}(\lambda)}\mu(L_{(\lambda \geq 3)-3})-(-q)^{t_{\ge 2}(\lambda)}\mu(L_{(\lambda \geq 2)-2}) \rbrace.
			\end{array}
		\end{equation*}
		Note that by Lemma \ref{lemma5.40}, we have that
		\begin{equation*}\begin{array}{l}
				\mu(L_{(\lambda \geq 2)-2}\obot L_{\lambda=1})=\mu(L_{(\lambda \geq 3)-2}\obot L_{\lambda=1})=q^{2t_{\ge 3}(\lambda)+2t_1(\lambda)-1}\mu(L_{(\lambda \geq 3)-3})-(q-1)(-q)^{\vert \lambda \vert-2t_{\ge 2}(\lambda)-1}.\\
			\end{array}
		\end{equation*}
		Combining these with \eqref{eq5.51}, it suffices to show that
		\begin{equation*}
			-(q-1)(-q)^{\vert \lambda \vert-1}=(-q)^{2t_{\ge 2}(\lambda)+t_1(\lambda)}\mu(L_{(\lambda \geq 2)-2})+(-q)^{2t_{\ge 3}(\lambda)+2t_{\ge 2}(\lambda)+t_1(\lambda)-1}\mu(L_{(\lambda \geq 3) -3}).
		\end{equation*}
		Since $\mu(L_{(\lambda \geq 2)-2})=\mu(L_{(\lambda \geq 3)-2})$, this follows from Lemma \ref{lemma5.40}. This shows that the coefficients of $D_{n,h}(a,b-1,c+2)$ on both sides are the same.

		For $D_{n,h}(a-1,b+2,c)$, by Proposition \ref{proposition5.36.1} (Case 3-1), we have that the coefficient of $D_{n,h}(a-1,b+2,c)$ on the left hand side of \eqref{eq5.50} is
		\begin{equation*}
			\mu(L_{(\lambda \geq 2) -1})-q^{2t_{\ge 3}(\lambda)}\mu(L_{(\lambda \geq 3)-3}).
		\end{equation*}
		On the other hand, the coefficient of $D_{n,h}(a-1,b+2,c)$ on the right hand side of \eqref{eq5.50} is
		\begin{equation*}
			((-q)^{t_{\ge 2}(\lambda)-1}-1)\lbrace q^{2t_{\ge 3}(\lambda)}\mu(L_{(\lambda \geq 3)-3})-(-q)^{t_{\ge 2}(\lambda)}\mu(L_{(\lambda \geq 2)-2}) \rbrace.
		\end{equation*}
		Therefore, it suffices to show that
		\begin{equation*}
			\mu(L_{(\lambda \geq 2) -1})+(-q)^{2t_{\ge 2}(\lambda)-1}\mu(L_{(\lambda \geq 2)-2})= (-q)^{t_{\ge 2}(\lambda)}\mu(L_{(\lambda \geq 2)-2})+(-q)^{2t_{\ge 3}(\lambda)+t_{\ge 2}(\lambda)-1}\mu(L_{(\lambda \geq 3)-3}).
		\end{equation*}
		Now, by Lemma \ref{lemma5.40}, we have that
		\begin{equation*}
			\mu(L_{(\lambda \geq 2) -1})+(-q)^{2t_{\ge 2}(\lambda)-1}\mu(L_{(\lambda \geq 2)-2})=-(-q)^{\vert \lambda \vert-t_1(\lambda)-t_{\ge 2}(\lambda)-1}(q-1),
		\end{equation*}
		and
		\begin{equation*}
			(-q)^{t_{\ge 2}(\lambda)}\mu(L_{(\lambda \geq 2)-2})+(-q)^{2t_{\ge 3}(\lambda)+t_{\ge 2}(\lambda)-1}\mu(L_{(\lambda \geq 3)-3})=(-q)^{t_{\ge 2}(\lambda)}(-(-q)^{\vert \lambda \vert-t_1(\lambda)-2t_{\ge 2}(\lambda)-1})(q-1).
		\end{equation*}
		This shows that the coefficients of $D_{n,h}(a-1,b+2,c)$ on both sides are the same.
		
		For $D_{n,h}(a-1,b,c+2)$, by Proposition \ref{proposition5.36.1} (Case 5), we have that the coefficient of $D_{n,h}(a-1,b,c+2)$ on the left hand side of \eqref{eq5.50} is
		\begin{equation*}
			\mu(L_{\lambda \geq 2} \obot L_{\lambda=1})-q^{2t_{\ge 2}(\lambda)}\mu(L_{(\lambda \geq 2) -2}\obot L_{\lambda=1}). 
		\end{equation*}
		On the other hand, the coefficient of $D_{n,h}(a-1,b,c+2)$ on the right hand side of \eqref{eq5.50} is
		\begin{equation*}
			(-q)^{2t_{\ge 2}(\lambda)+2t_1(\lambda)-1}(1-(-q)^{t_{\ge 2}(\lambda)-1})\lbrace q^{2t_{\ge 3}(\lambda)}\mu(L_{(\lambda \geq 3)-3})-(-q)^{t_{\ge 2}(\lambda)}\mu(L_{(\lambda \geq 2)-2}) \rbrace.
		\end{equation*}
		Note that by Lemma \ref{lemma5.40}, we have
		\begin{equation*}
			\mu(L_{\lambda \geq 2} \obot L_{\lambda=1})=q^{2t_{\ge 2}(\lambda)+2t_1(\lambda)-1}\mu(L_{(\lambda \geq 2)-1})-(q-1)(-q)^{\vert \lambda \vert-1},
		\end{equation*}
		and
		\begin{equation*}\begin{array}{l}
				\mu(L_{(\lambda \geq 2)-2}\obot L_{\lambda=1})=\mu(L_{(\lambda \geq 3)-2}\obot L_{\lambda=1})=q^{2t_{\ge 3}(\lambda)+2t_1(\lambda)-1}\mu(L_{(\lambda \geq 3)-3})-(q-1)(-q)^{\vert \lambda \vert-2t_{\ge 2}(\lambda)-1}.\\
			\end{array}
		\end{equation*}
		Therefore, it suffices to show that
		\begin{equation*}
			\begin{array}{l}
				q^{2t_{\ge 2}(\lambda)+2t_1(\lambda)-1}\mu(L_{(\lambda \geq 2)-1})-q^{2t_{\ge 3}(\lambda)+2t_{\ge 2}(\lambda)+2t_1(\lambda)-1}\mu(L_{(\lambda \geq 3)-3})\\
				=(-q)^{2t_{\ge 2}(\lambda)+2t_1(\lambda)-1}(1-(-q)^{t_{\ge 2}(\lambda)-1})\lbrace q^{2t_{\ge 3}(\lambda)}\mu(L_{(\lambda \geq 3)-3})-(-q)^{t_{\ge 2}(\lambda)}\mu(L_{(\lambda \geq 2)-2}) \rbrace.
			\end{array}
		\end{equation*}
		Note that this is equivalent to
		\begin{equation*}\begin{aligned}
				\mu(L_{(\lambda \geq 2)-1})-q^{2t_{\ge 2}(\lambda)-1}\mu(L_{(\lambda \geq 2)-2})&=(-q)^{t_{\ge 2}(\lambda)}\lbrace \mu(L_{(\lambda \geq 2)-2})-q^{2t_{\ge 3}(\lambda)-1}\mu(L_{(\lambda \geq 3) -3})\rbrace\\
				&=(-q)^{t_{\ge 2}(\lambda)}\lbrace \mu(L_{(\lambda \geq 3)-2})-q^{2t_{\ge 3}(\lambda)-1}\mu(L_{(\lambda \geq 3) -3})\rbrace.
			\end{aligned}
		\end{equation*}
		Now, by Lemma \ref{lemma5.40}, we have that both sides are equal to $-(-q)^{\vert \lambda \vert-t_1(\lambda)-t_{\ge 2}(\lambda)-1}(q-1)$. This shows that the coefficients of $D_{n,h}(a-1,b,c+2)$ on both sides are the same.
		
		This finishes the proof of \eqref{eq5.50}. Now, the theorem follows from Lemma \ref{lemma5.44}, Lemma \ref{lemma5.45}, Lemma \ref{lemma5.46}, and Lemma \ref{lemma5.47}.

	\end{proof}
	
	\begin{theorem}\label{theorem5.47}
		Assume that $x \perp L^{\flat}$, $\val(( x,x ))=-1$, and $L^{\flat} \subset L'^{\flat} \subset (L'^{\flat})^{\vee} \subset L^{\flat}_F$. Let $\lambda$ be the fundamental invariants of $L'^{\flat}$ and $(a,b,c)=(t_{\ge 2}(\lambda),t_1(\lambda),t_0(\lambda))$. Assume further that $(a,b,c) \neq (1,h,n-h-2), (1,h-2,n-h), (0,h+1,n-h-2), (0,h,n-h-1),(0,h-1,n-h)$. Then, we have
		\begin{equation*}
			\widehat{\pDen}_{L'^{\flat\circ}}^{n,h}(x)=-\dfrac{1}{q^h}D_{n-1,h-1}(a,b,c).
		\end{equation*} 
	\end{theorem}
	\begin{proof}
		Recall that
		\begin{equation*}
			\widehat{\pDen}_{L'^{\flat\circ}}^{n,h}(x)
			=\mathlarger{\sum}_{(u^{\flat},u^{\perp}) \in (L'^{\flat})^{\vee,\geq0}/L'^{\flat} \times (\langle x \rangle^{\vee}\backslash \lbrace 0\rbrace )/O_F^{\times}} D_{n,h}(L'^{\flat}+\langle u^{\flat}+u^{\perp}\rangle)\vol(L'^{\flat}+\langle u^{\flat}+u^{\perp}\rangle).
		\end{equation*}
		Since $\val(( x,x )) =-1$, we have that $\val(( u^{\perp},u^{\perp})) \geq 1$. In the proof of Theorem \ref{theorem5.46}, we proved that
		\begin{equation*}
			\mathlarger{\sum}_{(u^{\flat},u^{\perp}) \in (L'^{\flat})^{\vee,\geq0}/L'^{\flat} \times (\langle x \rangle^{\vee}\backslash \lbrace 0\rbrace )/O_F^{\times}, \val(( u^{\perp},u^{\perp}))\geq 2} D_{n,h}(L'^{\flat}+\langle u^{\flat}+u^{\perp}\rangle)\vol(L'^{\flat}+\langle u^{\flat}+u^{\perp}\rangle)=0.
		\end{equation*}
		
		Therefore, we have that
		\begin{equation*}
			\widehat{\pDen}_{L'^{\flat\circ}}^{n,h}(x)=\mathlarger{\sum}_{\substack{(u^{\flat},u^{\perp}) \in (L'^{\flat})^{\vee,\geq0}/L'^{\flat} \times (\langle x \rangle^{\vee}\backslash \lbrace 0\rbrace )/O_F^{\times},\\ \val(( u^{\perp},u^{\perp})=1}} D_{n,h}(L'^{\flat}+\langle u^{\flat}+u^{\perp}\rangle)\vol(L'^{\flat}+\langle u^{\flat}+u^{\perp}\rangle).
		\end{equation*}
		Therefore, by Proposition \ref{proposition5.34}, $D_{n,h}(L'^{\flat}+\langle u^{\flat}+u^{\perp}\rangle)$ depends only on $u^{\flat}$.
		Also, we have that
		\begin{equation*}
			\vol(L'^{\flat}+\langle u^{\flat}+u^{\perp}\rangle)=\vol(L'^{\flat})\vol(\langle u^{\perp} \rangle).
		\end{equation*}
		
		This implies that
		\begin{equation*}\begin{aligned}
				\widehat{\pDen}_{L'^{\flat\circ}}^{n,h}(x)
				&=\mathlarger{\sum}_{\substack{(u^{\flat},u^{\perp}) \in (L'^{\flat})^{\vee,\geq0}/L'^{\flat} \times (\langle x \rangle^{\vee}\backslash \lbrace 0\rbrace )/O_F^{\times},\\ \val(( u^{\perp},u^{\perp}))=1}} D_{n,h}(L'^{\flat}+\langle u^{\flat}+u^{\perp}\rangle)\vol(L'^{\flat}+\langle u^{\flat}+u^{\perp}\rangle)\\
				&=\vol(L'^{\flat})(q)^{-1}\mathlarger{\sum}_{u^{\flat} \in (L'^{\flat})^{\vee,\geq0}/L'^{\flat}} D_{n,h}(L'^{\flat}+\langle u^{\flat}+u^{\perp}\rangle).
			\end{aligned}
		\end{equation*}
		
		Note that Proposition \ref{proposition5.33} and Proposition \ref{proposition5.34} are different only when (Case 1-1) and (Case 1-2). Also, it is easy to see that
		\begin{equation*}\begin{array}{l}
				\text{(Case 1-2-1)}\\
				\vert \lbrace u^{\flat} \in (\pi^2(L_{2}'^{\flat})^{\vee}\obot \pi (L_{1}'^{\flat})^{\vee})^{\geq 1}-(\pi^2(L_{2}'^{\flat})^{\vee}\obot \pi (L_{1}'^{\flat})^{\vee})^{\geq 2}\mid \val(( u^{\flat}+u^{\perp},u^{\flat}+u^{\perp} )) =1 \rbrace \vert\\
				=\dfrac{q-2}{q-1}\text{(Case 1-2)}=\dfrac{q-2}{q-1}(q^{2t_{\ge 3}(\lambda)}\mu(L_{(\lambda\geq3)-3})-\mu(L_{(\lambda \geq2)-2})),
			\end{array}
		\end{equation*}
		\begin{equation*}\begin{array}{l}
				\text{(Case 1-2-2)}\\
				\vert \lbrace u^{\flat} \in (\pi^2(L_{2}'^{\flat})^{\vee}\obot \pi (L_{1}'^{\flat})^{\vee})^{\geq 1}-(\pi^2(L_{2}'^{\flat})^{\vee}\obot \pi (L_{1}'^{\flat})^{\vee})^{\geq 2} \mid \val(( u^{\flat}+u^{\perp},u^{\flat}+u^{\perp} )) \geq 2 \rbrace \vert\\
				=\dfrac{1}{q-1}\text{(Case 1-2)}=\dfrac{1}{q-1}(q^{2t_{\ge 3}(\lambda)}\mu(L_{(\lambda\geq3)-3})-\mu(L_{(\lambda \geq2)-2})).
			\end{array}
		\end{equation*}
		
		Now, note that if $\val(( u^{\perp},u^{\perp} )) \geq 2$, then $\mathlarger{\sum}_{u^{\flat} \in (L'^{\flat})^{\vee,\geq0}/L'^{\flat}} D_{n,h}(L'^{\flat}+\langle u^{\flat}+u^{\perp}\rangle)=0$. Therefore, for $\val((u^{\perp},u^{\perp}))=1$, we have
		\begin{equation*}\begin{array}{l}
				\widehat{\pDen}_{L'^{\flat\circ}}^{n,h}(x)=\dfrac{1}{q}\vol(L'^{\flat})\mathlarger{\sum}_{u^{\flat} \in (L'^{\flat})^{\vee,\geq0}/L'^{\flat}} D_{n,h}(L'^{\flat}+\langle u^{\flat}+u^{\perp}\rangle)\\
				=\dfrac{1}{q}\vol(L'^{\flat})
				\lbrace (D_{n,h}(a,b+1,c)-D_{n,h}(a+1,b,c))\mu(L_{(\lambda \geq 2)-2})\\
				\qquad\qquad+(D_{n,h}(a+1,b,c)-D_{n,h}(a,b+1,c))\dfrac{1}{q-1}(q^{2t_{\ge 3}(\lambda)}\mu(L_{(\lambda\geq3)-3})-\mu(L_{(\lambda \geq2)-2}))\rbrace\\
				=\dfrac{1}{q}\vol(L'^{\flat})(D_{n,h}(a+1,b,c)-D_{n,h}(a,b+1,c))\lbrace\dfrac{q^{2t_{\ge 3}(\lambda)}}{q-1}\mu(L_{(\lambda \geq 3)-3})-\dfrac{q}{q-1}\mu(L_{(\lambda \geq2)-2}) \rbrace.
			\end{array}
		\end{equation*}
		
		By Lemma \ref{lemma5.40}, we have
		\begin{equation*}
			\dfrac{q^{2t_{\ge 3}(\lambda)}}{q-1}\mu(L_{(\lambda \geq 3)-3})-\dfrac{q}{q-1}\mu(L_{(\lambda \geq2)-2})=-(-q)^{\vert \lambda \vert-b-2a}.
		\end{equation*}
		
		Also, by Theorem \ref{theorem5.24}, we have
		\begin{equation*}
			D_{n,h}(a+1,b,c)-D_{n,h}(a,b+1,c)=-(-q)^{2n-h-1-b-2c}D_{n-1,h-1}(a,b,c).
		\end{equation*}
		
		Finally, we have that $\vol(L'^{\flat})=q^{-\vert \lambda \vert}$. Therefore, we have
		\begin{equation*}
			\widehat{\pDen}_{L'^{\flat\circ}}^{n,h}(x)=\dfrac{(-q)^{\vert \lambda\vert+2n-2a-2b-2c-h-1}}{q^{\vert\lambda\vert+1}}D_{n-1,h-1}(a,b,c).
		\end{equation*}
		Now, note that $a+b+c=n-1$ and $\val(L'^{\flat} \obot \langle x \rangle)=\vert \lambda \vert-1 \equiv h+1 \pmod 2.$ This implies that
		\begin{equation*}
			\widehat{\pDen}_{L'^{\flat\circ}}^{n,h}(x)=-\dfrac{1}{q^h}D_{n-1,h-1}(a,b,c).			
		\end{equation*}
		This finishes the proof of the theorem.
	\end{proof}

	\begin{theorem}\label{theorem8.18}
		Assume that $x \perp L^{\flat}$, $\val(\langle x,x \rangle) \leq -2$, and $L^{\flat} \subset L'^{\flat} \subset (L'^{\flat})^{\vee} \subset L^{\flat}_F$. Let $\lambda$ be the fundamental invariants of $L'^{\flat}$ and $(a,b,c)=(t_{\ge 2}(\lambda),t_1(\lambda),t_0(\lambda))$. Then, we have

		\begin{equation*}\begin{array}{l}
				\widehat{\pDen}_{L'^{\flat\circ}}^{n,h}(x)
				=\left\lbrace \begin{array}{ll}
					q^{h-1}(q+1)\vol(\langle x \rangle^{\vee})(1-q^{-2})^{-1} & \text{ if }(a,b,c)=(1,h,n-h-2),\\
					q^{-h}(q+1)\vol(\langle x \rangle^{\vee})(1-q^{-2})^{-1} & \text{ if }(a,b,c)=(1,h-2,n-h),\\
					q^{-(h-1)}\vol(\langle x \rangle^{\vee})(1-q^{-2})^{-1} & \text{ if }(a,b,c)=(0,h-1,n-h),\\
					q^{-h}\vol(\langle x \rangle^{\vee})(1-q^{-2})^{-1} & \text{ if }(a,b,c)=(0,h,n-h-1),\\
					q^{-(h+1)}(q^{2h+1}+(-q)^h)\vol(\langle x \rangle^{\vee})(1-q^{-2})^{-1} & \text{ if }(a,b,c)=(0,h+1,n-h-2).
				\end{array}\right.
			\end{array}
		\end{equation*} 
	\end{theorem}
	\begin{proof}
		In the proof of Theorem \ref{theorem5.46}, we proved that
		\begin{equation*}\begin{array}{ll}
				\widehat{\pDen}_{L'^{\flat\circ}}^{n,h}(x)
				=\vol(L'^{\flat})\vol(\langle x \rangle^{\vee})(1-q^{-2})^{-1}\mathlarger{\sum}_{u^{\flat} \in (L'^{\flat})^{\vee,\geq0}/L'^{\flat}} D_{n,h}(L'^{\flat}+\langle u^{\flat}+u^{\perp}\rangle).
				
			\end{array}
		\end{equation*}
		
		Also, in \eqref{eq5.50}, we showed that
		\begin{equation*}\begin{aligned}
				\mathlarger{\sum}_{u^{\flat} \in (L'^{\flat})^{\vee,\geq0}/L'^{\flat}} D_{n,h}(L'^{\flat}+\langle u^{\flat}+u^{\perp}\rangle)&=\mu(L_{(\lambda\geq 2)-2})\times \eqref{eq5.34}\\&+\lbrace q^{2t_{\ge 3}(\lambda)}\mu(L_{(\lambda \geq3)-3})-(-q)^{t_{\ge 2}(\lambda)}\mu(L_{(\lambda\geq 2)-2}) \rbrace \times\eqref{eq5.42}.
			\end{aligned}
		\end{equation*}
		
		First, assume that $(a,b,c)=(1,h,n-h-2)$, $\lambda=(\alpha,\overset{h}{\overbrace{1,\dots,1}},\overset{n-h-2}{\overbrace{0,\dots,0}})$. In this case, we know that \eqref{eq5.34}$=0$. Also, it is easy to see that $\mu(L_{\alpha})=q^{2[\frac{\alpha}{2}]}$, and hence
		\begin{equation*}\begin{array}{l}
				q^{2t_{\ge 3}(\lambda)}\mu(L_{(\lambda \geq3)-3})-(-q)^{t_{\ge 2}(\lambda)}\mu(L_{(\lambda\geq 2)-2})=q^{2t_{\ge 3}(\alpha)}\mu(L_{\alpha-3})+q\mu(L_{\alpha-2})\\
				=\left\lbrace \begin{array}{ll}
					1+q & \text{ if }\alpha=2,\\
					q^{2+2[\frac{\alpha-3}{2}]}+q^{2[\frac{\alpha-2}{2}]+1}=q^{\alpha-2}(q+1) & \text{ if } \alpha \geq 3,
				\end{array}\right.\\
				=q^{\alpha-2}(q+1).
			\end{array}
		\end{equation*}
		Now, by Lemma \ref{lemma5.47}, we have that
		\begin{equation*}\begin{array}{l}
				\mathlarger{\sum}_{u^{\flat} \in (L'^{\flat})^{\vee,\geq0}/L'^{\flat}} D_{n,h}(L'^{\flat}+\langle u^{\flat}+u^{\perp}\rangle)=q^{\alpha-2}(q+1)q^{2h+1}.
			\end{array}
		\end{equation*}
		
		Similarly, for $(a,b,c)=(1,h-2,n-h)$, $\lambda=(\alpha,\overset{h-2}{\overbrace{1,\dots,1}},\overset{n-h}{\overbrace{0,\dots,0}})$, we have that
		\begin{equation*}\begin{array}{l}
				\mathlarger{\sum}_{u^{\flat} \in (L'^{\flat})^{\vee,\geq0}/L'^{\flat}} D_{n,h}(L'^{\flat}+\langle u^{\flat}+u^{\perp}\rangle)=q^{\alpha-2}(q+1).
			\end{array}
		\end{equation*}
		
		Therefore, we have that
		\begin{equation*}\begin{aligned}
				\widehat{\pDen}_{L'^{\flat\circ}}^{n,h}(x)
				&=\vol(L'^{\flat})\vol(\langle x \rangle^{\vee})(1-q^{-2})^{-1}\mathlarger{\sum}_{u^{\flat} \in (L'^{\flat})^{\vee,\geq0}/L'^{\flat}} D_{n,h}(L'^{\flat}+\langle u^{\flat}+u^{\perp}\rangle)\\
				&=\left\lbrace \begin{array}{ll}
					q^{h-1}(q+1)\vol(\langle x \rangle^{\vee})(1-q^{-2})^{-1} & \text{ if }(a,b,c)=(1,h,n-h-2),\\
					q^{-h}(q+1)\vol(\langle x \rangle^{\vee})(1-q^{-2})^{-1} & \text{ if }(a,b,c)=(1,h-2,n-h).
				\end{array}\right.
			\end{aligned}
		\end{equation*}
		
		Now, assume that $a=0$. Then, we have that
		\begin{equation*}
			\begin{array}{l}
				\mu(L_{(\lambda \geq2)-2})=1,\\
				q^{2t_{\ge 3}(\lambda)}(L_{(\lambda \geq 3)-3})-(-q)^{t_{\ge 2}(\lambda)}\mu(L_{(\lambda \geq 2)-2})=0.
			\end{array}
		\end{equation*}
		Therefore, by Lemma \ref{lemma5.46}, we have
		\begin{equation*}\begin{array}{l}
				\mathlarger{\sum}_{u^{\flat} \in (L'^{\flat})^{\vee,\geq0}/L'^{\flat}} D_{n,h}(L'^{\flat}+\langle u^{\flat}+u^{\perp}\rangle)=\left\lbrace\begin{array}{ll}
					1 & \text{ if }b=h-1,h,\\
					q^{2h+1}+(-q)^h & \text{ if } b=h+1.
				\end{array}\right.
			\end{array}
		\end{equation*}
		This implies that
		\begin{equation*}\begin{aligned}
				\widehat{\pDen}_{L'^{\flat\circ}}^{n,h}(x)
				&=\vol(L'^{\flat})\vol(\langle x \rangle^{\vee})(1-q^{-2})^{-1}\mathlarger{\sum}_{u^{\flat} \in (L'^{\flat})^{\vee,\geq0}/L'^{\flat}} D_{n,h}(L'^{\flat}+\langle u^{\flat}+u^{\perp}\rangle)\\
				&=\left\lbrace \begin{array}{ll}
					q^{-(h-1)}\vol(\langle x \rangle^{\vee})(1-q^{-2})^{-1} & \text{ if }(a,b,c)=(0,h-1,n-h),\\
					q^{-h}\vol(\langle x \rangle^{\vee})(1-q^{-2})^{-1} & \text{ if }(a,b,c)=(0,h,n-h-1),\\
					q^{-(h+1)}(q^{2h+1}+(-q)^h)\vol(\langle x \rangle^{\vee})(1-q^{-2})^{-1} & \text{ if }(a,b,c)=(0,h+1,n-h-2).
				\end{array}\right.
			\end{aligned}
		\end{equation*}
		This finishes the proof of the theorem.
		
	\end{proof}
	
	\begin{theorem}\label{theorem8.19}
		Assume that $x \perp L^{\flat}$, $\val(( x,x ))=-1$, and $L^{\flat} \subset L'^{\flat} \subset (L'^{\flat})^{\vee} \subset L^{\flat}_F$. Let $\lambda$ be the fundamental invariants of $L'^{\flat}$ and $(a,b,c)=(t_{\ge 2}(\lambda),t_1(\lambda),t_0(\lambda))$. Then, we have
		\begin{equation*}\begin{array}{l}
				\widehat{\pDen}_{L'^{\flat\circ}}^{n,h}(x)\\
				=\left\lbrace \begin{array}{ll}
					q^{h-2}(q+1)(1-q^{-2})^{-1} -q^{-h}D_{n-1,h-1}(a,b,c)& \text{ if }(a,b,c)=(1,h,n-h-2),\\
					q^{-h-1}(q+1)(1-q^{-2})^{-1} -q^{-h}D_{n-1,h-1}(a,b,c)& \text{ if }(a,b,c)=(1,h-2,n-h),\\
					q^{-h-1}(1-q^{-2})^{-1} -q^{-h}D_{n-1,h-1}(a,b,c)& \text{ if }(a,b,c)=(0,h,n-h-1).
				\end{array}\right.
			\end{array}
		\end{equation*}
		Note that two cases $(a,b,c)=(0,h-1,n-h), (0,h+1,n-h-2)$ are not possible  since $\val(( x,x))=-1$ and hence $\val(L^{\flat})\equiv h \pmod 2$.
	\end{theorem}
	\begin{proof}
		Recall that
		\begin{equation}\label{eq10.23.1}\begin{aligned}
				\widehat{\pDen}_{L'^{\flat\circ}}^{n,h}(x)
				&=\mathlarger{\sum}_{(u^{\flat},u^{\perp}) \in (L'^{\flat})^{\vee,\geq0}/L'^{\flat} \times (\langle x \rangle^{\vee}\backslash \lbrace 0\rbrace )/O_F^{\times}} D_{n,h}(L'^{\flat}+\langle u^{\flat}+u^{\perp}\rangle)\vol(L'^{\flat}+\langle u^{\flat}+u^{\perp}\rangle)\\
				&=\vol(L'^{\flat})\mathlarger{\sum}_{\substack{(u^{\flat},u^{\perp}) \in (L'^{\flat})^{\vee,\geq0}/L'^{\flat} \times (\langle x \rangle^{\vee}\backslash \lbrace 0\rbrace )/O_F^{\times}\\ \val(( u^{\perp},u^{\perp})=1}} D_{n,h}(L'^{\flat}+\langle u^{\flat}+u^{\perp}\rangle)\vol(\langle u^{\perp}\rangle)\\
				&+\vol(L'^{\flat})\mathlarger{\sum}_{\substack{(u^{\flat},u^{\perp}) \in (L'^{\flat})^{\vee,\geq0}/L'^{\flat} \times (\langle x \rangle^{\vee}\backslash \lbrace 0\rbrace )/O_F^{\times}\\ \val(( u^{\perp},u^{\perp}) \geq 3}} D_{n,h}(L'^{\flat}+\langle u^{\flat}+u^{\perp}\rangle)\vol(\langle u^{\perp}\rangle).
			\end{aligned}
		\end{equation}
		
		In the proof of the Theorem \ref{theorem8.18}, we showed that
		\begin{equation*}\begin{array}{l}
				\vol(L'^{\flat})\mathlarger{\sum}_{\substack{(u^{\flat},u^{\perp}) \in (L'^{\flat})^{\vee,\geq0}/L'^{\flat} \times (\langle x \rangle^{\vee}\backslash \lbrace 0\rbrace )/O_F^{\times}\\ \val(( u^{\perp},u^{\perp}) \geq 3}} D_{n,h}(L'^{\flat}+\langle u^{\flat}+u^{\perp}\rangle)\vol(\langle u^{\perp}\rangle)\\
				=\vol(L'^{\flat})\dfrac{1}{q^3}(1-q^{-2})^{-1}\mathlarger{\sum}_{\substack{u^{\flat} \in (L'^{\flat})^{\vee,\geq0}/L'^{\flat} \\ \val(( u^{\perp},u^{\perp})\geq 3}} D_{n,h}(L'^{\flat}+\langle u^{\flat}+u^{\perp}\rangle).
			\end{array}
		\end{equation*}
		Also, we already computed the sum $\mathlarger{\sum}_{\substack{u^{\flat} \in (L'^{\flat})^{\vee,\geq0}/L'^{\flat} \\ \val(( u^{\perp},u^{\perp}) \geq 3}} D_{n,h}(L'^{\flat}+\langle u^{\flat}+u^{\perp}\rangle)$ in the proof of the Theorem \ref{theorem8.18} (more precisely, we computed this for $\val(( u^{\perp},u^{\perp} )) \geq 2$). Therefore, we only need to compute 
		\begin{equation*}
			\vol(L'^{\flat})\mathlarger{\sum}_{\substack{(u^{\flat},u^{\perp}) \in (L'^{\flat})^{\vee,\geq0}/L'^{\flat} \times (\langle x \rangle^{\vee}\backslash \lbrace 0\rbrace )/O_F^{\times}\\ \val(( u^{\perp},u^{\perp}) =1}} D_{n,h}(L'^{\flat}+\langle u^{\flat}+u^{\perp}\rangle)\vol(\langle u^{\perp}\rangle).
		\end{equation*}
		
		Note that Proposition \ref{proposition5.33} and Proposition \ref{proposition5.34} are different only when (Case 1-1) and (Case 1-2). Therefore, as in the proof of Theorem \ref{theorem5.47}, we have that
		\begin{equation*}\begin{array}{l}
				\mathlarger{\sum}_{\substack{u^{\flat} \in (L'^{\flat})^{\vee,\geq0}/L'^{\flat} \\ \val(( u^{\perp},u^{\perp}) =1}} D_{n,h}(L'^{\flat}+\langle u^{\flat}+u^{\perp}\rangle)-\mathlarger{\sum}_{\substack{u^{\flat} \in (L'^{\flat})^{\vee,\geq0}/L'^{\flat} \\ \val(( u^{\perp},u^{\perp}) \geq 2}} D_{n,h}(L'^{\flat}+\langle u^{\flat}+u^{\perp}\rangle)\\
				=(D_{n,h}(a+1,b,c)-D_{n,h}(a,b+1,c))\lbrace\dfrac{q^{2t_{\ge 3}(\lambda)}}{q-1}\mu(L_{(\lambda \geq 3)-3})-\dfrac{q}{q-1}\mu(L_{(\lambda \geq2)-2}) \rbrace\\
				=(-(-q)^{2n-h-1-b-2c}D_{n-1,h-1}(a,b,c))(-(-q)^{\vert \lambda \vert-b-2a})\\
				=(-q)^{\vert \lambda \vert-h+1}D_{n-1,h-1}(a,b,c).
			\end{array}
		\end{equation*}
		
		Since $\vol(L'^{\flat})=q^{-\vert \lambda \vert}$ and $\val(L'^{\flat} \obot \langle x \rangle)=\vert \lambda \vert-1 \equiv h+1 \pmod 2$, we have that
		\begin{equation*}
			\begin{array}{l}
				\vol(L'^{\flat})\mathlarger{\sum}_{\substack{(u^{\flat},u^{\perp}) \in (L'^{\flat})^{\vee,\geq0}/L'^{\flat} \times (\langle x \rangle^{\vee}\backslash \lbrace 0\rbrace )/O_F^{\times}\\ \val(\langle u^{\perp},u^{\perp}\rangle =1}} D_{n,h}(L'^{\flat}+\langle u^{\flat}+u^{\perp}\rangle)\vol(\langle u^{\perp}\rangle)\\
				=\dfrac{(-q)^{\vert\lambda \vert-h+1}}{q^{\vert\lambda \vert+1}}D_{n-1,h-1}(a,b,c)+\dfrac{\vol(L'^{\flat})}{q}\mathlarger{\sum}_{\substack{u^{\flat} \in (L'^{\flat})^{\vee,\geq0}/L'^{\flat} \\ \val(( u^{\perp},u^{\perp}) \geq 2}} D_{n,h}(L'^{\flat}+\langle u^{\flat}+u^{\perp}\rangle)\\
				=-\dfrac{1}{q^h}D_{n-1,h-1}(a,b,c)+\dfrac{\vol(L'^{\flat})}{q}\mathlarger{\sum}_{\substack{u^{\flat} \in (L'^{\flat})^{\vee,\geq0}/L'^{\flat} \\ \val(( u^{\perp},u^{\perp}) \geq 2}} D_{n,h}(L'^{\flat}+\langle u^{\flat}+u^{\perp}\rangle).
				\					\end{array}
		\end{equation*}
		Now, the theorem follows from \eqref{eq10.23.1} and the fact that (see the proof of Theorem \ref{theorem8.18})
		\begin{equation*}
			\begin{array}{ll}
				\mathlarger{\sum}_{\substack{u^{\flat} \in (L'^{\flat})^{\vee,\geq0}/L'^{\flat} \\ \val(( u^{\perp},u^{\perp}) \geq 2}} D_{n,h}(L'^{\flat}+\langle u^{\flat}+u^{\perp}\rangle)=\left\lbrace \begin{array}{ll}
					q^{h-1}(q+1)q^{\vert \lambda \vert} & \text{ if }(a,b,c)=(1,h,n-h-2),\\
					q^{-h}(q+1)q^{\vert \lambda \vert} & \text{ if }(a,b,c)=(1,h-2,n-h-1),\\
					1 & \text{ if }(a,b,c)=(0,h,n-h).
				\end{array}\right.
			\end{array}
		\end{equation*}
	\end{proof}
	
	We will use the following lemma in the next section when we count the number of horizontal components.
	
	\begin{lemma}\label{lemma5.48}
		Assume that $\lambda \geq 2$. Consider $(\lambda,0^{n-1}), (\lambda,2^{n-1}) \in \CR_n^{0+}$ and $(0^{n-1}),(2^{n-1}) \in \CR_{n-1}^{0+}$. Then, we have
		\begin{equation*}
			\begin{array}{l}
				\dfrac{\alphad(A_{(\lambda,0^{n-1})},A_{(\lambda,2^{n-1})})/\alphad(A_{(\lambda,0^{n-1})},A_{(\lambda,0^{n-1})})}{\alphad(A_{(0^{n-1})},A_{(2^{n-1})})/\alphad(A_{(0^{n-1})},A_{(0^{n-1})})}=\left\lbrace \begin{array}{ll}
					q^{2n-2} & \text{ if } \lambda \geq 3,\\
					q^{2n-2}\dfrac{1-(-q)^{-n}}{1-(-q)^{-1}} &\text{ if } \lambda=2.
				\end{array}\right.
			\end{array}
		\end{equation*}
	\end{lemma}
	\begin{proof}
		We use \cite[Theorem 2.5]{Cho4} to prove this. To simplify notation, we use the following convention: For $1 \leq k \leq n$, $B_1 \in X_n(O_F)$, and $B_2 \in X_{n-1}(O_F)$, we define
		\begin{equation*}
			\begin{array}{lll}
				R_0^k=\alphad(A_{(\lambda,1^{k-1},0^{n-k})},B_1),& 	R_1^k=\alphad(A_{(\lambda,2^{k-1},1^{n-k})},B_1), &	R_2^k=\alphad(A_{(\lambda,3^{k-1},2^{n-k})},B_1), \\
				R_3^k=\alphad(A_{(1^{k-1},0^{n-k})},B_2),& 	R_4^k=\alphad(A_{(2^{k-1},1^{n-k})},B_2), &	R_5^k=\alphad(A_{(3^{k-1},2^{n-k})},B_2),
			\end{array}
		\end{equation*}
		and for a polynomial $f(X)=\sum_{i=1}^{n}a_i X^i$, we denote by $f(R_i)$ the sum $f(R_i)\coloneqq \sum_{i=1}^{n}a_i R_i^k$. For example, if $f(X)=X+X^2$, then $f(R_1)=\alphad(A_{(\lambda,0^{n-1})},B_1)+\alphad(A_{(\lambda,1^{1},0^{n-2})},B_1)$ (not $\alphad(A_{(\lambda,0^{n-1})},B_1)+\lbrace \alphad(A_{(\lambda,0^{n-1})},B_1)\rbrace^2$).
		
		Also, we define the following polynomials
		\begin{equation*}\begin{array}{ll}
				f_{1,k}(X)=X^k\prod_{l=k+1}^{n}(1-(-q)^{-l}X),&	f_{2,k}(X)=X\prod_{l=k+1}^{n}(1-(-q)^{-l}X),\\
				f_{3,k}(X)=X^k\prod_{l=k}^{n-1}(1-(-q)^{-l}X),&	f_{4,k}(X)=X\prod_{l=k}^{n-1}(1-(-q)^{-l}X).
			\end{array}
		\end{equation*}
		Finally, for $0 \leq j \leq i$, we define the following constants $k_{i,j}$:
		\begin{equation*}
			\prod_{l=1}^{i}(1+(-q)^{-l}X)=\sum_{j=0}^{i}k_{i,j}X^j.
		\end{equation*}
		
		First, we claim that
		\begin{equation}\label{eq5.52}
			f_{2,i+1}(X)=\sum_{j=0}^{i}k_{i,j}(-q)^{-\frac{j(j+1)}{2}}f_{1,j+1}(X), 0 \leq i \leq n-1,
		\end{equation}
		and
		\begin{equation}\label{eq5.53}
			f_{4,i+1}(X)=\sum_{j=0}^{i}k_{i,j}(-q)^{-\frac{j(j-1)}{2}}f_{3,j+1}(X), 0 \leq i \leq n-1.
		\end{equation}
		
		Let us prove \eqref{eq5.52} by induction on $i$. For $i=0$, the claim holds since $f_{2,1}(X)=f_{1,1}(X)$. Now, assume that the claim holds for $i$, i.e.,
		\begin{equation}\label{eq5.54}
			X\prod_{l=i+2}^{n}(1-(-q)^{-l}X)=\sum_{j=0}^{i}k_{i,j}(-q)^{-\frac{j(j+1)}{2}}X^{j+1}\prod_{l=j+2}^{n}(1-(-q)^{-l}X).
		\end{equation}
		Note that $\sum_{j=0}^{i+1}k_{i+1,j}X^j=\prod_{l=1}^{i+1}(1+(-q)^{-l}X)=(1+(-q)^{-i-1}X)(\sum_{j=0}^{i}k_{i,j}X^j)$, and hence
		\begin{equation}\label{eq5.55}
			\begin{aligned}
				\sum_{j=0}^{i+1}k_{i+1,j}(-q)^{-\frac{j(j+1)}{2}}f_{1,j+1}(X)&=\sum_{j=0}^{i}k_{i,j}(-q)^{-\frac{j(j+1)}{2}}f_{1,j+1}(X)\\
				&+\sum_{j=0}^{i}(-q)^{-i-1}k_{i,j}(-q)^{-\frac{(j+1)(j+2)}{2}}f_{1,j+2}(X).
			\end{aligned}
		\end{equation}
		Now, consider the equation \eqref{eq5.54} with $X\Rightarrow (-q)^{-1}X$, $n \Rightarrow n-1$. Then, we have
		\begin{equation*}
			\begin{array}{ll}
				&(-q)^{-1}X\prod_{l=i+3}^{n}(1-(-q)^{-l}X)=\sum_{j=0}^{i}k_{i,j}(-q)^{-\frac{(j+1)(j+2)}{2}}X^{j+1}\prod_{l=j+3}^{n}(1-(-q)^{-l}X)\\
				\Longleftrightarrow&\sum_{j=0}^{i}(-q)^{-i-1}k_{i,j}(-q)^{-\frac{(j+1)(j+2)}{2}}f_{1,j+2}(X)=(-q)^{-i-2}X^2 \prod_{l=i+3}^{n}(1-(-q)^{-l}X).
			\end{array}
		\end{equation*}
		Combining this with \eqref{eq5.55}, we have
		\begin{equation*}
			\begin{aligned}
				\sum_{j=0}^{i+1}k_{i+1,j}(-q)^{-\frac{j(j+1)}{2}}f_{1,j+1}(X)&=X \prod_{l=i+2}^{n}(1-(-q)^{-l}X)+(-q)^{-i-2}X^2 \prod_{l=i+3}^{n}(1-(-q)^{-l}X)\\
				&=X \prod_{l=i+3}^{n}(1-(-q)^{-l}X).
			\end{aligned}
		\end{equation*}
		This finishes the proof of \eqref{eq5.52}. The equation \eqref{eq5.53} can be proved in a similar way.

		Now, note that \cite[Theorem 2.5]{Cho4} implies that if $B_1=A_{\eta_1}$, $B_2=A_{\eta_2}$ such that $\eta_1 \geq (\overset{n}{\overbrace{2,\dots,2}})$, $\eta_2 \geq(\overset{n-1}{\overbrace{2,\dots,2}})$, then we have
		\begin{equation}\label{eq5.56}
			\begin{array}{ll}
				f_{1,j}(R_0)=\dfrac{(-1)^{n-j}}{(-q)^{n(n-j)}}f_{2,j}(R_1), & f_{1,j}(R_1)=\dfrac{(-1)^{n-j}}{(-q)^{n(n-j)}}f_{2,j}(R_2),\\
				f_{3,j}(R_3)=\dfrac{(-1)^{n-j}}{(-q)^{(n-1)(n-j)}}f_{4,j}(R_4), & f_{3,j}(R_4)=\dfrac{(-1)^{n-j}}{(-q)^{(n-1)(n-j)}}f_{4,j}(R_5).\\
			\end{array}
		\end{equation}
		Furthermore, if $B_1=A_{(\lambda,2^{n-1})}$ and $B_2=A_{(2^{n-1})}$, then $R_2^{k}=R_5^k=0$ for all $k > 1$. Therefore,
		\begin{equation}\label{eq5.57}
			\begin{array}{ll}
				f_{2,j}(R_2)=R_2=\alphad(A_{(\lambda,2^{n-1})},A_{(\lambda,2^{n-1})}), & f_{4,j}(R_5)=R_5=\alphad(A_{(2^{n-1})},A_{(2^{n-1})}).
			\end{array}
		\end{equation}
		Since $\alphad(A_{(\lambda,1,0^{n-2})},A_{(\lambda,2^{n-1})})=0$, we have
		\begin{equation}\label{eq5.58}
			\begin{array}{ll}
				\alphad(A_{(\lambda,0^{n-1})},A_{(\lambda,2^{n-1})})=R_0(1-(-q)^{-n}R_0)=f_{2,n-1}(R_0)&\\
				=\sum_{i=0}^{n-2}k_{n-2,i}(-q)^{-\frac{i(i+1)}{2}}f_{1,i+1}(R_0)&(\text{by }\eqref{eq5.52})\\
				=\sum_{i=0}^{n-2}k_{n-2,i}(-q)^{-\frac{i(i+1)}{2}}\frac{(-1)^{n-i-1}}{(-q)^{n(n-i-1)}}f_{2,i+1}(R_1) & (\text{by }\eqref{eq5.56})\\
				=\sum_{i=0}^{n-2}k_{n-2,i}(-q)^{-\frac{i(i+1)}{2}}\frac{(-1)^{n-i-1}}{(-q)^{n(n-i-1)}}\sum_{j=0}^{i}k_{i,j}(-q)^{-\frac{j(j+1)}{2}}f_{1,j+1}(R_1) & (\text{by } \eqref{eq5.52})\\
				=\lbrace\sum_{i=0}^{n-2}k_{n-2,i}(-q)^{-\frac{i(i+1)}{2}}\frac{(-1)^{n-i-1}}{(-q)^{n(n-i-1)}}\sum_{j=0}^{i}k_{i,j}(-q)^{-\frac{j(j+1)}{2}}\frac{(-1)^{n-j-1}}{(-q)^{n(n-j-1)}}\rbrace R_2 & (\text{by } \eqref{eq5.56}, \eqref{eq5.57}).
			\end{array}
		\end{equation}
		Similarly, note that $\alphad(A_{(1,0^{n-2})},A_{(2^{n-1})})=0$, and hence
		\begin{equation}\label{eq5.59}
			\begin{array}{ll}
				\alphad(A_{(0^{n-1})},A_{(2^{n-1})})=R_3(1-(-q)^{-(n-1)}R_3)=f_{4,n-1}(R_3)&\\
				=\sum_{i=0}^{n-2}k_{n-2,i}(-q)^{-\frac{i(i-1)}{2}}f_{3,i+1}(R_3)&(\text{by }\eqref{eq5.53})\\
				=\sum_{i=0}^{n-2}k_{n-2,i}(-q)^{-\frac{i(i-1)}{2}}\frac{(-1)^{n-i-1}}{(-q)^{(n-1)(n-i-1)}}f_{4,i+1}(R_4) & (\text{by }\eqref{eq5.56})\\
				=\sum_{i=0}^{n-2}k_{n-2,i}(-q)^{-\frac{i(i-1)}{2}}\frac{(-1)^{n-i-1}}{(-q)^{(n-1)(n-i-1)}}\sum_{j=0}^{i}k_{i,j}(-q)^{-\frac{j(j-1)}{2}}f_{3,j+1}(R_4) & (\text{by } \eqref{eq5.53})\\
				=\lbrace\sum_{i=0}^{n-2}k_{n-2,i}(-q)^{-\frac{i(i-1)}{2}}\frac{(-1)^{n-i-1}}{(-q)^{(n-1)(n-i-1)}}\sum_{j=0}^{i}k_{i,j}(-q)^{-\frac{j(j-1)}{2}}\frac{(-1)^{n-j-1}}{(-q)^{(n-1)(n-j-1)}}\rbrace R_5 & (\text{by } \eqref{eq5.56}, \eqref{eq5.57}).
			\end{array}
		\end{equation}
		
		Comparing \eqref{eq5.58} and \eqref{eq5.59}, we have
		\begin{equation}\label{eq5.60}	(-q)^{2n-2}\dfrac{\alphad(A_{(\lambda,0^{n-1})},A_{(\lambda,2^{n-1})})}{\alphad(A_{(\lambda,2^{n-1})},A_{(\lambda,2^{n-1})})}=\dfrac{\alphad(A_{(0^{n-1})},A_{(2^{n-1})})}{\alphad(A_{(2^{n-1})},A_{(2^{n-1})})}.
		\end{equation}
		Now, by \cite[Lemma 4.6]{Cho3}, we have
		\begin{equation*}
			\begin{array}{l}
				\alphad(A_{(2^{n-1})},A_{(2^{n-1})})=q^{2(n-1)^2}\alphad(A_{(0^{n-1})},A_{(0^{n-1})})=q^{2(n-1)^2}\prod_{l=1}^{n-1}(1-(-q)^{-l}),\\
				\alphad(A_{\lambda,(0^{n-1})},A_{(\lambda,0^{n-1})})=q^{\lambda}(1-(-q)^{-1})\prod_{l=1}^{n-1}(1-(-q)^{-l}), \text{ if }\lambda \geq 1,
			\end{array}
		\end{equation*}
		\begin{equation*}
			\alphad(A_{(\lambda,2^{n-1})},A_{(\lambda,2^{n-1})})=\left\lbrace\begin{array}{ll}
				q^{2(n^2-1)+\lambda}(1-(-q)^{-1})\prod_{l=1}^{n-1}(1-(-q)^{-l}) & \text{ if }\lambda \geq 3,\\
				q^{2n^2}\prod_{l=1}^{n}(1-(-q)^{-l}) & \text{ if }\lambda=2.
			\end{array}\right.
		\end{equation*}
		Combining these with \eqref{eq5.60}, we have
		\begin{equation*}
			\begin{array}{l}
				\dfrac{\alphad(A_{(\lambda,0^{n-1})},A_{(\lambda,2^{n-1})})/\alphad(A_{(\lambda,0^{n-1})},A_{(\lambda,0^{n-1})})}{\alphad(A_{(0^{n-1})},A_{(2^{n-1})})/\alphad(A_{(0^{n-1})},A_{(0^{n-1})})}\\
				=(-q)^{-2n+2}\dfrac{\alphad(A_{(\lambda,2^{n-1})},A_{(\lambda,2^{n-1})})}{\alphad(A_{(\lambda,0^{n-1})},A_{(\lambda,0^{n-1})})}\dfrac{\alphad(A_{(0^{n-1})},A_{(0^{n-1})})}{\alphad(A_{(2^{n-1})},A_{(2^{n-1})})}

				=\left\lbrace \begin{array}{ll}
					q^{2n-2} & \text{ if } \lambda \geq 3,\\
					q^{2n-2}\dfrac{1-(-q)^{-n}}{1-(-q)^{-1}} &\text{ if } \lambda=2.
				\end{array}\right.
			\end{array}
		\end{equation*}
		This finishes the proof of the lemma.
	\end{proof}

	\section{Tate conjectures and the proof of the main theorem}

	\subsection{The proof of the main theorem}
	Consider the Rapoport--Zink space $\CN^{[h]}_n$ and the space of special homomorphisms $\BV$. Assume that $L^{\flat} \subset \BV$ is an $O_F$-lattice of rank $n-1$. Recall that for $L'^{\flat}$ such that $L^{\flat} \subset L'^{\flat} \subset (L'^{\flat})^{\vee} \subset L^{\flat}_F$, we define the primitive part $\CZ(L'^{\flat})^{\circ}$ of $\CZ(L'^{\flat})$ inductively by setting
	\begin{equation*}
		\CZ(L'^{\flat})^{\circ}\coloneqq \CZ(L'^{\flat})-\mathlarger{\sum}_{\substack{L'^{\flat} \subset L''^{\flat}\\L''^{\flat} \subset (L''^{\flat})^{\vee} \subset L'^{\flat}_F}} \CZ(L''^{\flat})^{\circ}.
	\end{equation*}
	For example, for a lattice $L'^{\flat}$ with fundamental invariants $(0,0,\dots,0,1)$, there is no integral lattice $L''^{\flat}$ such that $L'^{\flat} \subset L''^{\flat}$ and hence $\CZ(L'^{\flat})^{\circ}=\CZ(L'^{\flat})$. For a lattice $L'^{\flat}$ with fundamental invariants $(0,0,\dots,0,3)$, there is one integral lattice $L''^{\flat}$ satisfying the above conditions, and its fundamental invariants are $(0,0,\dots,0,1)$. Therefore, we have $\CZ(L'^{\flat})^{\circ}=\CZ(L'^{\flat})-\CZ(L''^{\flat})^{\circ}=\CZ(L'^{\flat})-\CZ(L''^{\flat})$.
	
	Now, let us define the derived special cycles $\text{}^{\BL}\CZ(L)$.
	\begin{definition}\label{definition9.1}
		For a lattice $L \subset \BV$, choose its basis $x_1, \dots,x_r$. We define the derived special cycle $\text{}^{\BL}\CZ(L)$ as the image of $O_{\CZ(x_1)}\otimes^{\BL}\dots \otimes^{\BL}O_{\CZ(x_r)}$ in the $r$-th graded piece of the Grothendieck group $\text{Gr}^r K_0^{\CZ(L)}(\CN)$. This does not depend on the choice of the basis by Proposition \ref{prop: linear invariance}.
	\end{definition}

	Then, Conjecture \ref{conj: specialized}  is equivalent to the following statement: for $x \in \BV \backslash L_F^{\flat}$,
	\begin{equation}\label{eq9.1}
		\chi(\CN^{[h]}_n, \text{}^{\BL}\CZ(L'^{\flat})^{\circ} \otimes^{\BL}  O_{\CZ(x)})=\partial \mathrm{Den}_{L'^{\flat\circ}}^{n,h}(x)\coloneqq \mathlarger{\sum}_{L'^{\flat} \subset L' \subset L'^{\vee}, L' \cap L^{\flat}_F=L'^{\flat}} D_{n,h}(L')1_{L'}(x),
	\end{equation}
	where $L' \subset \BV$ are $O_F$-lattices of rank $n$. For example, for a lattice $L'^{\flat}$ with fundamental invariants $(0,0,\dots,1)$, it is obvious that Conjecture \ref{conj: specialized} is equivalent to \eqref{eq9.1}. For a lattice $L'^{\flat}$ with fundamental invariants $(0,0,\dots,3)$, there is one integral lattice $L''^{\flat}$ such that $L'^{\flat} \subsetneq L''^{\flat}$, and the fundamental invariants of $L''^{\flat}$ is $(0,0,\dots, 1)$. Therefore, Conjecture \ref{conj: specialized} is equivalent to
	\begin{equation*}\begin{array}{ll}
			\chi(\CN^{[h]}_n, \text{}^{\BL}\CZ(L'^{\flat})^{\circ} \otimes^{\BL}  O_{\CZ(x)})&=\mathlarger{\sum}_{L'^{\flat} \subset L' \subset L'^{\vee}} D_{n,h}(L')1_{L'}(x)-\mathlarger{\sum}_{L'^{\flat} \subset L' \subset L'^{\vee}, L' \cap L^{\flat}_F=L''^{\flat}} D_{n,h}(L')1_{L'}(x)\\
			&=\mathlarger{\sum}_{L'^{\flat} \subset L' \subset L'^{\vee}, L' \cap L^{\flat}_F=L'^{\flat}} D_{n,h}(L')1_{L'}(x),
		\end{array}
	\end{equation*}
	which is equivalent to \eqref{eq9.1}.
	
	Now, note that there is a decomposition of the derived special cycle $^{\BL}\CZ(L^{\flat})$ into a sum of horizontal and vertical parts (see Section \ref{subsection2.3}):
	\begin{equation*}
		\text{}^{\BL}\CZ(L'^{\flat})=\text{}^{\BL}\CZ(L'^{\flat})_{\ScH}+    \text{}^{\BL}\CZ(L'^{\flat})_{\ScV}.
	\end{equation*}
	We denote by $\text{}^{\BL}\CZ(L'^{\flat})_{\ScH}^{\circ}$ (resp. $\text{}^{\BL}\CZ(L'^{\flat})_{\ScV}^{\circ})$ the primitive part of $\text{}^{\BL}\CZ(L'^{\flat})_{\ScH}$ (resp. $\text{}^{\BL}\CZ(L'^{\flat})_{\ScV}$).
	
	We define
	\begin{equation*}
		\begin{array}{ll}
			\mathrm{Int}_{L'^{\flat},\ScH}(x)\coloneqq \chi(\CN^{[h]}_n, \text{}^{\BL}\CZ(L'^{\flat})_{\ScH} \otimes^{\BL}  O_{\CZ(x)}), &
			
			\mathrm{Int}_{L'^{\flat},\ScV}(x)\coloneqq \chi(\CN^{[h]}_n, \text{}^{\BL}\CZ(L'^{\flat})_{\ScV} \otimes^{\BL}  O_{\CZ(x)}),\\
			
			\mathrm{Int}_{L'^{\flat \circ},\ScH}(x)\coloneqq \chi(\CN^{[h]}_n, \text{}^{\BL}\CZ(L'^{\flat})_{\ScH}^{\circ} \otimes^{\BL}  O_{\CZ(x)}),&
			
			\mathrm{Int}_{L'^{\flat \circ},\ScV}(x)\coloneqq \chi(\CN^{[h]}_n, \text{}^{\BL}\CZ(L'^{\flat})_{\ScV}^{\circ} \otimes^{\BL}  O_{\CZ(x)}).
		\end{array}
	\end{equation*}
	Then, Conjecture \ref{conjecture5.5} is equivalent to  
	\begin{equation}\label{eq9.2}
		\mathrm{Int}_{L'^{\flat},\ScH}(x)+	\mathrm{Int}_{L'^{\flat },\ScV}(x)=\partial \mathrm{Den}_{L'^{\flat}}^{n,h}(x).
	\end{equation}
	or
	\begin{equation}\label{eq9.3}
		\mathrm{Int}_{L'^{\flat \circ},\ScH}(x)+	\mathrm{Int}_{L'^{\flat \circ},\ScV}(x)=\partial \mathrm{Den}_{L'^{\flat\circ}}^{n,h}(x).
	\end{equation}

	Now, we define $\partial \mathrm{Den}_{L'^{\flat\circ},\ScV}^{n,h}(x)$ by
	\begin{equation*}
		\partial \mathrm{Den}_{L'^{\flat\circ},\ScV}^{n,h}(x)\coloneqq \partial \mathrm{Den}_{L'^{\flat\circ}}^{n,h}(x)-\mathrm{Int}_{L'^{\flat \circ},\ScH}(x).
	\end{equation*}
	Let $\lambda \in \CR_{n-1}^{0+}$ be the fundamental invariants of $L'^{\flat}$. Then, by Lemma \ref{lemma4.3}, we know that if $L'^{\flat} \notin H(\BV)$ (i.e., $(t_{\ge 2}(\lambda),t_{1}(\lambda),t_0(\lambda)) \neq (1,h,n-h-2), (1,h-2,n-h), (0,h-1,n-h), (0,h,n-h-1), (0,h+1,n-h-2)$), then we have
	\begin{equation*}
		\mathrm{Int}_{L'^{\flat \circ},\ScH}(x)=0,
	\end{equation*}
	and hence we have that
	\begin{equation*}
		\partial \mathrm{Den}_{L'^{\flat\circ},\ScV}^{n,h}(x)= \partial \mathrm{Den}_{L'^{\flat\circ}}^{n,h}(x),
	\end{equation*} in this case.
	
	When $L'^{\flat} \in H(\BV)$, then the horizontal part of $\CZ(L'^{\flat})$ is not empty, and hence we need to be careful. Note that by Theorem \ref{thm: hori part}, the horizontal part of $\CZ(L'^{\flat})^{\circ}$ comes from $\CN^{[0]}_2$ or $\CN^{[2]}_2$. By \cite[Theorem 1.1]{KR1} and reductions in Proposition \ref{proposition2.6}, we have a precise formula for $\mathrm{Int}_{L'^{\flat \circ},\ScH}(x)$.
	
	Now, we will prove the following theorem.
	\begin{theorem}\label{theorem9.1.1}(cf. \cite[Theorem 7.4.1]{LZ})
		Assume that $x \perp L'^{\flat}$ and $(t_{\ge 2}(\lambda),t_1(\lambda),t_0(\lambda))=(a,b,c)$ where $\lambda \in \CR_{n-1}^{0+}$ is the fundamental invariants of $L'^{\flat}$. Then, we have the followings.
		\begin{enumerate}
			\item If $\val(( x,x )) \leq -2$, we have $\widehat{\pDen}_{L'^{\flat\circ},\ScV}^{n,h}(x)=0$.
			
			\item If $\val(( x,x))=-1$, we have
			\begin{equation*}
				\widehat{\pDen}_{L'^{\flat\circ},\ScV}^{n,h}(x)=\left\lbrace \begin{array}{ll}
					-\dfrac{1}{q^h}D_{n-1,h-1}(a,b,c) & \text{if }(a,b,c) \neq (1,h-2,n-h),\\\\
					-\dfrac{1}{q^h}D_{n-1,h-1}(a,b,c)+\dfrac{1}{q^h} & \text{if }(a,b,c) =(1,h-2,n-h).
				\end{array}\right.
			\end{equation*}
		\end{enumerate}
	\end{theorem}
	\begin{proof}
		First, if $(a,b,c)\coloneqq (t_{\ge 2}(\lambda),t_{1}(\lambda),t_0(\lambda)) \neq (1,h,n-h-2), (1,h-2,n-h), (0,h-1,n-h), (0,h,n-h-1), (0,h+1,n-h-2)$, then the assertion (1) follows from Theorem \ref{theorem5.46}, and the assertion (2) follows from Theorem \ref{theorem5.47}.
		
		Now, assume that $(a,b,c)= (1,h,n-h-2), (1,h-2,n-h), (0,h-1,n-h), (0,h,n-h-1), (0,h+1,n-h-2)$. Since we already know the Fourier transform of $\partial \mathrm{Den}_{L'^{\flat\circ}}^{n,h}(x)$ by Theorem \ref{theorem8.18} and Theorem \ref{theorem8.19}, we only need to know the Fourier transform of $\mathrm{Int}_{L'^{\flat \circ},\ScH}(x)$.
		
		When $(a,b,c)=(1,h,n-h-2)$, let us write $L'^{\flat}=L_2 \obot L_1 \obot L_0$, where the hermitian matrix of $L_2$, $L_1$, $L_0$ are $\pi^{\lambda}$ ($\lambda \geq 2$), $\pi I_h$, $I_{n-h-2}$, respectively. Then, by Proposition \ref{proposition2.6}, $\CZ(L'^{\flat})^{\circ}_{\ScH}$ in $\CN^{[h]}_{n}$ can be reduced to $\CZ(L_2 \obot L_1)^{\circ}_{\ScH}$ in $N^{[h]}_{h+2}$. Therefore, by Theorem \ref{thm: hori part}, we have
		\begin{equation}\label{eq9.4}
			\CZ(L_2 \obot L_1)^{\circ}_{\ScH}=\mathlarger{\sum}_{\substack{L_2 \obot L_1 \subset L_2 \oplus N \subset \pi^{-1}(L_2 \obot L_1)\\ N \simeq (\pi^{-1})^h}} \CZ(L_2)^{\circ}\cdot \CY(N)^{\circ}.
		\end{equation}
		
		By Proposition \ref{proposition2.6} again, $\CZ(L_2 \obot L_1)^{\circ}_{\ScH}$ can be reduced to $\CZ(L_2)^{\circ}_{\ScH}=\CZ(L_2)^{\circ}$ in $\CN^{[0]}_{2}$. Now, note that by \cite[Theorem 1.1]{KR1}, we know that the Kudla-Rapoport conjecture holds in the case of $\CN^{[0]}_2$, and hence
		\begin{equation*}
			\mathrm{Int}_{L_2^{\circ},\ScH}(x)=\mathrm{Int}_{L_2^{\circ}}(x)=\partial \mathrm{Den}_{L_2^{\circ}}^{2,0}(x).
		\end{equation*}
		By Theorem \ref{theorem8.18} and Theorem \ref{theorem8.19}, we have that
		\begin{equation*}
			\widehat{\Int}_{L_2^{\circ}}(x)=\widehat{\pDen}_{L_2^{\circ}}^{2,0}(x)=q^{-1}(q+1)\vol(\langle x \rangle^{\vee})(1-q^{-2})^{-1}.
		\end{equation*}
		Now, by \cite[Proposition 8.2]{ZhiyuZhang}, we know that via the embedding $\CN^{[0]}_{2} \hookrightarrow \CN^{[h]}_n$, the Fourier transform of $\chi(\CN^{[h]}_n, O_{\CZ(L_2)^{\circ}} \otimes^{\BL} O_{\CY(N_2)} \otimes^{\BL} O_{\CZ(L_0)}\otimes^{\BL}O_{\CZ(x)})$ is
		\begin{equation*}
			\dfrac{1}{q^h}\widehat{\Int}_{L_2^{\circ}}(x)=q^{-h-1}(q+1)\vol(\langle x \rangle^{\vee})(1-q^{-2})^{-1}.
		\end{equation*}
		
		Now, by Lemma \ref{lemma5.48}, the number of lattices $L_2 \oplus N$ in the sum in \eqref{eq9.4} is
		\begin{equation*}
			\begin{array}{l}
				\dfrac{\alphad(A_{(\lambda,(-1)^{h})},A_{(\lambda,1^{h})})/\alphad(A_{(\lambda,(-1)^{h})},A_{(\lambda,(-1)^{h})})}{\alphad(A_{((-1)^{h})},A_{(1^{h})})/\alphad(A_{((-1)^{h})},A_{((-1)^{h})})}\\
				=\dfrac{\alphad(A_{((\lambda+1),0^{h})},A_{((\lambda+1),2^{h})})/\alphad(A_{((\lambda+1),0^{h})},A_{((\lambda+1),0^{h})})}{\alphad(A_{(0^{h})},A_{(2^{h})})/\alphad(A_{(0^{h})},A_{(0^{h})})}=q^{2h}.
			\end{array}
		\end{equation*}
		
		This implies that
		\begin{equation*}\begin{aligned}
				\widehat{\Int}_{L'^{\flat \circ},\ScH}(x)&=q^{2h}q^{-h-1}(q+1)\vol(\langle x \rangle^{\vee})(1-q^{-2})^{-1}\\
				&=q^{h-1}(q+1)\vol(\langle x \rangle^{\vee})(1-q^{-2})^{-1}.
			\end{aligned}
		\end{equation*}
		Now, by Theorem \ref{theorem8.18}, Theorem \ref{theorem8.19}, we have that
		\begin{equation*}
			\widehat{\pDen}_{L'^{\flat\circ},\ScV}^{n,h}(x)=\widehat{\pDen}_{L'^{\flat\circ}}^{n,h}(x)-\widehat{\Int}_{L'^{\flat \circ},\ScH}(x)=\left\lbrace \begin{array}{ll}
				0 & \text{if }\val(( x,x )) \leq -2,\\
				-\dfrac{1}{q^h}D_{n-1,h-1}(a,b,c) & \text{if }\val(( x,x ))=-1.
			\end{array}\right.
		\end{equation*}
		This proves (1) and (2) when $(a,b,c)=(1,h,n-h-2)$.
		
		Assume that $(a,b,c)=(1,h-2,n-h)$ and let us write $L'^{\flat}=L_2 \obot L_1 \obot L_0$, where the hermitian matrix of $L_2$, $L_1$, $L_0$ are $\pi^{\lambda}$ ($\lambda \geq 2$), $\pi I_{h-2}$, $I_{n-h}$, respectively. Then, by Proposition \ref{proposition2.6}, $\CZ(L'^{\flat})^{\circ}_{\ScH}$ in $\CN^{[h]}_n$ can be reduced to $\CZ(L_2 \obot L_1)^{\circ}_{\ScH}$ in $\CN^{[h]}_{h}$. Also, in $\CN^{[h]}_{h}$, we have that $\CZ(w)=\CY(\pi^{-1}w)$ for any $w \in \BV$, and hence $\CZ(L_2 \obot L_1)^{\circ}_{\ScH}$ can be reduced to $\CZ(L_2)^{\circ}_{\ScH}=\CZ(L_2)^{\circ}$ in $\CN^{[2]}_2$. Now, note that by \cite[Theorem 1.1]{KR1}, we know that the Kudla-Rapoport conjecture holds in the case of $\CN^{[2]}_2$, and hence
		\begin{equation*}
			\mathrm{Int}_{L_2^{\circ},\ScH}(x)=\mathrm{Int}_{L_2^{\circ}}(x)=\partial \mathrm{Den}_{L_2^{\circ}}^{2,2}(x).
		\end{equation*}
		By Theorem \ref{theorem8.18} and Theorem \ref{theorem8.19}, we have that
		\begin{equation*}\begin{aligned}
				\widehat{\Int}_{L_2^{\circ}}(x)&=\widehat{\pDen}_{L_2^{\circ}}^{2,2}(x) 
				=\left\lbrace \begin{array}{ll}
					q^{-2}(q+1)\vol(\langle x \rangle^{\vee})(1-q^{-2})^{-1} & \text{if } \val(( x, x ))\leq -2,\\
					q^{-3}(q+1)(1-q^{-2})^{-1}-q^{-2} & \text{if } \val(( x, x))= -1.
				\end{array}\right.
			\end{aligned}
		\end{equation*}
		Now, by \cite[Proposition 8.2]{ZhiyuZhang}, we know that via the embedding $\CN^{[2]}_{2} \hookrightarrow \CN^{[h]}_n$, the Fourier transform of $\chi(\CN^{[h]}_n, O_{\CZ(L_2)^{\circ}} \otimes^{\BL} O_{\CY(\pi^{-1}L_1)} \otimes^{\BL} O_{\CZ(L_0)}\otimes^{\BL}O_{\CZ(x)})$ is
		\begin{equation*}\begin{aligned}
				\dfrac{1}{q^{h-2}}\widehat{\Int}_{L_2^{\circ}}(x)
				&=\left\lbrace \begin{array}{ll}
					q^{-h}(q+1)\vol(\langle x \rangle^{\vee})(1-q^{-2})^{-1} & \text{if } \val(( x, x ))\leq -2,\\
					q^{-h-1}(q+1)(1-q^{-2})^{-1}-q^{-h} & \text{if } \val(( x, x ))= -1.
				\end{array}\right.
			\end{aligned}
		\end{equation*}
		Now, by Theorem \ref{theorem8.18}, Theorem \ref{theorem8.19}, we have that
		\begin{equation*}
			\widehat{\pDen}_{L'^{\flat\circ},\ScV}^{n,h}(x)=\widehat{\pDen}_{L'^{\flat\circ}}^{n,h}(x)-\widehat{\Int}_{L'^{\flat \circ},\ScH}(x)=\left\lbrace \begin{array}{ll}
				0 & \text{if }\val(( x,x)) \leq -2,\\
				-\dfrac{1}{q^h}D_{n-1,h-1}(a,b,c)+\dfrac{1}{q^h} & \text{if }\val(( x,x ))=-1.
			\end{array}\right.
		\end{equation*}
		This proves (1) and (2) when $(a,b,c)=(1,h-2,n-h)$.
		
		Assume that $(a,b,c)=(0,h-1,n-h)$ and let us write $L'^{\flat}=L_2 \obot L_1 \obot L_0$, where the hermitian matrix of $L_2$, $L_1$, $L_0$ are $\pi I_{1}$. $\pi I_{h-2}$, $I_{n-h}$, respectively. Note that in this case, $\val((x,x)) \neq -1$ since $\val(L^{\flat} \obot \langle x \rangle)\equiv h+1 \pmod 2$. By Proposition \ref{proposition2.6}, $\CZ(L'^{\flat})^{\circ}_{\ScH}$ in $\CN^{[h]}_n$ can be reduced to $\CZ(L_2 \obot L_1)^{\circ}_{\ScH}$ in $\CN^{[h]}_{h}$. Also, in $\CN^{[h]}_{h}$, we have that $\CZ(w)=\CY(\pi^{-1}w)$ for any $w \in \BV$, and hence $\CZ(L_2 \obot L_1)^{\circ}_{\ScH}$ can be reduced to $\CZ(L_2)^{\circ}_{\ScH}=\CZ(L_2)^{\circ}$ in $\CN^{[2]}_2$. Now, note that by \cite[Theorem 1.1]{KR1}, we know that the Kudla-Rapoport conjecture holds in the case of $\CN^{[2]}_2$, and hence
		\begin{equation*}
			\mathrm{Int}_{L_2^{\circ},\ScH}(x)=\mathrm{Int}_{L_2^{\circ}}(x)=\partial \mathrm{Den}_{L_2^{\circ}}^{2,2}(x).
		\end{equation*}
		By Theorem \ref{theorem8.18}, we have that
		\begin{equation*}\begin{aligned}
				\widehat{\Int}_{L_2^{\circ}}(x)&=\widehat{\pDen}_{L_2^{\circ}}^{2,2}(x) 
				=q^{-1}\vol(\langle x \rangle^{\vee})(1-q^{-2})^{-1}.
			\end{aligned}
		\end{equation*}
		Now, by \cite[Proposition 8.2]{ZhiyuZhang}, we know that via the embedding $\CN^{[2]}_{2} \hookrightarrow \CN^{[h]}_n$, the Fourier transform of $\chi(\CN^{[h]}_n, O_{\CZ(L_2)^{\circ}} \otimes^{\BL} O_{\CY(\pi^{-1}L_1)} \otimes^{\BL} O_{\CZ(L_0)}\otimes^{\BL}O_{\CZ(x)})$is
		\begin{equation*}\begin{array}{l}
				\dfrac{1}{q^{h-2}}\widehat{\Int}_{L_2^{\circ}}(x)
				=q^{-h+1}\vol(\langle x \rangle^{\vee})(1-q^{-2})^{-1}.
			\end{array}
		\end{equation*}
		By Theorem \ref{theorem8.18} we have that
		\begin{equation*}
			\widehat{\pDen}_{L'^{\flat\circ},\ScV}^{n,h}(x)=\widehat{\pDen}_{L'^{\flat\circ}}^{n,h}(x)-\widehat{\Int}_{L'^{\flat \circ},\ScH}(x)=0,
		\end{equation*}
		This proves (1) when $(a,b,c)=(0,h-1,n-h)$.
		
		Assume that $(a,b,c)=(0,h,n-h-1)$ and let us write $L'^{\flat}=L_1 \obot L_0$, where the hermitian matrix of $L_1$, $L_0$ are $\pi I_{h}$, $I_{n-h-1}$, respectively. By Proposition \ref{proposition2.6}, $\CZ(L'^{\flat})^{\circ}_{\ScH}$ in $\CN^{[h]}_n$ can be reduced to $\CZ(L_1)^{\circ}_{\ScH}$ in $\CN^{[h]}_{h+1}$. Therefore, by Theorem \ref{thm: hori part}, we have that $\CZ(L_1)^{\circ}_{\ScH}=\CY(\pi^{-1}L_1)$ and hence $\langle \CY(\pi^{-1}L_1), \CZ(x)\rangle$ is $\langle \CZ(x) \rangle$ in $\CN^{[0]}_1$. Now, By \cite[Theorem 1.1]{KR1} and Proposition \ref{proposition2.6}, we know that $\langle \CZ(x)\rangle=\partial \mathrm{Den}_{\emptyset}^{1,0}(x)$.
		By Theorem \ref{theorem8.18} and Theorem \ref{theorem8.19}, we have that
		\begin{equation*}
			\widehat{\pDen}_{\emptyset}^{1,0}(x)=\vol(\langle x \rangle^{\vee})(1-q^{-2})^{-1}.
		\end{equation*}
		
		Now, by \cite[Proposition 8.2]{ZhiyuZhang}, we know that via the embedding $\CN^{[0]}_{1} \hookrightarrow \CN^{[h]}_n$, the Fourier transform of $\chi(\CN^{[h]}_n, O_{\CY(\pi^{-1}L_1)} \otimes^{\BL} O_{\CZ(L_0)}\otimes^{\BL}O_{\CZ(x)})$is
		\begin{equation*}
			\dfrac{1}{q^h}\widehat{\pDen}_{\emptyset}^{1,0}(x)=q^{-h}\vol(\langle x \rangle^{\vee})(1-q^{-2})^{-1}.
		\end{equation*}
		
		This implies that
		\begin{equation*}\begin{aligned}
				\widehat{\Int}_{L'^{\flat \circ},\ScH}(x)&=q^{-h}\vol(\langle x \rangle^{\vee})(1-q^{-2})^{-1}
			\end{aligned}
		\end{equation*}
		Now, by Theorem \ref{theorem8.18}, Theorem \ref{theorem8.19}, we have that
		\begin{equation*}
			\widehat{\pDen}_{L'^{\flat\circ},\ScV}^{n,h}(x)=\widehat{\pDen}_{L'^{\flat\circ}}^{n,h}(x)-\widehat{\Int}_{L'^{\flat \circ},\ScH}(x)=\left\lbrace \begin{array}{ll}
				0 & \text{if }\val(( x,x )) \leq -2,\\
				-\dfrac{1}{q^h}D_{n-1,h-1}(a,b,c) & \text{if }\val(( x,x ))=-1.
			\end{array}\right.
		\end{equation*}
		This proves (1) and (2) when $(a,b,c)=(0,h,n-h-1)$.

		When $(a,b,c)=(0,h+1,n-h-2)$, let us write $L'^{\flat}=L_2 \obot L_1 \obot L_0$, where the hermitian matrix of $L_2$, $L_1$, $L_0$ are $\pi^{1}I_1$, $\pi I_h$, $I_{n-h-2}$, respectively.  Note that in this case, $\val((x,x)) \neq -1$ since $\val(L^{\flat} \obot \langle x \rangle)\equiv h+1 \pmod 2$. By Proposition \ref{proposition2.6}, $\CZ(L'^{\flat})^{\circ}_{\ScH}$ in $\CN^{[h]}_{n}$ can be reduced to $\CZ(L_2 \obot L_1)^{\circ}_{\ScH}$ in $N^{[h]}_{h+2}$. Therefore, by Theorem \ref{thm: hori part}, we have
		\begin{equation}\label{eq9.5}
			\CZ(L_2 \obot L_1)^{\circ}_{\ScH}=\mathlarger{\sum}_{\substack{L_2 \obot L_1 \subset L_2 \oplus N \subset \pi^{-1}(L_2 \obot L_1)\\ N \simeq (\pi^{-1})^h}} \CZ(L_2)^{\circ}\cdot \CY(N)^{\circ}.
		\end{equation}
		
		By Proposition \ref{proposition2.6} again, $\CZ(L_2 \obot L_1)^{\circ}_{\ScH}$ can be reduced to $\CZ(L_2)^{\circ}_{\ScH}=\CZ(L_2)^{\circ}$ in $\CN^{[0]}_{2}$. Now, note that by \cite[Theorem 1.1]{KR1}, we know that the Kudla-Rapoport conjecture holds in the case of $\CN^{[0]}_2$, and hence
		\begin{equation*}
			\mathrm{Int}_{L_2^{\circ},\ScH}(x)=\mathrm{Int}_{L_2^{\circ}}(x)=\partial \mathrm{Den}_{L_2^{\circ}}^{2,0}(x).
		\end{equation*}
		By Theorem \ref{theorem8.18} and Theorem \ref{theorem8.19}, we have that
		\begin{equation*}
			\widehat{\Int}_{L_2^{\circ}}(x)=\widehat{\pDen}_{L_2^{\circ}}^{2,0}(x)=q^{-1}(q+1)\vol(\langle x \rangle^{\vee})(1-q^{-2})^{-1}.
		\end{equation*}
		Now, by \cite[Proposition 8.2]{ZhiyuZhang}, we know that via the embedding $\CN^{[0]}_{2} \hookrightarrow \CN^{[h]}_n$, the Fourier transform of $\chi(\CN^{[h]}_n, O_{\CZ(L_2)^{\circ}} \otimes^{\BL} O_{\CY(N_2)} \otimes^{\BL} O_{\CZ(L_0)}\otimes^{\BL}O_{\CZ(x)})$is
		\begin{equation*}
			\dfrac{1}{q^h}\widehat{\Int}_{L_2^{\circ}}(x)=q^{-h-1}(q+1)\vol(\langle x \rangle^{\vee})(1-q^{-2})^{-1}.
		\end{equation*}
		
		Now, by Lemma \ref{lemma5.48}, the number of lattices $L_2 \oplus N$ in the sum in \eqref{eq9.5} is
		\begin{equation*}
			\begin{array}{l}
				\dfrac{\alphad(A_{(1,(-1)^{h})},A_{(1,1^{h})})/\alphad(A_{(1,(-1)^{h})},A_{(1,(-1)^{h})})}{\alphad(A_{((-1)^{h})},A_{(1^{h})})/\alphad(A_{((-1)^{h})},A_{((-1)^{h})})}\\
				=\dfrac{\alphad(A_{(2,0^{h})},A_{(2,2^{h})})/\alphad(A_{(2,0^{h})},A_{(2,0^{h})})}{\alphad(A_{(0^{h})},A_{(2^{h})})/\alphad(A_{(0^{h})},A_{(0^{h})})}=q^{2h}\dfrac{1-(-q)^{-h-1}}{1-(-q)^{-1}}.
			\end{array}
		\end{equation*}
		
		This implies that
		\begin{align*}
			\widehat{\Int}_{L'^{\flat \circ},\ScH}(x)&=q^{2h}\dfrac{1-(-q)^{-h-1}}{1-(-q)^{-1}}q^{-h-1}(q+1)\vol(\langle x \rangle^{\vee})(1-q^{-2})^{-1}\\
			&=q^{h}(1-(-q)^{-h-1})\vol(\langle x \rangle^{\vee})(1-q^{-2})^{-1}.
		\end{align*}
		Now, by Theorem \ref{theorem8.18}, Theorem \ref{theorem8.19}, we have that
		\begin{equation*}
			\widehat{\pDen}_{L'^{\flat\circ},\ScV}^{n,h}(x)=\widehat{\pDen}_{L'^{\flat\circ}}^{n,h}(x)-\widehat{\Int}_{L'^{\flat \circ},\ScH}(x)=0.
		\end{equation*}
		This proves (1) when $(a,b,c)=(0,h+1,n-h-2)$.
		
		This finishes the proof of the theorem.
	\end{proof}
	
	\begin{proposition}\label{proposition9.1.3} (cf. \cite[Proposition 7.3.4]{LZ}, \cite[Proposition 2.22]{LL2})
		Assume that $L^{\flat}$ be an $O_F$-lattice of rank $n-1$ in $\BV$. Then, $\partial \mathrm{Den}_{L'^{\flat},\ScV}^{n,h}(x)$ extends uniquely to a compactly supported locally constant function on $\BV$ (we still denote it by $\partial \mathrm{Den}_{L'^{\flat},\ScV}^{n,h}(x)$).
	\end{proposition}
	\begin{proof}
		Even though our situation is more similar to \cite[Proposition 7.3.4]{LZ}, we do not have a functional equation like \cite[(3.2.0.2)]{LZ} yet. Therefore, let us follow the proof of \cite[Proposition 2.22]{LL2}. Basically, the proof is very similar to the proof of Theorem \ref{theorem9.1.1} (1).
		
		Note that if $L^{\flat}$ is not integral, then $\partial \mathrm{Den}_{L'^{\flat},\ScV}^{n,h}(x)=0$. Therefore, we can assume that $L^{\flat}$ is integral. Now, it suffices to show that for every $y \in L^{\flat}_F/ L^{\flat}$, there exists an integer $\delta(y)>0$ such that $\partial \mathrm{Den}_{L'^{\flat},\ScV}^{n,h}(y+x)$ is constant for $x \in \pi^{\delta(y)}((L_F^{\flat})^{\perp})^{\geq 0}\backslash \lbrace 0 \rbrace$, where $\BV=L^{\flat}_F \obot (L_F^{\flat})^{\perp}$ and $((L_F^{\flat})^{\perp})^{\geq 0}=\lbrace (L_F^{\flat})^{\perp} \mid (x,x) \in O_F \rbrace$. If $L^{\flat}+\langle y \rangle$ is not integral, then there exists $\delta(y)$ such that $L^{\flat}+\langle y+x\rangle$ is not integral for  $x \in \pi^{\delta(y)}((L_F^{\flat})^{\perp})^{\geq 0}\backslash \lbrace 0 \rbrace$, and hence $\partial \mathrm{Den}_{L'^{\flat},\ScV}^{n,h}(y+x)=0$. Therefore, we can assume that $L^{\flat}+\langle y\rangle$ is integral. Let $(a_1,a_2,\dots,a_{n-1})$ ($a_1 \leq a_2 \dots \leq a_{n-1}$) be the fundamental invariants of $L^{\flat}$ and let $\delta(y)=a_{n-1}+2$. Then, it suffices to show that for a fixed pair $(f_1,f_2)$ of generators of $((L_F^{\flat})^{\perp})^{\geq 0}$, we have
		\begin{equation*}
			\partial \mathrm{Den}_{L'^{\flat},\ScV}^{n,h}(y+\pi^{\delta}f_1)-\partial \mathrm{Den}_{L'^{\flat},\ScV}^{n,h}(y+\pi^{\delta-1}f_2)=0,
		\end{equation*}
		for $\delta > \delta(y)=a_{n-1}+2$.
		
		Now, let us introduce some notations from \cite[Lemma 2.24]{LL2}. Let $\FL$ be the set of $O_F$-lattices in $\BV$ containing $L^{\flat}$, and let $\FC$ be the set of triples $(L'^{\flat},\delta,\epsilon)$ such that $L'^{\flat}$ is an $O_F$-lattice of $L^{\flat}_F$ containing $L^{\flat}$, $\delta \in \BZ$, and $\epsilon:\pi^{\delta}((L_F^{\flat})^{\perp})^{\geq 0} \rightarrow L'^{\flat} \otimes_{O_F} F/O_F$ is an $O_F$-linear map. Then, consider the map $\theta:\FL \rightarrow \FC$ sending $L$ to $(L\cap L^{\flat}_F,\delta_L,\epsilon_L)$ where $\delta_L$ is the maximal integer such that the image of $L$ under the projection $\text{Pr}_{\perp}:\BV \rightarrow (L_F^{\flat})^{\perp}$ is contained in $\pi^{\delta_L}((L_F^{\flat})^{\perp})^{\geq 0}$, and $\epsilon_L$ is the extension map $\pi^{\delta_L}((L_F^{\flat})^{\perp})^{\geq 0} \rightarrow L'^{\flat} \otimes_{O_F} F/O_F$ induced by the short exact sequence
		\begin{equation*}
			0 \rightarrow L \cap L_F^{\flat} \rightarrow L \rightarrow \pi^{\delta_L}((L_F^{\flat})^{\perp})^{\geq 0} \rightarrow 0.
		\end{equation*}
		Then, as in the proof of \cite[Lemma 2.24 (1)]{LL2}, $\theta$ is a bijection and its inverse is given by sending $(L'^{\flat},\delta, \epsilon)$ to the $O_F$-lattice $L$ generated by $L'^{\flat}$ and $\epsilon(x)+x$ for every $x \in \pi^{\delta}((L_F^{\flat})^{\perp})^{\geq 0}$.

		Now, for every $\delta'\in\BZ$, we define the following sets
		\begin{equation*}
			\begin{array}{l}
				\FL_1^{\delta'}\coloneqq \lbrace L \in \FL \mid L \subset L^{\vee}, \delta_L=\delta', y+\pi^{\delta}f_1 \in L\rbrace,\\
				\FL_2^{\delta'}\coloneqq \lbrace L \in \FL \mid L \subset L^{\vee}, \delta_L=\delta', y+\pi^{\delta-1}f_2 \in L\rbrace,\\
			\end{array}
		\end{equation*}
		and for $L^{\flat} \subset L'^{\flat}$, we define
		\begin{equation*}
			\begin{array}{l}
				\FL_{1,L'^{\flat}}^{\delta'}\coloneqq \lbrace L \in \FL \mid L \subset L^{\vee}, \delta_L=\delta', y+\pi^{\delta}f_1 \in L , L \cap L^{\flat}_F=L'^{\flat}\rbrace,\\
				\FL_{2,L'^{\flat}}^{\delta'}\coloneqq \lbrace L \in \FL \mid L \subset L^{\vee}, \delta_L=\delta', y+\pi^{\delta-1}f_2 \in L, L \cap L^{\flat}_F=L'^{\flat}\rbrace.\\
			\end{array}
		\end{equation*}
		Since $\partial \mathrm{Den}_{L^{\flat},\ScV}^{n,h}(x)$ is a certain sum of $\partial \mathrm{Den}_{L'^{\flat\circ},\ScV}^{n,h}(x)$, it suffices to show that
		\begin{equation*}
			\partial \mathrm{Den}_{L'^{\flat\circ},\ScV}^{n,h}(y+\pi^{\delta}f_1)-\partial \mathrm{Den}_{L'^{\flat\circ},\ScV}^{n,h}(y+\pi^{\delta-1}f_2)=0,
		\end{equation*}
		for $\delta > \delta(y)=a_{n-1}+2$ and $L^{\flat} \subset L'^{\flat}$.
		
		By definition, we have that
		\begin{equation*}
			\begin{array}{l}
				\partial \mathrm{Den}_{L'^{\flat\circ},\ScV}^{n,h}(y+\pi^{\delta}f_1)=\mathlarger{\sum}_{\delta' \leq \delta} \quad \mathlarger{\sum}_{L \in \FL_{1,L'^{\flat}}^{\delta'}} D_{n,h}(L),\\
				
				\partial \mathrm{Den}_{L'^{\flat\circ},\ScV}^{n,h}(y+\pi^{\delta-1}f_2)=\mathlarger{\sum}_{\delta' \leq \delta-1} \quad \mathlarger{\sum}_{L \in \FL_{2,L'^{\flat}}^{\delta'}} D_{n,h}(L).
				
			\end{array}
		\end{equation*}
		Therefore, it suffices to show that
		\begin{equation}\label{eq9.1.6.1}
			\mathlarger{\sum}_{\delta' \leq \delta} \quad \mathlarger{\sum}_{L \in \FL_{1,L'^{\flat}}^{\delta'}} D_{n,h}(L)-\mathlarger{\sum}_{\delta' \leq \delta-1} \quad \mathlarger{\sum}_{L \in \FL_{2,L'^{\flat}}^{\delta'}} D_{n,h}(L)=0,
		\end{equation}
		for all $\delta>\delta(y)=a_{n-1}+2$ and $L^{\flat} \subset L'^{\flat}$.
		
		Since $\delta > a_{n-1}+2$, we have that for $\delta' \leq 2$, we have
		\begin{equation*}
			\FL_{1,L'^{\flat}}^{\delta'}=\FL_{2,L'^{\flat}}^{\delta'}=\lbrace L \in \FL \mid L \subset L^{\vee}, \delta_L=\delta', y \in L , L \cap L^{\flat}_F=L'^{\flat}\rbrace.
		\end{equation*}
		
		Therefore, \eqref{eq9.1.6.1} equals to
		\begin{equation*}
			\mathlarger{\sum}_{ \delta'=2}^{\delta} \quad \mathlarger{\sum}_{L \in \FL_{1,L'^{\flat}}^{\delta'}} D_{n,h}(L)-\mathlarger{\sum}_{\delta'=2}^{\delta-1} \quad \mathlarger{\sum}_{L \in \FL_{2,L'^{\flat}}^{\delta'}} D_{n,h}(L)=0.
		\end{equation*}
		
		Also, the automorphism of $\FC$ sending $(L'^{\flat},\delta',\epsilon)$ to $(L'^{\flat},\delta'-1,\epsilon(\pi \alpha \cdot))$, where $\alpha \in O_F^{\times}$, $f_1=\alpha f_2$, induces a bijection from $\FL_{1,L'^{\flat}}^{\delta'}$ to $\FL_{2,L'^{\flat}}^{\delta'-1}$. Therefore, it suffices to show that
		\begin{equation}\label{eq9.1.7.1}
			\mathlarger{\sum}_{L \in \FL_{1,L'^{\flat}}^{2}} D_{n,h}(L)=1_{L'^{\flat}}(y) \mathlarger{\sum}_{L \subset L^{\vee}, L \cap L_{F}^{\flat}=L'^{\flat}, \delta_L=2} D_{n,h}(L)=0.
		\end{equation}
		
		Note that by Theorem \ref{theorem9.1.1} (1), we have that for $\val((x,x)) \leq -2$, 
		\begin{align*}
			\widehat{\pDen}_{L'^{\flat\circ}}^{n,h}(x)	&=\mathlarger{\sum}_{L'^{\flat} \subset L \subset L^{\vee}, L \cap L^{\flat}_F=L'^{\flat}, x \in L^{\vee}} D_{n,h}(L)\vol(L)\\
			&=\vol(L'^{\flat})\mathlarger{\sum}_{L'^{\flat} \subset L \subset L^{\vee}, L \cap L^{\flat}_F=L'^{\flat}, x \in L^{\vee}} D_{n,h}(L)\vol(\text{Pr}_{\perp}(L))\\
			&=0.
		\end{align*}
		Now, choose $x$ to be generators of $\pi^{-2}((L_F^{\flat})^{\perp})^{\geq 0}$ and $\pi^{-3}((L_F^{\flat})^{\perp})^{\geq 0}$ and then take the difference. Then, we have that
		\begin{equation*}
			\mathlarger{\sum}_{L \subset L^{\vee}, L \cap L_{F}^{\flat}=L'^{\flat}, \delta_L=2} D_{n,h}(L)=0.
		\end{equation*}
		This shows that \eqref{eq9.1.7.1} holds which finishes the proof of the proposition.
	\end{proof}

	\begin{theorem}\label{theorem9.1.3}
		Assume that Conjecture \ref{conjecture9.2.1} above holds for $\CN^{[h]}_n$ and Conjecture \ref{conjecture5.5} holds for $\CZ$-cycles in $\CN^{[h-1]}_{n-1}$. Then, Conjecture \ref{conjecture5.5} holds for $\CZ$-cycles in $\CN^{[h]}_{n}$.
	\end{theorem}
	\begin{proof}
		As in \cite[section 8.2]{LZ}, we will prove this inductively. Let $L^{\flat} \subset \BV$ be a rank $n-1$ lattice such that $L^{\flat}_F$ is non-degenerate, and let $x \in \BV \backslash L^{\flat}_F$. By definition of $\partial \mathrm{Den}_{L^{\flat},\ScV}^{n,h}(x)$ and $\partial \mathrm{Den}_{L^{\flat\circ},\ScV}^{n,h}(x)$, it suffices to show that
		\begin{equation}\label{eq9.1.6}
			\partial \mathrm{Den}_{L^{\flat},\ScV}^{n,h}(x)=\mathrm{Int}_{L^{\flat},\ScV}(x),
		\end{equation}
		or equivalently
		\begin{equation}\label{eq9.1.7}
			\partial \mathrm{Den}_{L'^{\flat\circ},\ScV}^{n,h}(x)=\mathrm{Int}_{L'^{\flat\circ},\ScV}(x),
		\end{equation}
		for all $L^{\flat} \subset L'^{\flat} \subset L'^{\flat\vee}$.
		
		Now, assume that \eqref{eq9.1.6} holds for $L''^{\flat}$ such that $\val(L''^{\flat}) < \val(L^{\flat})$.
		
		Let $(a_1,a_2,\dots,a_{n-1})$ ($0 \leq a_1 \leq \dots \leq a_{n-1}$) be the fundamental invariants of $L^{\flat}$. Let $M=M(L^{\flat})=L^{\flat} \obot \langle u \rangle$ for some $u \in \BV$ such that $\val((u,u))=a_n\coloneqq a_{n-1}$ or $a_{n-1}+1$, so that
		\begin{equation*}
			a_1+a_2+\dots+a_n \equiv h+1 \pmod 2.
		\end{equation*}
		Now, assume that $(a_1',a_2',\dots, a_n')$ be the fundamental invariants of the lattice $L^{\flat}+\langle x \rangle$ with a basis $(e_1',\dots,e_n')$ such that $\val(e_i',e_i')=a_i'$. Let $L''^{\flat}=\langle e_1',\dots,e_{n-1}'\rangle$ ane let $x'=e_n'$. Then, we have that
		\begin{equation*}
			\begin{array}{ll}
				\mathrm{Int}_{L^{\flat}}(x)=\mathrm{Int}_{L''^{\flat}}(x), & \partial \mathrm{Den}_{L^{\flat}}^{n,h}(x)=\partial \mathrm{Den}_{L''^{\flat}}^{n,h}(x).
			\end{array}
		\end{equation*}
		By \cite[Lemma 8.2.2]{LZ}, if $x \notin M$, then $\val(L''^{\flat})<\val(L^{\flat})$, and hence by the inductive hypothesis, we have that
		\begin{equation*}
			\mathrm{Int}_{L''^{\flat}}(x)=\partial \mathrm{Den}_{L''^{\flat}}^{n,h}(x).
		\end{equation*}
		This implies that the support of
		\begin{equation*}
			\phi=	\mathrm{Int}_{L^{\flat},\ScV}(x)-\partial \mathrm{Den}_{L^{\flat},\ScV}^{n,h}(x) \in C_{c}^{\infty}(\BV)
		\end{equation*}
		is contained in the lattice $M$. Here, $\mathrm{Int}_{L^{\flat},\ScV}(x)-\partial \mathrm{Den}_{L^{\flat},\ScV}^{n,h}(x)$ is in $C_{c}^{\infty}(\BV)$ by Proposition \ref{proposition9.1.3}, Theorem \ref{theorem9.2.1}, \cite[Lemma 6.2.1]{LZ}, and \cite[Lemma 2.11]{San3}.
		
		Now, let us consider $x$ such that $\val((x,x)) <0$ and $x \perp L^{\flat}$. Since we have assumed that Conjecture \ref{conjecture9.2.1} below holds, by \cite[Lemma 6.3.1]{LZ} and \cite[Theorem 8.1]{ZhiyuZhang}, we have that
		\begin{equation*}
			\widehat{\chi(\CN^{[h]}_n, \text{}^{\BL}\CZ(L^{\flat})_{\ScV} \otimes^{\BL} O_{\CZ(x)})}=\widehat{\Int}_{L^{\flat},\ScV}(x)=-\frac{1}{q^h}\mathrm{Int}_{L^\flat,\sV}(\cY(x))\coloneqq-\dfrac{1}{q^h}\chi(\CN^{[h]}_n, \text{}^{\BL}\CZ(L^{\flat})_{\ScV} \otimes^{\BL} O_{\CY(x)}).
		\end{equation*}
		From now on, we use $\mathrm{Int}_{L^\flat,\sV}(\cY(x))$ (resp. $\mathrm{Int}_{L^{\flat\circ},\ScV}(\CY(x))$) to denote $\chi(\CN^{[h]}_n, \text{}^{\BL}\CZ(L^{\flat})_{\ScV} \otimes^{\BL} O_{\CY(x)})$ (resp. $\chi(\CN^{[h]}_n, \text{}^{\BL}\CZ(L^{\flat})^{\circ}_{\ScV} \otimes^{\BL} O_{\CY(x)})$).

		Furthermore, note that we can decompose this into primitive parts
		\begin{align*}
			-\dfrac{1}{q^h}\mathrm{Int}_{L^{\flat},\ScV}(\CY(x)))&=\mathlarger{\sum}_{L^{\flat} \subset L'^{\flat} \subset L'^{\flat\vee}}-\dfrac{1}{q^h}\mathrm{Int}_{L'^{\flat\circ},\ScV}(\CY(x))\\
			&=\mathlarger{\sum}_{L^{\flat} \subset L'^{\flat} \subset L'^{\flat\vee}} \widehat{\Int}_{L'^{\flat\circ},\ScV}(x).
		\end{align*}
		Now, let us compare this with the analytic side.
		
		When $\val((x,x)) \leq -2$, we have that $\CY(x)$ is empty. Therefore, combining this with Theorem \ref{theorem9.1.1}, we have that
		\begin{equation*}
			\widehat{\phi}(x)=\widehat{\Int}_{L^{\flat},\ScV}(x)-\widehat{\pDen}_{L^{\flat},\ScV}^{n,h}(x) =0-0=0.
		\end{equation*}
		
		When $\val((x,x))=-1$, by Proposition \ref{proposition2.5}, we have that
		\begin{align*}
			\widehat{\Int}_{L'^{\flat\circ},\ScV}(x)=-\dfrac{1}{q^h}\mathrm{Int}_{L^{\flat\circ},\ScV}(\CY(x))&=-\dfrac{1}{q^h}\chi(\CN^{[h]}_n, \text{}^{\BL}\CZ(L'^{\flat})^{
				\circ}_{\ScV} \otimes^{\BL} O_{\CY(x)})\\
			&=-\dfrac{1}{q^h}\chi(\CN^{[h-1]}_{n-1}, \text{}^{\BL}\CZ(L'^{\flat})^{
				\circ}_{\ScV})\\
			&=-\dfrac{1}{q^h}\chi(\CN^{[h-1]}_{n-1}, \text{}^{\BL}\CZ(L'^{\flat})^{
				\circ}-\text{}^{\BL}\CZ(L'^{\flat})^{
				\circ}_{\ScH}).
		\end{align*}
		Since we have assumed that Conjecture \ref{conjecture5.5} holds for $\CN^{[h-1]}_{n-1}$, we have that
		\begin{equation*}
			-\dfrac{1}{q^h}\chi(\CN^{[h-1]}_{n-1}, \text{}^{\BL}\CZ(L'^{\flat})^{
				\circ})=-\dfrac{1}{q^h}D_{n-1,h-1}(a,b,c),
		\end{equation*}
		where $(a,b,c)=(t_{\ge 2}(\lambda),t_1(\lambda),t_0(\lambda))$ and $\lambda \in \CR_{n-1}^{0+}$ is the fundamental invariants of $L'^{\flat}$.
		
		For $\chi(\CN^{[h-1]}_{n-1},\text{}^{\BL}\CZ(L'^{\flat})^{
			\circ}_{\ScH})$, we have the following two cases.
		
		First, assume that $(a,b,c)\neq (1,h-2,n-h)$. Then, by Theorem \ref{thm: hori part} (or see the proof of Theorem \ref{theorem9.1.1}), we have that $\CZ(L'^{\flat})_{\ScH}^{\circ}$ is empty or the sum of $\CZ(L)^{\circ} \cdot \CY(N)$ for some lattice $N \simeq (\pi^{-1})^h$. Since $\CY(N)$ is empty in $\CN^{[h-1]}_{n-1}$, we have that $\chi(\CN^{[h-1]}_{n-1},\text{}^{\BL}\CZ(L'^{\flat})^{
			\circ}_{\ScH})=0$.
		
		Now, assume that $(a,b,c)=(1,h-2,n-h)$. Then by Theorem \ref{thm: hori part} (or see the proof of Theorem \ref{theorem9.1.1}), we have that
		\begin{equation*}
			\CZ(L'^{\flat})^{\circ}_{\ScH}=\CZ(L_2)^{\circ}\cdot\CY(\pi^{-1}L_1)\cdot\CZ(L_0),
		\end{equation*}
		where $L_2 \simeq \pi^{\lambda}$ ($\lambda \geq 2$), $\pi^{-1}L_1 \simeq \pi^{-1}I_{h-2}$, and $L_0 \simeq I_{n-h}$. Therefore,
		\begin{equation*}
			\chi(\CN^{[h-1]}_{n-1},\text{}^{\BL}\CZ(L'^{\flat})^{
				\circ}_{\ScH})=\chi(\CN^{[1]}_{1},\text{}^{\BL}\CZ(L_2)^{
				\circ})=D_{1,1}(1,0,0)=1.
		\end{equation*}
		
		Combining these, we have that
		\begin{equation*}
			\widehat{\Int}_{L'^{\flat\circ},\ScV}(x)=\left\lbrace \begin{array}{ll}
				-\dfrac{1}{q^h}D_{n-1,h-1}(a,b,c) & \text{if }(a,b,c)\neq (1,h-2,n-h),\\
				-\dfrac{1}{q^h}D_{n-1,h-1}(a,b,c)+\dfrac{1}{q^h} & \text{if }(a,b,c)= (1,h-2,n-h).
			\end{array}\right.
		\end{equation*}
		Therefore, by Theorem \ref{theorem9.1.1}, we have that
		\begin{equation*}
			\widehat{\Int}_{L'{\flat},\ScV}(x)-\widehat{\pDen}_{L'^{\flat},\ScV}^{n,h}(x)=0 \text{ for }\val((x,x))< 0.
		\end{equation*}
		This implies that $\widehat{\phi}=0$ for $\val((x,x))< 0$.
		
		Now, we only need to follow the proof of \cite[Theorem 8.2.1]{LZ}.
		
		Since $\phi$ is invariant under $L^{\flat}$ and $\text{Supp}(\phi) \subset M$, we have that
		\begin{equation*}
			\phi=1_{L^{\flat}} \otimes \phi_{\perp},
		\end{equation*}
		where $\phi_{\perp} \in C_{c}^{\infty}((L^{\flat}_F)^{\perp})$ and $\phi_{\perp}$ is supported on $M_{\perp}=\langle u \rangle$. Then, we have
		\begin{equation*}
			\widehat{\phi}=\vol(L^{\flat})1_{L^{\flat\vee}} \otimes \widehat{\phi_{\perp}},
		\end{equation*}
		and $\widehat{\phi_{\perp}}$ is invariant under the translation by $\langle u^{\vee} \rangle=\langle \pi^{-a_n}u \rangle$. Since $\val((u^{\vee},u^{\vee}))=-a_n <0$ and $\widehat{\phi_{\perp}}=0$ for every $x$ such that $x \perp L^{\flat}$, $\val((x,x))<0$, we have that $ 
		\widehat{\phi_{\perp}}$ vanishes identically and hence $\phi=0$. This finishes the proof of the Theorem.
	\end{proof}
	
	\begin{theorem}(cf. \cite[Theorem 8.2.1, Theorem 10.5.1]{LZ}) \label{theorem9.1.4}
		Conjecture \ref{conjecture5.5} holds for $\CZ$-cycles in $\CN^{[0]}_{n}, \CN^{[1]}_{n},\CN^{[n-1]}_{n}, \CN^{[n]}_{n}$, and $\CN^{[2]}_{4}$.
	\end{theorem}
	\begin{proof}
		This follows from Theorem \ref{theorem9.1.3} and Theorem \ref{theorem9.2.1}.
	\end{proof}
	\begin{remark}
		We remark that the Kudla-Rapoport conjecture for $\CN^{[0]}_n$, $\CN^{[n]}_n$ (good reductions case), $\CN^{[1]}_n$ (almost self-dual case) is already proved in \cite{LZ}. Therefore, our new cases are $\CN^{[n-1]}_n$ and $\CN^{[2]}_4$. We also remark that our work gives a different proof of the conjecture for $\CN^{[1]}_n$.
	\end{remark}

	\bibliographystyle{plain}

\end{document}